\documentclass[10pt]{ociamthesis}  

\usepackage{silence}
    \usepackage[utf8]{inputenc}
    \usepackage[T1]{fontenc}
    \usepackage{siunitx}
    \usepackage{colonequals}
    \usepackage[super]{nth}
    \usepackage{algorithm}
    \usepackage{amsmath}
    \usepackage{amsthm}
    \usepackage{amssymb}
    \usepackage{tikz}
    \usepackage{url}
    \usepackage{float}
    \usepackage{tabularx}
    \usepackage{pbox}
    \usepackage{tabulary}
    \usepackage{breqn}
    \usepackage{longtable}
    \usepackage{mdframed}
    \usepackage{booktabs}

    \usepackage{graphicx}
    \graphicspath{{./images/}{./}}

    \usepackage{wrapfig}
    \usepackage[font={small}]{caption}
    \usepackage{enumitem}
    \usepackage{physics}
    \usepackage{lipsum}
    \usepackage{mwe}
    \usepackage{breqn}
    \usepackage[noend]{algpseudocode}
    \usepackage{subcaption}
    \usepackage{graphicx}
    \usepackage{enumitem}
    \usepackage{bm}
    \usepackage{textcmds}

    \usepackage{apptools}
    \AtAppendix{\counterwithin{lemma}{section}}
    \AtAppendix{\counterwithin{remark}{section}}
    
    \usepackage{hyperref}
    \hypersetup{
        colorlinks,
            citecolor=black,
        filecolor=black,
        linkcolor=black,
        urlcolor=black
    }

\newtheorem{corollary}{Corollary}

\newtheorem{theorem}{Theorem}
\newtheorem{lemma}{Lemma}
\newtheorem{remark}{Remark}

\newtheorem{definition}{Definition}
\newtheorem{assumption}{Assumption}
\newtheorem{example}{Example}

\numberwithin{equation}{section} 
\numberwithin{lemma}{section} 
\numberwithin{theorem}{section} 
\numberwithin{definition}{section} 
\numberwithin{corollary}{section} 

\DeclareMathOperator*{\argmin}{arg\,min}

\newcommand{\LH}{L_H}
\newcommand{\kappaS}{\kappa_S}
\newcommand{\Bepsilon}{\bar{\epsilon}}
\newcommand{\logOneOverDeltaS}{\log \bracket{\frac{1}{\deltaS}}}

\newcommand{\nPreFactorTRCubic}{\left[\frac{7}{8}(1-\delta_1) - 1 + \frac{c}{(c+1)^2} \right]^{-1}}
\newcommand{\deltaSThreeExpression}{\pow{e}{-\frac{l(\epSTwo)^2}{C_l} + r +1 }}
\newcommand{\newL}{\tau_{\alpha}}

\newcommand{\epsOne}{\epS}
\newcommand{\nThreeEpH}{N_{\epH}^{(3)}}

\newcommand{\nTwoEpsH}{N_{\epH}^{(2)}}

\newcommand{\epH}{\epsilon_{H}}
\newcommand{\epsHalf}{\epsilon^{\frac{1}{2} }  }
\newcommand{\oneMinusEpSHalf}{(1-\epS)^{\frac{1}{2}}}
\newcommand{\NsKTSKHat}{\normTwo{S_k^T \sKHat}}
\newcommand{\deltaSThree}{\deltaS^{(3)}}

\newcommand{\fZeroMinusfStarOverH}{\frac{f(x_0) - f^*}{h(\epsilon, \alphaZero\gammaOne^{c+\newL}
)}}
\newcommand{\ntsAlphaZeroGammaOneCL}{N_{TS, \underline{\alphaZero \gammaOne^{c+\newL}}}}
\newcommand{\fun}[2]{#1\bracket{#2}}
\newcommand{\pow}[2]{#1^{#2}}
\newcommand{\impliesSince}[1]{\overset{#1}{\implies}}
\newcommand{\hK}{\grad^2 f(x_k)}

\newcommand{\qKHat}[1]{\hat{q}_k \bracket{#1}}
\newcommand{\epSOne}{\epsilon_S^{(1)}}
\newcommand{\epSTwo}{\epsilon_S^{(2)}}
\newcommand{\MK}{M_k}
\newcommand{\hessFK}{\grad^2 f(x_k)}
\newcommand{\SK}{S_k}

\newcommand{\oneMinusEpSTwo}{1 - \epSTwo}

\newcommand{\sqrtFrac}[2]{\sqrt{\frac{#1}{#2}}}
\newcommand{\gDeltaSDeltaOne}{g(\deltaS, \deltaOne)}
\newcommand{\mathOmega}[1]{\Omega \left( #1 \right) }
\newcommand{\mathO}[1]{\mathcal{O}\left( #1\right)}
\newcommand{\ra}{\rangle}
\newcommand{\la}{\langle}
\newcommand{\norms}[1]{\|#1\|_2^2}
\newcommand{\Z}{\mathbb{Z}} 
\newcommand{\N}{\mathbb{N}}
\newcommand{\R}{\mathbb{R}}
\newcommand{\E}{\mathbb{E}}
\renewcommand{\P}{\mathbb{P}}
\newcommand{\dotp}[2]{\langle #1, #2 \rangle}
\newcommand{\alphaLow}{\alpha_{low}}
\newcommand{\nEps}{N_{\epsilon}}
\newcommand{\oneOverGammaOneC}{\frac{1}{\gamma_1^c}}
\newcommand{\nt}{N_T}
\newcommand{\nts}{N_{TS}}
\newcommand{\ns}{N_S}
\newcommand{\nuMe}{N_U}
\newcommand{\ntu}{N_{TU}}
\newcommand{\ntAlphaUpper}{N_{T, \overline{\alphaLowOne}}}
\newcommand{\nsAlphaUpper}{N_{S, \overline{\alphaLowOne}}}
\newcommand{\ntAlphaLower}{N_{T, \underline{\alphaLowOne}}}
\newcommand{\ntsAlphaLower}{N_{TS, \underline{\alphaLowOne}}}
\newcommand{\ntuAlphaLower}{N_{TU, \underline{\alphaLowOne}}}
\newcommand{\nuAlphaLower}{N_{U, \underline{\alphaLowOne}}}
\newcommand{\nsGammaAlpha}{N_{S, \underline{ \gamma_1^{c}\alphaLowOne}}}
\newcommand{\ntsGammaCAlpha}{N_{TS, \underline{ \gamma_1^c \alphaLowOne}}}
\newcommand{\nfsGammaAlpha}{N_{FS, \underline{ \gamma_1^c \alphaLowOne}}}
\newcommand{\nf}{N_F}

\newcommand{\logOneOverGammaOne}{\log(1/\gamma_1)}

\newcommand{\logGammaTwoOverGammaOne}{\log(\gamma_2/\gamma_1)}
\newcommand{\tCOne}{\newL}
\newcommand{\tCTwo}{c}

\newcommand{\expectation}[1]{\E \left[ #1 \right]}
\newcommand{\sumTk}{\sum_{k=0}^{N-1} T_k}
\newcommand{\eToMinusLambdaMinusOne}{e^{-\lambda}-1}
\newcommand{\oneMinusDeltaS}{1-\delta_S}
\newcommand{\conditionalE}[2]{\E \left[ #1 | #2 \right]}
\newcommand{\eToMinusLambdaSumTk}{e^{-\lambda \sumTk }}
\newcommand{\eToMinusLambdaSumTkNMinusTwo}{e^{-\lambda \sumTkToNMinusTwo }}
\newcommand{\exponentialMomentUpperTotal}{\left[ e^{( \eToMinusLambdaMinusOne)(\oneMinusDeltaS)}\right]^N}
\newcommand{\exponentialMomentUpper}{ e^{( \eToMinusLambdaMinusOne)(\oneMinusDeltaS)}}
\newcommand{\eToMinusLambda}[1]{e^{-\lambda #1}}
\newcommand{\conditionalP}[2]{\P \left[ #1 | #2 \right]}
\newcommand{\sumTkToNMinusTwo}{\sum_{k=0}^{N-2}T_k}
\newcommand{\SZeroDotsToXNMinusOne}{T_0, T_1, \dots, T_{N-2}, x_{N-1}}
\newcommand{\deltaOneComplicated}{-\delta_1 - (1-\delta_1) \log(1-\delta_1) }
\newcommand{\chernoffLowerExponential}{e^{-\frac{\delta_1^2}{2} (1-\delta_S) N}}

\newcommand{\nPreFactorTR}{\left[(1-\delta_S)(1-\delta_1) - 1 + \frac{c}{(c+1)^2} \right]^{-1}}
\newcommand{\betaK}{\beta_k}
\newcommand{\kStartOne}{k_{\small{start}}^{(1)}}
\newcommand{\kEndOne}{k_{\small{end}}^{(1)}}
\newcommand{\aOne}{A^{(1)}}
\newcommand{\closedInterval}[2]{\left[ #1, #2 \right]}
\newcommand{\twoCases}[4]{         \begin{cases}
             #1, & #2 \\
             #3, & #4.
        \end{cases} }
\newcommand{\closedOpenInterval}[2]{\left[ #1, #2 \right)}
\newcommand{\openClosedInterval}[2]{\left( #1, #2\right]}
\newcommand{\openInterval}[2]{\left( #1, #2\right)}
\newcommand{\mOneOne}{M_1^{(1)}}
\newcommand{\mTwoOne}{M_2^{(1)}}
\newcommand{\nOneOne}{n_1^{(1)}}
\newcommand{\nTwoOne}{n_2^{(1)}}
\newcommand{\kInStartEndOne}{k \in \openInterval{\kStartOne}{\kEndOne}}
\newcommand{\bOne}{B^{(1)}}
\newcommand{\kStartTwo}{k_{\small{start}}^{(2)}}
\newcommand{\kEndTwo}{k_{\small{end}}^{(2)}}
\newcommand{\kStartThree}{k_{\small{start}}^{(3)}}
\newcommand{\kStartI}{k_{\small{start}}^{(i)}}
\newcommand{\kStartIPlusOne}{k_{\small{start}}^{(i+1)}}
\newcommand{\kEndI}{k_{\small{end}}^{(i)}}
\newcommand{\aTwo}{A^{(2)}}
\newcommand{\mOneTwo}{M_1^{(2)}}
\newcommand{\mTwoTwo}{M_2^{(2)}}
\newcommand{\nOneTwo}{n_1^{(2)}}
\newcommand{\nTwoTwo}{n_2^{(2)}}

\newcommand{\bTwo}{B^{(2)}}

\newcommand{\kBar}{\overline{k}}
\newcommand{\onePlusTCTwo}{1+ \tCTwo}

\newcommand{\sHat}{\hat{s}}
\newcommand{\mKHat}[1]{\hat{m}_k\left( #1\right)}
\newcommand{\xK}{x_k}
\newcommand{\xKPlusOne}{x_{k+1}}
\newcommand{\sK}{s_k}
\newcommand{\normTwo}[1]{\left\lVert#1\right\rVert_2}
\newcommand{\gradFK}{\grad f(\xK)}
\newcommand{\DOneOverNMinusDTwo}{\frac{D_1}{N-D_2}}
\newcommand{\qInverseOfDOneOverNMinusDTwo}{q^{-1} \left( \DOneOverNMinusDTwo\right)}
\newcommand{\minGradient}[1]{\min_{k\leq #1} \normTwo{\gradFK}}
\newcommand{\minGradientN}{\minGradient{N}}

\newcommand{\DOneOverQPlusDTwo}{\frac{D_1}{q(\epsilon)}+D_2}

\newcommand{\logOneOverDelta}{\log\left( \frac{1}{\delta}\right)}
\newcommand{\eToMinusDThreeN}{e^{-D_3 N}}
\newcommand{\gammaOneC}{\gamma_1^c}
\newcommand{\hatNOneOne}{\hat{n}_1^{(1)}}
\newcommand{\hatNOneTwo}{\hat{n}_2^{(1)}}
\newcommand{\fK}{f(x_k)}
\newcommand{\innerProduct}[2]{\langle #1, #2 \rangle}
\newcommand{\sKGradFK}{S_k \gradFK}
\newcommand{\sKBKSKT}{S_k B_k S_k^T}
\newcommand{\rLTimesD}{\R^{l\times d}}
\newcommand{\fKPlusOne}{f(x_k + s_k)}
\newcommand{\epS}{\epsilon_S}
\newcommand{\sMax}{S_{\small{max}}}
\newcommand{\alphaMax}{\alpha_{\small{max}}}
\newcommand{\alphaK}{\alpha_k}
\newcommand{\hBar}{\bar{h}}
\newcommand{\deltaSOne}{\delta_S^{(1)}}
\newcommand{\deltaSTwo}{\delta_S^{(2)}}
\newcommand{\aKOne}{A_k^{(1)}}
\newcommand{\aKTwo}{A_k^{(2)}}
\newcommand{\intersect}{\cap}
\newcommand{\barXK}{\bar{x}_k}
\newcommand{\probabilityGivenXK}[1]{\probability{#1 | x_k = \barXK}}
\newcommand{\aKOneIntersectaKTwo}{\aKOne \intersect \aKTwo}
\renewcommand{\complement}[1]{\left( #1\right)^c}
\newcommand{\oneMiusEpsSToHalf}{\left( 1-\epS\right)^{1/2}}

\renewcommand{\abs}[1]{\lvert #1\rvert}

\newcommand{\alphaKPlusOne}{\alpha_{k+1}}
\newcommand{\gammaOne}{\gamma_1}
\newcommand{\gammaTwo}{\gamma_2}

\newcommand{\alphaZero}{\alpha_0}
\newcommand{\alphaMin}{\alpha_{\small{min}}}

\newcommand{\ceil}[1]{\left \lceil{#1}\right \rceil}
\newcommand{\logBaseGammaOne}[1]{\log_{\gammaOne}\left( #1\right)}
\newcommand{\minMe}[2]{\min \set{#1, #2}}
\newcommand{\alphaLowOne}{\alphaMin}
\newcommand{\pOneDef}{\frac{\log(\gammaTwo\alphaMin/\alphaZero)}{\logGammaTwoOverGammaOne}}

\newcommand{\gammaTwoExpression}{ \frac{1}{\gamma_1^c}}
\newcommand{\betaKPlusOne}{\beta_{k+1}}
\newcommand\muAOneHashing{C_1 \sqrt{C_2} \sqrt{\epsilon}  \min \left\{\frac{\log(E/ (C_2 \epsilon))}{4r + \log(1/\delta)}, \sqrt{\frac{\log(E)}{4r + \log(1/\delta)}} \right\} }
\newcommand\logDeltaI{\log(1/\delta)}
\newcommand\fourRplusLog{\left[ 4r + \logDeltaI\right]}
\newcommand{\nuOneHashing}{C_1 \sqrt{\epsilon} \min \left\{ \frac{\log(E/\epsilon)}{\log(1/\delta)}, \sqrt{\frac{\log(E)}{\log(1/\delta)}} \right\}}
 
\newcommand{\muASHashing}{\sqrt{s} C_{\nu} C_1^{-1} \muBarEpsDelta}
\newcommand{\muASHashingVariant}{\sqrt{s}\muBarEpsDelta}
\newcommand{\EUpperSHashing}{C_2^2 \epsilon^2 s[4r + \logDeltaI]^{-1} e^{C_s (C_2 \epsilon s)^{-1} \left[ 4r+\logDeltaI \right]}}
\newcommand{\muLower}{ \sqrt{r/n} + \sqrt{8 \log(n/\delta_1)/n } }
\newcommand{\muLowerNoN}{ \sqrt{r} + \sqrt{8 \log(n/\delta_1) } }
\newcommand{\nLower}{ \frac{ \left( \muLowerNoN  \right)^2}{ 
sC_{\nu}^2C_1^{-2}\bar{\mu}(\epsilon,\delta)^2 }}
\newcommand{\nLowerSVariant}{ \frac{ \left( \muLowerNoN  \right)^2}{ s\bar{\mu}(\epsilon,\delta)^2 } }
\newcommand{\mLower}{E C_2^{-2} \epsilon^{-2} \fourRplusLog}

\newcommand{\epsilonPrimeFactor}{\frac{(1-\gamma)(1-\gamma^2)}{1+2\gamma-\gamma^2}}
\renewcommand{\k}{p}
\newcommand{\set}[1]{\left\{ #1 \right\}}
\newcommand{\yPlus}{Y_+}
\newcommand{\yMinus}{Y_-}
\newcommand{\nPlus}{N_+}
\newcommand{\nMinus}{N_-}
\newcommand{\yPlusExpression}{\set{y_1 + y_2: y_1,y_2 \in Y}}
\newcommand{\yMinusExpression}{\set{y_1 - y_2: y_1,y_2 \in Y}}
\newcommand{\nPlusExpression}{\set{y_i + y_j: i,j \in \left[1,|N|\right]}}
\newcommand{\nMinusExpression}{\set{y_i - y_j: i,j  \in \left[1,|N|\right]}}
\newcommand{\union}[2]{#1\cup#2}
\newcommand{\yExpression}{\set{y_1, y_2, \dots y_{|Y|}}}
\newcommand{\probability}[1]{\P \left( #1 \right)}
\renewcommand{\complement}[1]{\left( #1 \right)^c}
\newcommand{\squareBracket}[1]{\left[ #1 \right]}

\newcommand{\nMinusOneExpression}{\set{y_i-y_j: 1\leq i <j \leq |N|}}
\newcommand{\nMinusTwoExpression}{\set{y_i - y_j: 1\leq j <i \leq |N|}}
\newcommand{\nMinusOne}{N_-^{(1)}}
\newcommand{\nMinusTwo}{N_-^{(2)}}
\newcommand{\muBarEpsDelta}{\bar{\mu}(\epsilon,\delta)}
\newcommand{\EUpperHashing}{\frac{2 e^{4r}}{\left[ 4r + \logDeltaI \right]\delta }}
\newcommand{\muHatSEpsDelta}{\hat{\mu}(s,\epsilon,\delta)}

\newcommand{\bracket}[1]{\left( #1\right)}

\newcommand{\lkHatSHat}{\hat{l}_k \bracket{\sHat}}
\newcommand{\mkHatSHat}{\mKHat{\sHat}}
\newcommand{\oneOverTwoAlphaK}{\frac{1}{2\alphaK}}
\newcommand{\oneOverAlphaK}{\frac{1}{\alphaK}}
\newcommand{\SKTransposed}{S_k^T}
\newcommand{\lkHat}[1]{\hat{l}_k \bracket{#1}}
\newcommand{\BMax}{B_{max}}
\newcommand{\sKHat}{\hat{s}_k}
\newcommand{\kappaT}{\kappa_T}
\newcommand{\barH}[2]{\bar{h} \bracket{#1,#2}}

\newcommand{\halfBMax}{\frac{1}{2}\BMax}

\newcommand{\SKTransposedsKHat}{S_k^T \sKHat}
\newcommand{\lPlusHalfBmax}{L+\halfBMax}

\newcommand{\deltaOne}{\delta_1}

\newcommand{\eToQRySqaured}{e^{q \normTwo{Ry}^2}}
\newcommand{\eToQLOneMinusEps}{e^{ql(1-\epS)}}
\newcommand{\eToMinusQLOneMinusEps}{e^{-ql(1-\epS)}}
\newcommand{\texteq}[1]{\text{\quad #1}}
\newcommand{\lambdaMin}[1]{\lambda_{min} \bracket{#1}}

\newcommand{\zK}{z^{k}}
\newcommand{\aZeroOne}{A_0^{(1)}}
\newcommand{\aZeroTwo}{A_0^{(2)}}

\newcommand{\normInf}[1]{\|#1\|_{\infty}}

\newcommand{\logOneOverDeltaSTwo}{\log\bracket{1/\deltaSTwo}}
\newcommand{\mkHatSkHat}{\hat{m}_k\bracket{\hat{s}_k}}
\newcommand{\gK}{g_k}
\newcommand{\mK}[1]{m_k \bracket{#1}}
\newcommand{\deltaS}{\delta_S}
\newcommand{\integral}[3]{\int_{#1}^{#2} #3}
\newcommand{\probNEpsGrN}{\probability{\nEps >M}dM}

\newcommand{\reply}[1]{#1}

\newcommand{\CorTwoDeltaOne}{Suppose \eqref{eqn::deltaSConditionThmTwo} hold and let $\delta_1 \in (0,1)$ satisfy \eqref{eqn:tmp32}}
\newcommand{\IterKTrueandSuccssfulWithAlphaKGeqAlphaMin}{\substack{ \text{Iteration $k$ is true and successful} \\ \text{with $\alphaK \geq \alphaZero \gammaOne^{c+\newL}$}}}

\newcommand{\constantsDescription}{problem-independent constants}
\newcommand{\whereMuBarIsDefined}{where $\muBarEpsDelta$ is defined in \eqref{A:mu-1}}
\newcommand{\theoremFiveFirstSentence}{Suppose that $\epsilon,\delta \in (0,1)$, $r \leq d \leq n, m\leq n \in \N$, $E >0$ satisfy }
\newcommand{\theoremThreeFirstSentence}{Let  $C_1, C_2, C_3, C_M, C_{\nu}, C_s >0$ be \constantsDescription. Suppose that $\epsilon,\delta \in (0, C_3)$, 
$m,\,s\in \N^+$ and $E>0$ satisfy }

\newcommand{\anSHashingMat}{an $s$-hashing matrix}
\newcommand{\anSHashingVariantMat}{an $s$-hashing variant matrix}
\newcommand{\refAlgOne}{Algorithm \ref{alg1}}
\newcommand{\solverName}{Ski-LLS}
\newcommand{\solverNameDense}{Ski-LLS-dense}
\newcommand{\solverNameSparse}{Ski-LLS-sparse}
\newcommand{\calibrationDenseCaptionSentenceOne}[1]{Runtime of #1 on dense matrices $A \in \R^{n \times d}$ from Test Set 1
with $n=50000, d=4000$ and $n=50000, d=7000$ and different values of $\gamma=m/d$} 
\newcommand{\calibrationDenseCaptionSentenceTwo}[1]{For each plot, #1 is run three times on (the same) randomly generated $A$. We see that the runtime has low variance despite the randomness in the solver} 
\newcommand{\calibrationDenseCaptionSentenceThree}[1]{We choose #1 to approximately minimize the runtime across the above plots}
\newcommand{\performanceProfileCaption}[1]{Performance profile comparison of \solverName{} with LSRN, LS\_HSL and LS\_SPQR for all matrices $A\in\R^{n\times d}$ in the Florida matrix collection with #1}
\newcommand{\defDenseTestMatrices}{\eqref{eq::A_co_dense}, \eqref{eq::A_semi_dense}, \eqref{eq::A_inco_dense}}

\newcommand{\calibrationSixFigures}[8]{
\begin{figure}[H]
    \centering
    \begin{minipage}{\mysize\textwidth}
        \centering
        \includegraphics[width=\textwidth]{#1} 
    \end{minipage}
    \begin{minipage}{\mysize\textwidth}
        \centering
        \includegraphics[width=\textwidth]{#2} 
    \end{minipage}
    \begin{minipage}{\mysize\textwidth}
        \centering
        \includegraphics[width=\textwidth]{#3} 
    \end{minipage}    
    \centering
    \begin{minipage}{\mysize\textwidth}
        \centering
        \includegraphics[width=\textwidth]{#4} 
    \end{minipage}
    \begin{minipage}{\mysize\textwidth}
        \centering
        \includegraphics[width=\textwidth]{#5} 
    \end{minipage}
    \begin{minipage}{\mysize\textwidth}
        \centering
        \includegraphics[width=\textwidth]{#6} 
    \end{minipage}        
    \caption{#7} \label{#8}
\end{figure}
}
\newcommand{\threeFigures}[9]{
\begin{figure}
    \begin{minipage}{\mysize\textwidth}
    \centering
        \includegraphics[width=\textwidth]{#1} 
        \caption{#2}
        \label{#3}
    \end{minipage}\hfill
    \begin{minipage}{\mysize\textwidth}
    \centering
        \includegraphics[width=\textwidth]{#4} 
        \caption{#5}
        \label{#6}
    \end{minipage}
    \begin{minipage}{\mysize\textwidth}
    \centering
        \includegraphics[width=\textwidth]{#7} 
        \caption{#8}
        \label{#9}
    \end{minipage}    
\end{figure}
}

\newcommand{\mathInTitle}[1]{\texorpdfstring{#1}{TEXT}}

\newcommand{\singleQuote}[1]{\lq #1\rq}

\title{On Random Embeddings and Their Application to Optimisation}   

\author{Zhen Shao}             
\college{St Anne's College}  

\degree{Doctor of Philosophy}     
\degreedate{Michaelmas 2021}         

\begin{document}
\ActivateWarningFilters[pdftoc] 

\baselineskip=18pt plus1pt

\setcounter{secnumdepth}{3}
\setcounter{tocdepth}{3}

\maketitle                  
\begin{acknowledgements}
I would like to thank my supervisor, Prof Coralia Cartis, for her patience, support and teaching over the last four years, without which this thesis would not be possible. 

I would also like to thank Chris Breward and Colin Please for making Industrially Focused Mathematical Modelling CDT possible, and the collaborations I had with Numerical Algorithm Group Ltd., in particular Dr Jan Fiala. 

Throughout my DPhil, I have been generously supported by my office mates, my friends in the Mathematical Institute and St Anne's College. Thank you all for being with me in this journey.

I am also grateful for all the teachings and support I received during my undergraduate years at Oxford. In particular, my tutors at Pembroke College. 

Finally, I would like to thank my parents, for raising me up and giving me the best environment for my education.
\end{acknowledgements}
\begin{abstract}
    
Random embeddings project high-dimensional spaces to low-dimensional ones; they are careful constructions which allow the approximate preservation of key properties, such as the pair-wise distances between points. Often in the field of optimisation, one needs to explore high-dimensional spaces representing the problem data or its parameters and thus the computational cost of solving an optimisation problem is connected to the size of the data/variables. This thesis studies the theoretical properties of norm-preserving random embeddings, and their application to several classes of optimisation problems. 

Our investigations into random projections present subspace embedding properties for $s$-hashing ensembles --- sparse random matrices with $s$ non-zero entries per column --- that are optimal in the projection dimension $m$ of the sketch, namely, $m = \mathcal{O}(d)$ where $d$ is the dimension of the subspace. A diverse set of results are presented that address the case when the input matrix has sufficiently low coherence; how the acceptable coherence changes with the number $s$ of non-zeros per column in the $s$-hashing matrices, or is reduced through suitable transformations. In particular, we propose a new random embedding, the Hashed Randomised Hadamard Transform,  that improves upon the Subsampled Randomised Hadamard Transform by replacing sub-sampling with hashing. 

We apply these sketching techniques to linear least squares problems,  as part of a Blendenpik-type algorithm, that uses a sketch of the data matrix to build a high quality preconditioner and then solves a preconditioned formulation of the original problem. We also include suitable linear algebra tools for rank-deficient and for sparse problems that lead to our implementation, Ski-LLS,  outperforming not only sketching-based routines on randomly-generated input, but also state of the art direct solver SPQR and iterative code HSL on certain subsets of the sparse Florida matrix collection; namely, on least squares problems that are significantly over-determined, or moderately sparse, or difficult. 

Instead of sketching in the data/observational space as in the linear least squares case above, we then consider sketching in the variable/parameter domain for a more generic problem and algorithm.
We propose a general random-subspace first-order framework for unconstrained non-convex optimisation that requires a weak probabilistic assumption on the subspace gradient, which we show to be satisfied by various random matrix ensembles, such as Gaussian and hashing sketching. We show that, when safeguarded with trust region or quadratic regularisation techniques, this random subspace approach satisfies, with high probability, a complexity bound of order $\mathO{\epsilon^{-2}}$ to drive the (full) gradient norm below $\epsilon$; matching in the accuracy order, deterministic counterparts of these methods and securing almost sure convergence. 
We then particularise this framework to random subspace Gauss-Newton  methods for nonlinear least squares problems, that only require the calculation of the Jacobian matrix in a subspace, with similar complexity guarantees.

We further investigate  second-order methods for non-convex optimisation, and propose a Random Subspace Adaptive Regularised Cubic (R-ARC) method, which we analyse under various assumptions on the objective function and the sketching matrices. We show that, when the sketching matrix achieves a subspace embedding of the augmented matrix of the gradient and the Hessian with sufficiently high probability, then the R-ARC method satisfies, with high probability, a complexity bound of order $\mathO{\epsilon^{-3/2}}$ to drive the (full) gradient norm below $\epsilon$; matching in the accuracy order  the deterministic counterpart (ARC). We also show that the same complexity bound is obtained when the Hessian matrix has sparse rows and appropriate sketching matrices are chosen. We also investigate R-ARC's convergence to  second order critical points. We show that the R-ARC method also drives the Hessian in the subspace to be approximately positive semi-definite with high probability, for a variety of sketching matrices; and furthermore if the Hessian matrix has low rank and scaled Gaussian sketching matrices are used, the R-ARC drives the (full) Hessian to be approximately positive semi-definite, with high probability, at the rate $\mathO{\epsilon^{-3}}$, again matching in the accuracy order its deterministic counterpart. 
\end{abstract}

\begin{romanpages}          
\tableofcontents            
\listoffigures              
\end{romanpages}            

\chapter{Introduction}\label{Ch1}

    \section{Background}
    This thesis is about random embeddings and their application to improving the efficiency of optimisation algorithms for different problem classes. 
In particular, regarding random embeddings, the novelty of our work is in the analysis of sparse projections and in proposing a new general random embedding with attractive theoretical properties and numerical performance. Then we transfer these results and existing random embeddings to improve optimisation algorithms for linear and non-linear least squares problems, as well as for general objectives using first and second order information.

Numerical optimisation designs algorithms that 
find an extreme value of a given function.
Such computational routines find numerous applications 
in data science, finance and machine learning. 
The computational cost of an optimisation algorithm typically 
grows with the dimension of the function being optimised, 
which in many applications increases as the data set 
becomes larger or the model for the data becomes more complex. 
For example, the classical computation of the solution of 
fitting a linear model to a data set (linear least squares) 
grows linearly with the size of the data set and quadratically 
with the number of variables in the model. 
Given the ever-increasing amount of data and complexity of models, 
recent research trends attempt to make  classical optimisation algorithms 
faster and more scalable, \cite{10.5555/1109557.1109682, Meng:2014ib, doi:10.1137/090767911, Cartis:2017fa, 10.1561/2200000035, MR3839333}. 
This work explores two topics in this context: 
algorithms for linear least squares that compute an accurate solution (up to the machine precision), and with computational complexities lower than classical methods; and algorithms for 
general unconstrained objective functions that 
compute an approximate solution with first and second order guarantees of optimality, with probability arbitrarily 
close to one and matching, in the order of desired accuracy of the solution, the complexity of classical optimization methods.

We begin with a simple example of how random embeddings can help solve linear least squares
\begin{align} 
\min_{x \in \R^d}\|Ax-b\|_2^2 \label{tmp-2022-1-8}
\end{align}
faster. Consider the problem $min_x f(x) = (x-6)^2 + (2x-5)^2 + (3x-7)^2 + (4x-10)^2$, corresponding to $A = \begin{pmatrix}
	1 & 2 & 3 & 4
\end{pmatrix}^T $ and $b = \begin{pmatrix}
	6 & 5 & 7 & 10
\end{pmatrix}^T $. Solving $f'(x)=0$ or equivalently $A^T Ax = A^Tb$, we obtain $x = \frac{77}{30} \approx 2.567$. Sketching with random embedding $S$ transforms the problem (\ref{tmp-2022-1-8}) to 

\begin{align} 
\min_{x \in \R^d}\|SAx-Sb\|_2^2, \label{Sketched-LLS-statement}
\end{align}
where $S \in \R^{m\times n}$ is some matrix we choose. We give two examples for $S$.

\begin{example}[Sampling]
Let\footnote{Namely, $S$ has one non-zero per row.} $S = \begin{pmatrix}
	1 & 0 & 0 & 0 \\ 0 & 0 & 1 & 0
\end{pmatrix}$. Then $SA = \begin{pmatrix}
	1  & 3
\end{pmatrix}^T$ gives the 1st and 3rd row of the matrix $A$. $Sb = \begin{pmatrix}
	6 & 7
\end{pmatrix}^T$ gives the 1st and 3rd entry of the vector $b$. Solving (\ref{Sketched-LLS-statement}) gives us $x = \frac{81}{30} = 2.700$.
\end{example}

\begin{example}[Hashing]
Let\footnote{Namely, $S$ has one non-zero per column.} $S = \begin{pmatrix}
	1 & 1 & 0 & 0 \\ 0 & 0 & 1 & 1
\end{pmatrix}$. Then $SA = \begin{pmatrix}
	3  & 7
\end{pmatrix}^T$ where the 1st row of $SA$ is the sum of the 1st and 2nd rows of $A$; the 2nd row of $SA$ is the sum of the 3rd and 4th rows of $A$. $Sb = \begin{pmatrix}
	11 & 17
\end{pmatrix}^T$ where the 1st entry of $Sb$ is the sum of the 1st and 2nd entries of $b$; the 2nd entry of $Sb$ is the sum of the 3rd and 4th entries of $b$. Solving (\ref{Sketched-LLS-statement}) gives us $x = \frac{152}{58} \approx 2.621$.
\end{example}

In both examples we reduce the number of rows of $A, b$ from four to two, but using hashing gives a more accurate result because it uses each row of $A$ and  entry in $b$. In later sections of this thesis we show that computing the solution of problem $\eqref{LLS-statement}$ with hashing sketching leads to improved performance. We also show how to use random embeddings to compute an accurate instead of just an approximate solution of linear least squares.

For the remainder of this chapter, we first review key concepts about random embeddings, and compare and contrast some  well-known random embeddings. Then we introduce the problem of linear least squares, classical techniques for its solution, and random embedding-based approaches known as sketching. We then introduce general non-convex optimisation problems; classical first and second order methods to solve them; and  existing theoretical results on these \singleQuote{full space} methods. We also introduce non-linear least squares as a particular type of non-convex optimisation problems. This chapter ends with a  summary of the structure and contributions of this thesis to random embeddings, their applications to linear least squares, and  to general non-convex optimisations. Our detailed contributions and relevant literature reviews can be found in individual chapters.
    
    \section{Random embedding}
    \paragraph{JL Lemma}
Dimensionality reduction techniques using random embeddings rely crucially  on the Johnson-Lindenstrauss lemma which first appeared in 1984 \cite{Johnson:1984aa}. It states that to calculate the approximate 2-norm\footnote{By 2-norm of a vector, we mean its usual Euclidean norm.} of a set of vectors, it suffices to randomly project them to lower dimensions, calculate their length in the projected space.
This is equivalent to multiplying the vectors representing the high-dimensional points (on the left) by an under-determined matrix with entries following some probabilistic distributions. We call such a matrix, a random embedding or projection. In particular, random embeddings for a set of points that approximately preserve their 2-norms are called Johnson-Lindenstrauss (JL)-embeddings (formally defined in \autoref{def::JL_embedding}). More specifically, we may choose scaled Gaussian matrices as the embedding \footnote{In the original paper \cite{Johnson:1984aa}, the lemma appears as an existence result concerning Lipschitz mappings. Here we state the `modern' form that is proved in \cite{MR1943859} and that is more relevant to this thesis.}.
\begin{lemma}[JL Lemma \cite{Johnson:1984aa, MR1943859}]
\label{def::JL_lemma}
Given a fixed, finite set $Y\subseteq \R^{n}$, $\epsilon, \delta > 0$, let $S \in\R^{m\times n}$ 
have entries independently distributed as the normal $N(0, n^{-1})$, with 
$m = \mathO{\epsilon^{-2}\log(\frac{|Y|}{\delta})}$ and where $|Y|$ refers to the cardinality of the set $Y$. Then we have, with probability at least $1-\delta$, that
\begin{equation}
(1-\epsilon)\|y\|_2^2 \leq \|Sy\|_2^2 \leq (1+\epsilon)\|y\|_2^2 \quad \text{for all}\,\, y \in Y.
\end{equation}
\end{lemma}

Intuitively, we are able to use the above dimensionality reduction technique because we are only concerned with  Euclidean distances, expressed as sum of squares. If a vector $x$ has $n$ entries with similar magnitudes, to calculate its 2-norm, we only need to sample some of its entries, say $m$ entries, then calculate the sum of squares of those entries, and rescale by $n/m$ to obtain the approximate 2-norm of $x$. This is illustrated in \autoref{fig:intro_1}, where we set $x$ to be a random Gaussian vector with independent identically distributed entries. We see that the error in the norm estimation is within 5\%. 

\begin{figure}
    \centering
    \includegraphics[width=0.5\textwidth]{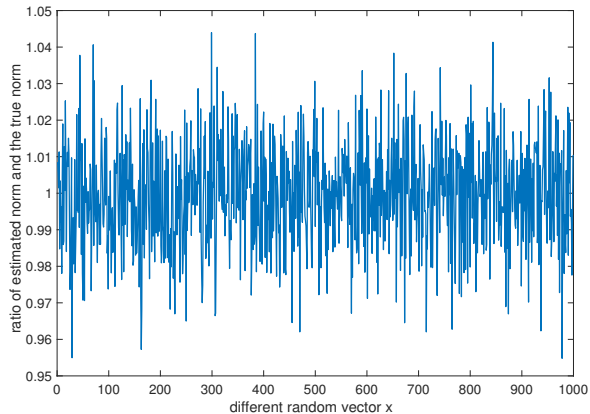}
    \caption{Randomly sampling and then re-scaling gives a good estimate of the norm when the vector components have similar magnitude. }
    \label{fig:intro_1}
\end{figure}

In general, the magnitudes of the entries are dissimilar. However, we can preprocess $x$ by applying a random, norm-preserving transformation, before sampling and re-scaling. In \autoref{fig:intro_2}, we apply a (randomised) Fourier transform to a vector $x$ with a non-uniform distribution of the  magnitude of entries. We observe that the square of the entries of $x$ are more uniformly distributed after the transform. 

\begin{figure}
    \centering
    \includegraphics[width=0.8\textwidth]{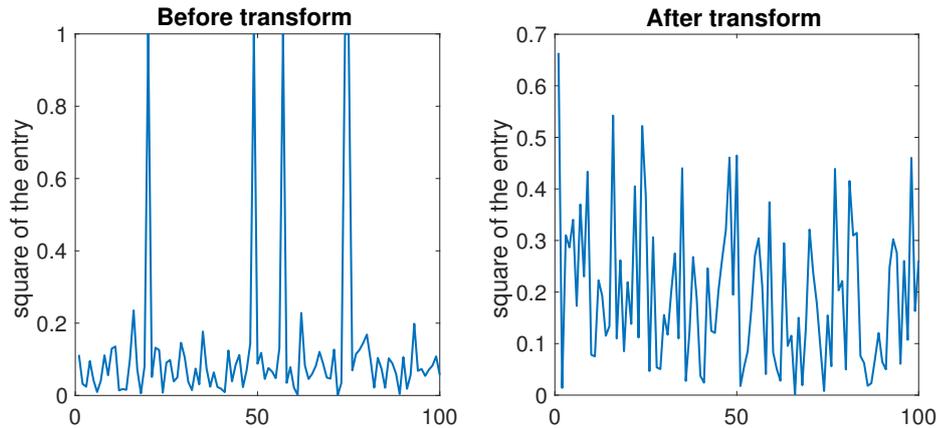}
    \caption{(Randomised) Fourier transform makes the magnitude of  the entries of a vector more similar}
    \label{fig:intro_2}
\end{figure}

Multiplying by a square Gaussian matrix has a similar effect in making the magnitude of the entries of a vector more similar. While
multiplying by an under-determined Gaussian matrix is a composition of multiplying by a square Gaussian matrix and then an under-determined sampling matrix (with one non-zero entry per row in a random column, whose value is one). This is the intuition behind the JL Lemma. For more  details on random matrices, see \cite{MR3837109}.

\paragraph{Subspace embedding}
Instead of a discrete group of points, Subspace embeddings (formally defined in \autoref{subspace_embedding_def1_statement}) also aim to approximately preserve the 2-norm of each point in a column subspace of some given matrix $A$. Subspace embeddings are useful when the point whose 2-norm is to be approximately preserved is unknown but lies in a given subspace; such as in the application of using random embeddings to solve linear least squares faster, where the optimal $Ax^* - b$ is unknown but lies in the subspace generated by the columns of $A$ and the vector $b$. Subspace embeddings also find applications in computing a high quality preconditioner of a linear system, and solving the low-rank matrix approximation problem \cite{MR3839333}. Often, a random matrix distribution can be both an (oblivious)\footnote{Data independent, see \autoref{Oblivious_embedding}.} JL-embedding and an (oblivious) subspace embedding, see \cite{10.1561/0400000060} where the author derives the oblivious subspace embedding property of the scaled Gaussian matrices from its oblivious JL-embedding property.

\paragraph{Oblivious subspace embedding}
A crucial advantage of random embeddings comparing to deterministic ones is that their embedding properties are data independent. For example, it is well known that the singular value decomposition (SVD) gives the most efficient low-dimensional embedding of a column subspace (Eckart--Young theorem). However, for each given matrix, its  SVD needs to be computed before the embedding can be applied; which is computationally expensive and the cost scales with the data size. 
By contrast, random embeddings are independent of the data matrix and hence no data-specific computation is required (aside from constructing the random embedding by sampling from the given distribution and applying the random embedding to the data matrix). 
Therefore, due to this property, random embeddings are oblivious embeddings (formally defined in \autoref{Oblivious_embedding}). 
A consequence of the embedding being oblivious to the data is that there is, in general, a positive probability that the randomly drawn embedding fails to embed the data (in the sense of providing a JL-embedding or a subspace-embedding). However the failure probability is exponentially small and can be bounded above by appropriately setting the dimension of the embedded space. Moreover, the iterative nature of our subspace algorithms in later chapters takes into account that in some iterations, the embedding may fail. But the probability that those algorithms fail to converge at the expected theoretical rate approaches zero exponentially fast with the total number of iterations. 

Next, we briefly review a list of commonly used random embeddings, which are represented by random matrices.

\paragraph{Popular random matrices and their properties}

Sampling matrices have one non-zero entry per row in a random column. 

\begin{definition}\label{def:sampling}
We define $S \in \R^{m \times n}$ to be a scaled sampling matrix if, independently for each $i \in [m]$, we sample $j \in [n]$ uniformly at random and let $S_{ij}=\sqrt{\frac{n}{m}}$. 
\end{definition}

The scaling factor is included so that given $x \in \R^n$, we have $\expectation{\|Sx\|_2} = \|x\|_2$ for any scaled sampling matrix $S$. 
Sampling matrices are computationally inexpensive to apply to vectors/matrices so that embeddings based on them can be computed efficiently. However, the success of sampling matrices is highly dependent on the data. Even if we have $\expectation{\|Sx\|_2} = \|x\|_2$, $\|Sx\|_2$ may have high variance, such as  when $x$ has a single non-zero entry in its first row.

Non-uniformity of a vector, formally defined in \autoref{def:non-uniform-vector}, provides a measure of how different the magnitudes of the entries are; and the success of sampling matrices as an oblivious JL embedding depends on this. Similarly, the success of the sampling matrices as an oblivious subspace embedding depends on the coherence of a matrix (formally defined in \autoref{def::coherence}), which provides a measure of the non-uniformity of vectors in the matrix column subspace (\autoref{non_uniformity_col_subspace_coherence}). 

There are broadly two types of approaches to tackle the high variance challenge of using sampling matrices. The first type is based on transforming the vector/matrix to one with the same norm/column subspace but with higher uniformity. For example, it is well known that for any fixed vector $x \in \R^n$, pre-multiplication by a square Gaussian matrix (with each entry following $N(0, n^{-1})$) transforms the vector into one with independent normally distributed entries while preserving $\|x\|_2$ in expectation. In high dimensions, the resulting vector has high uniformity (due to entries having the same distribution and the scaling factor) and is thus suitable for applying sampling. A scaled Gaussian matrix can be thought as the product of a scaled sampling matrix (with the scaling being $\sqrt{\frac{n}{m}}$) and a square Gaussian matrix (with each entry following $N(0, n^{-1})$).

\begin{definition}\label{def:Gaussian}
We say $S \in \R^{ m \times n }$ is a scaled Gaussian matrix if $S_{ij}$ are independently distributed as $N (0, {m}^{-1})$. 
\end{definition}

(Scaled) Gaussian matrices with appropriate dimensions have been shown to be an oblivious JL/subspace embeddings \cite{10.1561/0400000060}. However, Gaussian matrices are computationally expensive to apply, especially when embedding a linear subspace represented by a dense basis due to the cost of dense matrix-matrix multiplication.

Subsampled-Randomised-Hadamard-Transform (SRHT) \cite{10.1145/1132516.1132597, Tropp:wr} uses an alternative non-uniformity reduction technique based on the Hadamard transform, which is similar to the Fourier transform. A Fast-Fourier-type algorithm exists that allows applying SRHT in $O(n \log(n))$ time for $x\in \R^n$ \cite{10.1145/1132516.1132597}, thus having a better complexity than the naive matrix-matrix multiplication, while still achieving comparable theoretical properties as the scaled Gaussian matrix. For subspace embedding of matrices in $\R^{n \times d}$, the embedding dimension of SRHT has a $\log(d)$ multiplicative factor compared to that of scaled Gaussian matrices \cite{Tropp:wr}. We have SRHT formally defined as below.

\begin{definition}\label{def::SRHT}
A Subsampled-Randomised-Hadamard-Transform (SRHT) \cite{10.1145/1132516.1132597, Tropp:wr} is an $m\times n$ matrix of the form $S =  S_s HD$ with $m\leq n$, where
\begin{itemize}[topsep=0pt,itemsep=-1ex,partopsep=1ex,parsep=1ex]
    \item $D$ is a random $n \times n$ diagonal matrix with $\pm 1$ independent entries.
    \item $H$ is an $n\times n$ Walsh-Hadamard matrix defined by
        \begin{equation}
            H_{ij} = n^{-1/2}(-1)^{\la (i-1)_2, (j-1)_2\ra},
        \end{equation}
        where $(i-1)_2$, $(j-1)_2$ are binary representation vectors of the numbers $(i-1), (j-1)$ respectively\footnote{For example, $(3)_2 = (1,1)$.}.
    \item $S_s$ is a random scaled $m\times n$ sampling matrix (defined in \autoref{def:sampling}), independent of $D$.
\end{itemize}
\end{definition}

A crucial drawback of SRHT is that if the column space is represented by a sparse matrix,  the embedded matrix, although of smaller dimensions, is generally dense. 
Though sampling matrices preserve sparsity, we have mentioned above their downsides concerning high variance. 

The second way to tackle the disadvantages of sampling is to use another sparse embedding ensemble instead. The 1-hashing matrices have one non-zero per column instead of one non-zero per row as in the sampling matrix; moreover, the value of the non-zero is $\pm 1$ with equal probability so that \reply{$\expectation{\|Sx\|_2}^2 = \|x\|_2^2$}.  We have
the following formal definition.

\begin{definition}[1-hashing \cite{10.1145/3019134, MR3167920}]\label{1-hashing}
$S \in \R^{m \times n}$ is a $1$-hashing matrix if independently for each $j \in [n]$, we sample $i$ uniformly at random and let $S_{i j} = \pm 1$ with equal probability.
\end{definition}

Conceptually, unlike sampling, which discards a number of rows of the vector/matrix to be embedded, hashing uses every single row. The dimensionality reduction is achieved by hashing those rows into $m$ slots, and adding them with sign-randomisation if multiple rows are hashed into a single slot. Therefore intuitively, hashing is more robust than sampling because it uses all the rows, and theoretical results have been established to show 1-hashing matrices with appropriate dimensions are oblivious JL/subspace embeddings without any requirement on the non-uniformity of the input \cite{10.1145/3019134, Nelson:te}.

Finally, $1$-hashing can be generalised to $s$-hashing, that is, matrices with $s$ non-zeros per column, defined below.

\begin{definition}\cite{10.1145/3019134} \label{def::sampling_and_hashing}
$S \in \R^{m \times n}$ is a $s$-hashing matrix if independently for each $j \in [n]$, we sample without replacement $i_1, i_2, \dots, i_s \in [m]$ uniformly at random and let $S_{i_k j} = \pm 1/\sqrt{s}$ for $k = 1, 2, \dots, s$.
\end{definition}

Conceptually, $s$-hashing sketches each row of the input $s$ times (into $s$ different rows of the output), with each row being scaled by $\frac{1}{\sqrt{s}}$, and has better theoretical properties when used as a JL/subspace embedding than $1$-hashing \cite{10.5555/2884435.2884456}. However, we note that while $1$-hashing does not increase the number of non-zeros in the vector/matrix to be embedded, $s$-hashing may increase it by up to $s$ times.

    \section{Linear least squares}
    Linear Least Squares (LLS) problems arising from fitting observational data to a linear model are mathematically formulated as the optimisation problem,
\begin{align} 
\min_{x \in \R^d}\|Ax-b\|_2^2, \label{LLS-statement}
\end{align}
where  $A \in \R^{n\times d}$ is the data matrix that has (potentially unknown) rank $r$,  $b\in \R^n$ is the vector of observations, and
$n \geq d \geq r$. Thus \eqref{LLS-statement} represents an optimisation problem where we have $n$ data points and a model of $d$ variables.

Problem (\ref{LLS-statement}) is equivalent to solving the normal equations
\begin{align}
 A^TA x= A^T b. \label{NE}
 \end{align}

Numerous techniques have been proposed for the solution of (\ref{NE}), and they traditionally involve the factorization of $A^TA$, just $A$, or iterative methods. The ensuing cost in the worst case is $\mathcal{O}(nd^2)$, which is prohibitive for large $n$ and $d$ \cite{10.5555/248979}.  
We briefly survey here the main classical techniques for solving LLS (\ref{LLS-statement})/(\ref{NE}) following \cite{Nocedal:2006uv}. For iterative methods and preconditioning, see \cite{Bjorck:1996uz}, while for sparse input matrices, see also \cite{MR2270673}.

\subsection{Dense linear least squares}
We say  problem \eqref{LLS-statement} is a dense Linear Least Squares (dense LLS) if the matrix $A$ is a dense matrix. Namely, the matrix $A$ does not have sufficiently many zero entries for specialised sparse linear algebra routines to be advantageous. To solve dense LLS, we may employ direct methods based on factorizations, or iterative methods typically based on conjugate gradient techniques.

\paragraph{Direct methods for dense LLS}

Cholesky factorization computes $A^TA = LL^T$, where $L \in \R^{d \times d}$ is a lower triangular matrix. Then the normal equations \eqref{NE} are solved by forward-backward substitutions involving the  matrix $L$. The main costs are computing $A^TA$ and factorizing it, both taking $\mathcal{O}(nd^2)$ though many practical algorithms compute the factorization without forming $A^TA$ explicitly.\footnote{Factorising $A^TA$ takes $\mathO{d^3}$ only but given that $n\geq d$, it is still $\mathO{nd^2}$.} This method is affected by the potential ill-conditioning of $A^TA$ (since the condition number of  $A^T A$ is the square of the condition number of $A$) and so may not solve (\ref{NE}) accurately.

Employing the QR factorization aims to solve \eqref{LLS-statement} directly without using \eqref{NE}.
 Computing $A = QR$, where $Q\in \R^{n\times n}$ is orthogonal and $R\in \R^{n\times d}$ is upper triangular, we have that $\normTwo{Ax-b}^2 = \normTwo{Rx - Q^Tb}^2$. As $R$ is both over-determined and upper triangular, its last $n-d$ rows are zeros. Therefore, $\normTwo{Rx - Q^Tb}^2$ is minimised by making the first $d$ rows of $Rx$ equal to the first $d$ rows of $Q^Tb$ which involves solving a linear system of equations involving the upper triangular matrix $R$. 
 Hence the dominant cost is the QR factorization, which is $\mathcal{O}(nd^2)$.

When $A$ is rank-deficient or approximately rank-deficient, Cholesky factorization may break down and (un-pivoted) $QR$ factorization may give a rank-deficient $R$, introducing numerical instabilities in solving systems involving $R$. Singular Value Decomposition(SVD)-based methods are the most robust in dealing with rank-deficient problems, as an SVD reveals  the spectrum, and therefore the extent of rank-deficiency, of the matrix $A$.

A (full) SVD of $A$ is computed as $A = U_f \Sigma_f V_f^T$, where $U_f \in \R^{n\times d}, V_f \in \R^{d \times d}$ have orthonormal columns, and $\Sigma_f$ is diagonal with non-negative and decreasing diagonal entries. The rank deficiency in $A$ is then controlled by carefully selecting a cut-off point in the diagonal of $\Sigma_f$. After which the method proceeds similarly to QR-based approach by replacing $A$ in \eqref{LLS-statement} with its factorization and using the fact that left multiplying matrices with orthonormal columns/right multiplying orthogonal matrices does not change the 2-norm. However SVD-based methods are more computationally expensive than QR-based ones \cite{10.5555/248979}. 

\paragraph{Iterative methods for dense LLS}
LSQR \cite{10.1145/355984.355989} and related algorithms such as LSMR \cite{LSMR_2011} apply conjugate gradient method to solve the normal equations (\ref{NE}), only requiring matrix-vector multiplications involving $A$ or $A^T$. In the worst case, $\mathcal{O}(d)$ iterations with $\mathcal{O}(nd)$ floating-point arithmetic operations per iteration are required. Therefore the worst case cost of iterative methods for dense LLS is $\mathcal{O}(nd^2)$ \footnote{We note that iterative methods may not converge in $\mathcal{O}(d)$ iterations if $A$ has a large condition number due to the effect of floating-point arithmetic.}.
But if the spectrum of A (the distribution of the singular values of $A$) is favorable these methods may take less iterations \cite{10.5555/248979}.

Preconditioning techniques lower the condition number of the system by transforming the problem \eqref{LLS-statement} into an equivalent form before applying iterative methods. For example, a sophisticated preconditioner for \eqref{LLS-statement} is the incomplete Cholesky preconditioner \cite{10.1145/2617555}, that uses an incomplete Cholesky factor of $A$ to transform the problem. In general, some preconditioning should be used together with iterative methods \cite{10.1145/3014057}.

\subsection{Sparse linear least squares}
When the matrix $A$ is sparse (that is, there is a significant number of zeros in $A$ such that specialised sparse linear algebra algorithms could be advantageous), we refer to the problem as sparse Linear Least Squares (sparse LLS). 

\paragraph{Direct methods for sparse LLS}

We refer the reader to \cite{MR2270673}, where sparse Cholesky and QR factorizations are described. The main difference compared to the dense counterparts is that when the positions of a large number of zero entries of $A$ are given, it is possible to use that information alone to predict positions of some zero entries in the resulting factors so that no computation is required to compute their values. Therefore, sparse factorizations could be faster than their dense counterparts on sparse $A$. 

\paragraph{Iterative methods for sparse LLS}
The LSQR and LSMR algorithms mentioned in solving dense LLS automatically take advantage of the sparsity of $A$, as the matrix-vector multiplications involving $A$ or $A^T$ are faster when $A$ is sparse.

\subsection{Sketching for linear least squares}

Over the past fifteen years, sketching techniques 
 have been investigated for improving the computational efficiency and scalability of methods
 for the solution of \eqref{LLS-statement}; see, for example, the survey papers  \cite{10.1561/2200000035, 10.1561/0400000060}. Sketching  uses a carefully-chosen random  matrix $S\in\R^{m\times n}$, $m\ll n$, to sample/project the measurement matrix $A$ to lower dimensions, while approximately preserving the geometry of the entire column space of $A$; this quality of $S$ (and of its associated distribution) is captured by the (oblivious) subspace embedding property \cite{10.1561/0400000060} in Definition \ref{subspace_embedding_def1_statement}. 
The sketched  matrix $SA$ is then used to either directly compute an approximate solution to  problem (\ref{LLS-statement}) or to generate a high-quality preconditioner for the iterative solution of   (\ref{LLS-statement}); the latter has been the basis of state-of-the-art randomized linear algebra codes such as Blendenpik \cite{doi:10.1137/090767911} and LSRN \cite{Meng:2014ib}, where the latter improves on the former by exploiting sparsity  and allowing rank-deficiency of the input matrix\footnote{LSRN also allows and exploits parallelism but this is beyond our scope here.}. 

Using sketching to compute an approximate solution of \eqref{LLS-statement} in a computationally efficient way is proposed by Sarlos \cite{10.1109/FOCS.2006.37}. 
Using sketching to compute a preconditioner for the iterative solution of \eqref{LLS-statement} via a QR factorization of the sketched matrix is proposed by Rokhlin \cite{Rokhlin:2008wb}. 
The choice of the sketching matrix is very important as it needs to approximately preserves the geometry of the column space of a given matrix (a subspace embedding, see Definition \ref{subspace_embedding_def1}) with high-probability, while allowing efficient computation of the matrix-matrix product $SA$. 
Sarlos \cite{10.1109/FOCS.2006.37} proposed using the fast Johnson-Lindenstrauss transform (FJLT) discovered by Ailon and Chazelle \cite{10.1145/1132516.1132597}. 
The FJLT (and similar embeddings such as the SRHT) is a structured matrix that takes $O \left( nd \log(d) \right) $ flops to apply to $A$, while requiring the matrix $S$ to have about $m = O \left( d \log(d) \right) $ rows to be a subspace embedding (also see Tropp \cite{Tropp:wr}, Ailon and Liberty \cite{10.1145/2483699.2483701}). 
More recently, Clarkson and Woodruff \cite{10.1145/3019134} proposed the hashing matrix that has one non-zero per column in random rows, with the value of the non-zero being $\pm1$. This matrix takes $O \left( nnz(A) \right) $ flops to apply to $A$, but needs $m=\Theta \left( d^2 \right) $ rows to be a subspace embedding (also see Meng et al \cite{10.1145/2488608.2488621}, Nelson at al \cite{Nelson:2014uu,10.1145/2488608.2488622}). Recent works have also shown that increasing number of non-zeros per column of the hashing matrix leads to reduced requirement of number of rows (Cohen \cite{10.5555/2884435.2884456}, Nelson \cite{Nelson:te}). 
Further work on hashing by Bourgain \cite{Bourgain:2015tc} showed that if the coherence\footnote{Maximum row norm of the left singular matrix $U$ from the compact SVD of the matrix $A=U \Sigma V^T$. Formally defined in \autoref{def::coherence}.} is low, hashing requires fewer rows. 
These sketching matrices have found applications in practical implementations of sketching algorithms, namely, Blendenpik \cite{doi:10.1137/090767911} used a variant of FJLT; LSRN \cite{Meng:2014ib} used Gaussian sketching; Iyer \cite{10.5555/3019094.3019103} and Dahiya\cite{10.1145/3219819.3220098} experimented with 1-hashing\footnote{hashing matrices with 1 non-zero per column}. 
The state-of-the-art sketching solvers Blendenpik \cite{doi:10.1137/090767911} and LSRN \cite{Meng:2014ib} demonstrated several times speed-ups comparing to solvers based on QR/SVD factorizations in LAPACK \cite{laug}, and LSRN additionally showed significant speed-up comparing to the solver based on sparse QR in SuiteSparse \cite{10.1145/2049662.2049670} when the measurement matrix $A$ is sparse. However, Blendenpik and LSRN have not fully exploited the power of sketching, namely, Blendenpik only solves problem \eqref{LLS-statement} when the measurement matrix $A$ has full column rank, and LSRN uses Gaussian sketching matrices with dense SVD even for a sparse measurement matrix $A$. We propose a new solver in Chapter 3. 


\subsection{Alternative approaches for preconditioning and solving large-scale LLS problems}

 On the preconditioning side for linear least squares, \cite{CERDAN2020112621} considered alternative regularization strategies to compute an Incomplete Cholesky preconditioner for rank-deficient least squares. LU factorization may alternatively be used for preconditioning. After a factorization $PAQ=LU$ where $P,Q$ are permutation matrices, the normal equation $A^TAx=A^Tb$ is transformed as $L^TL y = L^Tc$ with $y = UQ^Tx$ and $c=Pb$. In \cite{inproceedings}, $L$ is further preconditioned with $L_1^{-1}$ where $L_1$ is the upper square part of $L$. On the other hand, \cite{gnanasekaran2021hierarchical} proposed and implemented a new sparse QR factorization method, with C++ code and encouraging performance on  Inverse Poisson Least Squares problems. For a survey, see \cite{10.1145/3014057, Gould:2016vg}.

In addition to the sketch-precondition-solve methodology we use, large-scale linear least squares may alternatively be solved by first-order methods, zeroth-order methods (including Stochastic Gradient Descent (SGD)) and classical sketching (namely, solve the randomly embedded linear least square problem directly, as in see \cite{10.1109/FOCS.2006.37}). First order methods construct iterates $x_{t+1} = x_t - \mu_t H_t^{-1} g(x_t) + \beta_t(x_t - x_{t-1})$, where $H_t = A^T S_t^T S_t A$, $g(x_t) = A^TAx_t - A^T b$ and the last term represents the momentum. This is proposed in  $\cite{lacotte2020optimal,lacotte2019faster}$, deriving optimal sequences $\mu_t,\beta_t$ for Gaussian and subsampled randomised Hadamard transforms, for $S_t$ fixed or refreshed at each iteration. See also \cite{RichtarikAndGowerLinearSystem} for a randomised method for consistent linear least squares (namely, the residual at the optimal solution is zero). 
On the other hand, because linear least squares is a convex problem, SGD can be used, with 
\cite{LoizouEtAl} investigating using SGD with heavy ball momentum 
and \cite{kahale2020leastsquares} investigating using SGD with sketched Hessian.  Using gradient-based sampling instead of uniform or leverage-scores-based sampling is explored in \cite{zhu2018gradientbased}. Finally, \cite{lopes2018error} provides techniques for a posteriori error estimates for classical/explicit sketching.

    
    \section{Minimising a general unconstrained objective function}
    We consider the following problem
\begin{equation}
\min_{x \in \R^d} f(x),  \label{general_objective_statement}
\end{equation}
where $f$ is smooth and non-convex. 
We will be satisfied if our algorithm returns an (approximate)
local minimum of the objective function $f$ -- a point at which, if $f$ is continuously differentiable, its gradient $\grad f(x)$ is (approximately) zero; if $f$ is twice continuously differentiable, then in addition to its gradient being zero, its Hessian $\grad^2 f(x)$ is (approximately) positive semi-definite. 
This may not be the global minimum -- namely, the smallest value of $f$ over the entire $\R^d$. Finding the latter is much more computationally challenging and the remit of the field of global optimization \cite{GlobalOptimisationReview}. Though global optimization is beyond our scope here, local optimization algorithms are often key ingredients in the development of techniques for the former.

Starting from a(ny) initial guess $x_0\in \R^d$, classical (local) algorithms for solving \eqref{general_objective_statement} are iterative approaches that generate iterates $x_k$, $k\geq 0$, recursively, based on an update of the form
\begin{equation*}
    \xKPlusOne = \xK - \sK,
\end{equation*}
where the step $\sK$ is chosen so that 
the objective $f$ typically decreases at each iteration. Depending, for example, on the problem information used in the computation of $\sK$, algorithms
can be classified into those that use only up to first order information
(i.e. $f(x), \grad f(x)$); and those that also use  second order information
(i.e, $\grad^2 f(x)$).

\subsection{First order methods}
For a(ny) user-provided tolerance $\epsilon>0$, first order methods find an iterate $x_k$ such 
that $\normTwo{\gradFK} < \epsilon$. 
This ensures the approximate achievement of the necessary optimality condition 
$\nabla f(x) = 0$ that holds at any local minimiser 
of problem \eqref{general_objective_statement}. 
However, note that this condition
is not sufficient, e.g. $x$ with $\nabla f(x) = 0$ could be a 
saddle point or even a local maximiser. 
However, as the iterates also progressively decrease $f$, we increase the chance that we find an approximate local minimiser. 
Indeed, several recent works have shown that for a diverse set of landscapes, first order methods such as the gradient descent methods escape/do not get trapped in saddle points and approach local minimisers; see for example, \cite{ChiJin, BenRecht}.
We briefly review the three main classical first order methods: steepest descent with line search, the trust region method and the adaptive regularisation method. 

\paragraph{Steepest descent with linesearch}
Steepest descent method seek the update
\begin{equation}
    \xKPlusOne = \xK - \alphaK \gradFK,
\end{equation}
where the step $\alphaK$ is determined by a line-search detailed below.

Given a constant $0< \beta < 1$, the linesearch algorithm starts with some initial guess of $\alpha_k > 0$, and repeatedly decreases $\alpha_k$ until the following Armijo condition is satisfied:

\begin{equation}
f(x_k) - f(x_k - \alphaK \gradFK) \geq \beta \alphaK \normTwo{\gradFK}^2. 
\label{eq:armijo}
\end{equation}

It can be shown that assuming $f(x)$ is continuously differentiable, one can always find $\alphaK > 0 $ satisfying \eqref{eq:armijo}; moreover, provided the gradient of $f$ is Lipschitz continuous\footnote{For some $L>0$, we have that $\normTwo{\grad f(x) - \grad f(y)} \leq L \normTwo{x-y}$ for all $x, y \in \R^d$.} and $f$ is bounded below, the steepest descent algorithm with  linesearch requires at most $\mathO{\epsilon^{-2}}$ evaluations of the objective function and its gradient  to converge to an iterate $x_k$ such that $\normTwo{\gradFK} \leq \epsilon$ \cite{CoraBook, NesterovTextBook}.
This complexity bound is sharp, as shown, for example, in Theorem 2.2.3 of \cite{CoraBook}.

\paragraph{The trust region method} \label{Intro_trust_region}
In the trust region method, the step $\sK$ is calculated by minimizing a local model $m_k(s) = f(x_k) + \gradFK^T s + \frac{1}{2}s^T B_k s$, where $B_k$ is a  symmetric matrix (that is required to be uniformly bounded above with $k$). The matrix $B_k$ could be some approximation of the Hessian/curvature information, if possible.

We then compute $s_k$ by approximately minimising $m_k(s)$ within a trust region $\|s\|_2 \leq \Delta_k$ so that the decrease in $m_k$ achieved by taking the step $\sK$ is at least as much as that can be achieved by considering the steepest descent direction in the trust region. The trust region radius $\Delta_k$ is initialised at some value $\Delta_0$ and subsequently dynamically adjusted: for the computed step $s_k$, if we have sufficient decrease in $f$, $f(x_k) - f(x_k + s_k)\geq \eta [m_k(0) - m_k(s_k)]$, for some (iteration-independent) constant $0 < \eta < 1$ and $0 < \gamma < 1 $, then $\Delta_{k+1} = \gamma^{-1} \Delta_k$ and $x_{k+1}=x_k+s_k$. Otherwise, $\Delta_{k+1} = \gamma \Delta_k $ and we do not take the step $s_k$ ($x_{k+1}=x_k$).

It has been shown that the first order trust region method also has a global complexity of $\mathO{\epsilon^{-2}}$ in terms of gradient and objective function evaluations \cite{gratton2008recursive}.
This complexity bound is also sharp \cite{CoraBook}.

\paragraph{The adaptive regularisation method}
\label{Intro_QR}
Like the trust region method,  adaptive regularisation uses a local model around the current iterate $x_k$ to compute a suitable step $s_k$. Unlike the trust region method, which explicitly restricts the size of the potential step, a regularisation term  is imposed to a first-order Taylor model to implicitly restrict the size of the step:
\begin{equation}
    m_k(s) = f(x_k) + \gradFK^T s + \frac{1}{2}\sigma_k \|s\|_2^2 = T_{f,1}(x_k,s) + \frac{1}{2}\sigma_k \normTwo{s}^2, 
\end{equation}
where $T_{f,1}(x_k,s)$ is the first-order Taylor series expansion of $f$ around $x_k$. We minimise the local regularised model to compute a trial step $s_k$. Then, as in the trust region method, we evaluate the objective at the trial step $x_k+s_k$ and dynamically adjust the regularisation parameter by computing $\rho_k = 
\frac{f(x_k) - f(x_k+s_k)}{T_{f,1}(x_k,0) - T_{f,1}(x_k,s_k)}$; and, if $\rho_k \geq \eta$, we set $x_{k+1}=x_k+s_k$ and set $\sigma_{k+1} = \max(\gamma \sigma_k, \sigma_{min})$, otherwise we do not take the step ($x_{k+1}=x_k$) and increase the regularisation by setting $\sigma_{k+1} = \frac{1}{\gamma}\sigma_k$, where $\gamma, \sigma_{min}\in (0,1)$ are constants. 

The first order adaptive regularisation method also has a (sharp) complexity of $\mathO{\epsilon^{-2}}$ for both gradient and objective function evaluations, under the same assumptions on $f$ as for the trust region methods \cite{CoraBook}.

\paragraph{Subspace first order methods}
The coordinate descent method \cite{Wright2015} iteratively computes the next iterate by fixing most components of the variable $x$ at their values from the current iterate, while approximately minimising the objective with respect to the remaining components; thus effectively searches a potential step in a restricted subspace of the full variable space. The coordinate descent method is convergent for some convex problems \cite{Wright2015}, but fails to converge for nonconvex problems \cite{MR321541}. 
The coordinate descent has found many applications in solving large-scale problems \cite{Bach2011, Richtarik2015}.
Randomised coordinate descent methods have been an intense topic of recent investigations due to the demands of large scale problems; see
\cite{MR2968857_Nesterov, MR3179953_Richtarik, Lacotte2019, Grishchenko2021, Yang2020, snobfit_URL, Facchinei2015, Xu2015, Xu2017, Lu2018, Patrascu2015}. For a survey see \cite{Wright2015}.
In Chapter 4 of this
thesis, we will study a probabilistic subspace first order algorithm for general non-convex problems, that
only needs directional gradient evaluations $S_k \gradFK$ (so that the algorithm only searches the step in a subspace), where 
$S_k \in \rLTimesD$ is some random matrix to be specified with $l$ being a
user-chosen constant. Our work builds on the framework in \cite{Cartis:2017fa}. However, here we establish more general results and use subspaces explicitly to save gradient evaluation/computation cost. 
We show that the appropriate replacement of the full gradient
with subspace gradients $S_k \gradFK$ does not harm the worst-case
complexity; although in our specific algorithm, since $S_k$ is a random
matrix, there is a probabilistic component in our convergence result. That
is, we have $\normTwo{\gradFK}< \epsilon$ with $k = \mathO{\epsilon^{-2}}$
with probability proportional to $1 - e^{-\mathO{k}}$. The failure probability
is very small for any reasonable value of $k$, and we show some numerical
examples illustrating the effectiveness of our algorithm compared to 
the classical first order based methods when applied to
non-linear least squares. 

\newcommand{\epsTwo}{\epsilon_2}

\subsection{Second order methods}
In order to improve both the performance and the optimality guarantees of first order methods, we add curvature information both in the construction and in the termination of algorithms, when this is available. 
Given accuracy tolerances $\epsOne, \epsTwo >0$, we may strengthen our goal to try to find
a point where simultaneously, 
\begin{equation}
    \normTwo{\gradFK} < \epsOne, 
\lambdaMin{\hessFK}> -\epsTwo \label{eqn:secondOrderCriPoint}
\end{equation}
where $\lambdaMin{.}$ denotes the left-most
eigenvalue of a matrix. These conditions secure approximate second order criticality conditions and strengthen the guarantee that we are close to a local minimiser, where $\nabla^2 f(x)$ is positive semidefinite. Clearly, 
in order to achieve this aim, the algorithm needs to be provided
with second order information, the Hessian $\hessFK$, at each iteration. 
Let us review the classical optimisation methods for 
smooth, non-convex objective where both first and second order
information is available. 

\paragraph{Newton's method}
In Newton's method, the iterates 
are constructed according to

\begin{equation}
    \xKPlusOne = \xK - \hessFK^{-1} \gradFK,
\end{equation}
that is, the step $s_k$ satisfies the linear system $\hessFK s_k = - \gradFK$. Note that here one assumes that the matrix $\hessFK$ is positive definite for each $k$. For general functions, one may add regularisation terms, or use a linesearch or trust region, which we will discuss later. 

Newton's method is attractive
because it has a quadratic convergence property once $x_k$ gets into
a neighbourhood of a nondegenerate solution. However such a neighbourhood is typically not
known a priori, before the run of the algorithm. It turns out that starting from 
an arbitrary starting point, the complexity of Newton's method
can be the same as the steepest descent method, $\mathO{\epsilon^{-2}}$, even if we assume the Newton direction is always well-defined \cite{CoraBook}.

\paragraph{The second order trust region method}

Trust-region methods could also use  additional second order information,
and the local model at each iteration becomes
\begin{equation}
    \mK{s} = f(x_k) + \gradFK^T{s} + 
    \frac{1}{2}s^T \hessFK {s}. 
\end{equation}
We then compute an approximate minimiser of the model, subject to $s$ being smaller than the trust-region radius. As before, it suffices to compute an approximate solution for convergence. In addition to requiring the model decrease to be at least as much as that of in the steepest descent direction, for second order criticality, we also require that the model decrease is at least as much as that obtained in the direction of the eigenvector of $\hessFK$ corresponding to the smallest eigenvalue (if such an eigenvalue is negative, otherwise  this second condition is not required). Then, the objective at the trial step is again evaluated, and the ratio of the function decrease with the model decrease is compared to a pre-defined constant, and steps are taken/not taken; trust region radius is increased/decreased accordingly. 

The second order trust-region algorithm has been shown to converge to a first order critical point $\normTwo{\gradFK} < \epsOne$ in $\mathO{\epsOne^{-2}}$ iterations; moreover, this bound is sharp. Thus  the first order complexity of the (first order) trust region method is not improved by upgrading the trust region model to include accurate second order information. However, one can further show that the second order trust-region algorithm converges to a second order critical point $\lambdaMin{\hessFK} > -\epsTwo$ and $\normTwo{\gradFK} \leq \epsOne$ in $\mathO{\max(\epsOne^{-2}, \epsTwo^{-3})}$ iterations. We see that the main advantage of the second order trust region over the first order one is that it allows one to quantifiably compute an approximate second order critical point \cite{CoraBook}.

\paragraph{The second order adaptive regularisation method}
\label{page:second:order:cubic:classical}
The second order adaptive regularisation method, however, is able to 
not only allow convergence to a second order critical point, but also allow an improved speed of convergence to a first order critical point. 
The algorithm is the following: at each iteration, we build the model as
\begin{equation}
    \mK{s} = f(x_k) + \gradFK^T{s} + 
    \frac{1}{2}s^T \hessFK {s} + \frac{1}{3}
    \sigma_k \normTwo{s}^3, \label{tmp-2022-01-09}
\end{equation}
where $\sigma_k$ is an adaptive parameter whose value increases or 
decreases depending on the amount of objective function decrease
achieved by a step calculated from (approximately) minimising such a model. 
Compared to the first order regularisation method, the approximate minimisation here requires that $\normTwo{\nabla_s m_k(s_k)} \leq \frac{1}{2}\theta_1 \normTwo{s_k}^2$ and $\lambdaMin{\grad^2_s m_k(s_k)} \geq \theta_2 \normTwo{s_k}$ (for two iteration independent constants $\theta_1, \theta_2>0$). Assuming that $f$ is bounded below and that its Hessian is Lipschitz continuous\footnote{For some $L>0$, we have that $\normTwo{\grad^2 f(x) - \grad^2 f(y)} \leq L \normTwo{x-y}$ for all $x, y \in \R^d$.}, it can be shown  that this algorithm converges to a point where $\normTwo{\gradFK} < \epsOne$ in $\mathO{\epsOne^{-3/2}}$ evaluations of the objective function, gradient and Hessian; and  to a point where $\normTwo{\gradFK} < \epsOne$ and $\lambdaMin{\hessFK} > - \epsTwo$ in $\mathO{\max (\epsOne^{-3/2}, \epsTwo^{-3})}$ evaluations of the objective function, gradient and Hessian. Moreover, both of these bounds are sharp and the first-order one is provably optimal for second order methods  \cite{CoraBook, MR4163541_Carmon}.
Minimising (a slight variant of) \eqref{tmp-2022-01-09} to compute a step was first suggested in \cite{Griewank}. Independently, 
\cite{NesterovAndPolyak} considered using \eqref{tmp-2022-01-09} from a different perspective. The above mentioned method was first proposed in \cite{Cartis:2009fq} that improves on previous attempts, namely, allowing inexact model minimisation and without requiring the knowledge of problem-specific constants.


\paragraph{Subspace second order adaptive regularisation methods}
Methods that only use the gradient/Hessian in a subspace have been studied in the recent years.
The sketched Newton algorithm \cite{pilanci2017newton} requires a sketching matrix that is proportional to the rank of the Hessian. Sketched online Newton \cite{luo2016efficient}  uses streaming sketches to scale up a second-order method, comparable to Gauss–Newton, for solving online learning problems.
The randomised subspace Newton \cite{gower2019rsn} efficiently sketches the full Newton direction for a family of generalised linear models, such as logistic regression. Other randomised versions of Newton's method include \cite{Gower2020, Gower2016a, Berahas2020}.
The global convergence of the above methods, however, require the objective function $f$ to be convex (or even strongly convex).
For general non-convex optimisation, \cite{Cartis:2017fa, Gratton:2017kz} give generic frameworks that apply to the first order/second order general optimisation methods. Our main focus is non-convex functions and so we build on these works.
Chapter 5 of this thesis proposes a second order adaptive regularisation method when
operating in a random subspace. Specifically, both the gradient and the Hessian
will be replaced by their subspace equivalent. We are able to show that
under suitable assumptions on the subspace sketching matrix $S_k \in 
\rLTimesD$ (that could be a scaled Gaussian matrix with $l = \mathO{r}$,
where $r$ is the rank of the Hessian $\hessFK$), both the fast convergence
rate $\mathO{\epsOne^{-3/2}}$ to the first order critical point, 
and the convergence to the second order critical point with a rate
$\mathO{\epsTwo^{-3}}$ can be retained.

\subsection{Applications to non-linear least squares problems}

Non-linear least squares are a subclass of general unconstrained optimisation problems. We aim to solve

\begin{equation}
\min_{x \in \R^d} f(x) = \frac{1}{2}\sum_{i=1}^n (r_i(x) )^2=\frac{1}{2}\norm{r(x)}_2^2, \label{NLS}
\end{equation}
where $r(x)=(r_1,r_2,\ldots,r_n)(x)$ is a vector-valued smooth residual function $r:\R^d\rightarrow\R^n$.
This formulation has a wide range of applications in weather forecasting, finance and machine learning problems. We briefly overview some classical solution methods here, following \cite{Nocedal:2006uv}.

\paragraph{The Gauss--Newton method}
The Gauss--Newton method is a simplification of Newton's method that exploits the structure of non-linear least squares problems. In particular, we can approximate the Hessian as $\hessFK=J(x_k)^T J(x_k)+\sum_{i=1}^n r_i(x_k)\nabla^2 r_i(x_k) \approx J(x_k)^T J(x_k)$ where 
\begin{align*}
J(x) = \left(\pdv{r_i(x)}{x_j}\right)_{ij} \in \R^{n \times d}.
\end{align*}
This approximation is justified in the case when $r(x)\approx 0$ at a solution $x$ or when $r$ is approximately linear in the variables.
Despite using only first-derivative information about $r$, the Gauss-Newton method has been shown to enjoy a local super-linear convergence rate to a first order critical point. When the Gauss-Newton direction is safeguarded with 
a trust region or regularization technique (which is often referred to as a Levenberg-Marquardt method), it can be shown to have global convergence  provided for example, that $J$ is Lipschitz continuous. To ensure global convergence of linesearch variants of Gauss-Newton, we additionally need to require that the Jacobian's singular values are uniformly bounded away from zero -- a very strong assumption.


\paragraph{Subspace Gauss Newton method}
Expanding on our work in \cite{zhen:icml_BCGN}, in Chapter 4, we present such an algorithmic variant and its numerical performance when compared to the full  Gauss Newton method.
Subspace Gauss-Newton variants can also be found in \cite{gower2019rsn}.

    \section{Contributions of this thesis}
    \subsection{New theoretical analysis of hashing matrices and development of new random embeddings}
In Chapter 2, we study theoretical properties of hashing matrices and propose a new oblivious subspace embedding based on hashing.
We show that hashing matrices --- with one nonzero entry per column and  of size proportional to the rank of the data matrix --- generate a subspace embedding with high probability, provided the given data matrix has low coherence. We then show that $s$-hashing matrices, with $s>1$ nonzero entries per column, satisfy similarly good sketching properties for a larger class of low coherence data matrices. 

More specifically, we show  that a hashing matrix $S\in\R^{m\times n}$ with $m=\mathcal{O}(r)$ is an oblivious subspace embedding for matrices $A$ with low coherence. Hashing matrices have been shown empirically to be almost as good as Gaussian matrices \cite{10.1145/3219819.3220098} in terms of their embedding properties, but the theoretical results show that they need at least $m = \mathcal{O}(r^2)$ rows \cite{Nelson:2014uu}. Our result explains the  phenomenon observed in \cite{10.1145/3219819.3220098}. In addition, it was observed empirically that one needs at least 2 non-zeros per hashing column for the projection to be  accurate. We show that using $s$ non-zeros per column instead of $1$ non-zero per column relaxes the coherence requirement by $\sqrt{s}$. Thus we expect more performance improvement if we increase $s=1$ to $s=2$ than from $s=2$ to $s=3$. Previous work on $s$-hashing has independently discovered the $\sqrt{s}$ factor \cite{Bourgain:2015tc}, but our result relies on a single, intuitive proof, and is not tied to any particular proof for the case of $1-$hashing. So if the coherence requirement bound for $1$-hashing is subsequently improved, our work allows the result on $s$-hashing to improve automatically. 

Cascading this result, we also introduce and analyse a new random embedding: Hashed-Randomised-Hadamard-Transform (HRHT), that combines the coherence reduction properties of randomised Hadamard Transform with the coherence-based theory of hashing embeddings. Compared to Subsampled-Randomised-Hadamard-Transform, HRHT is able to achieve subspace embedding with the embedding dimension $\mathO{r}$ where $r$ is the rank of the matrix to be embedded, matching  the optimal bound known for the Gaussian embedding. Experiments using random embeddings for preconditioning  linear least squares show the improved performance of HRHT over SRHT. 

\subsection{Analysis, state-of-the-art implementation and benchmarking of new large-scale linear least squares solver using random embeddings}

Chapter 3 concerns the solution of large-scale Linear Least Squares (LLS) problems, by applying random embeddings to reduce the dimensionality of the observation/sample space. The sketched matrix $SA$ is used to generate a high-quality preconditioner for the iterative solution of (\ref{LLS-statement}) and has been the basis of state-of-the-art randomized linear algebra codes such as Blendenpik \cite{doi:10.1137/090767911} and LSRN \cite{Meng:2014ib}, where the latter improves on the former by exploiting input sparsity, parallelization and allowing rank-deficiency of the input matrix.

We propose and analyse a  sketching framework for potentially rank-deficient LLS. Our framework includes the algorithm used by Blendenpik \cite{doi:10.1137/090767911} and LSRN \cite{Meng:2014ib}; and additionally it allows one to use a wide range of rank-revealing factorizations to build a preconditioner using the sketched matrix $SA$. Our analysis shows that one can recover a minimal residual solution with a rank-revealing QR factorization with sketching, or the minimal norm solution with a total orthogonal factorization with sketching. This framework allows us to use (randomised) column pivoted QR factorization for dense LLS so that our solver solves rank-deficient LLS satisfactorily without using the expensive SVD. We are also able to use a sparse rank-revealing factorization for sparse LLS, obtaining a significant speed-up over LSRN \cite{Meng:2014ib}, state-of-the-art sparse solvers LS\_SPQR \cite{10.1145/2049662.2049670} and incomplete Cholesky factorization preconditioned Krylov subspace method LS\_HSL \cite{10.1145/2617555}.

Numerically, we developed a solver SKi-LLS (SKetchIng-Linear-Least-Square) combining our theoretical and algorithmic ideas and state-of-the-art C++ implementations. For dense inputs, the solver is more robust than Blendenpik (as it solves rank-deficient or approximately rank-deficient problems); while being quicker than Blendenpik for matrices with high coherence and comparable in speed with Blendenpik for other matrices. In order to overcome the speed deterioration of the column-pivoted QR comparing to the un-pivoted QR, we used a recent development of randomised column pivoted QR that exploits randomisation and the importance of memory usage and cache in modern computing architecture \cite{Martinsson:2017eh}. For sparse inputs, by using a sparse QR factorization code developed by Davis \cite{10.1145/2049662.2049670}, our solver is more than 10 times faster than LSRN, LS\_SPQR and LS\_HSL for sparse Gaussian inputs. We extensively compared our solver with LSRN, LS\_SPQR and LS\_HSL on the Florida Matrix Collection \cite{10.1145/2049662.2049663}, and our solver is extremely competitive on strongly-over-determined inputs or ill-conditioned inputs.

\subsection{First-order subspace methods and their application to non-linear least squares}

In Chapter 4, we analyse a general randomised algorithmic framework for minimizing a general objective function \eqref{general_objective_statement}, that improves upon the one introduced in \cite{Cartis:2017fa}, 
so that an arbitrarily high probability convergence/complexity result can be derived. We formulate more specific conditions on the reduced local models that are based on random embeddings of the variable space (in contrast to embedding the observational space in the linear least squares case). Compared to \cite{Gratton:2017kz}, our algorithm applies more generally\footnote{Also, in the case of trust region methods, our framework does not need to compare the norm of the model gradient with the trust region radius at each iteration in order to decide if the trust region radius should be increased (see \cite{Gratton:2017kz}, Algorithm 2.5).}, also to quadratic regularisation (see later sections). 
  
  Compared to \cite{Cartis:2017fa, Gratton:2017kz}, we use a weaker/different definition of a `true' iteration, when the approximate problem information is sufficiently accurate; this definition is  based on the random embedding satisfying a (one-sided) JL-embedding property (see \eqref{Intro:one-side} below), which is novel.
 Using the latter property and typical smoothness assumptions on the problem, we show that our framework of random subspace methods  has complexity $\mathO{\epsilon^{-2}}$ to generate an approximate 
first-order stationary point, with exponentially high probability. To ensure this, the random subspace needs to sketch the gradient, replacing $\gradFK$ with $S_k \gradFK$ in the algorithm, where 
$S_k \in \rLTimesD$ satisfies, with positive probability,
\begin{equation}\label{Intro:one-side}
    \normTwo{S_k \gradFK} \geq (1-\epS) \normTwo{\gradFK}.
\end{equation}
We show that the above is achieved when $S_k$ is a scaled Gaussian matrix, a hashing (sparse-embedding) 
matrix, a sampling matrix, and many more.

    We note again that this framework marks a significant departure from probabilistic
    model assumptions \cite{Cartis:2017fa, Gratton:2017kz}, since our model gradient, $S_k \gradFK$, does not even have 
    the same dimension as the true gradient\footnote{Hence the probabilistic model condition
    which bounds $\normTwo{\grad m(x_k) - \gradFK}$ is not applicable here.}.
    The intuition behind this requirement is that in classical 
     trust-region or adaptive regularization, the norm of the 
    gradient is a key ingredient in the recipe for ensuring a convergence result, and hence
    by preserving the norm of the gradient with sketching, we are able to 
    obtain a similar worst case complexity. 
    Another interesting observation is that in the case of $S_k$ being the sampling
    matrix, for which our method reduces to a randomised block-coordinate approach,
    we show how the success of the algorithm on non-convex smooth problems is 
    connected with the 'non-uniformity' of the gradient; thus leading to almost sure convergence under some strong  assumptions related to the objective's gradient.
We then particularize this framework to Gauss-Newton techniques for nonlinear least squares problems, where the Jacobian is computed in a subspace.

\subsection{Random subspace cubic regularisation algorithm, R-ARC}
In Chapter 5,
we further analyse the random subspace  framework when second order information is added and applied to  general non-convex optimisation. We propose and analyse an algorithm that is a random subspace version of the second order adaptive cubic regularisation method. We show that the subspace variant matches the optimal convergence rate $\mathO{\epsilon_1^{-3/2}}$ of the full-dimensional variant to generate $\|\nabla f(x_k)\|_2 \leq \epsilon_1$ under suitable assumptions: either the embedding matrix $S_k$ provides a subspace embedding of the Hessian $\grad^2 f(x_k)$, or the Hessian is sparse in the sense that only few rows are non-zero.

We further analyse convergence to  second order critical points of the second order adaptive regularisation method. We first show that in general, the algorithm converges to a point where the subspace Hessian $S_k \hessFK S_k^T$ is approximately positive semi-definite. Then we prove that if scaled Gaussian matrices are used as  random embeddings, the algorithm converges to a point where the full Hessian is approximately positive semi-definite, at a rate $\mathO{\epsilon_2^{-3}}$ that  matches the full-dimensional second order cubic regularisation method.

    
    \section{Structure of thesis}
    In Chapter 2, we first give the necessary technical background on random embeddings, which will be used throughout this thesis. We then state and prove our theorem on the coherence requirement needed to use 1-hashing with $m=\mathcal{O}(d)$ as an oblivious subspace embedding (defined in \autoref{Oblivious_embedding}). Cascading this result, we show how increasing the number of non-zeros per column from 1 to $s$ relaxes the coherence requirement for $s$-hashing matrices by a factor of $\sqrt{s}$, and propose a new random matrix distribution for subspace embeddings that has at most $s$ non-zeros per column. Then, we propose a carefully constructed random matrix distribution that uses hashing matrices, achieving $m=\mathcal{O}(d)$ as a subspace embedding with high probability for any sufficiently over-determined matrix $A$.

In Chapter 3, we propose and analyse an algorithmic framework that uses random embedding (sketching) for potentially rank-deficient linear least squares. Then we introduce our linear least squares solver \solverName{} which implements the framework and discuss its key features and implementation details. We test and benchmark \solverName{} against state of the art algorithms and test problems in the remainder of Chapter 3. 

In Chapter 4, we move onto problem \eqref{general_objective_statement}. We first propose and analyse an algorithmic framework that relies on stochastic reduced local models. We then show how sketching-based subspace methods fit into this framework, and derive results for quadratic-regularisation  and trust-region  algorithms for general unconstrained objective optimisation. We then apply this framework to Gauss-Newton method and nonlinear least squares, obtaining a subspace Gauss-Newton method and illustrating its performance numerically. 

In Chapter 5, we propose and analyse the subspace cubic-regularisation based approach for solving \eqref{general_objective_statement}. We first show how subspace embedding of the Hessian of the objective function allows the same convergence rate as the (full-space) cubic-regularisation methods. We then show how the sparsity of the Hessian allows a similar convergence result to be derived. We then go on to analyse the convergence to second-order critical points using the subspace cubic-regularisation based approach. We show that using scaled Gaussian embeddings allows convergence of R-ARC to a second order critical point with a rate essentially the same as the full-space method. 

Finally in Chapter 6, we summarise the main results in this thesis and set some future directions. 

{\bf Notation.} Throughout the thesis, we let $\langle\cdot,\cdot\rangle$ and $\|\cdot\|_{2}$ denote the usual Euclidean inner product and norm, 
respectively, and  $\|\cdot\|_{\infty}$, the $l_{\infty}$ norm. Also, for some $n\in \N$, $[n]=\{1,2,\ldots,n\}$. For a(ny) symmetric positive definite matrix $\overline{W}$,
we define the norm $\|x\|_{\overline{W}}:=x^T\overline{W}x$, for all $x$,  as the norm induced  by $\overline{W}$.
The notation $\Theta \left( \cdot\right)$ denotes both lower and upper bounds of the respective order. $\Omega \left( \cdot \right)$ denotes a lower bound of the respective order. $\mathO{ \cdot}$ denotes an upper bound of the respective order.

\chapter{Random embedding}

    \section{Introduction and relevant literature}
    This chapter is based and expands materials in \cite{zhen:icml_LLS, 2021arXiv210511815C}.
\paragraph{Main contributions}
This chapter aims to explore the theoretical properties of sparse sketching matrices $S$ for improved efficiency and scalability of methods for solving the LLS problem \eqref{LLS-statement}, when $A$ is sparse or dense. After introducing the necessary technical background in Section 2, we firstly investigate projection and computational properties of $1$-hashing matrices, random sparse matrices with $1$ non-zero per column, that were first proposed in the randomised linear algebra context by Clarkson and Woodruff \cite{10.1145/3019134}. Sparse matrices allow faster computation of the matrix-matrix product $SA$, leading to faster embeddings than their dense counterparts. Moreover, sparse matrices preserve the sparsity of the data matrix $A$, allowing the sparsity of the embedded matrix $SA$ to be exploited by specialized numerical linear algebra routines for faster computation.

It has been observed numerically in \cite{10.1145/3219819.3220098} that $1$-hashing matrices, with the same projection dimensions as Gaussian matrices (matrices with i.i.d. Gaussian entries) are as efficient in projecting and solving the LLS problem \eqref{LLS-statement}. However, it was shown in \cite{Nelson:2014uu, 10.1145/2488608.2488621} that $1$-hashing matrices require at least an $\mathO{r^2}$ projection dimension to work effectively (as an oblivious subspace embedding, defined in \autoref{Oblivious_embedding}) comparing to an $\mathO{r}$ projection dimension required by Gaussian matrices \cite{10.1561/0400000060}, where $r$ is the rank of $A$ in \eqref{LLS-statement}. Thus a gap exists between the theory and the empirical observation. Our main result on $1$-hashing matrices shows that $1$-hashing matrices can have the same theoretical properties as the Gaussian matrices, namely, being an oblivious subspace embedding for matrices of rank $r$ with the projection dimension being $\mathO{r}$, given that the matrix $A$ has low coherence (defined in \autoref{def::coherence}). 

Cascading on this result, we then show in Section \ref{s-hashing} , firstly that $s$-hashing matrices, which are sparse matrices with (fixed) $s$ non-zeros per column first proposed as a candidate matrix distribution for oblivious subspace embeddings in \cite{Nelson:te}, achieves being an oblivious subspace embedding with the projection dimension of $\mathO{r}$ for matrices $A$ of rank $r$, but with the coherence requirement on $A$ being relaxed by $\sqrt{s}$ comparing to $1$-hashing matrices. Our numerical illustration (\autoref{fig::1-2-3-semi-co}) shows that $2$-hashing is more effective as a sketching matrix in solving \eqref{LLS-statement}. 

Secondly in Section \ref{s-hashing}, we propose a new matrix distribution called $s$-hashing variant matrices that has at most $s$ non-zeros per column for oblivious subspace embeddings. Using a novel result that allows us to connect any coherence-restricted embedding result for $1$-hashing matrices to $s$-hashing matrices, we show that $s$-hashing variant matrices have a similar coherence restricted embedding property as $s$-hashing matrices. 

At the end of Section \ref{s-hashing}, we combine $s$-hashing (variant) matrices with coherence reduction transformations that were first proposed in \cite{10.1145/1132516.1132597} to derive a new random matrix distribution that will be an oblivious subspace embedding with the projection dimension $\mathO{r}$ for any matrix $A$ sufficiently over-determined; and will take $\mathO{nd\log(n)}$ flops to apply. This so-called Hashed-Randomised-Hadamard-Transform (HRHT) improves upon the previously proposed Subsampled-Randomized-Hadamard-Transform (SRHT) by lowering the projection dimension from $\mathO{r \log(r)}$ to $\mathO{r}$ while maintaining the complexity of embedding time up to a $\log(n)/\log(d)$ multiplicative factor. 

\paragraph{Related literature}

After the Johnson-Lindenstrauss Lemma appeared in 1984, Indyk and Motawani proved scaled Gaussian matrices is a JL-embedding \cite{MR1715608} for which an elementary proof was provided by Dasgupta and Gupta \cite{MR1943859}. Achlioptas \cite{MR2005771} showed that matrices with all entries being $\pm1$ with equal probability, or indeed matrices with all entries being $+1, -1, 0$ with equal probability are JL-embeddings. However these random ensembles take $\mathO{nd^2}$ to apply to an $n \times d$ matrix in general. The Fast-Johnson-Lindenstrauss-Transform (FJLT) as a JL-embedding was proposed in \cite{10.1145/1132516.1132597} and as a subspace-embedding was proposed in \cite{10.1109/FOCS.2006.37}. The FJLT is based on Fast Fourier Transform-like algorithms and is faster to apply to matrices and vectors. The construction and analysis of FJLT are subsequently improved in \cite{Tropp:wr,10.1145/2483699.2483701,doi:10.1137/090767911,Rokhlin:2008wb}, ending with \cite{Tropp:wr} analysing a variant of FJLT called Subsampled Randomised Hadamard Transform (SRHT) using matrix concentration inequalities. SRHT takes $\mathO{nd \log(d)}$ flops to apply to $A$, while requiring $S$ to have about $ m = \mathO{r \log(r)}$ rows to be an oblivious subspace embedding, where $r$ is the rank of $A$. Clarkson and Woodruff \cite{10.1145/3019134} proposed and analysed using the $1$-hashing matrices as a candidate distribution for subspace embeddings, and Nelson and Nguyen \cite{Nelson:te} proposed and analysed using the $s$-hashing matrices. These sparse matrices are subsequently analysed in \cite{10.1145/2488608.2488621,Nelson:2014uu,10.1145/2488608.2488622,10.5555/2884435.2884456}, showing that for $1$-hashing matrices, $m = \mathOmega{r^2}$ is required for being an oblivious subspace embedding (see also Example \ref{ex:1}); and $m = \mathO{r^2}$ is sufficient. And for $s$-hashing matrices, $m = \mathO{r\log(r)}$ is sufficient for being an oblivious subspace embedding with $s = \mathO{\log(r)}$. Recently, Bourgain, Dirksen and Nelson \cite{Bourgain:2015tc} showed a coherence dependent result of $s$-hashing matrices (see \autoref{Bourgain}). 

Comparing to the existing results, our results on $1$ and $s$-hashing matrices are the first oblivious subspace embedding results on $1$ and $s$-hashing matrices with $m = \mathO{r}$; though our result does have a strict coherence requirement. Our result on Hashed-Randomised-Hadamard-Transform has a lower embedding dimension than the SRHT. (Note that it has also been shown in \cite{Tropp:wr} that the embedding dimension of SHRT could not be further lowered due to the Coupon Collector's Problem.)

Finally, we mention some recent works on random embeddings.
Recent results concerning oblivious (tensor) subspace embeddings \cite{iwen2020lower} could be particularized to  oblivious (vector) subspace embeddings, leading to a matrix distribution  $S\in\mathbb{R}^{m\times n}$ with $m = \mathcal{O} (r \log^4 r \log n )$ (where $r$ is the rank of the matrix to be embedded) that requires $\mathcal{O}(n\log n)$ operations to apply to any vector. This has slightly worse space and time complexity than sub-sampled randomized Hadamard transform. Regarding sparse embeddings, \cite{CHEN2020105639} proposed a 'stable' 1-hashing matrix that has the $(\epsilon,\delta)$-oblivious JL embedding property with the optimal $m = \mathcal{O}\bracket{\epsilon^{-2} \log(1/\delta)}$ (same as scaled Gaussian matrices) while each row also has approximately the same number of non-zeros. The algorithm samples $n$ non-zero row indices for $n$ columns of $S$ by sampling without replacement from the set $\{[m], [m], \dots, [m]\}$ where $[m]=\{ 1,2, \dots, m\}$ is repeated $\lceil \frac{n}{m} \rceil$ times. \cite{liu2021extending} proposed learning the positions and values of non-zero entries in 1-hashing matrices by assuming the data comes from a fixed distribution. 
    
    \section{Technical Background}

In this section, we review some important concepts and their properties that we then use throughout \reply{the thesis}. 
We employ several variants of the notion of random embeddings for finite or infinite sets, as we define next.

\subsection{Random embeddings}

We start with a very general concept of embedding a (finite or infinite) number of points; throughout, we let $\epsilon \in (0,1)$ be the user-chosen/arbitrary error
tolerance in the embeddings and $n, k\in \N$. \footnote{\reply{Note that here $\epsilon$ is not the error tolerance of the algorithms that we will discuss later in this thesis, e.g. linear least squares, general non-convex optimisations. While the error tolerance in the embedding influences the performance of embedding-based algorithms, it is not necessary to have a small error in the embedding in order to achieve a small error tolerance in the actual algorithm. Because the inaccuracy of the embedding may be mitigated by repeated iterations of the algorithm, or an indirect use of the embedding. In particular, although $\epsilon \in (0,1)$, we do not require the embedding accuracy $\epsilon$ to be close to zero in this thesis.}}

\begin{definition}[Generalised JL\footnote{Note that `JL' stands for Johnson-Lindenstrauss, recalling their pioneering lemma \cite{Johnson:1984aa}.} embedding \cite{10.1561/0400000060}]
\label{genJL}
A generalised $\epsilon$-JL embedding for a set $Y\subseteq \R^n$
is a matrix $S \in\R^{m\times n}$ such that
\begin{equation}\label{JL-plus}
-\epsilon \|y_i\|_2\cdot \|y_j\|_2 \leq \langle Sy_i, Sy_j \rangle - \langle y_i, y_j \rangle \leq \epsilon \|y_i\|_2 \cdot \|y_j\|_2, \quad \text{for all }\,\, y_i, y_j \in Y.
\end{equation}
\end{definition}

If we let $y_i=y_j$ in (\ref{JL-plus}),  we recover the common notion of an $\epsilon$-JL embedding, that approximately preserves the length of vectors in a given set. 

\begin{definition}[JL embedding \cite{10.1561/0400000060}]
\label{def::JL_embedding}
An $\epsilon$-JL embedding for a set $Y\subseteq \R^{n}$ 
is a matrix $S \in\R^{m\times n}$ such that
\begin{equation}\label{JL}
(1-\epsilon)\|y\|_2^2 \leq \|Sy\|_2^2 \leq (1+\epsilon)\|y\|_2^2 \quad \text{for all}\,\, y \in Y.
\end{equation}
\end{definition}

Often, in the above definitions, the set $Y=\{y_1,\ldots,y_k\}$ is a finite collection of vectors in $\R^n$. But an infinite number of points may also be embedded, such as in the
case when $Y$ is an entire subspace. Then, an embedding approximately preserves pairwise distances between any points in the column 
space of a  matrix $B\in \R^{n\times k}$.
\begin{definition}[$\epsilon$-subspace embedding \cite{10.1561/0400000060}]\label{subspace_embedding_def1_statement}
An $\epsilon$-subspace embedding for a matrix $B \in \R^{n\times k}$ is a matrix $S\in\R^{m\times n}$ such that
\begin{equation}\label{subspace_embedding_def1}
(1-\epsilon)\|y\|_2^2 \leq \|Sy\|_2^2 \leq (1+\epsilon)\|y\|_2^2 \quad \text{ for all $y\in Y=\{y: y=Bz, z\in \R^k\}$}. 
\end{equation}
\end{definition}
In other words, $S$ is an $\epsilon$-subspace embedding for $B$ if and only if $S$ is 
an $\epsilon$-JL embedding for  the column subspace $Y$ of $B$.

 Oblivious embeddings are matrix distributions such that given a(ny)  subset/column subspace of vectors in $\R^n$, a random matrix drawn 
 from such a distribution is an embedding for these vectors with high probability. We let $1-\delta\in [0,1]$ denote a(ny) success probability of an embedding.
 
\begin{definition}[Oblivious embedding \cite{10.1561/0400000060,10.1109/FOCS.2006.37}] \label{Oblivious_embedding}
A distribution $\cal{S}$ on $S \in \R^{m \times n}$  is an $(\epsilon,\delta)$-oblivious embedding if given a fixed/arbitrary set of vectors, we have that, with probability at least $1-\delta$, a matrix $S$ from the distribution is an $\epsilon$-embedding for these vectors.
\end{definition}
Using the above definitions of embeddings, we have distributions that are {\it oblivious JL-embeddings} for a(ny) given/fixed set $Y$ of some vectors $y\in \R^n$, and distributions that are   {\it oblivious subspace embeddings} 
 for a(ny) given/fixed matrix $B\in \R^{n\times k}$ (and for the corresponding subspace $Y$ of its columns).  We note that depending on the quantities being embedded, in addition to $\epsilon$ and $\delta$ dependencies, the size $m$
 of $S$ may depend on $n$ and  the `dimension' of the embedded sets; for example, in the case of a finite set $Y$ of $k$ vectors in $\R^n$, $m$ additionally may depend on $k$ while in the subspace embedding case, $m$ may depend on  the rank $r$ of $B$.

\subsection{Generic properties of subspace embeddings}


A necessary condition for a matrix $S$ to be an $\epsilon$-subspace embedding for a given matrix is that the sketched matrix  has the same rank.  

\begin{lemma}\label{rank-of-sketched-equal-to-unsketched}
If the matrix $S$ is an $\epsilon$-subspace embedding for a given matrix $B$ for some $\epsilon \in (0,1)$, then  $rank(SB) = rank(B)$, where $rank(\cdot)$ denotes the rank of the argument matrix.
\end{lemma}

\begin{proof}
Let $B\in \R^{n \times k}.$
By rank-nullity theorem, $\rank(B) + \dim \ker(B) = \rank(SB) + \dim \ker(SB) =k$. Clearly, $\dim \ker(SB) \geq \dim \ker(B)$. If the previous inequality is strict, then there exists $z\in \R^k$ such that $\|SBz\|_2 = 0$ and $\|Bz\|_2 >0$, contradicting the assumption that $S$ is an $\epsilon$-subspace embedding for $B$ according to \eqref{subspace_embedding_def1}.
\end{proof}

Given any matrix $A \in \R^{n\times d}$ of rank $r$, the compact singular value decomposition (SVD) of $A$ provides a perfect subspace embedding.
In particular, let
\begin{align}
A = U \Sigma V^T,\label{thin-SVD}
\end{align}
where $U \in \R^{n \times r}$ with orthonormal columns, $\Sigma \in \R^{r \times r}$ is diagonal matrix with strictly positive diagonal entries, and $V \in \R^{d \times r}$ with orthonormal columns  \cite{10.5555/248979}.  Then the matrix $U^T$ is a $\epsilon$-subspace embedding for $A$ for any $\epsilon\in (0,1)$.

Next, we connect the embedding properties of $S$ for $A$ with those for $U$ in \eqref{thin-SVD}, using a proof technique in Woodruff \cite{10.1561/0400000060}.

\begin{lemma} \label{subspace-embedding-def-2}
Let $A \in \R^{n \times d}$ with rank $r$ and SVD-decomposition factor 
$U \in \R^{n \times r}$ defined in \eqref{thin-SVD}, and let
$\epsilon \in (0,1)$. Then the following equivalences hold:
\begin{itemize}
\item[(i)] 
a matrix $S$ is an $\epsilon$-subspace embedding for  $A$ if and only if $S$ is an $\epsilon$-subspace embedding for $U$, namely,
\begin{align}
(1-\epsilon)\|Uz\|_2^2 \leq \|SUz\|_2^2 \leq (1+\epsilon)\|Uz\|_2^2, \quad \text{for all $z \in \R^{r}$}.\label{U-condition}
\end{align}
\item[(ii)] A matrix $S$ is an $\epsilon$-subspace embedding for  $A$ if and only if for all $z\in \R^{r}$ with $\|z\|_2=1$, we have\footnote{We note that since
$\|z\|_2=1$ and $U$ has orthonormal columns, $\|Uz\|_2=\|z\|_2=1$ in \eqref{U-condition-2}.}
\begin{align}
(1-\epsilon)\|Uz\|_2^2 \leq \|SUz\|_2^2 \leq (1+\epsilon)\|Uz\|_2^2. \label{U-condition-2}
\end{align}
\end{itemize}
\end{lemma}
\begin{proof}
\begin{itemize}
    \item[(i)]
Let $A = U \Sigma V^T$ be defined as in (\ref{thin-SVD}). If $S$ is an $\epsilon$-subspace embedding for $A$, let $z \in \R^{r}$ and define $x = V\Sigma^{-1}z \in \R^{d}$. Then we have $Uz = Ax$ and

\begin{equation}
\|SUz\|_2^2 = \|SAx\|_2^2 
			\leq (1+\epsilon) \|Ax\|_2^2 \\
			= (1+\epsilon) \|Uz\|_2^2,
\end{equation}
where we have used $Uz = Ax$ and (\ref{subspace_embedding_def1}). Similarly, we have $\|SUz\|_2^2 \geq (1-\epsilon) \|Uz\|_2^2$. Hence $S$ is an $\epsilon$-subspace embedding for $U$.

Conversely, given $S$ is an $\epsilon$-subspace embedding for $U$, let $x \in \R^{d}$ and $z = \Sigma V^T x \in \R^{r}$. Then we have $Ax = Uz$, and $\|SAx\|_2^2 = \|SUz\|_2^2 \leq (1+\epsilon)\|Uz\|_2^2 = (1+\epsilon) \|Ax\|_2^2$. Similarly $\|SAx\|_2^2 \geq (1-\epsilon) \|Ax\|_2^2$. Hence $S$ is an $\epsilon$-subspace embedding for $A$.
\item[(ii)] Since the equivalence in (i) holds, note that \eqref{U-condition} clearly implies \eqref{U-condition-2}. The latter also implies the former if 
\eqref{U-condition-2} is applied to $z/\|z\|_2$ for any nonzero $z\in \R^r$.
\end{itemize}
\end{proof}

\begin{remark}\label{rem_rd}
Lemma \ref{subspace-embedding-def-2}  shows that to obtain a subspace embedding for an $n\times d$ matrix $A$ it is sufficient (and necessary) to embed correctly its left-singular matrix that has rank $r$. Thus, the dependence on $d$ in subspace embedding results can be replaced by dependence on $r$, the rank of the input matrix $A$. As rank deficient matrices $A$ are important in this \reply{thesis}, we opt to state our results in terms of their $r$ dependency (instead of $d$).
\end{remark}

The matrix $U$ in \eqref{thin-SVD} can be seen as the ideal `sketching' matrix for $A$; however,  there is not much computational gain in doing this as computing the compact SVD has similar complexity as computing a minimal residual solution to (\ref{LLS-statement}) directly.

\subsection{Sparse matrix distributions and their embeddings properties}

In terms of optimal embedding properties, it is well known that (dense)
scaled Gaussian matrices  $S$ with $m = \mathcal{O} \left( \epsilon^{-2} (r + \log(1/\delta) )\right)$ provide an $(\epsilon, \delta)$-oblivious subspace embedding for $n\times d$ matrices $A$ of rank $r$ \cite{10.1561/0400000060}. However, the computational cost of the matrix-matrix product $SA$  is  $\mathcal{O}(nd^2)$, which is similar 
to the complexity of solving the original LLS problem \eqref{LLS-statement}; thus it seems difficult to achieve computational gains by calculating a sketched solution of 
\eqref{LLS-statement} in this case.
In order to improve the computational cost of using sketching for solving LLS problems, and to help preserve input sparsity (when $A$ is sparse), sparse random matrices have been proposed, 
namely, such as random matrices with one non-zero per row.  However, uniformly sampling rows of $A$ (and entries of $b$ in  \eqref{LLS-statement}) may miss choosing some (possibly important) row/entry.  A more robust proposal, both theoretically and numerically, is to use hashing matrices, with one (or more) nonzero entries per column, which when applied to $A$ (and $b$), captures all rows of $A$  (and entries of $b$) by adding two (or more) rows/entries with randomised signs. The definition of $s$-hashing matrices and was given in \autoref{def::sampling_and_hashing} and when $s=1$, we have $1$-hashing matrices defined in \autoref{1-hashing}.



Still, in general, the optimal dimension dependence present in Gaussian subspace embeddings cannot be replicated even for hashing distributions, as our next example illustrates.

\begin{example}[The $1$-hashing matrix distributions fails to yield an oblivious subspace embedding with $m = \mathcal{O}(r)$]
\label{ex:1}
	Consider the matrix
	\begin{align}
	A = \begin{pmatrix}
		I_{r\times r} & 0 \\ 0 & 0
	\end{pmatrix} \in \R^{n \times d}. \label{eqn::problematic_matrix}
	\end{align}
	If $S$ is a 1-hashing matrix with $m=\mathcal{O}(r)$, then 
	$\displaystyle SA = \bigl( S_1 \quad 0 \bigr)$, where the $S_1$ block contains the first $r$ columns of  $S$. 
	To ensure that the rank of $A$ is preserved (cf. Lemma \ref{rank-of-sketched-equal-to-unsketched}), a necessary condition for $S_1$ to have rank $r$ is that the $r$ non-zeros of $S_1$ are in different rows. Since, by definition, the respective  row is chosen independently and uniformly at random for each column of $S$, the probability of $S_1$ having rank $r$ is no greater than

	\begin{align}\label{hashing-fails}
	\left(1 - \frac{1}{m}\right) \cdot \left(1 - \frac{2}{m}\right)\cdot \ldots \cdot\left(1-\frac{r-1}{m}\right) 
	\leq e^{-\frac{1}{m}-\frac{2}{m}-\ldots -\frac{r-1}{m}}  
	= e^{-\frac{r(r-1)}{2m}},
	\end{align}
For the probability\footnote{The argument in the example relating  to 1-hashing sketching is related to the birthday paradox, as mentioned (but not proved) in Nelson and Nguyen \cite{10.1145/2488608.2488622}.}  \eqref{hashing-fails}  to be at least $1/2$, we must have
 $m \geq \frac{r(r-1)}{2 \log (2)}$.
\end{example}

The above example  improves upon the lower bound in Nelson et al.  \cite{10.1145/2488608.2488622} by slightly relaxing the requirements on $m$ and $n$\footnote{We note that
in fact,  \cite{Nelson:te} considers a more general set up, namely, any matrix distribution with column sparsity one.}.
We note that in the order of $r$ (or equivalently\footnote{See Remark \ref{rem_rd}.}, $d$), the lower bound $m=\mathcal{O}(r^2)$ for $1$-hashing matches the upper bound given in Nelson and Nguyen \cite{Nelson:te}, Meng and Mahoney \cite{10.1145/2488608.2488621}. 


When $S$ is an $s$-hashing matrix, with $s>1$, the tight bound $m=\Theta(r^2)$ can be improved to 
$m=\Theta (r\log r)$ for $s$ sufficiently large. In particular, Cohen \cite{10.5555/2884435.2884456} derived a general upper bound  that implies, for example, subspace embedding properties of 
$s$-hashing matrices provided $m=\mathcal{O}(r\log r)$ and $s=\mathcal{O}(\log r)$; the value of $s$ may be further reduced to a constant (that is not equal to $1$) at the expense of increasing $m$ and worsening its dependence of $d$. 
A lower bound for guaranteeing oblivious embedding properties of $s$-hashing matrices is given in \cite{Nelson:2014uu}.
Thus we can see that for $s$-hashing (and especially for $1$-hashing) matrices, their general subspace embedding properties are suboptimal in terms of the dependence of $m$ on $d$ when compared to the Gaussian sketching results. To improve the embedding properties of hashing matrices, we must focus on special structure input matrices.

\subsubsection{Coherence-dependent embedding properties of sparse random matrices}

A  feature of the problematic matrix (\ref{eqn::problematic_matrix}) is that its rows are separated into two groups, with the first $r$ rows containing all the information. If the rows  of $A$ were more `uniform' in the sense of equally important in terms of relevant information content, hashing may perform better as a sketching matrix. Interestingly, it is not the uniformity of the rows of $A$ but the uniformity of the rows of $U$, the left singular matrix from the compact SVD of $A$, that plays an important role. The concept of coherence is a useful proxy for the uniformity of the rows of $U$ and $A$\footnote{We note that sampling matrices were shown to have good subspace embedding properties for input matrices with low coherence \cite{10.1145/1132516.1132597, Tropp:wr}.
Even if the coherence is minimal, the size of the sampling matrix has a $d\log d$ dependence where the $\log d $ term cannot be removed due to the coupon collector problem \cite{Tropp:wr}.}.

\begin{definition} (Matrix coherence \cite{10.1561/2200000035})
\label{def::coherence}
The coherence of a matrix $A \in \R^{n\times d}$, denoted $\mu(A)$, is the largest Euclidean norm of the rows of $U$ defined in (\ref{thin-SVD}). Namely,
\begin{align}
\mu(A) = \max_{i\in [n]} \|U_i\|_2,
\end{align}
where  $U_i$ denotes the $i$th row of $U$.\footnote{\reply{Note that the concept of coherence is different to the (in)coherence used in compressed sensing literature \cite{DonohoMutualCoherence}, in particular our notion of coherence is not invariant under a different coordinate representation. }}
\end{definition}

Some useful properties follow.

\begin{lemma}
Let $A \in \R^{n\times d}$ have rank $r\leq d\leq n$. Then
\begin{equation} \label{mu_bound}
\sqrt{\frac{r}{n}} \leq \mu(A) \leq 1.
\end{equation}
Furthermore, if $\mu(A) = \sqrt{\frac{r}{n}}$, then $\|U_i\|_2= \sqrt{\frac{r}{n}}$ for all $i\in [n]$ where $U$ is defined in \eqref{thin-SVD}.
\end{lemma}

\begin{proof}
Since the matrix $U \in \R^{n\times r}$ has orthonormal columns, we have that
\begin{equation}
    \sum_{i=1}^{n} \|U_i\|_2^2 = r. \label{eqn:sum_of_U_i}
\end{equation} Therefore the maximum 2-norm of $U$ must not be less than $\sqrt{\frac{r}{n}}$, and thus $\mu(A) \geq \sqrt{\frac{r}{n}}$. Furthermore, if $\mu(A) = \sqrt{\frac{r}{n}}$, then \eqref{eqn:sum_of_U_i} implies $\|U_i\|_2 = \sqrt{\frac{r}{n}}$ for all $i \in [n]$.

Next, by expanding the set of columns of $U$ to a basis of $R^n$, there exists $U_f \in \R^{n \times n}$ such that $U_f = \bigl(  U \; \hat{U}  \bigr)$ orthogonal where $\hat{U} \in \R^{n \times (n-d)}$ has orthonormal columns. The 2-norm of $i$th row of $U$ is bounded above by the 2-norm of $i$th row of $U_f$, which is one. Hence $\mu(A) \leq 1$. 
\end{proof}

We note that for $A$ in (\ref{eqn::problematic_matrix}), we have $\mu(A) = 1$. The  maximal coherence of this matrix sheds some light on the ensuing poor embedding properties we noticed in Example 1.

Bourgain et al \cite{Bourgain:2015tc}  gives a general result that captures the coherence-restricted subspace embedding properties 
of $s$-hashing matrices.

\begin{theorem}[Bourgain et al \cite{Bourgain:2015tc}] \label{Bourgain}
Let $A \in \R^{n \times d}$ with coherence $\mu(A)$  and rank $r$; and  let $0 < \epsilon, \delta <1$.
Assume also that
\begin{align}
m \geq c_1\max\left\{\delta^{-1}, \left[ (r + \log m )  \min \left\{ \log^2(r/\epsilon), \log^2(m)
\right\} + r\log(1/\delta) \right]\epsilon^{-2}\right\} \label{Bourgain-m}\\
{\rm and}\quad s \geq c_2 \left[ \log(m) \log(1/\delta) \min \left\{ \log^2(r/\epsilon), \log^2(m)
\right\} + \log^2(1/\delta) \right] \mu(A)^2 \epsilon^{-2},\label{Bourgain_s_eqn} 
\end{align}
where $c_1$ and $c_2$ are  positive constants.
Then  a(ny) s-hashing matrix $S \in \R^{m\times n}$ is an $\epsilon$-subspace embedding for  $A$  with probability at least $1-\delta$.

\end{theorem}
Substituting $s=1$ in (\ref{Bourgain_s_eqn}), we can use the above Theorem to deduce an upper bound $\mu$ of acceptable  coherence values of the input matrix $A$, namely,
\begin{align}
\mu(A)\leq c_2^{-1/2} \epsilon \left[ \log(m) \log(1/\delta) \min \left\{ \log^2(r/\epsilon), \log^2(m)
\right\} + \log^2(1/\delta) \right]^{-1/2} :=\mu.\label{Bourgain-1-hashing-mu}
\end{align}
Thus Theorem \ref{Bourgain} implies that the distribution of $1$-hashing matrices with
$m$ satisfying (\ref{Bourgain-m})  is 
 an oblivious subspace embedding for any input matrix $A$ with 
 $\mu(A)\leq \mu$, where  $\mu$ is defined in \eqref{Bourgain-1-hashing-mu}.

\subsubsection{Non-uniformity of vectors and their relation to embedding properties of sparse random matrices}

In order to prove some of our main results, we need a corresponding notion of coherence of vectors, to be able to measure the `importance' of their respective entries; this is captured by the so-called non-uniformity of a vector.

\begin{definition}[Non-uniformity of a vector]\label{def:non-uniform-vector}
Given $x \in \R^{n}$, the non-uniformity of $x$, $\nu(x)$, is defined as
\begin{align}
\nu(x) = \frac{\|x\|_{\infty}}{\|x\|_2}. 
\end{align}
\end{definition}

We note that for any vector $x \in \R^{n}$, we have $\frac{1}{\sqrt{n}} \leq \nu(x) \leq 1$.

\begin{lemma}\label{non_uniformity_col_subspace_coherence}
Given $A \in \R^{n\times d}$, let $y=Ax$ for some $x \in \R^{d}$. Then
\begin{align}
\nu(y) \leq \mu(A).
\end{align}
\end{lemma}
\begin{proof}
Let $A = U\Sigma V^T$ be defined as in \eqref{thin-SVD}, and let $z = \Sigma V^T x \in \R^{r}$. Then $y = Ax = Uz$.
Therefore 
\begin{align}
\|y\|_{\infty} = \|Uz\|_{\infty} = \max_{1\leq i \leq n} | \langle U_i, z \rangle| \leq \max_{1 \leq i \leq n} \|U_i\|_2 \|z\|_2 \leq  \mu(A) \|z\|_2,
\end{align}
where $U_i$ denotes the $i^{th}$ row of $U$.
Furthermore, $\|y\|_2 = \|Uz\|_2 = \|z\|_2$ which then implies $\nu(y) = \|y\|_{\infty} / \|y\|_2 \leq \mu(A)$.
\end{proof}

The next lemmas are crucial to our results in the next section; the proof of the first lemma can be found in the paper \cite{10.5555/3327345.3327444}.

We also note, in subsequent results, the presence of {\it problem-independent constants}, also called absolute constants that will be implicitly or explicitly defined, depending on the context. Our convention here is as expected, that the same notation denotes the same constant across all results in this chapter. 

The following expression will be needed in our results,
\begin{equation}
\bar{\nu}(\epsilon,\delta):= \nuOneHashing,
      \label{L5:nu}
\end{equation}
where $\epsilon, \delta \in (0,1)$ and $E, C_1>0$.

\begin{lemma}[\cite{10.5555/3327345.3327444}, Theorem 2] \label{Freksen}
Suppose that $\epsilon, \delta \in (0,1)$, and $E$ satisfies $C \leq E < \frac{2}{\delta \log(1/\delta)}$, where $C>0$ and $C_1$ are  problem-independent constants. Let 
$m \leq n \in \N$ with
$m\geq E \epsilon^{-2} \log(1/\delta)$.

Then,  for any $x\in \R^n$ with 
\begin{equation}
    \nu(x) \leq\bar{\nu}(\epsilon,\delta),
    \label{L5:nu-1}
\end{equation}
where $\bar{\nu}(\epsilon,\delta)$ is defined in \eqref{L5:nu},
a randomly generated 1-hashing matrix $S\in \R^{m\times n}$ is an $\epsilon$-JL embedding for $\{x\}$ with probability at least $1-\delta$.

\end{lemma}

\begin{lemma} \label{point to JL plus}
Let $\epsilon, \delta \in (0,1)$, $\nu \in (0,1]$ and $m,n \in \N$. Let $\cal{S}$ be a distribution of $m\times n$ random matrices. Suppose that for any given 
y with $\nu(y) \leq \nu$, a matrix $S \in \R^{m\times n}$ randomly drawn from $\cal{S}$ is an $\epsilon$-JL embedding for $\{y\}$ with probability at least $1-\delta$.
Then for any given set $Y\subseteq \R^n$ with $\max_{y \in Y} \nu(y) \leq \nu$ and cardinality $|Y|\leq 1/\delta$, a matrix $S$ randomly drawn from $\cal{S}$ is an $\epsilon$-JL embedding for Y with probability at least $1-|Y|\delta$.
\end{lemma}
\begin{proof}
Let $Y = \yExpression$. Let $B_i$ be the event that $S$ is an $\epsilon$-JL embedding for $y_i\in Y$. Then $\probability{B_i} \geq 1-\delta$ by  assumption. We have
\begin{align}
    \probability{\text{S is an $\epsilon$-JL embedding for Y}}=
    \probability{\cap_i B_i} & = 1 - \probability{\complement{\cap_i B_i}} \\
    & = 1 - \probability{\cup_i B_i^c} \\
    & \geq 1 - \sum_i \probability{B_i^c} \\
    & = 1 - \sum_i \squareBracket{1 - \probability{B_i}}\\
    & \geq 1 - \sum_i \squareBracket{1 - (1-\delta)} = 1-|Y|\delta.
\end{align}
\end{proof}

\begin{lemma}\label{lem::lemma7}
Let $\epsilon\in (0,1)$, and $Y \subseteq \R^n$ be a finite set such that
$\|y\|_2=1$ for each $y\in Y$. Define 
\begin{align}
    & \yPlus= \yPlusExpression \\
    & \yMinus = \yMinusExpression.
\end{align} 
If $S \in \R^{m\times n}$ is an $\epsilon$-JL embedding for $\set{ \union{\yPlus}{\yMinus} }$, then $S$ is a generalised $\epsilon$-JL embedding for $Y$.
\end{lemma}
\begin{proof}
Let $y_1, y_2 \in Y$. We have that
\begin{align}
\left| \la Sy_1, Sy_2 \ra - \la y_1, y_2 \ra \right| &=  | ( \|S(y_1+y_2)\|^2 - \|S(y_1-y_2)\|^2 ) \\\nonumber
&\quad\quad - ( \|(y_1+y_2)\|^2 - \|(y_1-y_2)\|^2 ) | /4 \\\nonumber
&\leq \epsilon( \|(y_1+y_2)\|^2 + \|(y_1 -y_2)\|^2)/4 \\\nonumber
& = \epsilon( \|y_1\|^2 + \|y_2\|^2)/2 \\\nonumber
& = \epsilon,
\end{align}
where to obtain the inequality, we use that $S$ is an $\epsilon$-JL embedding for $\set{\union{\yPlus}{\yMinus}}$;  the last equality follows from $\|y_1\|_2=\|y_2\|_2=1$.
\end{proof}

    \section{Hashing sketching with \mathInTitle{$m=\mathcal{O}(r)$}}

Our first result shows that if the coherence of the input matrix is sufficiently low, the distribution of $1$-hashing matrices with $m = \mathcal{O}(r)$ is an $(\epsilon,\delta)$-oblivious subspace embedding.

The following expression will be useful later,
\begin{equation}\label{A:mu-1}
\bar{\mu}(\epsilon,\delta):=  \muAOneHashing,
\end{equation}
where $\epsilon, \delta \in (0,1)$,  $r, E>0$ are to be chosen/defined depending on the context, and $C_1, C_2>0$ are problem-independent constants.

\begin{theorem} \label{thm1}
Suppose that $\epsilon,\delta \in (0,1)$, $r \leq d \leq n, m\leq n \in \N$, $E >0$ satisfy 
\begin{align}
&C \leq E \leq \EUpperHashing,\label{eqn::theoremTwoOne}\\[0.5ex]
&m \geq \mLower,\label{eqn::theoremTwoTwo}
\end{align}
where $C>0$ and $C_1, C_2>0$  are problem-independent constants. Then for any matrix $A\in\R^{n\times d}$ with rank $r$ and 
\begin{equation}\label{A:mu}
    \mu(A) \leq  \bar{\mu}(\epsilon,\delta),
\end{equation}
where $\bar{\mu}(\epsilon,\delta)$ is defined in \eqref{A:mu-1}, a randomly generated 1-hashing matrix $S\in\R^{m\times n}$ is an $\epsilon$-subspace embedding for $A$ with probability at least $1-\delta$.
\end{theorem}

The proof of Theorem \ref{thm1} 
relies on the fact that the coherence of the input matrix gives a bound on the non-uniformity of the entries for all vectors in its column space (Lemma \ref{non_uniformity_col_subspace_coherence}),
adapting standard arguments in \cite{10.1561/0400000060} involving set covers.

\begin{definition}
A $\gamma$-cover of a set $M$ is a subset $N\subseteq M$ with the property that given any point $y \in M$, there exists $w \in N$ such that
 $\|y-w\|_2 \leq \gamma$. 
\end{definition}

Consider a given real  matrix $U\in\R^{n\times r}$ with orthonormal columns, and let
\begin{equation}\label{MU}
M:= \{Uz \in \R^n: z \in \R^r, \|z\|_2=1\}.
\end{equation} 

The next two lemmas show the existence of a $\gamma$-net $N$ for $M$, and connect generalised JL embeddings for $N$ with JL embeddings for $M$.

\begin{lemma} \label{Gamma-cover-existance}

Let $0< \gamma <1$, $U\in\R^{n\times r}$ have orthonormal columns and $M$ be defined in \eqref{MU}. Then there exists a $\gamma$-cover $N$ of $M$ such that $|N| \leq (1 + \frac{2}{\gamma})^r$.
\end{lemma}

\begin{proof}
Let $\tilde{M} = \{z \in \R^r: \|z\|_2=1\}$. Let $\tilde{N} \subseteq \tilde{M}$ be the maximal set such that no two points in $\tilde{N}$ are within distance $\gamma$ from each other. Then it follows that the r-dimensional balls centred at points in $\tilde{N}$ with radius $\gamma/2$ are all disjoint and contained in the r-dimensional ball centred at the origin with radius $(1+\gamma/2)$. Hence
\begin{align}
\frac{\text{Volume of the r-dimensional ball centred at the origin with radius $(1+\gamma/2)$}}{\text{Total volume of the r-dimensional balls centred at points in $\tilde{N}$ with radius $\gamma/2$ }} \\ 
= \frac{1}{|\tilde{N}|} \frac{(1+\frac{\gamma}{2})^{r}}{(\frac{\gamma}{2})^{r}} \geq 1,
\end{align}
which implies $|\tilde{N}| \leq (1+\frac{2}{\gamma})^r$.

Let $N = \{Uz \in \R^n: z \in \tilde{N}\}$. Then $|N| \leq |\tilde{N}| \leq (1+\frac{2}{\gamma})^r$ and we show $N$ is a $\gamma$-cover for $M$. Given $y_M \in M$, there exists $z_M \in \tilde{M}$ such that $y_M = Uz_M$. By definition of $\tilde{N}$, there must be $z_N \in \tilde{N}$ such that $\|z_M - z_N\|_2 \leq \gamma$ as otherwise $\tilde{N}$ would not be maximal. 
Let $y_N = Uz_N \in N$. Since $U$ has orthonormal columns, we have $\|y_M -y_N\|_2 = \|z_M-z_N\|_2 \leq \gamma$.

\end{proof}

\begin{lemma} \label{approximation-of-net}
Let $\epsilon, \gamma \in (0,1), U\in \R^{n\times d}, M \subseteq \R^n$ associated with $U$ be defined in \eqref{MU}. Suppose $N \subseteq M$ is a $\gamma$-cover of $M$ and $S \in \R^{m\times n}$ is a generalised $\epsilon_1$-JL embedding for $N$, where $\epsilon_1 = \epsilonPrimeFactor \epsilon$. Then $S$ is an $\epsilon$-JL embedding for $M$. 
\end{lemma}

To prove Lemma \ref{approximation-of-net}, we need the following Lemma. 

\begin{lemma} \label{approximation_unit_ball}
Let $\gamma \in (0,1)$, $U\in\R^{n\times r}$ having orthonormal columns and $M\subseteq \R^{n}$ associated with $U$ be defined in \eqref{MU}. Let $N$ be a $\gamma$-cover of $M$, $y\in M$. Then for any $k \in \N$, there exists 
 $\alpha_0, \alpha_1, \dots, \alpha_k \in \R$, $y_0, y_1, y_2, \dots, y_k \in N$ such that
\begin{align}
 \|y - \sum_{i=0}^k \alpha_i y_i \|_2 \leq \gamma^{k+1}, \label{A-1-1}\\
 |\alpha_i| \leq \gamma^i, i=0,1,\dots, k. \label{A-1-2}
 \end{align}

\end{lemma}

\begin{proof}
We use induction. Let $k=0$. Then by definition of a $\gamma$-cover, there exists $y_0 \in N$ such that $\|y - y_0\| < \gamma$. Letting $\alpha_0=1$, we have covered the $k=0$ case.

Now assume \eqref{A-1-1} and \eqref{A-1-2} are true for $k = K \in \N$. Namely there exists $\alpha_0, \alpha_1, \dots, \alpha_K \in \R$, $y_0, y_1, y_2, \dots, y_K \in N$ such that
\begin{align}
 \|y - \sum_{i=0}^K \alpha_i y_i \|_2 \leq \gamma^{K+1} \\
 |\alpha_i| \leq \gamma^i, i=0,1,\dots, K.
 \end{align}
Because $y, y_0, y_1 \dots, y_K \in N \subseteq M$, there exists $z, z_0, z_1, \dots, z_K \in \R^r$ such that $y=Uz, y_0=Uz_0, y_1=Uz_1, \dots y_K=Uz_K$ with $\|z\| = \|z_0\| = \dots = \|z_K\|=1$. Therefore

\begin{align}
 &\frac{y - \sum_{i=0}^K \alpha_i y_i}{ \|y - \sum_{i=0}^K \alpha_i y_i\|_2} =
 \frac{ U \left( z - \sum_{i=0}^K \alpha_i z_i \right)  }{\|U \left( z - \sum_{i=0}^K \alpha_i z_i \right) \|_2} = 
 U \frac{z - \sum_{i=0}^K \alpha_i z_i}{ \|z - \sum_{i=0}^K \alpha_i z_i\|_2} \in M,
\end{align}
where we have used that the columns of $U$ are orthonormal. 

Since $N$ is a $\gamma$-cover for $M$,  there exists $y_{K+1} \in N$ such that

\begin{align}
 \left\| \frac{y - \sum_{i=0}^K \alpha_i y_i}{ \|y - \sum_{i=0}^K \alpha_i y_i\|_2} - y_{K+1} \right\|_2 \leq \gamma.
\end{align}

Multiplying both sides by $\alpha_{K+1}:=\|y - \sum_{i=0}^K \alpha_i y_i\|_2 \leq \gamma^{K+1}$, we have

\begin{align}
  \|y - \sum_{i=0}^{K+1} \alpha_i y_i \|_2 \leq \gamma^{K+2}, \\
 |\alpha_i| \leq \gamma^i, i=0,1, \dots, K+1.
\end{align}
 
\end{proof}

\begin{proof}[Proof of Lemma \ref{approximation-of-net}]
Let $y \in M$ and $k\in \N$, and consider the approximate representation of $y$ provided in Lemma \ref{approximation_unit_ball}, namely, assume that 
\eqref{A-1-1} and \eqref{A-1-2} hold. Then we have

\begin{align*}
 \|S \sum_{i=0}^k \alpha_i y_i\|_2^2 
		   & = \sum_{i=0}^k \|S\alpha_i y_i\|_2^2 + \sum_{0 \leq i < j \leq k} 2 \langle S\alpha_i y_i, S\alpha_j y_j \rangle \\
		   & = \sum_{i=0}^k \|S\alpha_i y_i\|_2^2 + \sum_{0 \leq i < j \leq k} 2 \langle \alpha_i y_i, \alpha_j y_j \rangle +\\
		   &\hspace*{2.5cm}
		   		+\left[  \sum_{0 \leq i < j \leq k} 2 \langle S\alpha_i y_i, S\alpha_j y_j \rangle - \sum_{0 \leq i < j \leq k} 2 \langle \alpha_i y_i, \alpha_j y_j \rangle\right] \\
		   & \leq (1+\epsilon_1) \sum_{i=0}^k \|\alpha_i y_i\|_2^2 +\sum_{0 \leq i < j \leq k} 2 \langle \alpha_i y_i, \alpha_j y_j \rangle +
 				2 \sum_{0 \leq i <j \leq k} \epsilon_1 \|\alpha_i y_i\|_2 \|\alpha_j y_j\|_2 \\
 		   & = \|\sum_{i=0}^k \alpha_i y_i\|^2_2 + \epsilon_1 \left[
 		   \sum_{i=0}^k \|\alpha_i y_i\|_2^2 + 2 \sum_{0 \leq i <j \leq k}  \|\alpha_i y_i\|_2 \|\alpha_j y_j\|_2
 		   \right],
\end{align*}
where to deduce the inequality, we use that $S$ is a generalised $\epsilon_1$-JL embedding for $N$. Using $\|y_i\|_2=1$ and $|\alpha_i| \leq \gamma^i$, we have
\begin{align}\label{L9:gamma_i}
\frac{1}{\epsilon_1}\left\{\|S \sum_{i=0}^k \alpha_i y_i\|_2^2 - \|\sum_{i=0}^k \alpha_i y_i\|^2_2\right\} &=
    \sum_i^k \|\alpha_i y_i\|_2^2 + 2 \sum_{0 \leq i <j \leq k} \|\alpha_i y_i\|_2 \|\alpha_j y_j\|_2 \\\nonumber
   & \leq \sum_{i=0}^{k} \gamma^i + 2 \sum_{0 \leq i <j \leq k} \gamma^i \gamma^j\\\nonumber
   &\leq \frac{1-\gamma^{k+1}}{1-\gamma}+\frac{2\gamma (1-\gamma^{k-i})(1-\gamma^{2k})}{\left( 1-\gamma \right) \left( 1-\gamma^2 \right)},
\end{align}
where we have used 
\begin{align*}
    \sum_{0 \leq i <j \leq k} \gamma^i \gamma^j = 
\sum_{i=0}^{k-1} \gamma^i \sum_{j=i+1}^{k} \gamma^j &=
\sum_{i=0}^{k-1} \gamma^{2i+1} \sum_{j=0}^{k-i-1} \gamma^j \\
&=
\frac{\gamma(1-\gamma^{k-i})}{1-\gamma} \sum_{i=0}^{k-1} \gamma^{2i} = 
\frac{\gamma (1-\gamma^{k-i})(1-\gamma^{2k})}{ \left( 1-\gamma \right) \left( 1-\gamma^2 \right) }.
\end{align*} 
Letting $k\rightarrow \infty$ in \eqref{L9:gamma_i}, we deduce
\[
\frac{1}{\epsilon_1}\left\{\|S \sum_{i=0}^{\infty} \alpha_i y_i\|_2^2 - \|\sum_{i=0}^{\infty} \alpha_i y_i\|^2_2\right\}\leq 
\frac{1}{1-\gamma}+\frac{2\gamma}{\left( 1-\gamma \right) \left( 1-\gamma^2 \right)}=\frac{1+2\gamma-\gamma^2}{\left( 1-\gamma \right) \left( 1-\gamma^2 \right)},
\]
Letting $k\rightarrow \infty$ in  \eqref{A-1-1} implies 
 $y = \sum_{i=0}^\infty \alpha_i y_i$, and so the above gives
 \[
\|S y\|_2^2 - \|y\|^2_2\leq 
\epsilon_1\frac{1+2\gamma-\gamma^2}{\left( 1-\gamma \right) \left( 1-\gamma^2 \right)}=\epsilon \|y\|_2^2,
\]
where to get the equality, we used
 $\|y\|_2=1$ and the definition of $\epsilon_1$.
The lower bound in the $\epsilon_1$-JL embedding follows similarly.
\end{proof}

We are ready to prove Theorem \ref{thm1}.
\begin{proof}[Proof of Theorem \ref{thm1}]
Let $A\in \R^{n\times d}$ with rank $r$ and satisfying \eqref{A:mu}.
Let  $U\in \R^{n\times r}$ be an SVD factor of $A$ as defined in \eqref{thin-SVD}, which by definition of coherence, implies
\begin{equation}
    \mu(U)=\mu (A)\leq \bar{\mu}(\epsilon,\delta),
    \label{th2:barmu}
\end{equation}
where $\bar{\mu}(\epsilon,\delta)$ is defined in \eqref{A:mu-1}. 
We let $\gamma, \epsilon_1, \delta_1 \in (0,1)$ be defined as
  \begin{equation}
        \gamma=\frac{2}{e^2-1}, \quad C_2 =  \epsilonPrimeFactor,\quad 
        \epsilon_1 = C_2 \epsilon \quad {\text{and}}\quad \delta_1 = e^{-4r}\delta,
    \label{th2:eps}
  \end{equation}
and note that $C_2\in (0,1)$ and 
\begin{equation}
\bar{\nu}(\epsilon_1,\delta_1) = \bar{\mu} (\epsilon,\delta),
\label{th2:numu}
\end{equation}
where $\bar{\nu}(\cdot,\cdot)$ is defined in \eqref{L5:nu}.
Let $M \in \R^n$ be associated to $U$ as in \eqref{MU}
and let $N\subseteq M$ be the $\gamma$-cover of $M$  as guaranteed by Lemma \ref{Gamma-cover-existance}, with $\gamma$ defined in \eqref{th2:eps} which implies that $|N| \leq e^{2r}$.

Let $S \in \R^{m\times n}$ be a randomly generated 1-hashing matrix with $m\geq E \epsilon_1^{-2}\log (1/\delta_1)=E C_2^{-2}\epsilon^{-2}[4r+\log(1/\delta)]$, where to obtain the last equality, we used \eqref{th2:eps}.

To show that the sketching matrix $S$ is an $\epsilon$-subspace embedding for $A$ (with probability at least $1-\delta$),
it is sufficient to show that $S$ is an $\epsilon_1$-generalised JL embedding for $N\subseteq M$ (with probability at least $1-\delta$). To see this, recall \eqref{MU} and Lemma \ref{subspace-embedding-def-2}(ii) which show that $S$ is an $\epsilon$-subspace embedding for $A$ if and only if 
$S$ is an $\epsilon-$JL embedding for $M$. Our sufficiency claim now follows by invoking Lemma \ref{approximation-of-net} for $S$, $N$ and $M$.

We are left with considering in detail the cover set $N = \set{y_1,y_2, \dots, y_{|N|}}$ and the following useful ensuing  sets
\begin{align*}
    &\nPlus = \nPlusExpression \quad
    \text{and}\quad \nMinus = \nMinusExpression, \\\nonumber
    &\nMinusOne = \nMinusOneExpression \quad \text{and}\quad
    \nMinusTwo = \nMinusTwoExpression.
\end{align*}
Now let  $Y := \nPlus \cup \nMinusOne$  and show that 
\begin{equation}
    \nu(y)\leq \bar{\nu}(\epsilon_1,\delta_1) \quad\text{for all}\quad y\in Y.
  \label{th2:nu-barnu}
\end{equation}
 To see this, assume first that $y=y_i+y_j\in \nPlus$, with 
$y_i, y_j \in N\subseteq M$. Thus there exist $z_i, z_j\in R^r$ such that 
$y_i=Uz_i$ and $y_j=Uz_j$, and so $y=U(z_i+z_j)$. Using Lemma \ref{non_uniformity_col_subspace_coherence}, $\nu(y)\leq \mu(U)=\mu(A)$, which together with \eqref{th2:barmu}
and \eqref{th2:numu}, gives \eqref{th2:nu-barnu} for points $y\in\nPlus$;
the proof for $y\in \nMinusOne$ follows similarly. 

Lemma \ref{Freksen} with $(\epsilon,\delta):= (\epsilon_1,\delta_1)$ provides that for any $x\in\R^n$ with $\nu(x) \leq \bar{\nu}(\epsilon_1,\delta_1)$, $S$ is an $\epsilon_1$-JL embedding for $\{x\}$ with probability at least $1-\delta_1$.  This and  \eqref{th2:nu-barnu} imply that the conditions of Lemma \ref{point to JL plus} are satisfied for $Y = \nPlus \cup \nMinusOne$, from which we conclude that 
$S$ is an $\epsilon_1$-JL embedding for $Y$ with probability at least $1-|Y|\delta_1$. Note that 
\[
|Y|\leq |\nPlus|+|\nMinusOne|\leq \frac{1}{2} |N|(|N|+1) +  \frac{1}{2}
 |N|(|N|-1)=|N|^2.
 \]
 This, the definition of $\delta_1$ in \eqref{th2:eps} and $|N|\leq e^{2r}$ imply that $1-|Y|\delta_1\geq 1-\delta$. 
Therefore $S$ is an $\epsilon_1$-JL embedding for $\nPlus \cup \nMinusOne$ with probability at least $1-\delta$.

Finally, Definition \ref{def::JL_embedding} of JL-embeddings implies that the sign of the embedded vector is irrelevant and that $\{0\}$ is always embedded, and so if
$S$ is an $\epsilon_1$-JL embedding for $\nPlus \cup \nMinusOne$, it is also an $\epsilon_1$-JL embedding for $\union{\nPlus}{\nMinus}$.
Lemma \ref{lem::lemma7} now provides us with the desired result that then, 
$S$ is a generalised $\epsilon_1$-JL embedding for $N$. 
\end{proof}

Next we discuss the results in Theorem \ref{thm1}.

\paragraph{Conditions for a well-defined coherence requirement}

While our result guarantees optimal dimensionality reduction for the sketched matrix, using a very sparse 1-hashing matrix for the sketch, it imposes  implicit restrictions on the number $n$ of rows of $A$. Recalling 
\eqref{mu_bound}, we note that condition \eqref{A:mu} is well-defined
when 
\begin{equation}\label{th2:nr-mu}
\sqrt{\frac{r}{n}}\leq \bar{\mu}(\epsilon,\delta).
\end{equation}
Using the definition of $\bar{\mu}(\epsilon,\delta)$ in \eqref{A:mu-1} and
assuming reasonably that $\logDeltaI=\mathcal{O}(r)$, 
we have the lower bound

\begin{equation*}
\bar{\mu}(\epsilon,\delta) \geq C_1\sqrt{C_2}\sqrt{\epsilon} \frac{\min\left\{ \log(E/(C_2\epsilon)), \sqrt{\log E}\right\}}{4r+\log(1/\delta)},
\end{equation*}
and so
\eqref{th2:nr-mu} is satisfied if
\begin{align}
n \geq \frac{r (4r+\log(1/\delta)^2}{C_1^2C_2\epsilon\min\left\{ \log^2(E/(C_2\epsilon)), \log E\right\}} = \mathcal{O} \left(  \frac{r^3}{\epsilon \log^2(\epsilon)} \right).
\end{align}

\paragraph{Comparison with data-independent bounds}
    Existing results show that $m = \Theta \left( r^2\right)$ is both necessary and sufficient in order to secure an oblivious subspace embedding property for 1-hashing matrices with no restriction on the coherence of the input matrix \cite{10.1145/2488608.2488622,Nelson:te, 10.1145/2488608.2488621}. Aside from requiring more projected rows than in Theorem \ref{thm1}, these results implicitly impose $n \geq \mathcal{O}(r^2)$ for the size/rank of data matrix in order to secure  meaningful dimensionality reduction.

    \paragraph{Comparison with data-dependent bounds}
    To the best of our knowledge, the only data-dependent result for hashing matrices is   \cite{Bourgain:2015tc} (see Theorem \ref{Bourgain}). From \eqref{Bourgain-m}, we have that $m\geq c_1 r \min\left\{ \log^2(r/\epsilon), \log^2(m) \right\}\epsilon^{-2}$ and hence $m = \Omega(r \log^2 r)$;  while Theorem \ref{thm1} only needs $m = \mathcal{O}(r)$. However, the coherence requirement on $A$ in Theorem \ref{Bourgain} is weaker than \eqref{A:mu} and so \cite{Bourgain:2015tc}  applies to a wider range of inputs at the expense of a larger value of $m$ required for the sketching matrix.
    
    \paragraph{Summary and look ahead}
Table \ref{tab::m_and_mu_1_hashing} summarises existing results and we see stricter coherence assumptions lead to improved dimensionality reduction properties. In the next section, we investigate relaxing coherence requirements by using hashing matrices with increased column sparsity ($s$-hashing) and coherence reduction transformations. 

    \begin{table}
        \caption{Summary of results for $1$-hashing}
        \label{tab::m_and_mu_1_hashing}
        \centering
\begin{tabular}{|c|c|c|}
\hline
Result                     & \mbox{$\mu$ (coherence of $A$)}                         & \mbox{$m$ (size of sketching $S$)}                        \\ \hline
\cite{10.1145/2488608.2488621} & --                              & $\Theta(r^2)$                \\ \hline
\cite{Bourgain:2015tc}         & $\mathcal{O} \left( \log^{-3/2} (r)\right)$ & $\mathcal{O}\left(r \log^2(r) \right)$ \\ \hline
Theorem \ref{thm1}             & $\mathcal{O} \left( r^{-1}\right)$        & $\mathcal{O}(r)$                       \\ \hline
\end{tabular}
    \end{table}

    \section{Relaxing the coherence requirement using
    \mathInTitle{$s$}-hashing matrices}\label{s-hashing}

This section investigates the embedding properties of  $s$-hashing matrices when $s\geq 1$. 
Indeed, \cite{Bourgain:2015tc} shows that $s$-hashing relaxes their particular coherence requirement by $\sqrt{s}$. 
Theorem \ref{thm::s-hashing} presents a similar result for our particular coherence requirement  \eqref{A:mu} that again guarantees embedding properties for $m=\mathcal{O}(r)$. Then we present a new $s$-hashing variant that allows us to give a general result showing that (any) subspace embedding properties of $1$-hashing matrices immediately translate into similar properties for these $s$-hashing matrices when applied to a larger class of data matrices, with larger coherence.  
A simplified embedding result with $m=\mathcal{O}(r)$ is then deduced for this $s$-hashing variant.
Finally, $s$-hashing or $s$-hashing variant is combined with the randomised Hadamard transform with
Theorem \ref{thm::HRHT} and \ref{thm::HRHT_variant} guaranteeing embedding properties for $m=\mathcal{O}(r)$ given that the data matrix $A$ has $n = \mathO{r^3}$.

Numerical benefits of $s$-hashing (for improved preconditioning) are investigated in later sections; see Figures \ref{fig::1-2-3-inco} and \ref{fig::1-2-3-semi-co} for example.

\subsection{The embedding properties of \mathInTitle{$s$}-hashing matrices}
Our next result shows that using $s$-hashing (Definition \ref{def::sampling_and_hashing}) relaxes the particular coherence requirement in Theorem \ref{thm1} by $\sqrt{s}$. 

\begin{theorem} \label{thm::s-hashing} 
Let $ r \leq d\leq n \in \N^+$. Let  $C_1, C_2, C_3, C_M, C_{\nu}, C_s >0$ be \constantsDescription. Suppose that $\epsilon,\delta \in (0, C_3)$, 
$m,\,s\in \N^+$ and $E>0$ satisfy\footnote{Note that the expressions
of the lower bounds in \eqref{eqn::theoremTwoTwo} and \eqref{s:m_lower}
are identical apart from  the choice of $E$ and the condition $m \geq se$.}
\begin{align}
&1\leq s \leq C_s C_2^{-1} \epsilon^{-1} \fourRplusLog,\label{eqn::theoremThreeOne}\\[0.5ex]
& C_M \leq E \leq \EUpperSHashing,\label{eqn::theoremThreeTwo}\\[0.5ex]
&m \geq \set{\mLower, se}.\label{s:m_lower}
\end{align}
Then for any matrix $A\in \R^{n\times d}$ with rank $r$ and $\mu(A) \leq \sqrt{s}C_{\nu} C_1^{-1} \bar{\mu}(\epsilon,\delta)$, where $\bar{\mu}(\epsilon,\delta)$ is defined in \eqref{A:mu-1}, a randomly generated $s$-hashing matrix $S \in \R^{m\times n}$ is an $\epsilon$-subspace embedding for $A$ with probability at least $1-\delta$.
\end{theorem}

Theorem \ref{thm::s-hashing} parallels Theorem \ref{thm1}; and its proof relies on the following lemma which parallels Lemma \ref{Freksen}.

\begin{lemma}[\cite{NIPS2019_9656}, Theorem 1.5] \label{Jaga-lemma}
Let $C_1, C_3, C_M, C_{\nu}, C_s>0$ be \constantsDescription. Suppose that $\epsilon,\delta \in (0, C_3), m,s\in \N^+, E\in \R$ satisfy 
\begin{align}\nonumber
&1\leq s \leq C_s \epsilon^{-1} \logDeltaI,\\\nonumber
& C_M \leq E < \epsilon^2 s\log^{-1}(1/\delta) e^{C_s (\epsilon s)^{-1}\logDeltaI },\\\nonumber
& m\geq \max \set{ E \epsilon^{-2} \log(1/\delta), se}.
\end{align}
Then for any $x\in \R^n$ with $\nu(x) \leq \sqrt{s}C_{\nu}C_1^{-1}\bar{\nu}(\epsilon,\delta)$, where $\bar{\nu}(\epsilon,\delta)$ is defined in \eqref{L5:nu},
a randomly generated $s$-hashing matrix $S\in \R^{m\times n}$ is an $\epsilon$-JL embedding for $\{x\}$ with probability at least $1-\delta$.
\end{lemma}

The proof of Theorem \ref{thm::s-hashing} follows the same argument as Theorem \ref{thm1}, replacing $1$-hashing with $s$-hashing and using Lemma \ref{Jaga-lemma} instead of Lemma \ref{Freksen}. We omit the details.

\subsection{A general embedding property for an \mathInTitle{$s$}-hashing variant}

Note that in both Theorem \ref{Bourgain} and Theorem \ref{thm::s-hashing}, allowing column sparsity of hashing matrices to increase from $1$ to $s$ results in coherence requirements being relaxed by $\sqrt{s}$. We introduce an
$s$-hashing variant that allows us to generalise this result.

\begin{definition} 
\label{def::s-hashing-variant}
 We say $T \in \R^{m \times n}$ is an s-hashing variant matrix if independently for each $j \in [n]$, we sample with replacement $i_1, i_2, \dots, i_s \in [m]$ uniformly at random and add $\pm 1/\sqrt{s}$ to $T_{i_k j}$, where $k = 1, 2, \dots, s$. \footnote{We add $\pm 1/\sqrt{s}$ to $T_{i_k j}$ because we may have $i_k = i_l$ for some $l <k$, as we have sampled with replacement.  }
\end{definition}

Both $s$-hashing and $s$-hashing variant matrices reduce to $1$-hashing matrices when $s=1$. For $s\geq 1$, the $s$-hashing variant has at most $s$ non-zeros per column, while the usual $s$-hashing matrix has precisely $s$ nonzero entries per same column.

The next lemma connects $s$-hashing variant matrices to $1$-hashing matrices. 

\begin{lemma}
\label{decompose_s_hashing_varaint}
An $s$-hashing variant matrix $T\in \R^{m\times n}$ (as  in Definition \ref{def::s-hashing-variant}) could alternatively be generated by calculating $T = \frac{1}{\sqrt{s}} \left[ S^{(1)} + S^{(2)} + \dots + S^{(s)} \right]$, where $S^{(k)} \in \R^{m\times n}$ are independent $1$-hashing matrices for $k = 1, 2, \dots, s$.
\end{lemma}

\begin{proof}
In Definition \ref{def::s-hashing-variant}, an s-hashing variant matrix $T$ is generated by the following procedure:
\begin{algorithmic}
      \For{\texttt{$j = 1, 2, \dots n$}}
        \For{ \texttt{$k = 1, 2, \dots, s$}}
            \State \texttt{Sample $i_k \in [m]$ uniformly at random and add $\pm 1/\sqrt{s}$ to $T_{i_k,j}$. }
        \EndFor
      \EndFor
\end{algorithmic}
Due to the independence of the entries, the 'for' loops in the above routine can be swapped, leading to the equivalent formulation,
\begin{algorithmic}
      \For{\texttt{$k = 1, 2, \dots s$}}
        \For{ \texttt{$j = 1, 2, \dots, n$}}
            \State \texttt{Sample $i_k \in [m]$ uniformly at random and add $\pm 1/\sqrt{s}$ to $T_{i_k,j}$}.
        \EndFor
      \EndFor
\end{algorithmic}
For each $k\leq s$, the 'for' loop over $j$ in the above routine generates an independent random $1$-hashing matrix $S^{(k)}$ and adds $\left( 1/\sqrt{s} \right)S^{(k)}$ to $T$.

\end{proof}

We are ready to state and prove the main result in this section. 
\begin{theorem} \label{1-hashing-and-s-hashing}
Let $s, r \leq d\leq n \in \N^+$, $\epsilon, \delta\in (0, 1)$. Suppose that $m \in \N^+$ is chosen such that the distribution of $1$-hashing matrices $S \in \R^{m\times ns}$ is an $(\epsilon,\delta)$-oblivious subspace embedding for any matrix $B\in \R^{ns\times d}$ with rank $r$ and $\mu(B)\leq \mu$ for some $\mu>0$. Then the distribution of $s$-hashing variant matrices $T\in\R^{m\times n}$ is an $(\epsilon, \delta)$-oblivious subspace embedding for any matrix $A\in \R^{n\times d}$ with rank $r$ and  $\mu(A) \leq \mu\sqrt{s}$.

\end{theorem}

\begin{proof}
Applying Lemma \ref{decompose_s_hashing_varaint}, we let 
\begin{equation}
T = \frac{1}{\sqrt{s}} \left[ S^{(1)} + S^{(2)} + \ldots + S^{(s)} \right]
\end{equation}
be a randomly generated $s$-hashing variant matrix where $S^{(k)} \in \R^{m\times n}$ are independent $1$-hashing matrices, $k\in \{1,\ldots, s\}$. Let $A\in \R^{n\times d}$ with rank $r$ and  with $\mu(A)\leq \mu\sqrt{s}$; let
$U \in \R^{n \times r}$ be an SVD-factor of $A$ as defined in \eqref{thin-SVD}. Let 
\begin{equation}
W = \frac{1}{\sqrt{s}} \begin{pmatrix} U\\ \vdots \\ U \end{pmatrix} \in \R^{ns \times r}.
\end{equation}
As $U$ has orthonormal columns, the matrix $W$ also has orthonormal columns and hence the coherence of $W$ coincides with the largest Euclidean norm of its rows
	\begin{equation}
	\mu(W) = \frac{1}{\sqrt{s}} \mu(U)=\frac{1}{\sqrt{s}} \mu(A) \leq \mu.
	\end{equation}
Let $S = \begin{pmatrix} S^{(1)} \hdots  S^{(s)} \end{pmatrix} \in \R^{m \times ns}$.
We note that the $j$-th column of $S$ is generated by sampling $i \in [m]$ and setting $S_{ij} = \pm 1$. Moreover, as $S^{(k)}$,   $k\in \{1,\ldots, s\}$, are independent, the sampled entries are independent. Therefore, $S$ is distributed as a $1$-hashing matrix. Furthermore, due to our assumption on the distribution of $1$-hashing matrices, 
$m$ is chosen such that  $S \in \R^{m \times ns}$ is an $(\epsilon, \delta)$-oblivious subspace embedding for $(ns)\times r$ matrices of coherence at most $\mu$. Applying this to input matrix $W$, we have that 
 with probability at least $1-\delta$, 
	\begin{equation}\label{s-hash:SW}
	(1-\epsilon)\|z\|^2_2=	(1-\epsilon)\|Wz\|^2_2 \leq \|SWz\|^2_2 \leq (1+\epsilon)\|Wz\|^2_2=	(1+\epsilon)\|z\|^2_2,
	\end{equation}
	for all $z\in \R^r$, where in the equality signs, we used that  $W$ has orthonormal columns.
On the other hand, we have that 
\begin{equation*}
SW = \frac{1}{\sqrt{s}} \begin{pmatrix} S^{(1)} \hdots  S^{(s)} \end{pmatrix}  \begin{pmatrix} U\\ \vdots \\ U \end{pmatrix}  
   = \frac{1}{\sqrt{s}} \left[ S^{(1)} U + S^{(2)} U + \dots + S^{(s)} U \right]
   = TU. 
\end{equation*}
This and \eqref{s-hash:SW} provide that, with probability 
at least $1-\delta$, 
	\begin{equation}
	(1-\epsilon)\|z\|^2_2 \leq \|TUz\|^2_2 \leq (1+\epsilon)\|z\|^2_2,
	\end{equation}
	which implies that $T$ is an $\epsilon$-subspace embedding for $A$ by Lemma \ref{subspace-embedding-def-2}.
\end{proof}

 Theorem \ref{thm1} and Theorem \ref{1-hashing-and-s-hashing} imply an $s$-hashing variant version of Theorem \ref{thm::s-hashing}.

\begin{theorem}\label{thm::s-hashing-variant}
Suppose that $\epsilon,\delta \in (0,1)$, $s, r \leq d \leq n, m\leq n \in \N^+$, $E >0$ satisfy \eqref{eqn::theoremTwoOne} and \eqref{eqn::theoremTwoTwo}. Then for any matrix $A\in\R^{n\times d}$ with rank $r$ and 
$\mu(A) \leq \bar{\mu}(\epsilon,\delta)\sqrt{s}$, \whereMuBarIsDefined,
a randomly generated $s$-hashing variant matrix $S\in\R^{m\times n}$ is an $\epsilon$-subspace embedding for $A$ with probability at least $1-\delta$.
\end{theorem}

\begin{proof}
Theorem \ref{thm1} implies that the distribution of $1$-hashing matrices $S\in \R^{m\times ns}$ is an $(\epsilon,\delta)$-oblivious subspace embedding
for any matrix $B \in \R^{ns \times d}$ with rank $r$ and $\mu(B) \leq \bar{\mu}(\epsilon,\delta)$. We also note that this result is invariant to the number of rows in $B$ (as long as the column size of $S$ matches the row count of $B$), and so the expressions for $m$, $\bar{\mu}(\epsilon,\delta)$ and the constants therein remain unchanged. 
 
 Theorem \ref{1-hashing-and-s-hashing} then provides that the distribution of $s$-hashing variant matrices $S \in \R^{m\times n}$ is an $(\epsilon,\delta)$-oblivious subspace embedding for any matrix $A \in \R^{n\times d}$ with rank $r$ and $\mu(A) \leq \bar{\mu}(\epsilon,\delta) \sqrt{s}$; the desired result follows. 
\end{proof}

Theorem \ref{thm::s-hashing} and Theorem \ref{thm::s-hashing-variant} provide similar results, and we find that the latter provides simpler constant expressions (such as for $E$).

\subsection{The Hashed-Randomised-Hadamard-Transform sketching}

Here we consider the Randomised-Hadamard-Transform \cite{10.1145/1132516.1132597}, to be applied to the input matrix $A$ before sketching, as
another approach that allows reducing the coherence requirements under which  good subspace embedding properties can be guaranteed. It is common to use the Subsampled-RHT (SHRT) \cite{10.1145/1132516.1132597}, but the size of the sketch needs to be at least 
$\mathcal{O}(r\log r)$; this prompts us to consider using hashing instead of subsampling in this context (as well), and obtain an optimal order sketching bound. Figure \ref{fig::1-2-3-inco} illustrates numerically the benefit of HRHT sketching for preconditioning compared to SRHT.

\begin{definition}\label{def::HRHT}
A Hashed-Randomised-Hadamard-Transform (HRHT) is an $m\times n$ matrix of the form $S =  S_h HD$ with $m\leq n$, where
\begin{itemize}[topsep=0pt,itemsep=-1ex,partopsep=1ex,parsep=1ex]
    \item $D$ is a random $n \times n$ diagonal matrix with $\pm 1$ independent entries.
    \item $H$ is an $n\times n$ Walsh-Hadamard matrix defined by
        \begin{equation}
            H_{ij} = n^{-1/2}(-1)^{\la (i-1)_2, (j-1)_2\ra},
        \end{equation}
        where $(i-1)_2$, $(j-1)_2$ are binary representation vectors of the numbers $(i-1), (j-1)$ respectively\footnote{For example, $(3)_2 = (1,1)$.}.
    \item $S_h$ is a random $m\times n$ $s$-hashing or $s$-hashing variant matrix, independent of $D$.
\end{itemize}
\end{definition}

Our next results show 
that if the input matrix is sufficiently over-determined, the distribution of $HRHT$ matrices with optimal sketching size and either choice of $S_h$, is an $(\epsilon, \delta)$-oblivious subspace embedding.

\begin{theorem}[$s$-hashing version] \label{thm::HRHT}
$ r \leq d\leq n \in \N^+$. \theoremThreeFirstSentence \eqref{eqn::theoremThreeOne}, \eqref{eqn::theoremThreeTwo} and \eqref{s:m_lower}. Let $\delta_1 \in (0,1)$ and suppose further that
\begin{equation}
    n \geq \nLower \label{eqn::theoremSixN},
\end{equation}
\whereMuBarIsDefined. Then for any matrix $A \in \R^{n \times d}$ with rank $r$, an $HRHT$ matrix $S \in \R^{m\times n}$ with an $s$-hashing matrix $S_h$, is an $\epsilon$-subspace embedding for $A$ with probability at least $(1-\delta)(1-\delta_1)$. 

\end{theorem}

\begin{theorem}[$s$-hashing variant distribution] \label{thm::HRHT_variant}
\theoremFiveFirstSentence \eqref{eqn::theoremTwoOne} and \eqref{eqn::theoremTwoTwo}. Let $\delta_1 \in (0,1)$ and suppose further that 
\begin{equation}
    n \geq \nLowerSVariant \label{eqn::theoremSevenN},
\end{equation}
\whereMuBarIsDefined.
Then for any matrix $A \in \R^{n \times d}$ with rank $r$, an $HRHT$ matrix $S \in \R^{m\times n}$ with an $s$-hashing variant matrix  $S_h$,  is an $\epsilon$-subspace embedding for $A$ with probability at least $(1-\delta)(1-\delta_1)$. 
\end{theorem}

The proof of Theorem \ref{thm::HRHT} and Theorem \ref{thm::HRHT_variant} relies on the analysis in \cite{Tropp:wr} of Randomised-Hadamard-Transforms, which are shown to reduce the coherence of any given matrix with high probability.

\begin{lemma} \cite{Tropp:wr} \label{thm::Tropp_HD}
Let $r\leq n \in \N^+$ and $U \in \R^{n \times r}$ have orthonormal columns. Suppose that $H, D$ are defined in Definition \ref{def::HRHT} and $\delta_1 \in (0,1)$. Then $$\mu(HDU) \leq \sqrt{\frac{r}{n}} + \sqrt{\frac{8\log(n/\delta_1)}{n}} $$ with probability at least $1-\delta_1$.
\end{lemma}

We are ready to prove Theorem \ref{thm::HRHT}.

\begin{proof}[Proof of Theorem \ref{thm::HRHT} and Theorem \ref{thm::HRHT_variant}]
Let $A = U\Sigma V$ be defined in \eqref{thin-SVD}, $S = S_h HD$ be an HRHT matrix. Define the following events:
\begin{itemize}[topsep=0pt,itemsep=-1ex,partopsep=1ex,parsep=1ex]
    \item $B_1 = \left\{ \mu(HDU) \leq \muLower \right\}$,
    \item $B_2 = \left\{ \mu(HDA) \leq \muLower \right\}$, 
    \item $B_3 = \left\{ \mu(HDA) \leq \muHatSEpsDelta \right\}$,
    \item $B_4 = \left\{ \text{$S_h$ is an $\epsilon$-subspace embedding for $HDA$ } \right\}$,
    \item $B_5 = \left\{ \text{$S_h HD$ is an $\epsilon$-subspace embedding for $A$} \right\}$,
\end{itemize}
where $\muHatSEpsDelta = \muASHashing$ if $S_h$ is an $s$-hashing matrix and $\muHatSEpsDelta = \muASHashingVariant$ if $S_h$ is an $s$-hashing variant matrix, and \whereMuBarIsDefined.

Observe that $B_4$ implies $B_5$ because $B_4$ gives 
\begin{equation}
    (1-\epsilon) \|Ax\|^2 \leq (1-\epsilon)\|HDA x\|^2 \leq \| S_h HDA x\|^2 \leq (1+\epsilon) \|HDA x\|^2 \leq (1+\epsilon) \|Ax\|^2,
\end{equation}
where the first and the last equality follows from $HD$ being orthogonal. Moreover, observe that $B_1 = B_2$ because $\mu(HDA) = \max_i \| (HDU)_i \|_2 = \mu(HDU)$, where the first equality follows from $HDA = (HDU) \Sigma V^T$ being an $SVD$ of $HDA$. Furthermore, $B_2$ implies $B_3$ due to \eqref{eqn::theoremSixN} in the $s$-hashing case; and \eqref{eqn::theoremSevenN} in the $s$-hashing variant case.

Thus $\P(B_5) \geq \P(B_4) = \P(B_4 | B_3) \P(B_3) \geq P(B_4|B_3) \P(B_2) = \P(B_4 |B_3) \P(B_1)$. If $S_h$ is \anSHashingMat, Theorem \ref{thm::s-hashing} gives $\P(B_4 | B_3) \geq 1-\delta$. If $S_h$ is \anSHashingVariantMat, Theorem \ref{thm::s-hashing-variant} gives $\P(B_4 | B_3) \geq 1-\delta$. Therefore in both cases, we have
\begin{equation}
     P(B_5) \geq \P(B_4 |B_3) \P(B_1) \geq (1-\delta) \probability{B_1} \geq (1-\delta)(1-\delta_1),
\end{equation}
where the third inequality uses Lemma \ref{thm::Tropp_HD}.
\end{proof}

\chapter{Sketching for linear least squares}

    \section{Introduction and relevant literature}
    This chapter is based and expands materials in \cite{zhen:icml_LLS, 2021arXiv210511815C}.

\paragraph{Main contribution}
This chapter builds on the insight from the theoretical results in the last chapter to propose, analyse and benchmark a sketching based solver of \eqref{LLS-statement}. We first propose and analyse a rank-deficient generic sketching framework for \eqref{LLS-statement}, which includes the algorithm used by two previous sketching-based solvers, Blendenpik \cite{doi:10.1137/090767911} and LSRN \cite{Meng:2014ib} but additionally allows more flexibility of the choice of factorizations of the sketched matrix $SA$ for building a preconditioner for \eqref{LLS-statement}. Our analysis shows that under this framework, one can compute a minimal residual solution of \eqref{LLS-statement} with sketching if a rank-revealing factorization is used; or the minimal norm solution of \eqref{LLS-statement} if a total orthogonal factorization is used. Next, based on this algorithmic framework, we propose \solverName{}, a software package for solving \eqref{LLS-statement} where we carefully distinguish whether $A$ is dense or sparse. If $A$ is dense, \solverName{} combines our novel hashed coherence reduction transformation \footnote{For better numerical performance we use DHT as in Blendenpik instead of the Hadamard Transform analysed in the theory for \solverName{}.} analysed in \autoref{thm::HRHT} with a recently proposed randomized column pivoted QR factorization \cite{Martinsson:2017eh}, achieving better robustness and faster speed than Blendenpik and LSRN. If $A$ is sparse, \solverName{} combines $s$-hashing analysed in \autoref{thm::s-hashing} with the state-of-the-art sparse QR factorization in \cite{10.1145/2049662.2049670}, achieving 10 times faster speed on random sparse ill-conditioned problems and competitive performance on a test set of 181 matrices from the Florida Matrix Collection \cite{10.1145/2049662.2049663} comparing to the state-of-the-art direct and iterative solvers for sparse \eqref{LLS-statement}, which are based on sparse QR factorization \cite{10.1145/2049662.2049670} and incomplete Cholesky factorization preconditioned LSQR \cite{10.1145/2617555} respectively.

\paragraph{Relevant literature}
Classically, dense \eqref{LLS-statement} is solved by LAPACK \cite{laug}, and sparse \eqref{LLS-statement} is either solved by a sparse direct method implemented, say in \cite{10.1145/2049662.2049670} or a preconditioned LSQR (see a comparison of different preconditioners in \cite{10.1145/3014057}).
Sarlo \cite{10.1109/FOCS.2006.37} first proposed using sketching matrices that are oblivious subspace embeddings to solve \eqref{LLS-statement} by solving $\min_{x\in\R^d} \normTwo{SAx-Sb}$. This approach requires the row of $S$ to grow proportionally to the inverse square of the residual accuracy; hence is impractical for obtaining high accuracy solutions. Instead, Rokhlin \cite{Rokhlin:2008wb} proposed using the sketch $SA$ to compute a preconditioner of \eqref{LLS-statement}; and then solve \eqref{LLS-statement} using preconditioned LSQR. This approach allows machine precision solutions to be computed in a small number of LSQR iterations if the matrix $S$ is a subspace embedding of $A$. This algorithmic idea was carefully implemented in Blendenpik \cite{doi:10.1137/090767911}, achieving four times speed-up against LAPACK. Noting that Blendenpik only solves full rank \eqref{LLS-statement} and does not take advantage of sparse $A$, Meng, Saunders and Mahoney \cite{Meng:2014ib} proposed LSRN, which takes advantage of sparse $A$ and computes an accurate solution even when $A$ is rank-deficient by using Gaussian matrices to sketch and the SVD to compute a preconditioner. However, the run-time comparisons are conducted in a multi-core parallel environment, unlike Blendenpik, which uses the serial environment. 

Recently, the numerical performance of using $1$-hashing matrices as the sketching matrix to solve \eqref{LLS-statement} was explored in \cite{10.1145/3219819.3220098}. \cite{10.5555/3019094.3019103,IYER2019100547} further explored using Blendenpik-like solvers in a distributed computing environment. \cite{Lacotte2020} explores using random embedding to solve $L2$-regularised least squares. 

To the best of our knowledge, \solverName{} is the first solver that uses $s$-hashing (with $s>1$); uses a sparse factorization when solving sparse \eqref{LLS-statement}; and uses the hashing combined with coherence reduction transformations for dense problems. This work also presents the first large scale comparison of sketching-based LLS solvers with the state-of-the-art classical sparse solvers on the Florida Matrix Collection.

    \section{Algorithmic framework and a_alysis}
    
\label{sec:algo_analysis}
We now turn our attention to the LLS problem \eqref{LLS-statement} we are interested in solving. 
Building on the Blendenpik \cite{doi:10.1137/090767911} and LSRN \cite{Meng:2014ib} techniques, we introduce a generic algorithmic framework 
for \eqref{LLS-statement} that can employ any rank-revealing factorization of $SA$, where $S$ is a(ny) sketching matrix; we then analyse its convergence.

\subsection{A generic algorithmic framework for solving linear least squares with sketching}

\begin{algorithm}[H]
\begin{description}
\item[Initialization] \ \\
Given $A \in \R^{n\times d}$ and $b \in \R^n$, set positive integers $m$ and $it_{max}$, and accuracy tolerances $\tau_a$ and $ \tau_r $, and an $m\times n$ random matrix distribution $\cal{S}$.

	\item[1.] Randomly draw a sketching matrix $S \in \R^{m\times n}$ from $\cal{S}$, compute the matrix-matrix product $SA \in \R^{m\times d}$ and the matrix-vector product $Sb \in \R^{m}$.

	\item[2.] Compute a factorization of $SA$ of the form, 
		\begin{align}
		SA= QR\hat{V}^T, \label{SA-QRV-fac}
		\end{align}
		where 
			\begin{itemize}
				\item $R = 
					\left( \begin{matrix} R_{11} & R_{12} \\ 0 & 0 	\end{matrix} 
					 \right) \in \R^{d \times d}$, where $R_{11} \in \R^{\k \times \k}$ is nonsingular.

				\item $Q =\left( \begin{matrix} Q_{1} & Q_{2} \end{matrix} \right) \in\R^{m \times d} $, where $Q_1 \in \R^{m \times \k}$ and $Q_2 \in \R^{m \times (d-\k)}$ have orthonormal columns.

				\item $\hat{V}= \left(  \begin{matrix} V_{1} \quad  V_{2}\end{matrix} \right)\in \R^{d \times d}$ is an orthogonal matrix with $V_1\in \R^{d \times \k}$.
			\end{itemize}

	\item[3.] Compute $x_{s} = V_1 R_{11}^{-1} Q_1^TSb$. If $\|Ax_{s}-b\|_2 \leq \tau_a$, terminate with solution $x_s$.

	\item[4.]
	Else, iteratively, compute 
    \begin{equation}\label{ytau}	 
	 y_{\tau} \approx \argmin_{ y \in \R^{\k}} \|Wy - b\|_2,
	 \end{equation}
	 where 
	  \begin{equation}
	      W = A V_1 R_{11}^{-1},\label{def::W}
	  \end{equation}
	  using LSQR \cite{10.1145/355984.355989} with (relative) 
	 tolerance $\tau_r$ and maximum iteration count $it_{max}$. 	Return  $x_{\tau} = V_1 R_{11}^{-1} y_{\tau}$. 

\caption{\bf{Generic Sketching Algorithm for Linear Least Squares}} \label{alg1} 

\end{description}
\end{algorithm}

\begin{remark}
\begin{itemize}
\item[(i)]
The factorization $SA = QR\hat{V}^T$ allows column-pivoted QR, or other rank-revealing factorization, complete orthogonal decomposition ($R_{12}=0$) and the SVD ($R_{12} = 0$, $R_{11}$ diagonal). It also includes the usual QR factorisation if $SA$ is full rank; then the $R_{12}$ block is absent. 

\item[(ii)]						Often in  implementations, the factorization \eqref{SA-QRV-fac} has $R=\left( \begin{matrix} R_{11} & R_{12} \\ 0 & R_{22}  \end{matrix} \right)$, where $R_{22} \approx 0$ and is treated as the zero matrix.

\item[(iii)] For computing $x_s$ in Step 3, we note that in practical implementations, $R_{11}$ in \eqref{SA-QRV-fac}  is upper triangular, enabling efficient calculation
	 of matrix-vector products involving $R_{11}^{-1}$; then, there is no need to form/calculate $R_{11}^{-1}$ explicitly.
	 \item[(iv)] For the solution of \eqref{ytau},
	 we use the termination criterion \reply{ $\|y_\tau-y_*\|_{W^TW} \leq \tau_r \|y_\tau-y_* \|_{W^T W}$} in the theoretical analysis, 
	 where $y_*$ is defined in \eqref{def-ystar}. In practical implementations different termination criteria need to be employed (see Section \ref{subsection:alg_implementation_discussion}).
\end{itemize} 
\end{remark}

	\subsection{Analysis of Algorithm \ref{alg1}}
	Given problem (\ref{LLS-statement}), we denote its minimal Euclidean norm solution as follows
	\begin{equation}\label{def-xstar}
	x_{*,2} = \argmin_{x^*\in \R^d} \|x_*\|_2\quad{\rm subject\,\, to}\quad \|Ax_*-b\|_2 =\min_x \|Ax-b\|_2.
	\end{equation}
	and let
	\begin{equation}\label{def-ystar}
	 y_* = \argmin_{y\in \R^\k} \| Wy-b \|_2, \quad \text{where $W$ is defined in \eqref{def::W}}.
	\end{equation}

The following two lemmas provide basic properties of Algorithm \ref{alg1}. 
\begin{lemma}
\label{Lemma::w}
$W\in\R^{n\times \k}$ defined in \eqref{def::W} has full rank $\k$.
\end{lemma}
\begin{proof}
Note $SW$ has rank $\k$ because $SW = Q_1$, where $Q_1$ is defined in \eqref{SA-QRV-fac}. By rank-nullity theorem in $\R^\k$, $\rank(W) + \dim \ker(W) = \rank(SW) + \dim \ker(SW)$ where $\ker(W)$ denotes the null space of $W$; and since $\text{dim} \ker(SW) \geq \dim \ker(W)$, we have that $\rank(SW) \leq \rank(W)$. So $\rank(W) \geq \k$. It follows that $\rank(W)=\k$ because $W\in\R^{n\times \k}$ can have at most rank $\k$.
\end{proof}

\begin{lemma}
\label{Lemma::p_equals_r}
In Algorithm \ref{alg1}, if $S$ is an $\epsilon$-subspace embedding for $A$ for some $\epsilon\in (0,1)$, then $\k = r$ where $r$ is the rank of $A$. 
\end{lemma}
\begin{proof}
Lemma \ref{rank-of-sketched-equal-to-unsketched} gives $r= \rank(A) = \rank(SA)= \k$.
\end{proof}

	  If the LLS problem (\ref{LLS-statement}) has a sufficiently small optimal residual, then
	Algorithm \ref{alg1} terminates early in Step 3 with the solution $x_s$ of the sketched problem $\min \|SAx-Sb\|_2$; then,
	 no LSQR iterations are required.
	

\begin{lemma}[Explicit Sketching Guarantee] \label{explicit-sketching-guarantee}
Given problem \eqref{LLS-statement},
suppose that the matrix $S \in \R^{m\times n}$ in Algorithm \ref{alg1} is an $\epsilon$-subspace embedding for the augmented matrix $\left(A\; \;b \right)$ for some $0<\epsilon<1$. Then 
\begin{align}\label{explicit-sketch}
\|Ax_s - b\|_2 \leq \frac{1+\epsilon}{1-\epsilon} \|Ax_*-b\|_2,
\end{align}
where $x_s$ is defined in Step 3 of Algorithm \ref{alg1} and $x_*$ is a(ny) solution of \eqref{LLS-statement}.
\end{lemma}

The proof is similar to the result in \cite{10.1561/0400000060} that shows that any solution of the sketched problem $\min_x \|SAx-Sb\|_2$ satisfies \eqref{explicit-sketch}.  For completeness, the proof is included here.

\begin{proof}
We have that $x_s \in \argmin \|SAx-Sb\|_2$ by checking the optimality condition $(SA)^T SA x_s = (SA)^T Sb$. Hence we have that
\begin{align}
\norms{Ax_s-b} \leq \frac{1}{1-\epsilon} \norms{SAx_s-Sb} \leq \frac{1}{1-\epsilon} \norms{SAx^*-Sb}
 \leq \frac{1+\epsilon}{1-\epsilon} \norms{Ax^*-b},
\end{align}
where the first and the last inequality follow from $S$ being a subspace embedding for $\left(A \,\,b \right)$, while the second inequality is due to $x_s$ minimizing $\|SAx-Sb\|$.
\end{proof}
		
The following technical lemma is needed in the proof of our next theorem.
		\begin{lemma}\label{VA-cap}
Let $A \in \R^{n\times d}$ and $V_1 \in \R^{d \times \k}$ be defined in Algorithm \ref{alg1}. Then
$\ker(V_1^T) \cap {\rm range}(A^T) = \{0\}$, where $\ker(V_1^T)$ and ${\rm{range}}(A^T)$ denote the null space of $V_1^T$ and range subspace 
generated by the rows of $A$, respectively.
\end{lemma}

\begin{proof}
Let $z \in \ker(V_1^T) \cap {\rm range}(A^T)$. Then $V_1^T z = 0$ and $z = A^T w$ for some $w \in \R^n$. Let $U, \Sigma, V$ be the SVD factors of $A$ as defined in \eqref{thin-SVD}. Since $S$ is an $\epsilon$-subspace embedding for $A$, $\rank \left( (SU)^T \right) = \rank(SU) = \rank(SA) = r$, where $r$ is the rank of $A$ and hence there exists $\hat{w} \in \R^m$ such that $(SU)^T \hat{w} = U^T w$. Note that 
\begin{equation}
    0 = V_1^T z = V_1 A^T w = V_1^T V \Sigma U^T w = V_1^T V \Sigma U^T S^T \hat{w} = V_1^T A^T S^T \hat{w} = R_{11}^T Q_1^T \hat{w},
\end{equation}
which implies $Q_1^T \hat{w} = 0$ because $R_{11}^T$ is nonsingular.
It follows that
\begin{equation}
    z = A^T w = V \Sigma U^T w = V\Sigma U^T S^T \hat{w} = (SA)^T \hat{w} = V 
\left( \begin{smallmatrix} R_{11}^T Q_1^T \\ R_{12}^T Q_1^T \end{smallmatrix} \right) \hat{w} = 0, 
\end{equation}
where we have used $Q_1^T \hat{w} = 0$ for the last equality.
\end{proof}

Theorem \ref{Implicit-sketching-guarantee} shows that when the LSQR algorithm in Step 4 converges, Algorithm \ref{alg1} returns a minimal residual solution of
\eqref{LLS-statement}.

\begin{theorem}[Implicit Sketching Guarantee]\label{Implicit-sketching-guarantee}
Given problem \eqref{LLS-statement},
suppose that the matrix $S \in \R^{m\times n}$ in Algorithm \ref{alg1} is an $\epsilon$-subspace embedding  for the augmented matrix $\left(A\; \;b \right)$ for some $0<\epsilon<1$. 
If  $y_{\tau} = y_*$ in Step 4 of Algorithm \ref{alg1} (by setting $\tau_r := 0$), where $y_*$ is defined in \eqref{def-ystar},  then $x_{\tau}$ in Step 5 satisfies
$x_{\tau}=x_*$, where $x_*$ is a solution of \eqref{LLS-statement}.
\end{theorem}

\begin{proof}
Using the optimality conditions (normal equations) for the LLS in \eqref{def-ystar}, and $y_{\tau} = y_*$, we deduce 
$W^T Wy_{\tau} = W^T b$, 
where $W$ is defined in \eqref{def-ystar}. 
Substituting the definition of $x_{\tau}$ from Step 5 of Algorithm \ref{alg1}, 
we deduce
$$
(R_{11}^{-1})^TV_1^TA^TAx_{\tau}=(R_{11}^{-1})^TV_1^TA^Tb.
$$
Multiplying the last displayed equation by $R_{11}^T$, we obtain
\begin{align}
V_1^T \left( A^TA x_{\tau} - A^T b \right)  = 0. \label{tmp8}
\end{align}
It follows from
(\ref{tmp8}) that $A^T Ax - A^T b \in \ker(V_1^T) \cap {\rm range}(A^T)$. But Lemma \ref{VA-cap} implies that the latter set intersection only contains the origin, and so
 $A^T Ax_{\tau} - A^T b =0$; this and the normal equations for \eqref{LLS-statement} imply that $x_{\tau}$ is an optimal solution of \eqref{LLS-statement}.
\end{proof}

The following technical lemma is needed for our next result; it re-states Theorem 3.2 from  \cite{Meng:2014ib} in the context of Algorithm \ref{alg1}.

		\begin{lemma} \cite{Meng:2014ib} \label{Meng-min-norm}
		Given problem \eqref{LLS-statement}, let $x_{*,2}$ be its minimal Euclidean norm solution defined in \eqref{def-xstar} and $P\in\R^{d\times \k}$, a nonsingular matrix. 	Let $x_{\tau}:=Py_{\tau}$, where $y_{\tau}$ is assumed to be the minimal Euclidean norm solution  of $\min_{y\in \R^\k}\|APy-b\|_2$.
		Then $x_{\tau}=x_{*,2}$ if ${\rm range}(P)={\rm range}(A^T)$.	
\end{lemma}

Theorem \ref{tmp9} further guarantees that if $R_{12}=0$ in \eqref{SA-QRV-fac} such as when a complete orthogonal factorization is used, then the minimal Euclidean norm solution of \eqref{LLS-statement} is obtained. 

\begin{theorem}[Minimal-Euclidean Norm Solution Guarantee] \label{tmp9}
Given problem \eqref{LLS-statement},
suppose that the matrix $S \in \R^{m\times n}$ in Algorithm \ref{alg1} is an $\epsilon$-subspace embedding  for the augmented matrix $\left(A\; \;b \right)$ for some $0<\epsilon<1$. 
If  $R_{12}=0$ in \eqref{SA-QRV-fac} and $y_{\tau} = y_*$ in Step 4 of Algorithm \ref{alg1} (by setting $\tau_r := 0$), where $y_*$ is defined in \eqref{def-ystar}, 
then $x_{\tau}$ in Step 5 satisfies
$x_{\tau}=x_{*,2}$, where $x_{*,2}$ is the minimal Euclidean norm solution \eqref{def-xstar} of \eqref{LLS-statement}.

\end{theorem}

\begin{proof}
The result follows from Lemma \ref{Meng-min-norm} with $P:=V_1R_{11}^{-1}$, provided  ${\rm range}(V_1 R_{11}^{-1}) = {\rm range}(A^T)$. To see this, note that 
\[
{\rm range}(V_1 R_{11}^{-1}) = {\rm range}(V_1) = {\rm range}( (SA)^T),
\]
where the last equality follows from $(SA)^T=V_1R_{11}^TQ_1^T+V_2R_{12}Q_1^T$ and $R_{12}=0$. Using the SVD decomposition \eqref{thin-SVD} of $A$, we further have
\[
{\rm range}(V_1 R_{11}^{-1})  = {\rm range}( A^T S^T) ={\rm range}( V\Sigma U^T S^T) ={\rm  range}(V\Sigma (SU)^T).
\]
Since $S$ is an $\epsilon-$subspace embedding for $A$, it is also an $\epsilon$-subspace embedding for $U$ by Lemma \ref{subspace-embedding-def-2} and therefore by Lemma \ref{rank-of-sketched-equal-to-unsketched}, $\rank(SU) = \rank(U) = r$. Since $S U \in \R^{m\times r}$ has full column rank, we have that ${\rm range}(V\Sigma (S U)^T) = {\rm range}(V) = {\rm range}(A^T)$. 
\end{proof}

Theorem \ref{Speed-of-convergence} gives an iteration complexity bound for the inner solver in Step 4  of Algorithm \ref{alg1}, as well as particularising this result for a special starting point for which  an optimality guarantee can be given. It relies crucially on the quality of the preconditioner provided by the sketched factorization in \eqref{SA-QRV-fac}, and its proof uses standard LSQR results.

\begin{theorem}[Rate of convergence] \label{Speed-of-convergence}
Given problem \eqref{LLS-statement},
suppose that the matrix $S \in \R^{m\times n}$ in Algorithm \ref{alg1} is an $\epsilon$-subspace embedding  for the augmented matrix $\left(A\; \;b \right)$ for some $0<\epsilon<1$. 
Then: 
\begin{itemize}
\item[(i)]
Step 4 of Algorithm \ref{alg1} takes at most
\begin{align}\label{LSQR_iterations}
\tau \leq O\left( \frac{|\log\tau_r|}{|\log\epsilon|}\right)
\end{align}
LSQR iterations to return a solution $y_{\tau}$ such that 
\begin{equation}\label{LSQR-tc}
\|y_{\tau}- y_*\|_{ W^T W} \leq \tau_r \|y_0 - y_*\|_{ W^T W}, 
\end{equation}
where $y_*$ and $W$ are defined in \eqref{def-ystar}.
\item[(ii)] 
If we initialize $y_0 : = Q^TSb$ for the LSQR method in Step 4, then at termination of Algorithm \ref{alg1}, we can further guarantee that
\begin{align}
		\|Ax_{\tau} - b\|_2 \leq \left( 1 + \frac{2\epsilon \tau_r}{1-\epsilon} \right) \|Ax_*-b\|_2,
		\end{align}
		where $x_{\tau}$ is computed in \reply{Step 4} of Algorithm \ref{alg1} and $x_*$ is a solution of \eqref{LLS-statement}.

\end{itemize}
\end{theorem}

\begin{proof}
(i) Using results in \cite{10.5555/248979},  LSQR applied to \eqref{ytau} converges as follows 
\begin{align}
\frac{\|y_j - y_*\|_{ W^T W}}{ \|y_0 - y_*\|_{ W^T W}} \leq 
2 \left( \frac{ \sqrt{\kappa \left[ W^T W \right]} -1 } 
{\sqrt{\kappa \left[ W^T W \right]} +1 } \right)^j, \label{CG_guarantee}
\end{align}
where $y_j$ denotes the $j$th iterate of LSQR and $\kappa(W^TW)$ refers to the condition number of $W^TW$.
Since $S$ is an $\epsilon$-subspace embedding for $A$, we have that the largest singular value of $W$ satisfies
\begin{align*}
\sigma_{\max}(W) = \max_{\|y\|=1} \|AV_1R_{11}^{-1}y\| \leq \bracket{1-\epsilon}^{-1/2} \max_{\|y\|=1} \|SAV_1R_{11}^{-1}y\| = \bracket{1-\epsilon}^{-1/2} \max_{\|y\|=1} \|Q_1y\| = \bracket{1-\epsilon}^{-1/2} ,
\end{align*}
where we have used that $SAV_1R_{11}^{-1} = Q_1$ from (\ref{SA-QRV-fac}). Similarly, it can be shown that the smallest singular
value of $W$ satisfies $\sigma_{\min}(W) \geq \bracket{1+\epsilon}^{-1/2}$. Hence
\begin{align}
\kappa(W^T W) \leq  \frac{1+\epsilon}{1-\epsilon} . \label{Low condition number}
\end{align}
Hence we have
\begin{align*}
    \frac{ \sqrt{\kappa \left[ W^T W \right]} -1 } 
{\sqrt{\kappa \left[ W^T W \right]} +1 } \leq \frac{\sqrt{1+\epsilon} - \sqrt{1-\epsilon}}{\sqrt{1+\epsilon} + \sqrt{1-\epsilon}} =  \frac{\bracket{\sqrt{1+\epsilon} - \sqrt{1-\epsilon}} \bracket{\sqrt{1+\epsilon} + \sqrt{1-\epsilon}} }{\bracket{\sqrt{1+\epsilon} + \sqrt{1-\epsilon}}^2} \leq \epsilon.
\end{align*}
Thus (\ref{CG_guarantee}) implies $\|y_{\tau} - y_*\|_{ W^T W} \leq \tau_r \|y_0 - y_*\|_{ W^T W}$ whenever $\tau \geq \frac{\log(2) + |\log\tau_r|}{|\log\epsilon|}$.

(ii) If we initialize $y_0 : = Q^TSb$ for the LSQR method in Step 4, then we have
\begin{align*}
\|y_0 - y_*\|_{ W^T W} &= \|Ax_s - Ax_*\|_2 = \|Ax_s-b - (Ax_*-b)\|_2 
					                 \leq \|Ax_s -b\| - \|Ax_* -b\| \\
					                 & \leq \left( \sqrt{\frac{1+\epsilon}{1-\epsilon}} - 1 \right) \|Ax_*-b\| 
					                  \leq \frac{2\epsilon}{1-\epsilon} \|Ax_*-b\|,
\end{align*}
where we have used $\sqrt{\frac{1+\epsilon}{1-\epsilon}} \leq \frac{1+\epsilon}{1-\epsilon}$ to get the last inequality.
Using part (i), after at most $\frac{\log(2) + |\log \tau_r|}{|\log\epsilon|}$  LSQR iterations, we have that 
\begin{align}
\|y_{\tau} - y_*\|_{ W^T W} \leq \frac{2\epsilon \tau_r}{1-\epsilon} \|Ax_*-b\|_2. 
\end{align}
Note that $\|Ax_{\tau} - Ax_*\|_2 = \|y_{\tau} - y^*\|_{W^T W}$. Using the triangle inequality, we deduce 
\begin{align}
\|Ax_{\tau} - b\|_2 = \| A x_{\tau} - Ax_* + Ax_* -b  \|_2 \leq \left( 1 + \frac{2\epsilon \tau_r}{1-\epsilon} \right) \|Ax_*-b\|_2.
\end{align}
\end{proof}
    
    \section{Implementation details}

\subsection{Ski-LLS, an implementation of Algorithm \ref{alg1}}
\label{subsec::implemntation_of_alg1}
\textbf{Sk}etch\textbf{i}ng-for-\textbf{L}inear-\textbf{L}east-\textbf{S}quares (\solverName{}) implements \autoref{alg1} for solving \eqref{LLS-statement}. We distinguish two cases based on whether the data matrix $A$ is stored as a dense matrix or a sparse matrix. 

\paragraph{Dense $A$}
When $A$ is stored as a dense matrix \footnote{This does not necessarily imply that every entry of $A$ is non-zero, however, we presume a large number of the entries are zero such that specialized sparse numerical linear algebras are ineffective.}, we employ the following implementation of \refAlgOne. The resulting solver is called \solverNameDense{}.
\begin{enumerate}
    \item In Step 1 of \refAlgOne, we let 
    \begin{equation}
        S = S_h F D, \label{eqn::HR-DHT}
    \end{equation}
    where
        \begin{enumerate}
            \item $D$ is a random $n \times n$ diagonal matrix with $\pm1$ independent entries, as in \autoref{def::HRHT}.
            \item F is a matrix representing the normalized Discrete Hartley Transform (DHT), defined as $F_{ij} = \sqrt{1/n} \squareBracket{ \cos{ \bracket{ 2\pi (i-1)(j-1)/n}} + \sin{ \bracket{2 \pi (i-1)(j-1)/n}}}$ \footnote{Here we use the same transform (DHT) as that in Blendenpik for comparison of other components of \solverNameDense{}, instead of the Walsh-Hadamard transform defined in \autoref{def::HRHT}.}. We use the (DHT) implementation in FFTW 3.3.8 \footnote{Available at http://www.fftw.org.}.  
            \item $S_h$ is an $s$-hashing matrix, defined in \autoref{def::sampling_and_hashing}. We use the sparse matrix-matrix multiplication routine in SuiteSparse 5.3.0 \footnote{Available at https://people.engr.tamu.edu/davis/suitesparse.html.} to compute $S_h  \times (FDA)$. 
        \end{enumerate}
    
    \item In Step 2 of \refAlgOne, we use the randomized column pivoted QR (R-CPQR) proposed in \cite{Martinsson:2017eh,martinsson2015blocked} \footnote{The implementation can be found at https://github.com/flame/hqrrp/. The original code only has a 32-bit integer interface. We wrote a 64-bit integer wrapper as our code has 64-bit integers. }. 
    
    \item In Step 3 of \refAlgOne, since $R_{11}$ from R-CPQR is upper triangular, we do not explicitly compute its inverse, but instead, use back-solve from the LAPACK provided by Intel MKL 2019 \footnote{See https://software.intel.com/content/www/us/en/develop/tools/oneapi/components/onemkl.html.}. 
    
    \item In Step 4 of \refAlgOne, we use the LSQR routine implemented in LSRN \cite{Meng:2014ib}\footnote{Available at https://web.stanford.edu/group/SOL/software/lsrn/. We fixed some very minor bugs in the code. }.
\end{enumerate}

The user can choose the value of the following parameters: $m$ (default is $1.7d$), $s$ (default is $1$), ${\tau_a}$ (default is $10^{-8}$), ${it_{max}}$ (default value is $10^4$). ${rcond}$ (default value is $10^{-12})$, which is a parameter used in Step 2 of \refAlgOne. The R-CPQR we use computes
            $SA = Q \tilde{R} \hat{V}^T$, which is then used to compute $R_{11}$ by 
        letting $p = \max \set{q: \tilde{R}_{qq}\geq rcond}$, $R_{11}$ be the upper left $p \times p$ block of $\tilde{{R}}$. 
$wisdom$ (default value is $1$). The DHT we use is faster with pre-tuning, see Blendenpik \cite{Avron:2009aa} for a detailed discussion. If the DHT has been pre-tuned, the user needs to set $wisdom=1$, otherwise set $wisdom=0$.  In all our experiment, the default is to tune the DHT using the crudest tuning mechanism offered by FFTW, which typically takes less than one minute.

We also offer an implementation without using R-CPQR for dense full-rank problems. The only difference  is that
 in Step 2 of \refAlgOne, we assume that the matrix $A$ has full-rank $r=d$. Hence we use DGEQRF from LAPACK to compute a QR factorization of $SA$ (the same routine is used in Blendenpik) instead of R-CPQR. It has the same list of parameters with the same default values, except the parameter ${rcond}$ is absent because it does not use R-CPQR.

\paragraph{Sparse $A$}
When $A$ is stored as a sparse matrix \footnote{Here we assume that the user stored the matrix in a sparse matrix format because a large number of entries are zero. Throughout computations, we maintain the sparse matrix format for effective numerical linear algebras.}, we employ the following implementation of \refAlgOne. The resulting solver is called \solverNameSparse{}.
\begin{enumerate}
    \item In Step 1 of \refAlgOne, we let $S$ be an $s$-hashing matrix, defined in \autoref{def::sampling_and_hashing}.
    \item In Step 2 of \refAlgOne, we use the sparse QR factorization (SPQR) proposed in \cite{10.1145/2049662.2049670} and implemented in SuiteSparse.
    \item In Step 3 of \refAlgOne, since $R_{11}$ from SPQR is upper triangular, we do not explicitly compute its inverse, but instead, use the sparse back-substitution routine from SuiteSparse.
    \item In Step 4 of \refAlgOne, we use the LSQR routine implemented in LSRN, extended to include the use of sparse preconditioner and sparse numerical linear algebras from SuiteSparse. 
\end{enumerate}

The user can choose the value of the following parameters: $m$ (default value is $1.4d$), $s$ (default value is $2$), $\tau_a$ (default value is $10^{-8}$), $\tau_r$ (default value is $10^{-6}$), $it_{max}$ (default value is $10^4$). And $rcond_{thres}$ (default value $10^{-10}$), which checks the conditioning of $R_{11}$ computed by SPQR. If $\kappa(R_{11}) \geq 1/rcond_{thres}$, we use the perturbed back-solve for upper triangular linear systems involving $R_{11}$ (see the next point).  $perturb$ (default value $10^{-10}$). When $\kappa(R_{11}) \geq 1/rcond_{thres}$, any back-solve involving $R_{11}$ or its transpose will be modified in the following way: When divisions by a diagonal entry $r_{ii}$ of $R_{11}$ is required where $1\leq i \leq p$, we divide by $r_{ii} + perturb$ instead. \footnote{This is a safe-guard when SPQR fails to detect the rank of $A$. This happens infrequently \cite{10.1145/2049662.2049670}.} $ordering$ (default value $2$) which is a parameter to the SPQR routine that influences the permutation matrix $\hat{V}$ and the sparsity of $R$. \footnote{Note that this is slightly different from the SPQR default, which is to use to use COLAMD if m2<=2*n2; otherwise try AMD. Let f be the flops for chol((S*P)’*(S*P)) with the ordering P found by AMD. Then if f/nnz(R) $\geq$ 500 and nnz(R)/nnz(S) $\geq$ 5 then try METIS, and take the best ordering found (AMD or METIS), where typically $m_2 = m$, $n_2 = n$ for $SA \in \R^{m \times n}$. In contrast, \solverName{} by default always use the AMD ordering.}.
    
\subsection{Discussion of our implementation}
\label{subsection:alg_implementation_discussion}

\paragraph{Subspace embedding properties achieved via numerical calibration}

Our analysis of Algorithm \ref{alg1} in Section $\ref{sec:algo_analysis}$ relies crucially on $S$
being an $\epsilon$-subspace embedding of $A$. For dense matrices, Blendenpik previously used SR-DHT, defined in \eqref{eq::SR-DHT} with theoretical guarantees of the oblivious $\epsilon$-subspace embedding property for full rank $A$ if $m = \mathcal{O} \bracket{d \log(d)}$. Theorem \ref{thm::HRHT} shows when using hashing instead of sampling with randomised Walsh-Hadamard transform, hashing achieves being an oblivious $\epsilon$-subspace embedding with $m = \mathcal{O} \bracket{d}$ (note that $r=d$ for full rank $A$) under the addition dimensional assumption of $A$. In \solverNameDense{}, HR-DHT is used instead of HRHT analyzed in Theorem \ref{def::HRHT} because as mentioned in Blendenpik paper \cite{doi:10.1137/090767911}, DHT is more flexible (Walsh-Hadamard transform only allows n to be an integer power of $2$ so that padding is needed); and SR-DHT based sketching solver has stabler and shorter running time comparing to when DHT is replaced by Walsh-Hadamard transform. Moreover, we aim to show in additional to the theoretical advantage of hashing (Theorem \ref{thm::HRHT}), numerically using hashing instead of sampling combined with coherence-reduction transformations yields a more effective solver for \eqref{LLS-statement} in terms of running time. 

Therefore to compare to Blendenpik, we chose to use HR-DHT instead of HRHT. We then use numerical calibration as used in Blendenpik to determine the default value of $m$ for \solverNameDense{} such that $\epsilon$-subspace embedding of $A$ is achieved with sufficiently high (all the matrices in the calibration set) probability. (See the next section and Appendices). Note that the U-shaped curve appears in Figure \ref{fig::new_blen_engineering}, because as $\gamma:= m/d$ grows, we have better subspace embeddings so that $\epsilon$ decreases, resulting in fewer LSQR iterations according to \eqref{LSQR_iterations}. However the factorization cost in Step 2 and the sketching cost in Step 1 will grow as $m$ grows. Thus a trade-off is achieved when $m$ is neither too big nor too small. 

For sparse matrices, Theorem \ref{thm::s-hashing} guarantees the oblivious $\epsilon$-subspace embedding property $s$-hashing matrices for matrices $A$ with low coherence. However as Figure \ref{fig::Ls_qr_engineering_time}, \ref{fig::Ls_qr_engineering_residual} suggest, $s$-hashing with $s>1$ and $m=\mathcal{O} \bracket{d}$ tends to embed higher coherence $A$ as well. The specific default values of $m,s$ are again chosen using numerical calibration; and the characteristic U-shape is because of a similar trade-off as in the dense case. 

\paragraph{What if $S$ is not an $\epsilon$-subspace embedding of $A$}
Note that even $\cal{S}$ is an oblivious subspace embedding for matrices $A \in \R^{n\times d}$, for a given $A \in \R^{n\times d}$, there is a chance that a randomly generated matrix $S$ from $\cal{S}$ fails to embed $A$. However, in this case, \solverName{} will still compute an accurate solution of $\eqref{LLS-statement}$ given that $A$ has full rank and $SA$ has the same rank as $A$. Because then the preconditioner $V_1 R_{11}^{-1}$ is an invertible square matrix. The situation is less clear when $A$ is rank-deficient and $S$ fails to embed $A$. However, with the default parameters chosen from numerical calibrations, the accuracy of \solverName{} is excellent for $A$ being both random dense/sparse matrices and for $A$ in the Florida matrix collection. 

\paragraph{Approximate numerical factorization in Step 2}
In both our dense and sparse solvers, Step 2 $SA=QR\hat{V}^T$ is not guaranteed to be accurate when $A$ is rank-deficient. This is because R-CPQR, like CPQR, does not guarantee detection of rank although in almost all cases the numerical rank is correctly determined (in the sense that if one follows the procedure described in the definition of the parameter $rcond$, the factorization $SA=QR\hat{V}^T$ will be accurate up to approximately $rcond$ error). Similarly, SPQR performs heuristic rank-detection for speed efficiency and therefore rank-detection and the resulting accuracy is not guaranteed. Also, we have not analysed the implication of floating-point arithmetic for \solverName{}. The accuracy of \solverName{}, however, is demonstrated in a range of dense and sparse test problems, see later sections. 

\paragraph{Practical LSQR termination criterion}
The termination criterion proposed in Step 4 of the algorithm is not practical as we do not know $y^*$. In practice, we terminate Step 4 of Algorithm \ref{alg1} if $\frac{\|W^T (Wy_k - b)\|}
			{\|W\|\|Wy_k-b\|} \leq \tau_r$ where $W$ is defined in \eqref{def::W} similarly to what is used in LSRN \cite{Meng:2014ib}. See the original LSQR paper \cite{10.1145/355984.355989}, Section 6 for a justification.

    \section{Numerical study}

\subsection{Test Set}

\paragraph{The matrix $A$}
\begin{enumerate} \label{List_of_test_set}
	\item The following are three different types of random dense matrices that are the same type of test matrices used by Avron et.al. \cite{doi:10.1137/090767911} for comparing Blendenpik with LAPACK least square solvers. They have different \singleQuote{non-uniformity} of rows.
		\begin{enumerate}
			\item Incoherent dense, defined by
			    \begin{equation}
			        A = U\Sigma V^T \in \R^{n \times d}, \label{eq::A_inco_dense}
			    \end{equation}
			where $U \in \R^{n \times d}$, $V \in \R^{d \times d}$ are matrices generated by orthogonalising columns of two independent matrices with i.i.d. N(0,1) entries. $\Sigma \in \R^{d \times d}$ is a diagonal matrix with diagonal entries equally spaced from $1$ to $10^6$ (inclusive).

			\item Semi-coherent dense, defined by
			    \begin{equation}
			      A = \bigl( \begin{smallmatrix} B & 0 \\ 0 & I_{d/2} \end{smallmatrix} \bigr) + 
			10^{-8}  J_{n,d} \in \R^{n \times d} , \label{eq::A_semi_dense}  
			    \end{equation}
			where $B$ is an incoherent dense matrix defined in \eqref{eq::A_inco_dense}, $J_{n,d} \in \R^{n \times d}$ is a matrix of all ones.

			\item Coherent dense, defined by
			    \begin{equation}
			        A = \bigl(\begin{smallmatrix} I_{d\times d} \\ 0  \end{smallmatrix}\bigr) + 10^{-8}J_{n,d} \in \R^{n\times d} , \label{eq::A_co_dense}
			    \end{equation}
			where $J_{n,d}$ is a matrix of all ones. 
			
		\end{enumerate}
	
	\item The following are three different types of random sparse matrices with different \singleQuote{non-uniformity} of rows.
	    \begin{enumerate}
	        \item Incoherent sparse, defined by 
                    \begin{equation}
                        A = \text{sprandn}(n, d, 0.01, 1e-6) \in \R^{n \times d}, \label{eq:A_inco_sparse}
                    \end{equation}
                    where \singleQuote{sprandn} is a command in MATLAB that generates a matrix with approximately $0.01nd$ normally distributed non-zero entries and a condition number approximately equals to $10^6$.
            \item Semi-coherent sparse, defined by
                \newcommand{\dTemp}{\hat{D}}
                \begin{equation}
                    A = \dTemp^5 B, \label{eq:A_semi-co_sparse} 
                \end{equation}
                where $B \in \R^{n\times d}$ is an incoherent sparse matrix defined in \eqref{eq:A_inco_sparse} and $\dTemp$ is a diagonal matrix with independent $N(0,1)$ entries on the diagonal. 
            \item Coherent sparse, defined by
                \begin{equation}
                    A = \dTemp^{20} B, \label{eq:A_co_sparse} 
                \end{equation}
                where $B \in \R^{n\times d}, \dTemp$ are the same as in \eqref{eq:A_semi-co_sparse}.
	    \end{enumerate}
	
	\item (Florida matrix collection) A total of 181 matrices in the Florida (SuiteSparse) matrix collection \cite{10.1145/2049662.2049663} satisfying:
		\begin{enumerate}
			\item If the matrix is under-determined, we transpose it to make it over-determined.
			\item We only take a matrix $A \in \R^{n \times d}$ if $n \geq 30000$ and $n \geq 2 d$.
		\end{enumerate}
		
\end{enumerate}

\begin{remark}
Note that here \singleQuote{coherence} is a more general concept indicating the non-uniformity of the rows of $A$. Although \singleQuote{coherent dense/sparse} $A$ tends to have higher values of $\mu(A)$ than that of \singleQuote{incoherent dense/sparse} $A$, the value of $\mu(A)$ may be similar for semi-coherent and coherent test matrices. The difference is that for \singleQuote{coherent} test matrices, the row norms of $A$ (and $U$ from the SVD of $A$) tend to be more non-uniform. \reply{In the dense matrix cases, the rows of $A$ are progressively less uniform due to the presence of the identity blocks. In the sparse matrix cases, the rows of $A$ are progressively less uniform due to the scalings from the increasing powers of a diagonal Gaussian matrix. }
\end{remark}

\paragraph{The vector $b$} In all the test, the vector $b \in \R^n$ in \eqref{LLS-statement} is chosen to be a vector of all ones.

\subsection{Numerical illustrations}

\paragraph{The case for using hashing instead of sampling}
In \autoref{fig::blen_motivation}, we generate a random coherent dense matrix $A \in \R^{4000 \times 400}$ defined in \eqref{eq::A_co_dense}, and for each $m/d$ (where $d=400$), we sketch the matrix using $S \in \R^{m \times n}$ being a HR-DHT defined in \eqref{eqn::HR-DHT} and using SR-DHT defined by 
    \begin{equation}
        S = S_s FD, \label{eq::SR-DHT}
    \end{equation}
    where $S_s \in \R^{m\times n} $ is a scaled sampling matrix, whose individual rows contain a single non-zero entry at a random column with value $\sqrt{\frac{n}{m}}$; $F, D$ are defined the same as in \eqref{eqn::HR-DHT}. This is the sketching used in Blendenpik. We then compute an (non-pivoted) QR factorization of each sketch $SA=QR$, and the condition number of $AR^{-1}$. 

    We see that using hashing instead of sampling allows the use of smaller $m$ to reach a given preconditioning quality. 

\begin{figure}
\centering
\includegraphics[width=0.6\textwidth]{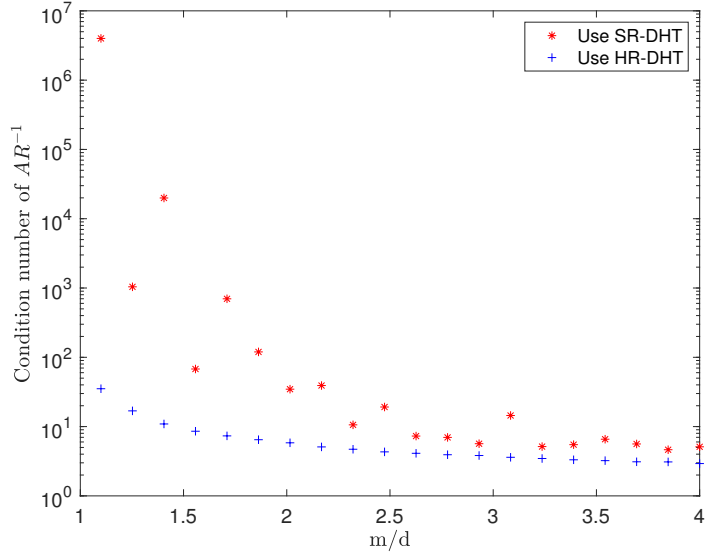}
\caption{Hashing combined with a randomised Discrete Hartley Transform produces more accurate sketched matrix $SA$ for a given $m/d$ ratio comparing to sampling combined with a randomised Discrete Hartley Transform; the accuracy of the sketch is reflected in the quality of the preconditioner $R$ constructed from the matrix $SA$, see \eqref{Low condition number}.}
\label{fig::blen_motivation}
\end{figure}

\paragraph{The case for using $s$-hashing with $s>1$}
In Figure \ref{fig::1-2-3-inco} , we let $A \in \R^{4000\times 400}$ be a random incoherent sparse matrix defined in \eqref{eq:A_inco_sparse}, while in Figure
\ref{fig::1-2-3-semi-co}, $A$ be defined as 
\begin{equation}
    A = \bigl( \begin{smallmatrix} B & 0 \\ 0 & I_{d/2} \end{smallmatrix} \bigr) + 
			10^{-8}  J_{n,d} \in \R^{n \times d},
\end{equation}
where $ B\in \R^{n\times d}$ is a random incoherent sparse matrix, and $J_{n \times d}$ is a matrix of all ones \footnote{We use this type of random sparse matrix instead of one of the types defined in \eqref{eq:A_semi-co_sparse} because this matrix better showcases the failure of 1-hashing.}. Comparing \autoref{fig::1-2-3-inco} with \autoref{fig::1-2-3-semi-co}, we see that using $s$-hashing matrices with $s>1$ is essential to produce a good preconditioner.

		\begin{figure}[H]
		    \centering
		    \begin{minipage}{0.48\textwidth}
		        \centering
		        \includegraphics[width=\textwidth]{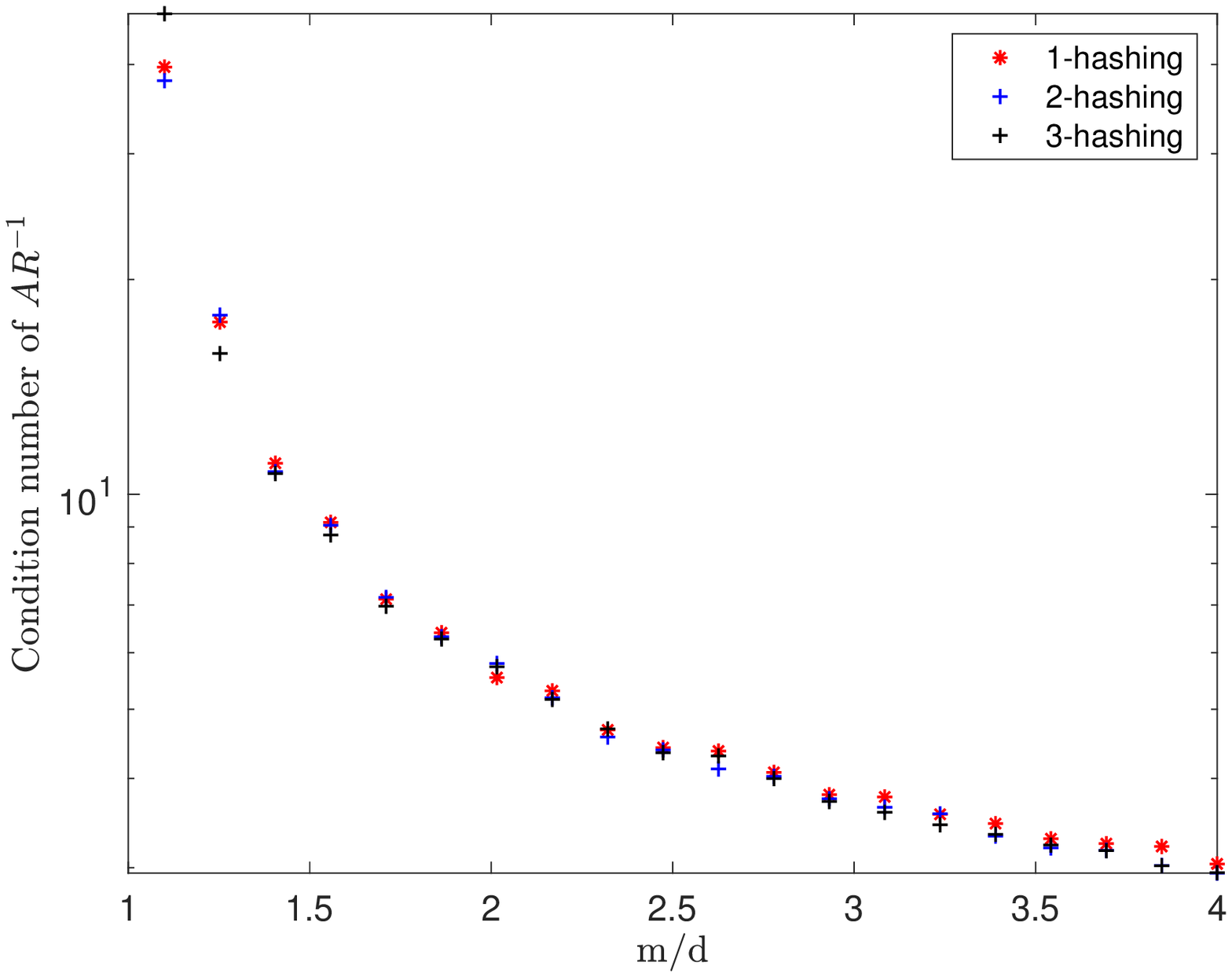} 
		        \caption{When the data matrix $A$ is an ill-conditioned sparse Gaussian matrix, using $1,2,3-$hashing produces similarly good preconditioners.}
		        \label{fig::1-2-3-inco}
		    \end{minipage}\hfill
		    \begin{minipage}{0.48\textwidth}
		        \centering
		        \includegraphics[width=\textwidth]{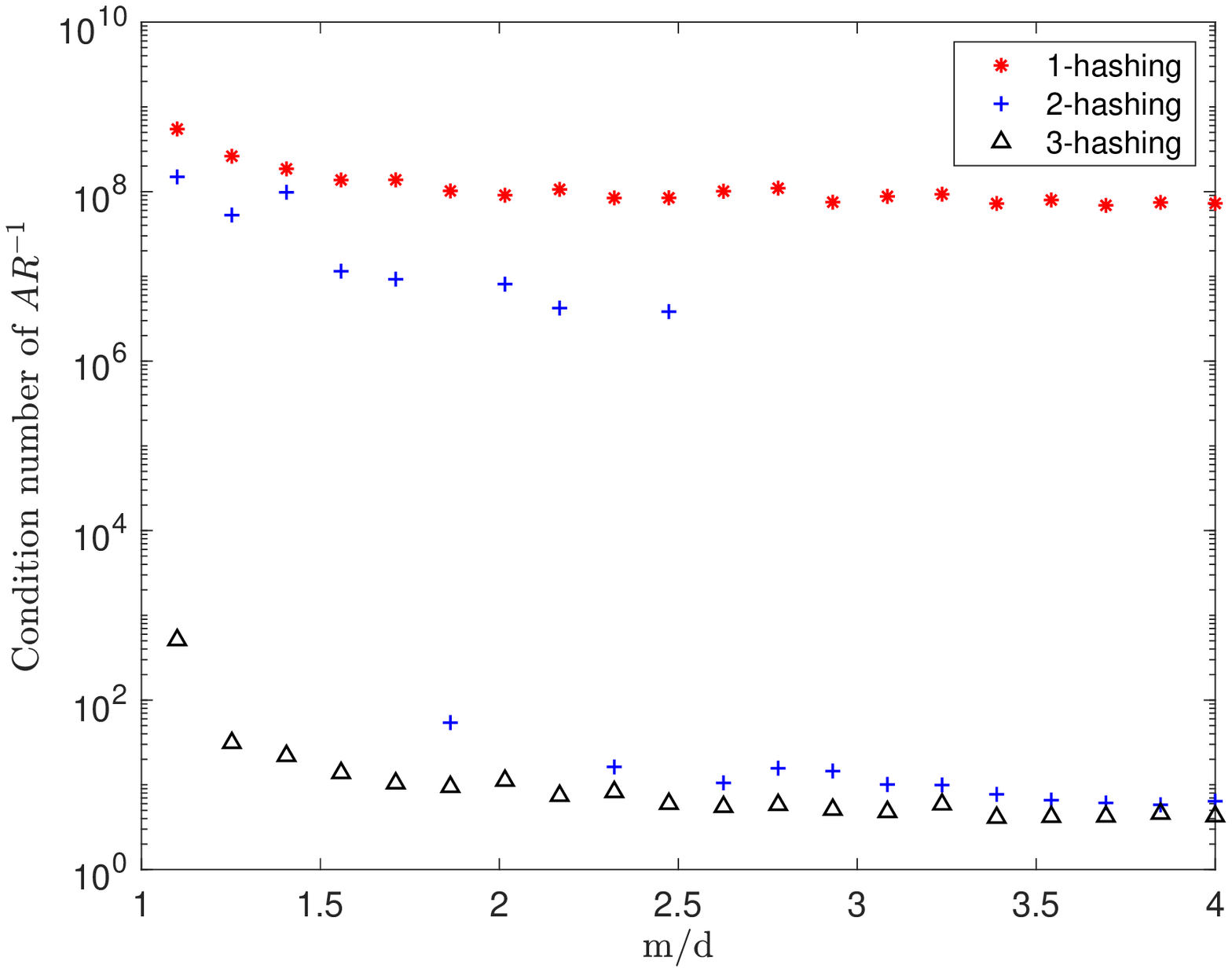} 
		        \caption{When the data matrix $A$ has higher coherence, using $s-$hashing with $s> 1$ is crucial to produce an acceptable preconditioner.}
		        \label{fig::1-2-3-semi-co}
		    \end{minipage}   
		\end{figure}

\subsection{Compilation and running environment for timed experiments}
\label{subsec::compilation_and_running_env}
The above numerical illustrations are done in MATLAB as it does not involve running time. For all the other studies,  unless otherwise mentioned, we use Intel C compiler icc with optimisation flag -O3 to compile all the C code, and Intel Fortran compiler ifort with -O3 to compile Fortran-based code. All code has been compiled in sequential mode and linked with sequential dense/sparse linear algebra libraries provided by Intel MKL, 2019 and Suitesparse 5.3.0. 
The machine used has Intel(R) Xeon(R) CPU E5-2667 v2 @ 3.30GHz with 8GB RAM.

\subsection{Tuning to set the default parameters}

The default parameter values $m$ for \solverNameDense{} (both with and without R-CPQR) solvers and $m,s$ for \solverNameSparse{} are chosen using a calibrating random matrix set. See the below graphs.

\paragraph{Calibration for Dense Solvers}
In \autoref{fig::new_blen_engineering}, 
\autoref{fig::new_blen_noCPQR_engineering},
\autoref{fig::blen_engineering},
\autoref{fig::LSRN_engineering}
we tested \solverNameDense{}, \solverNameDense{} without R-CPQR, Blendenpik and LSRN on the same calibration set and chose the optimal parameters for them for fair comparison. The default parameters chosen are $m = 1.7d$, $m=1.7d$, $m=2.2d$ and $m=1.1d$ for \solverNameDense{}, \solverNameDense{} without R-CPQR, Blendenpik and LSRN respectively.

\newcommand{\mysize}{0.3}

\calibrationSixFigures{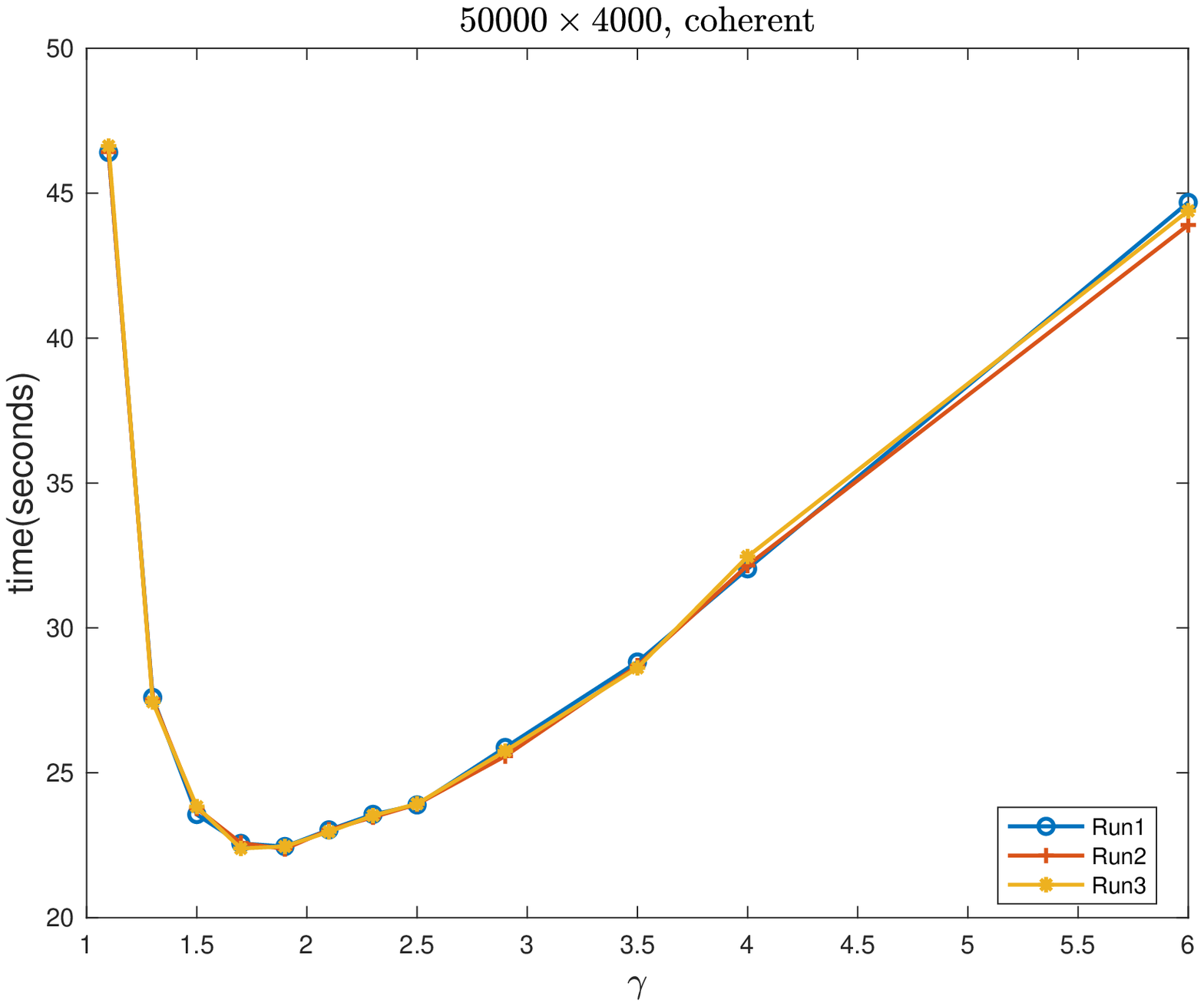}
{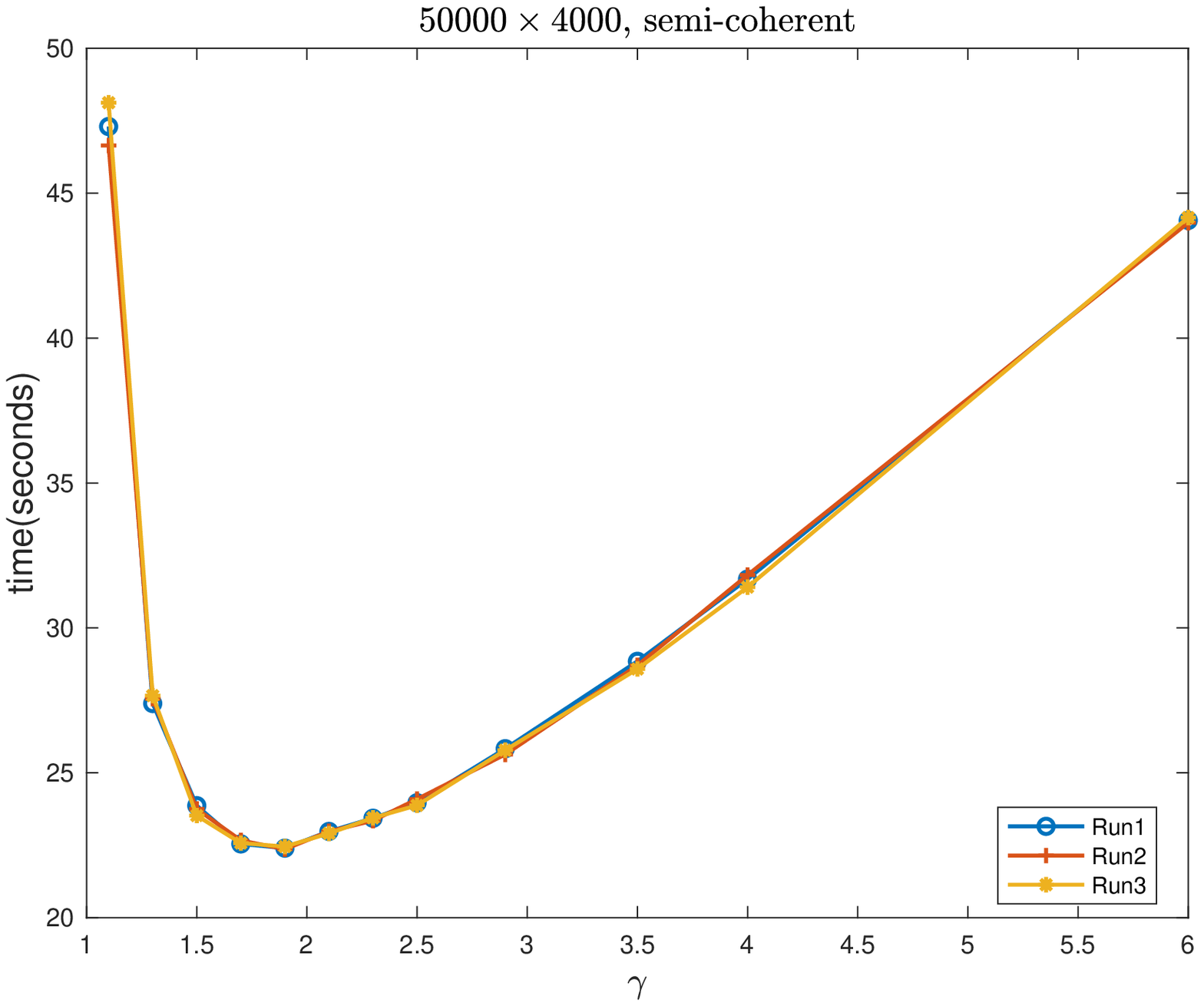}
{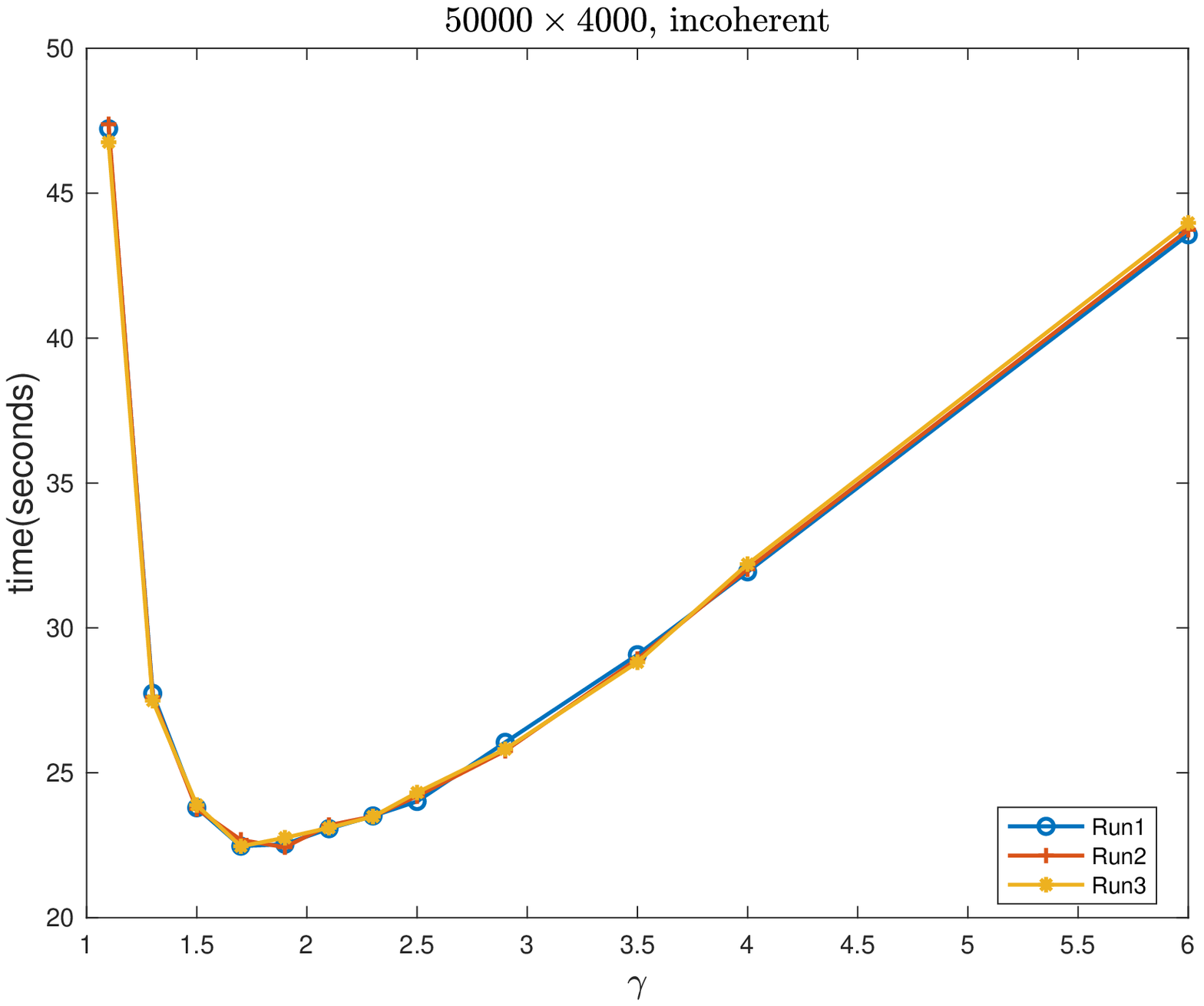}
{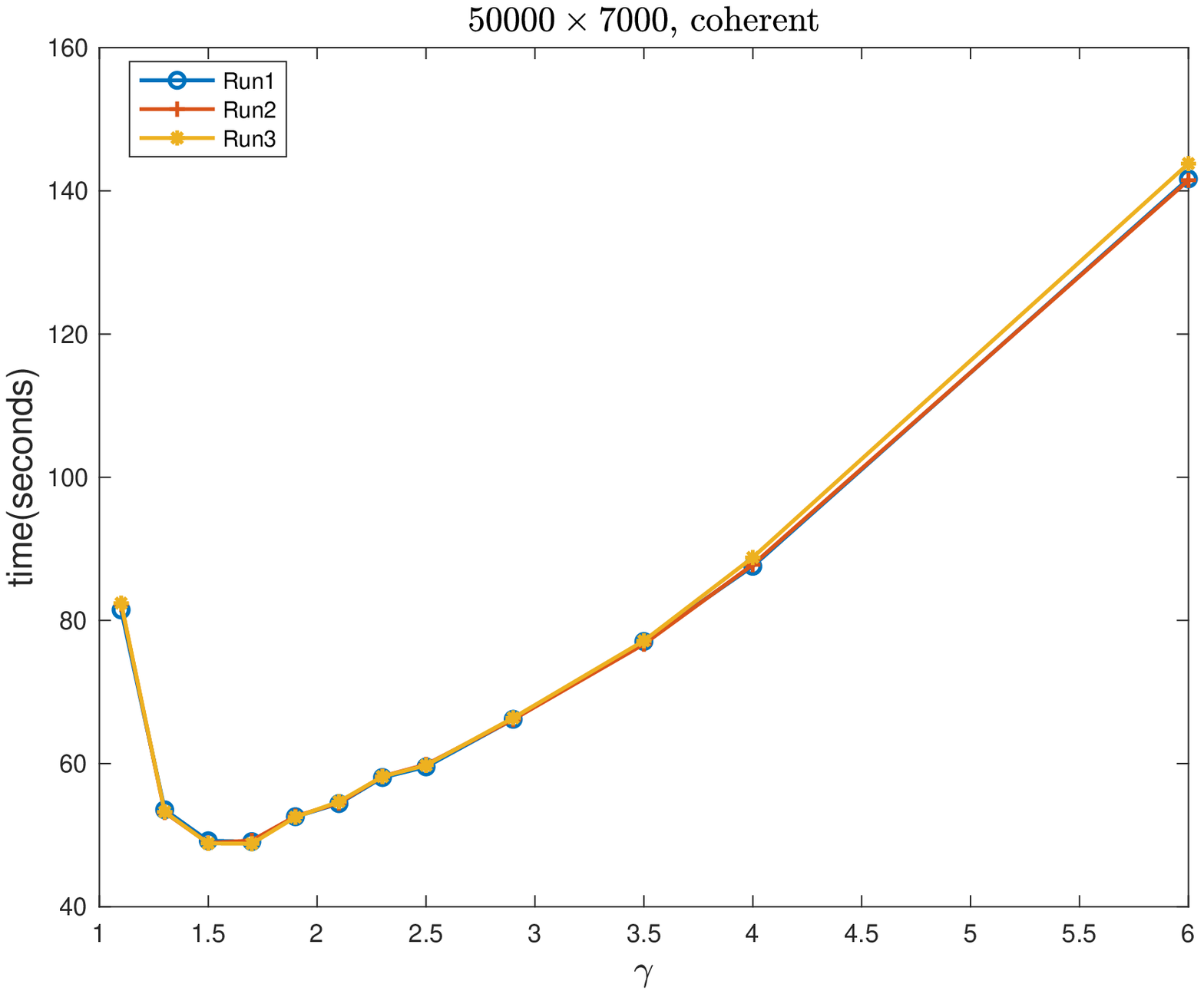}
{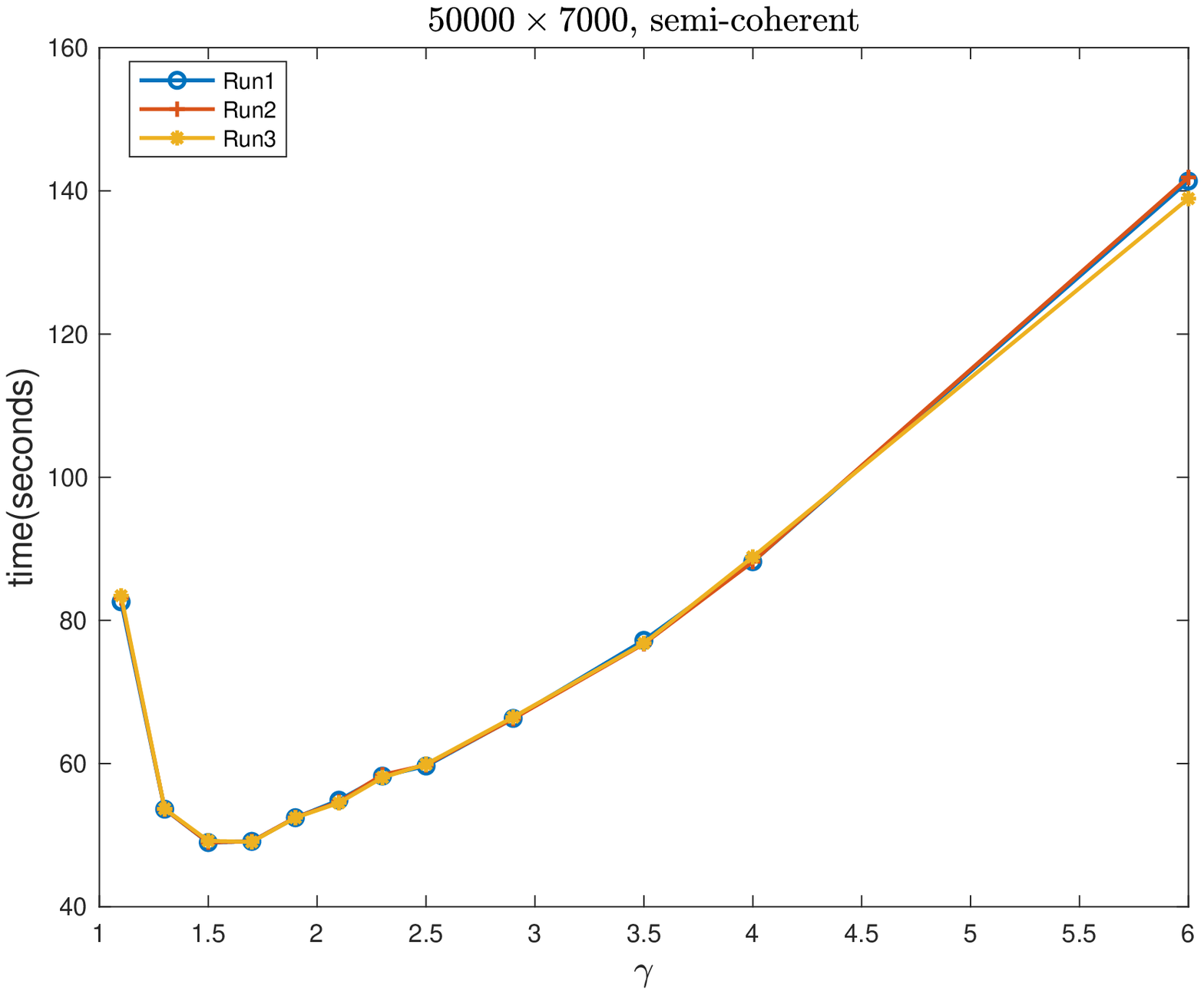}
{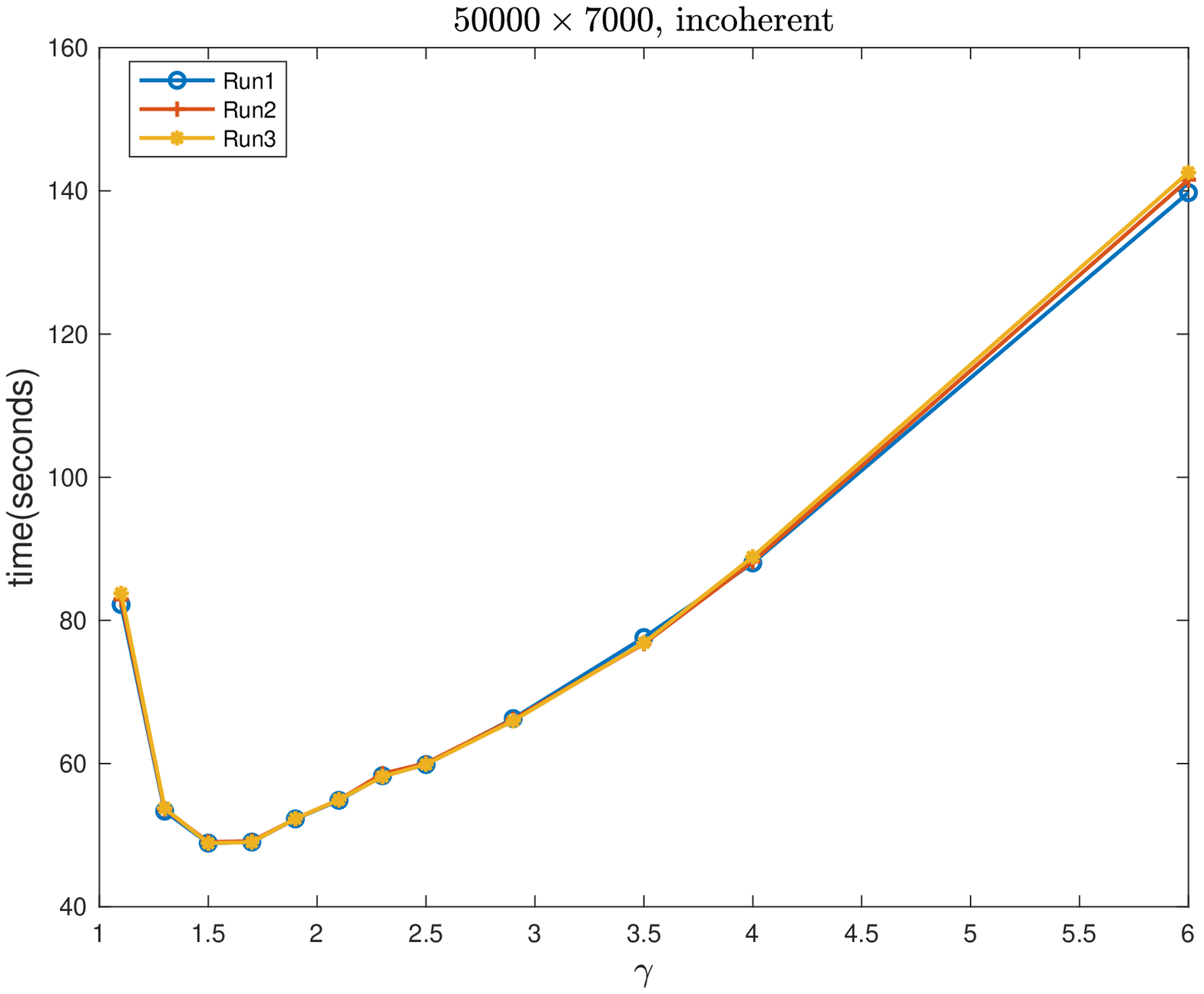}
{\calibrationDenseCaptionSentenceOne{\solverNameDense{}}. \calibrationDenseCaptionSentenceTwo{\solverNameDense{}}. \calibrationDenseCaptionSentenceThree{$
\gamma=1.7$}.}
{fig::new_blen_engineering}

\calibrationSixFigures{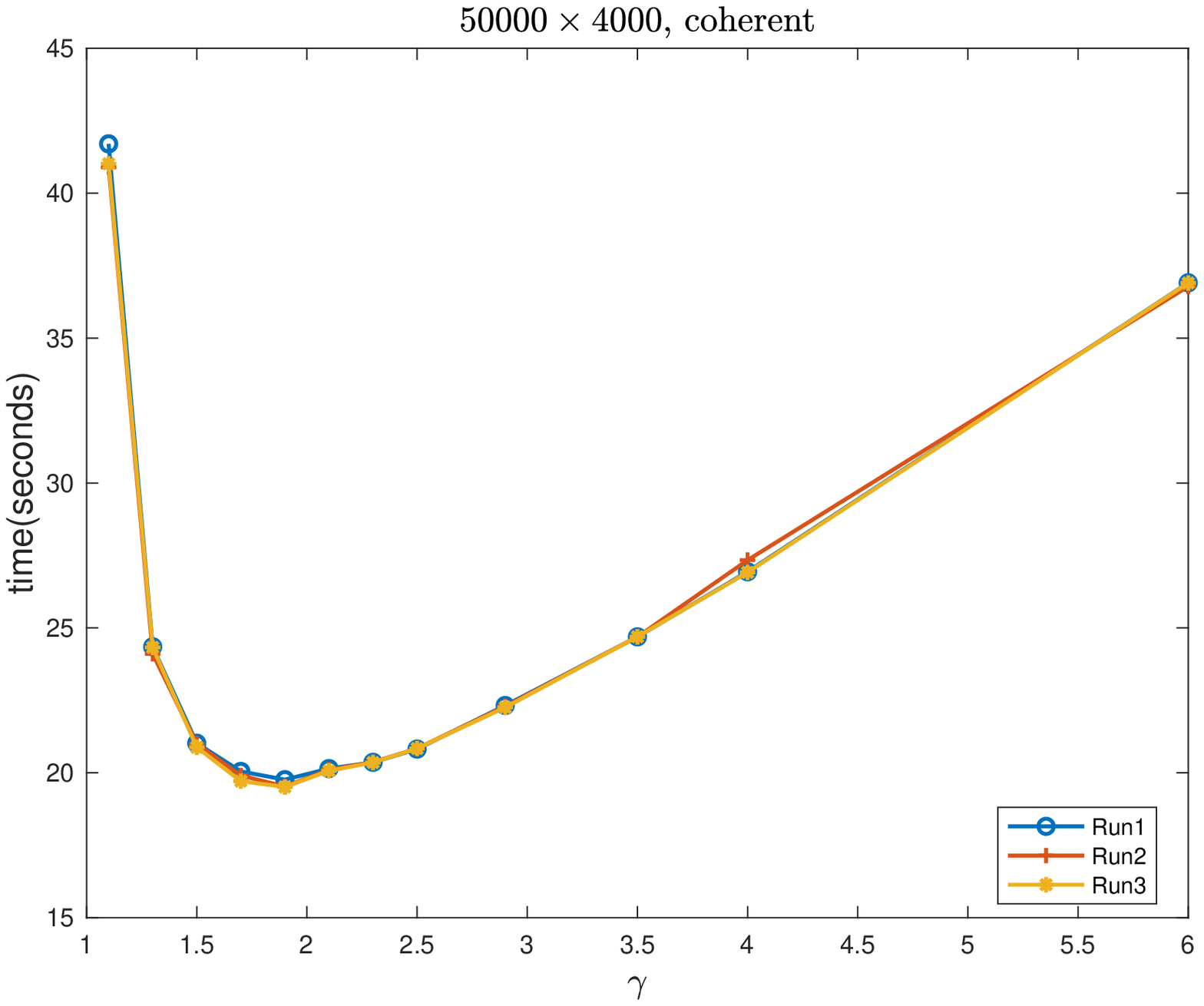}
{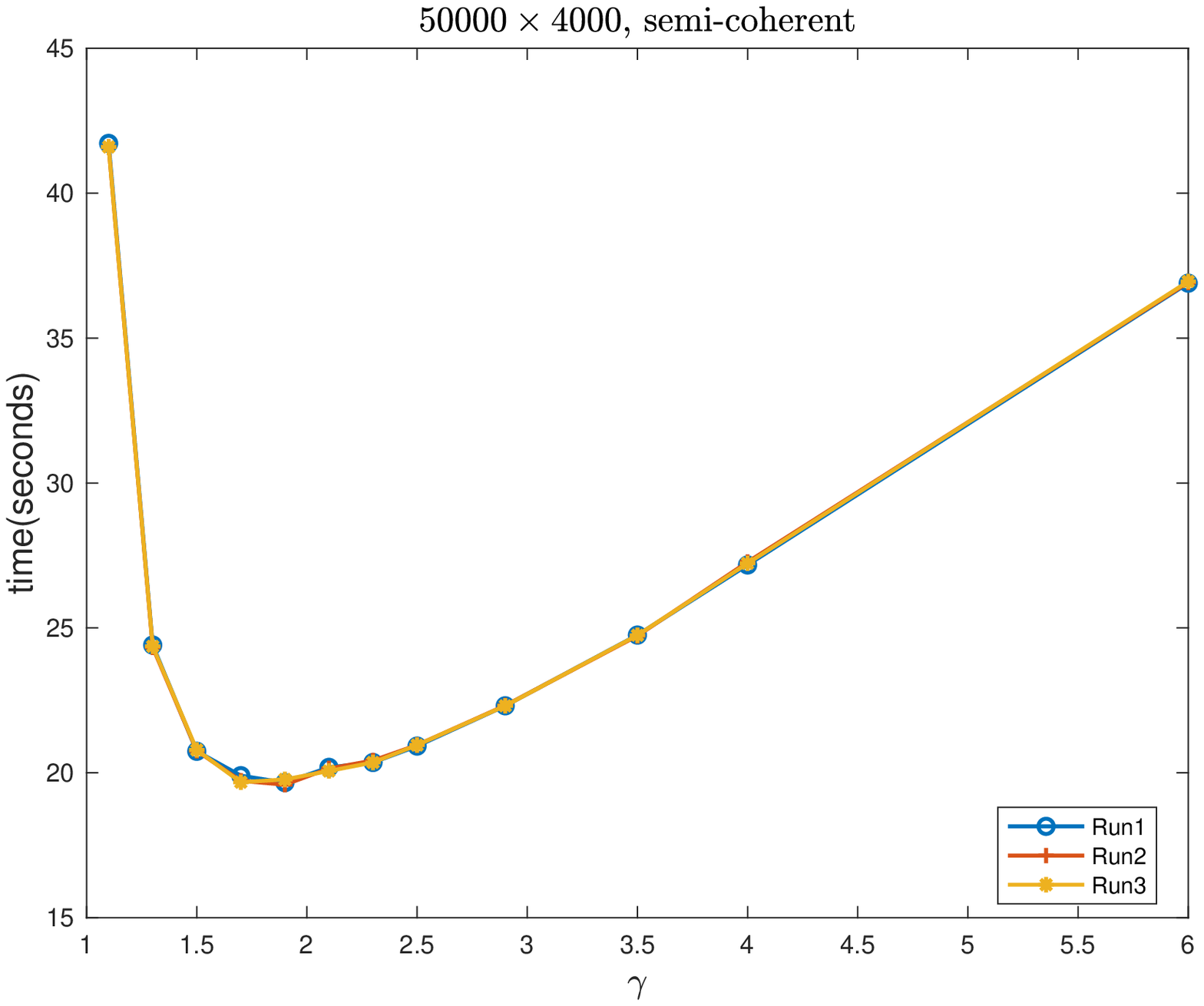}
{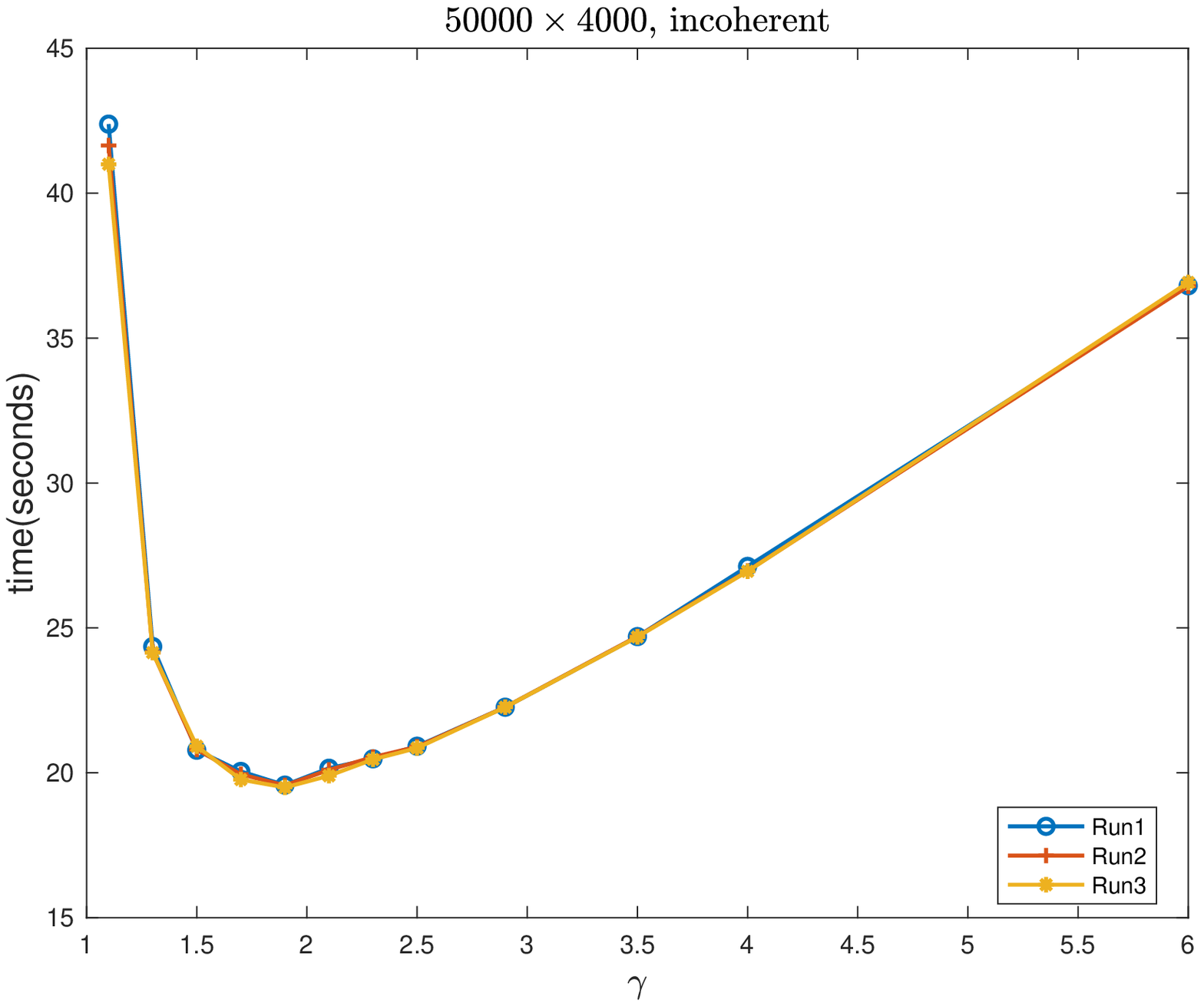}
{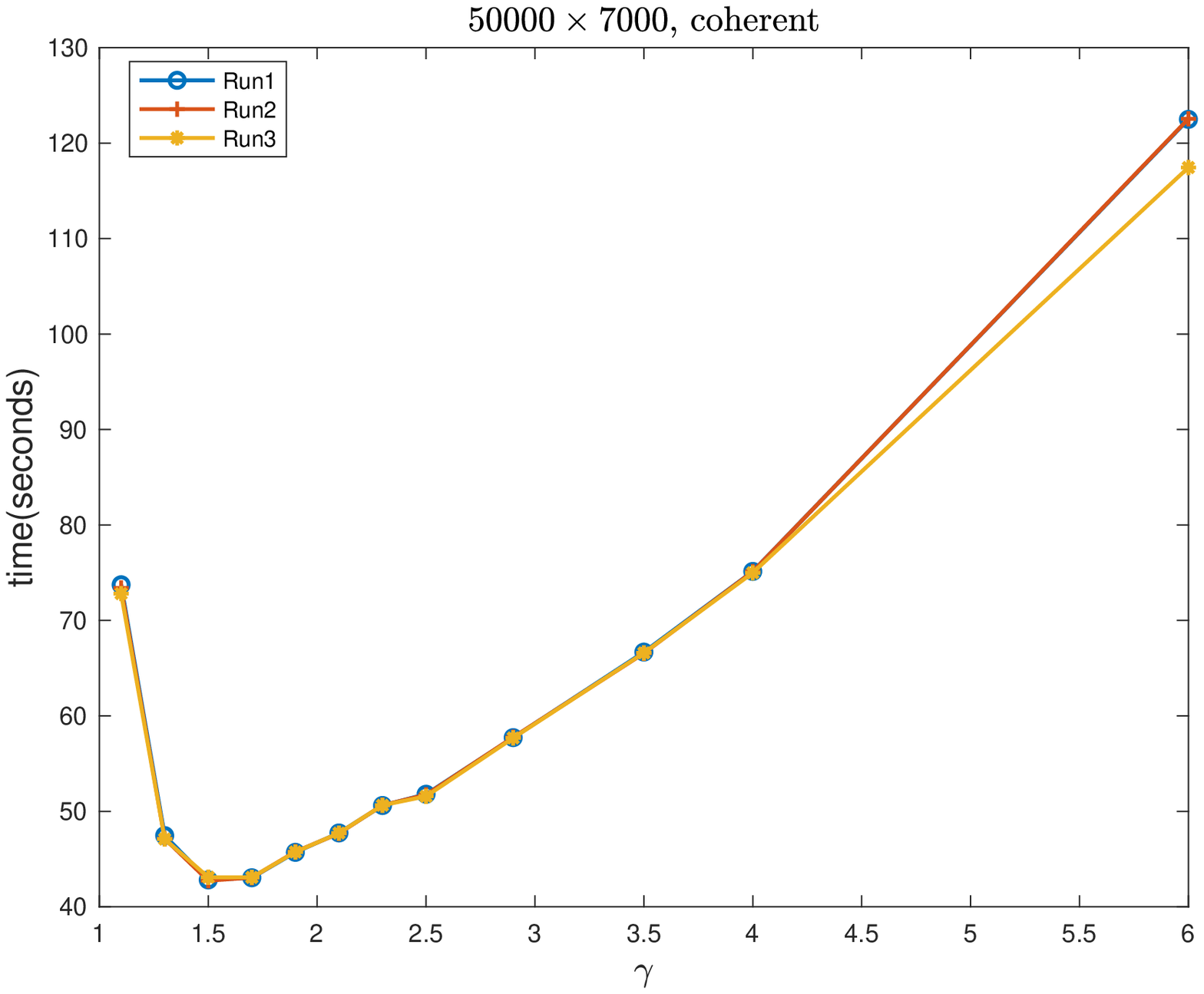}
{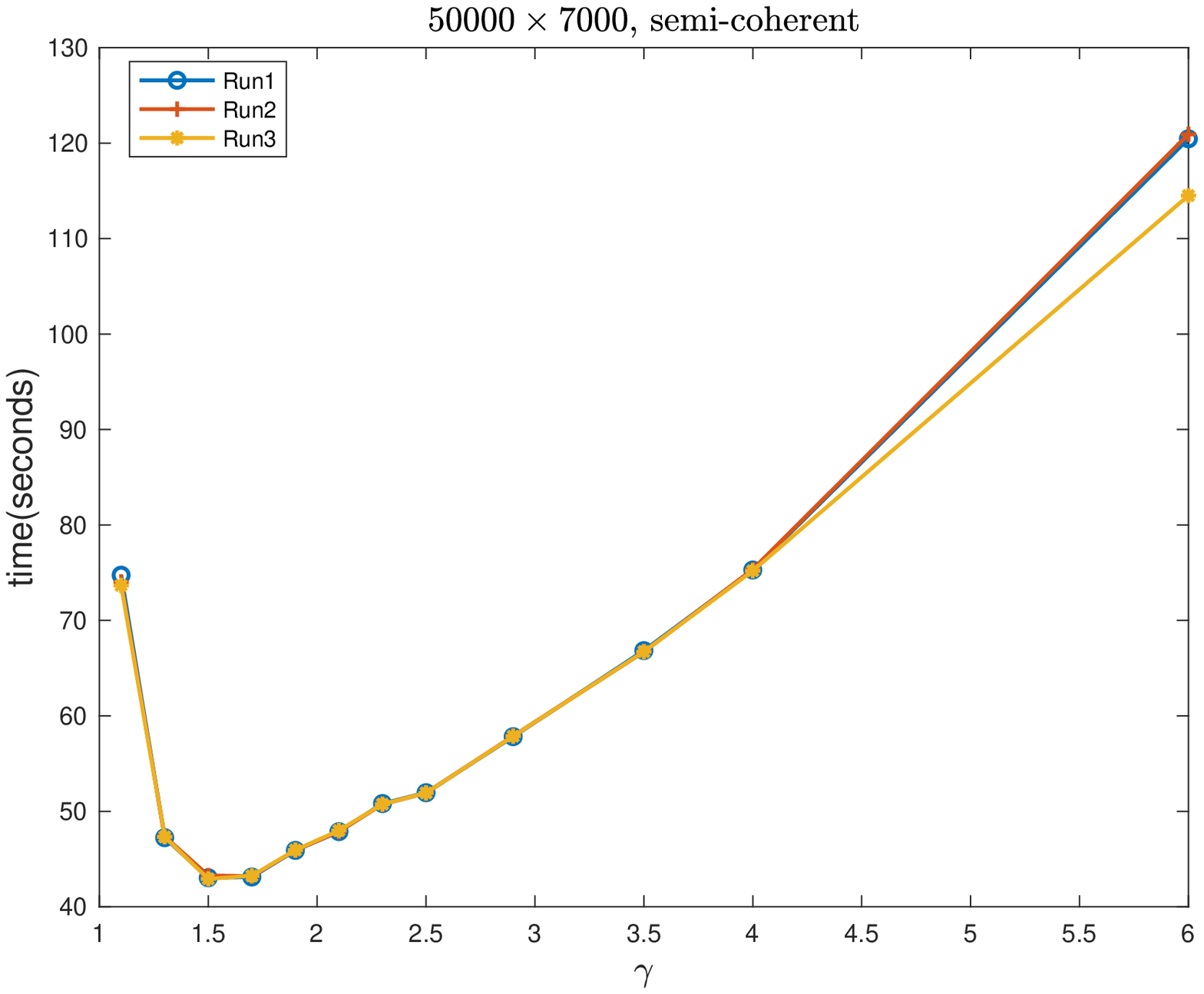}
{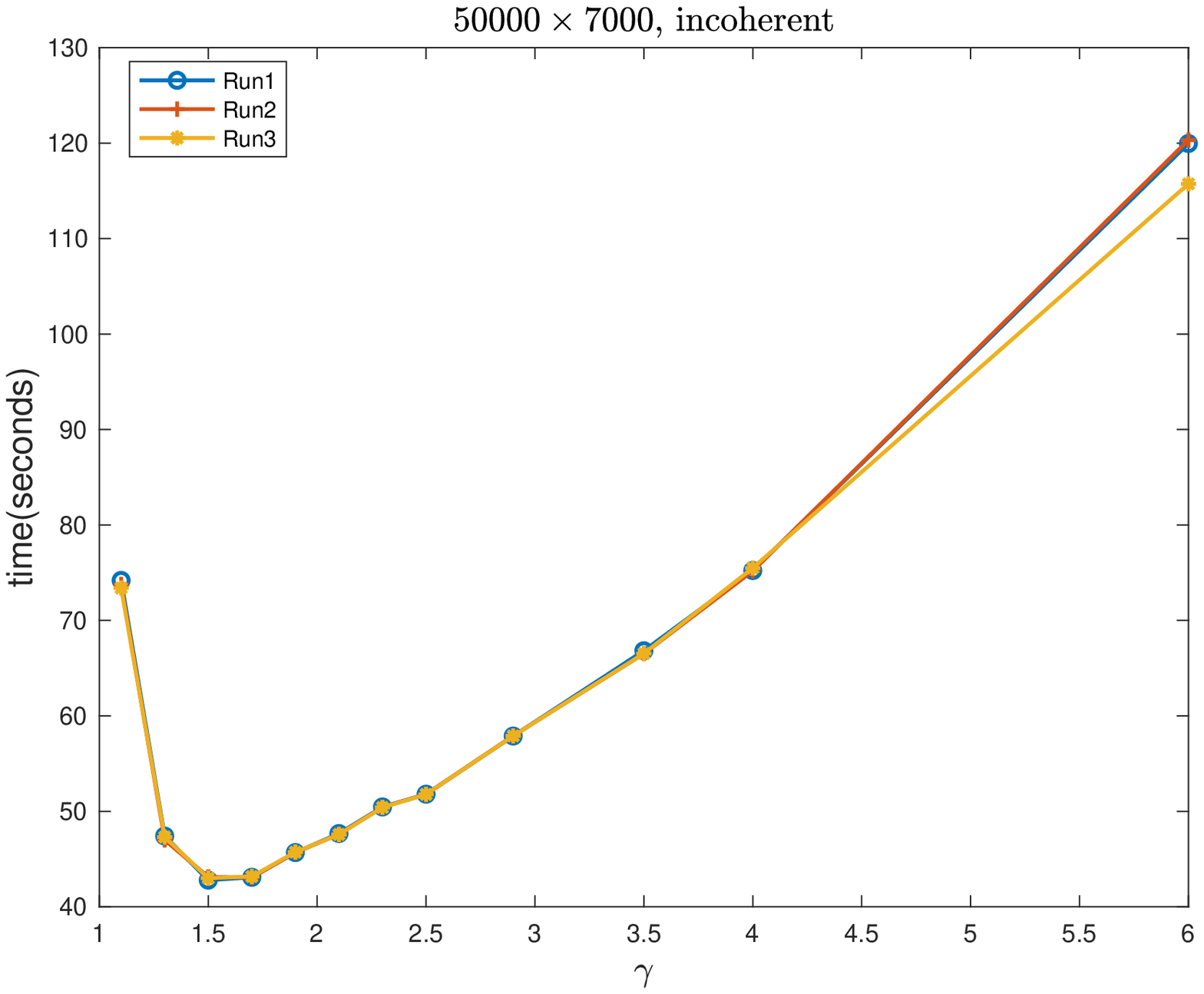}
{\calibrationDenseCaptionSentenceOne{\solverNameDense{} without R-CPQR}. \calibrationDenseCaptionSentenceTwo{\solverNameDense{} without R-CPQR}. Note that using LAPACK QR instead of R-CPQR results in slightly shorter running time (c.f. Figure \ref{fig::new_blen_engineering} ). \calibrationDenseCaptionSentenceThree{$\gamma=1.7$}.}
{fig::new_blen_noCPQR_engineering}

\calibrationSixFigures{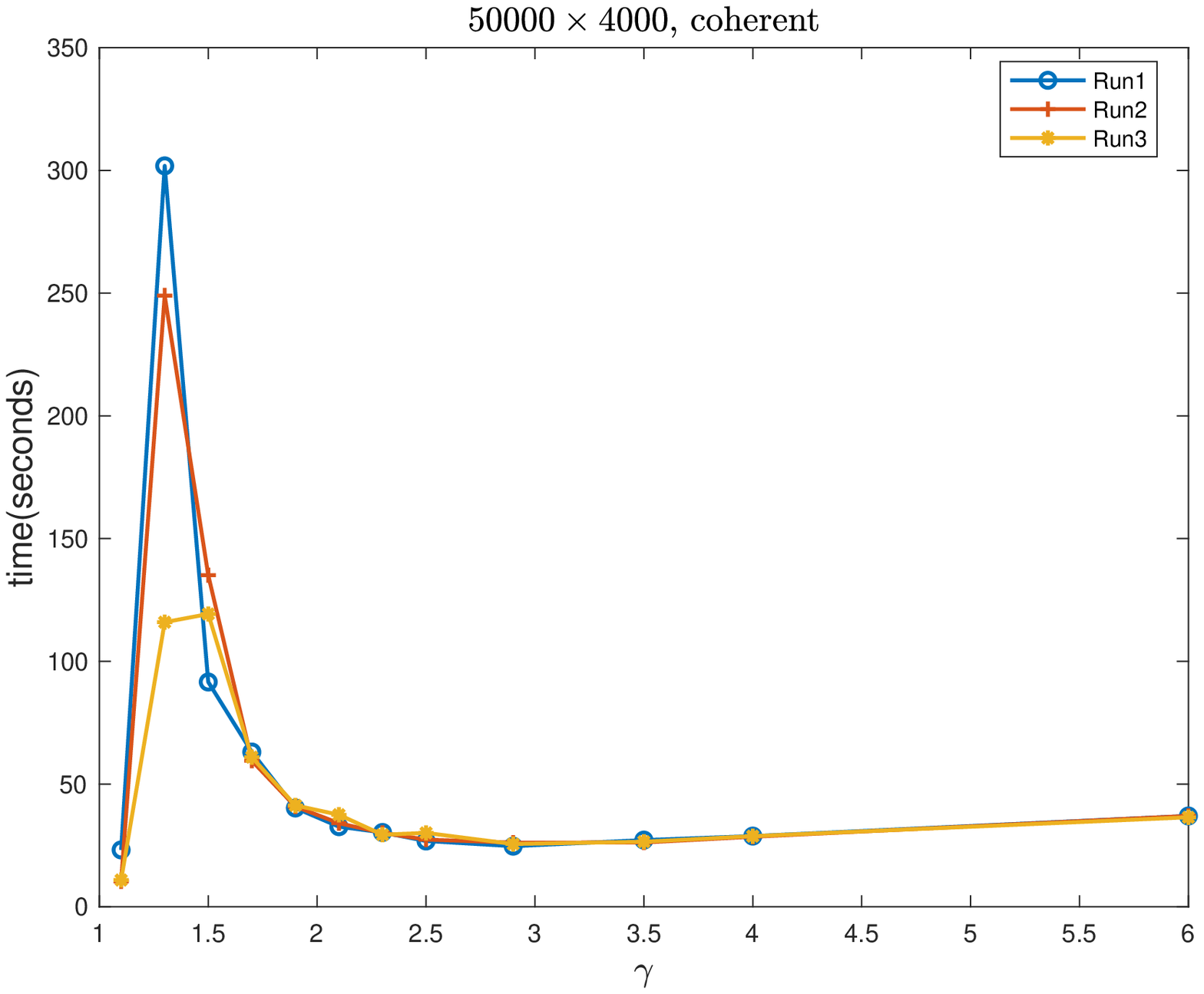}
{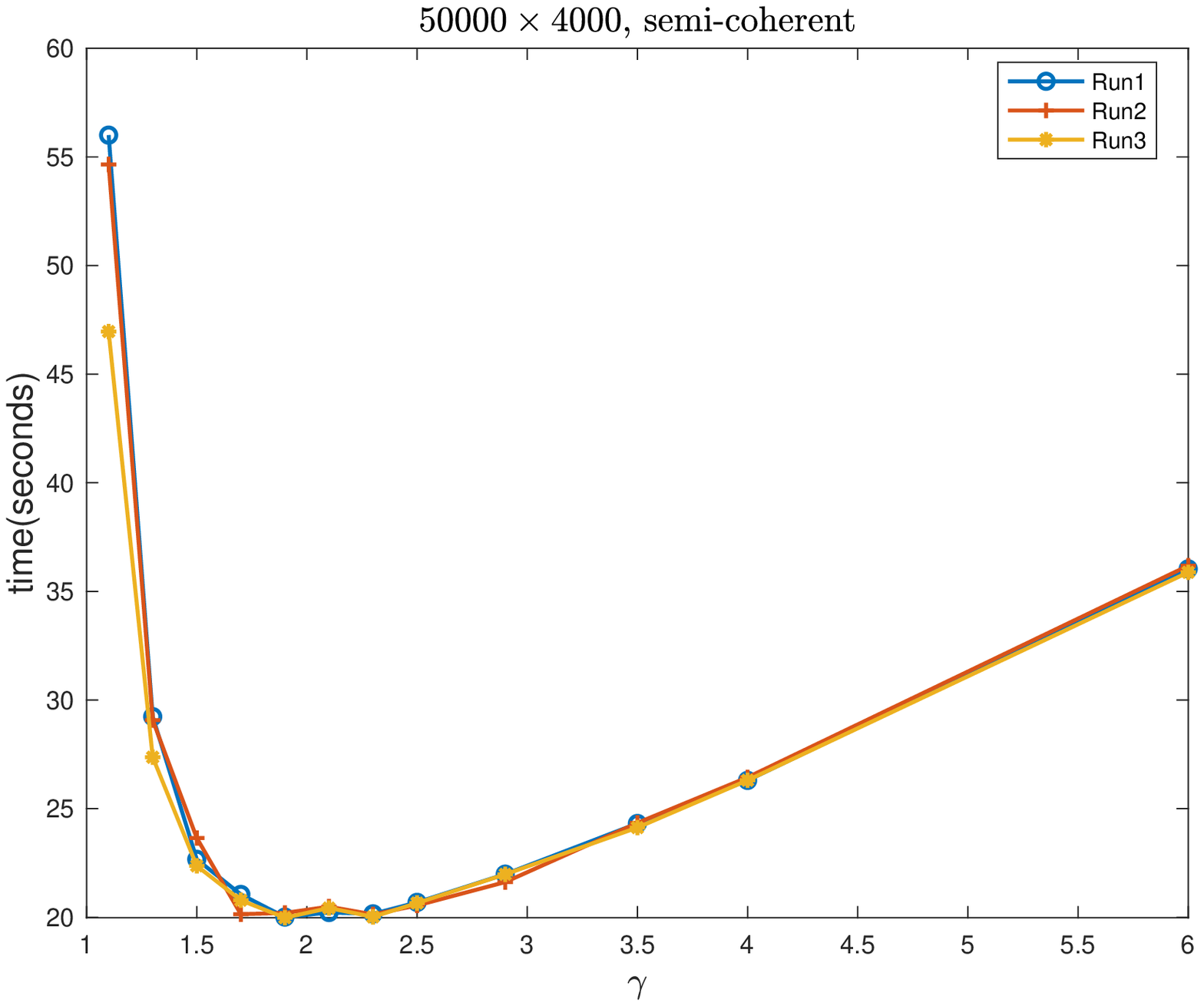}
{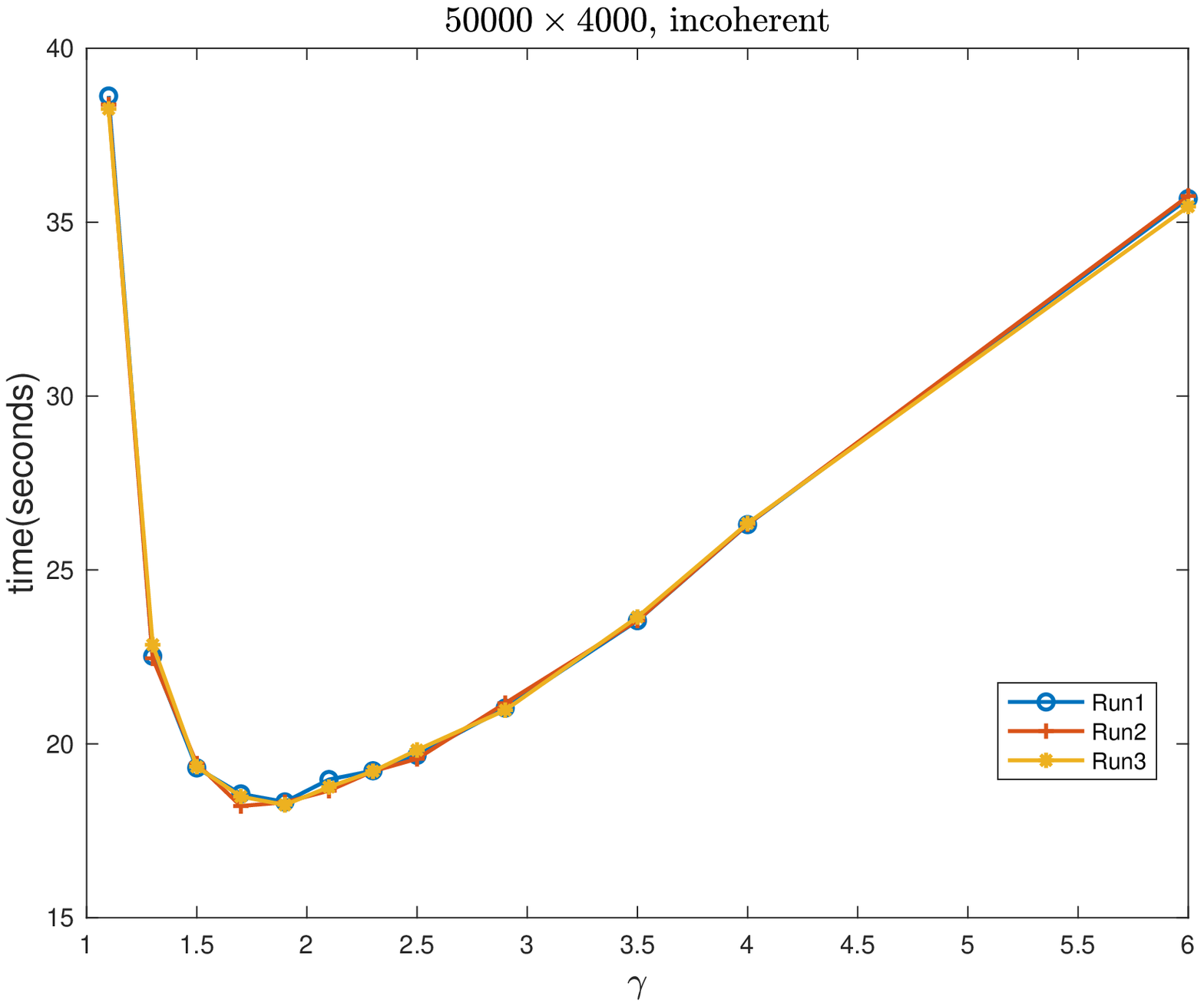}
{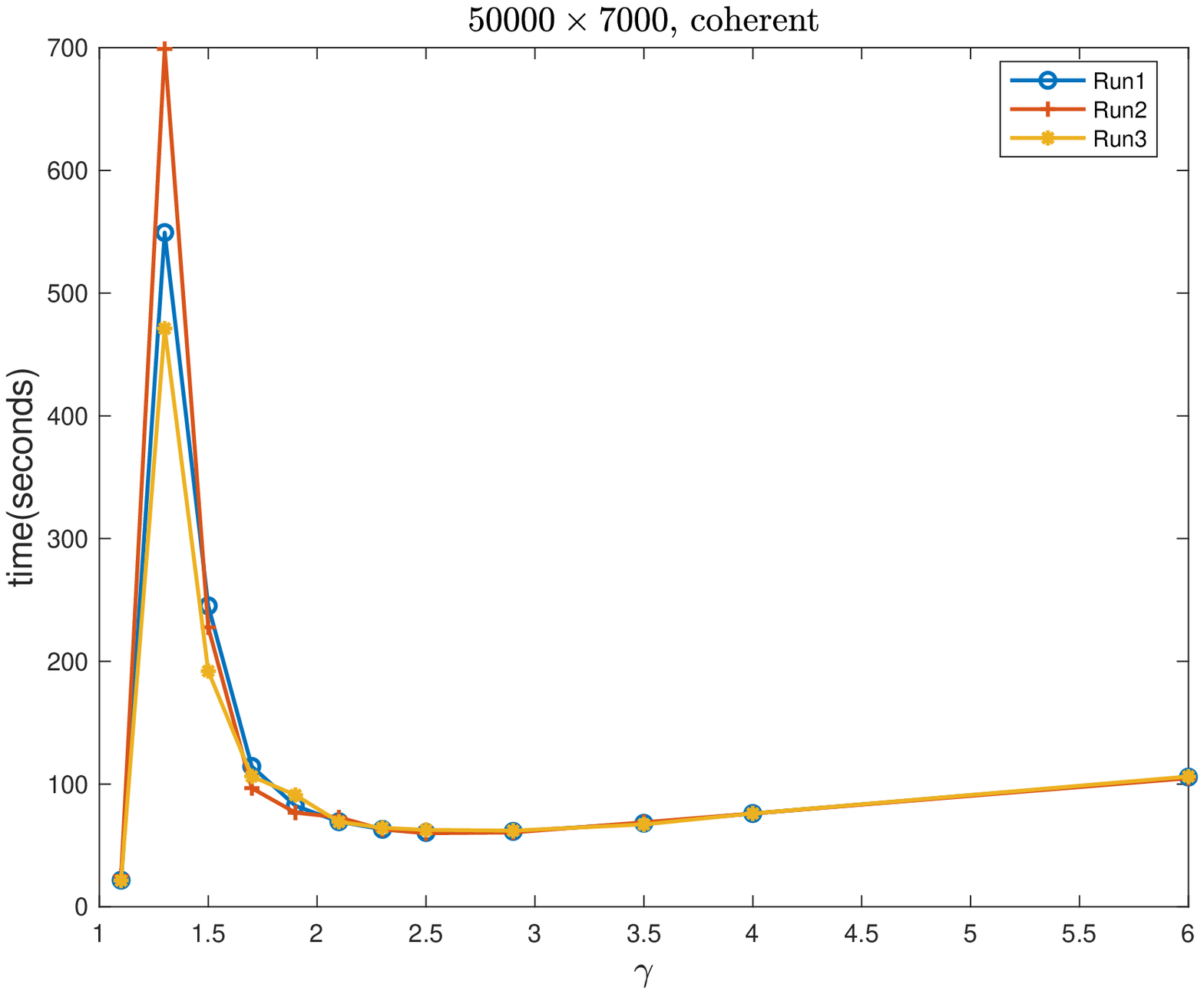}
{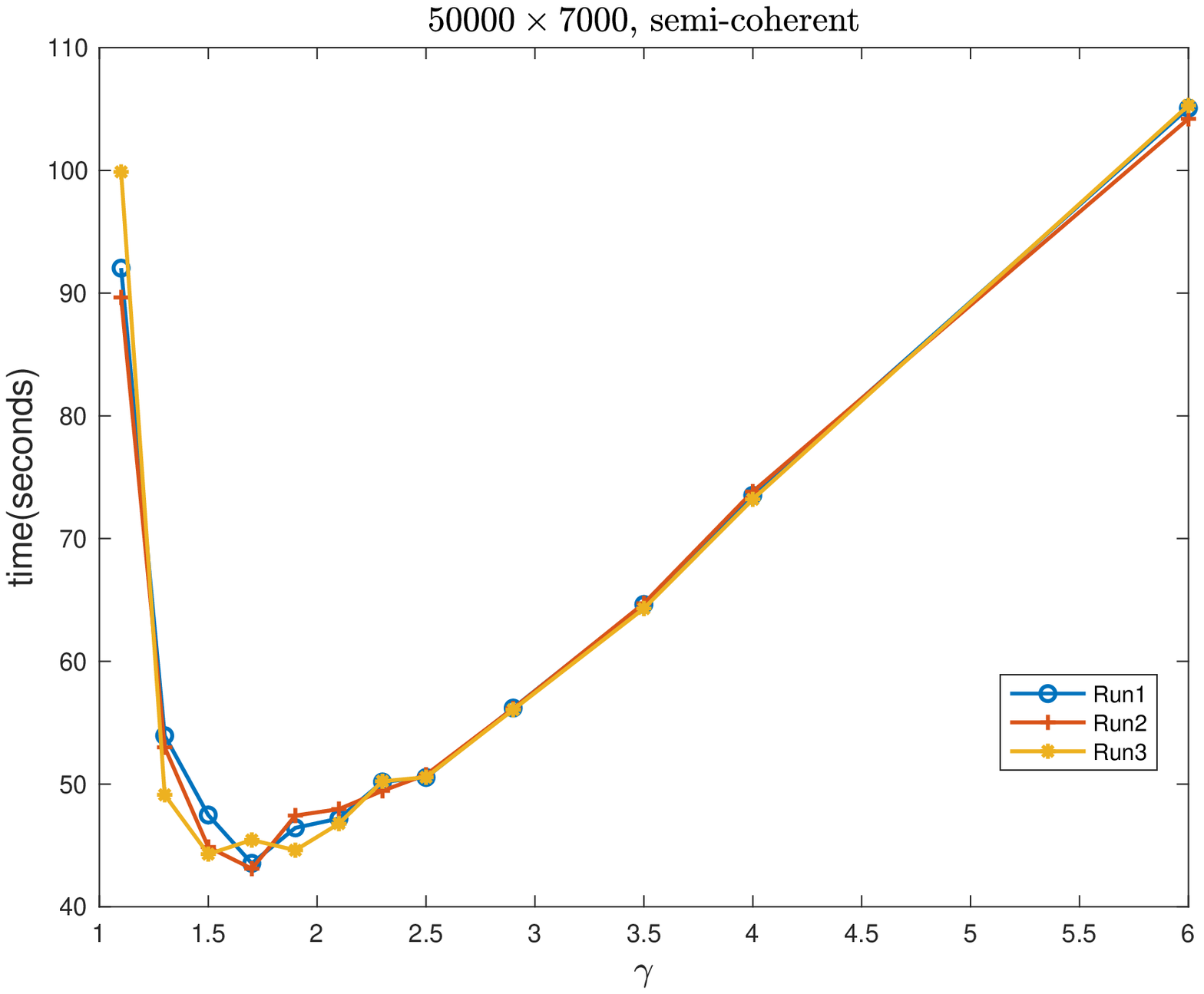}
{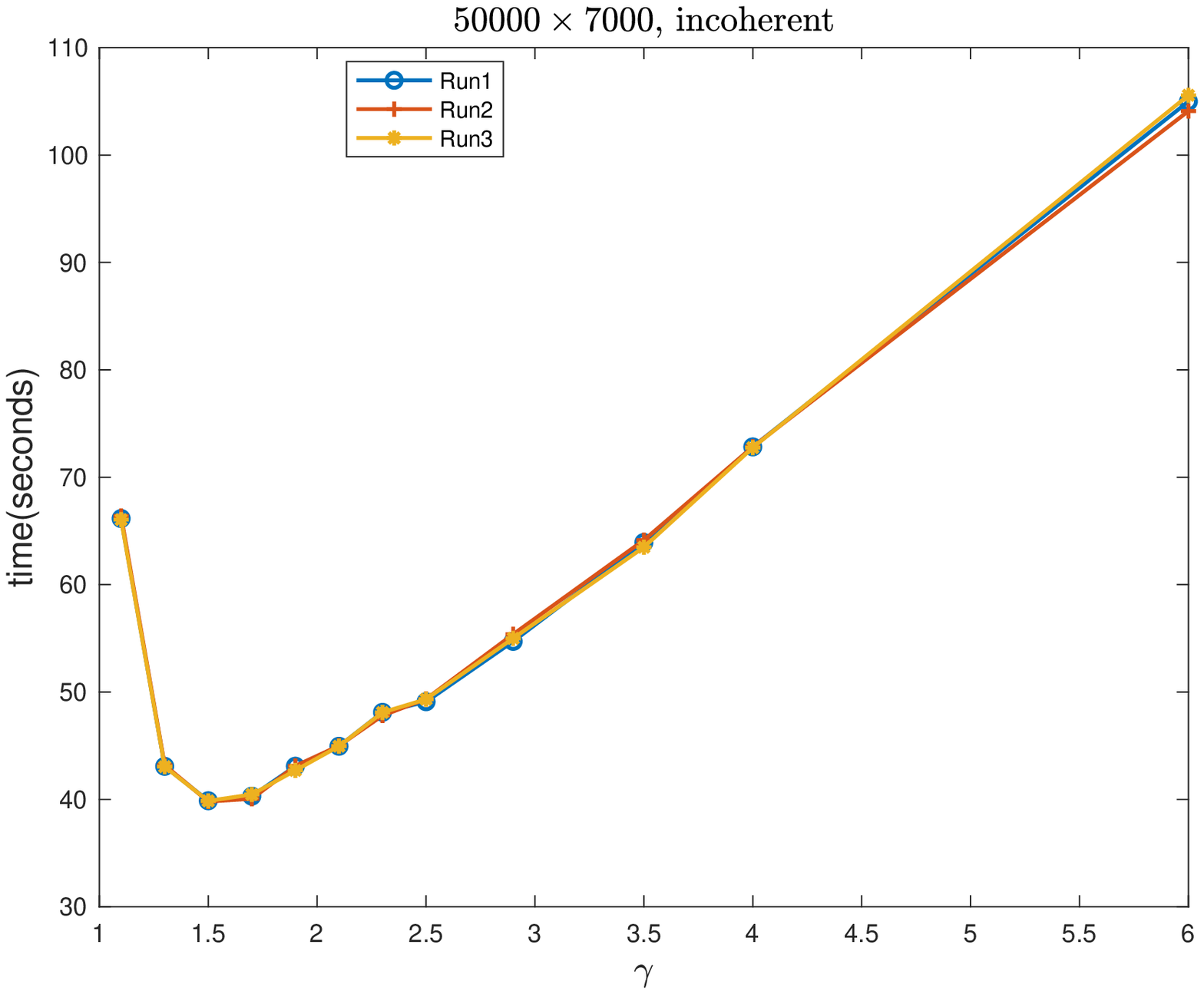}
{\calibrationDenseCaptionSentenceOne{Blendenpik}. \calibrationDenseCaptionSentenceTwo{Blendenpik}. Note that Blendenpik handles coherent dense $A$ significantly less well than our dense solvers. \calibrationDenseCaptionSentenceThree{$\gamma=2.2$}.}
{fig::blen_engineering}

\calibrationSixFigures{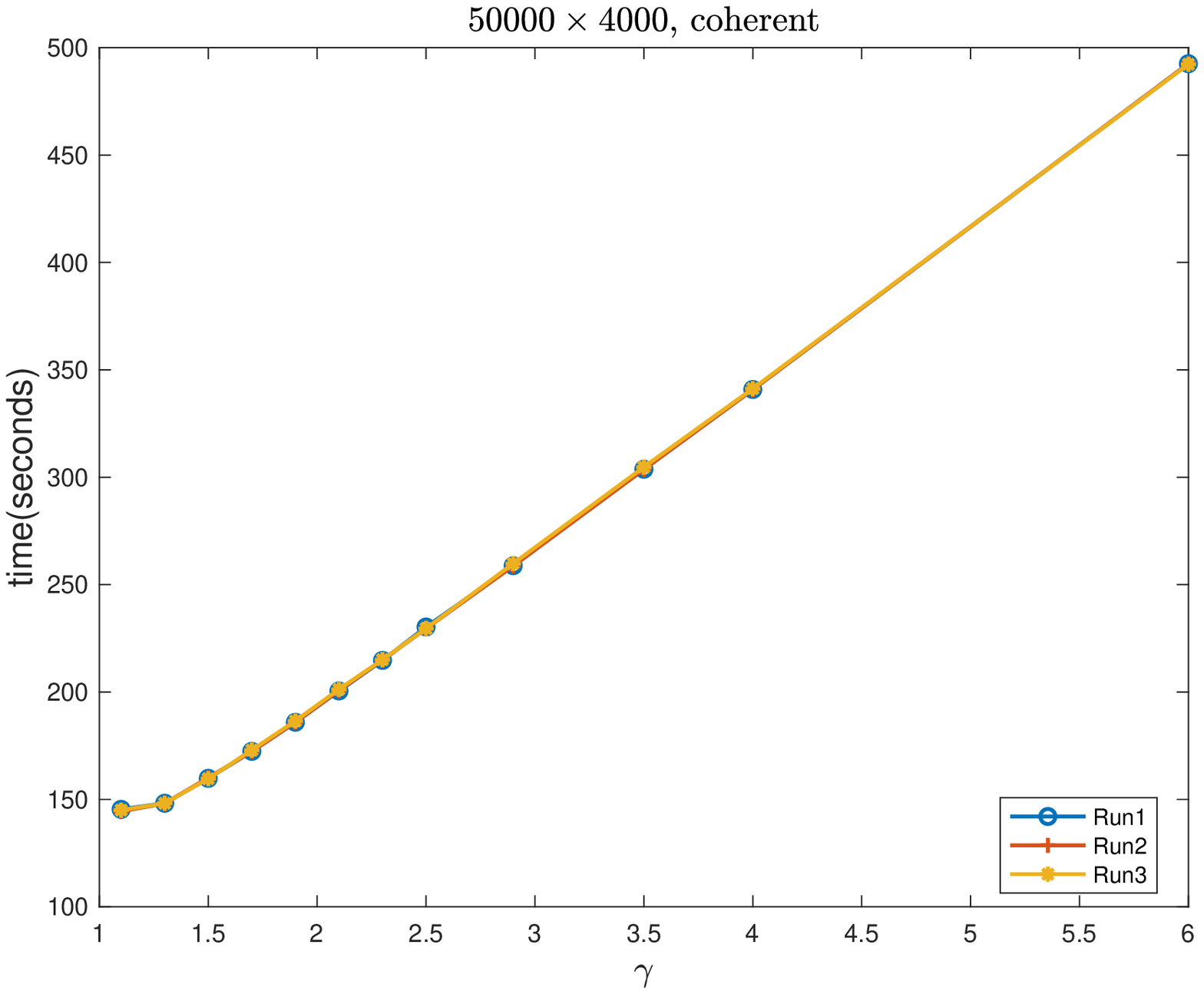}
{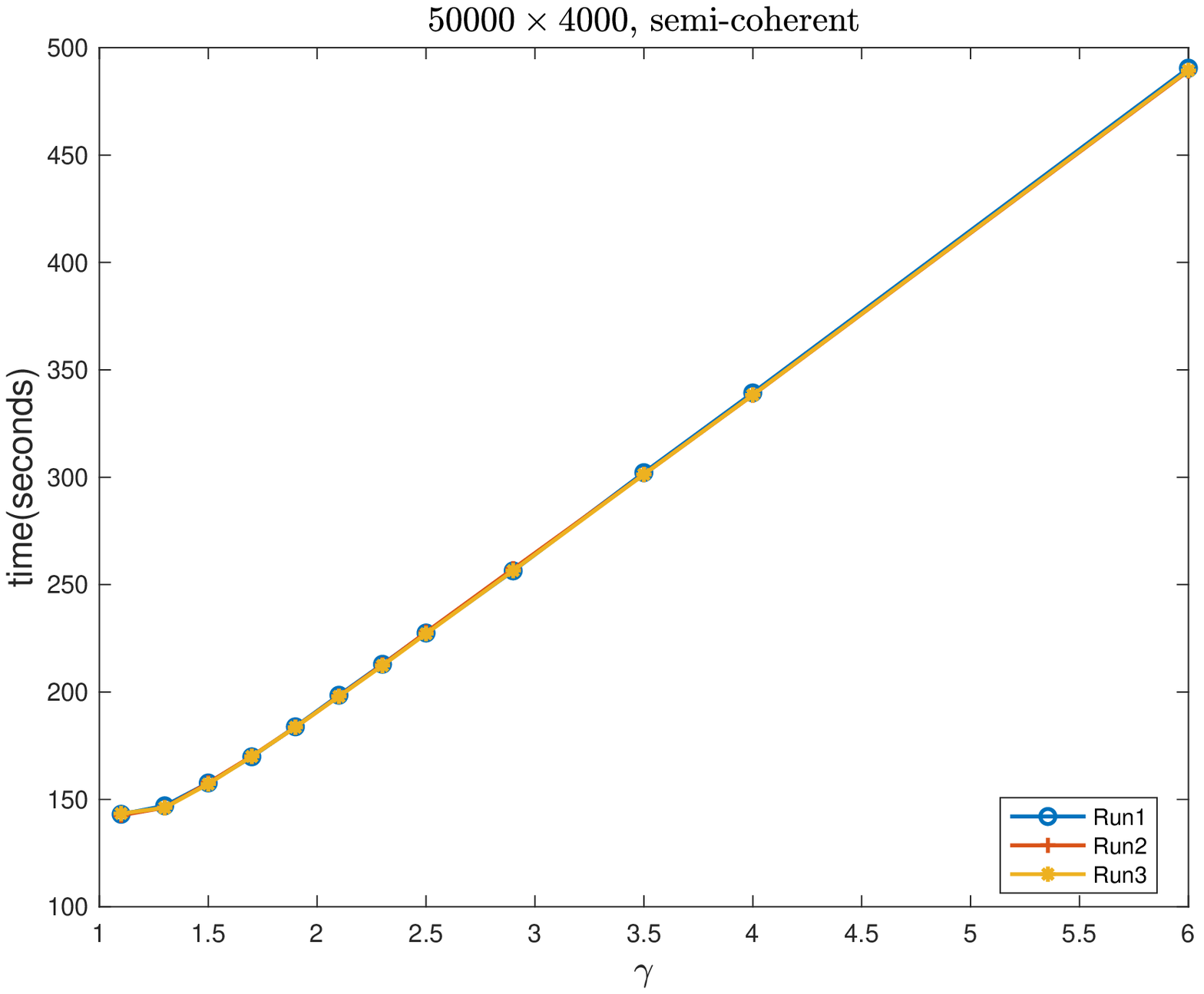}
{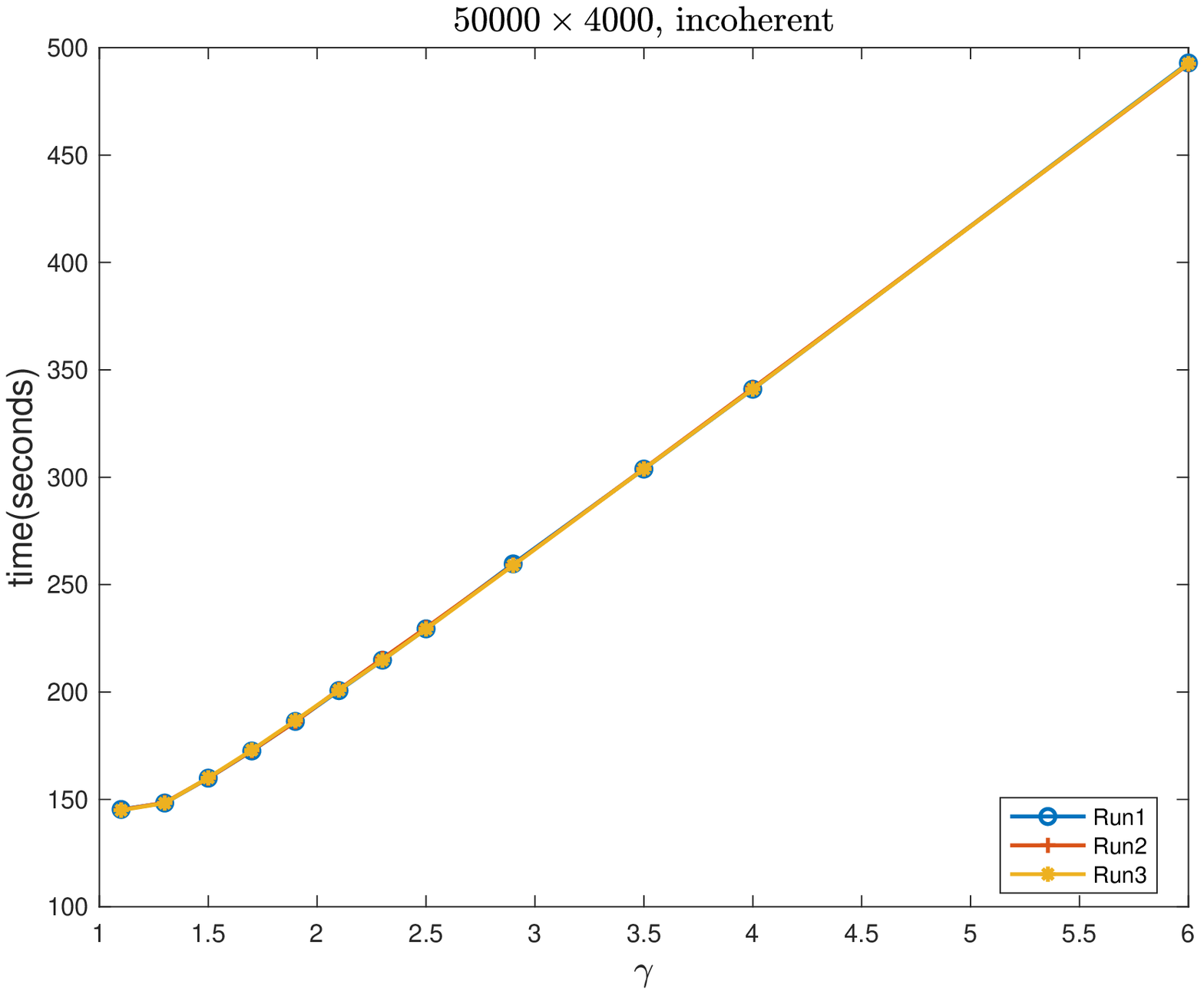}
{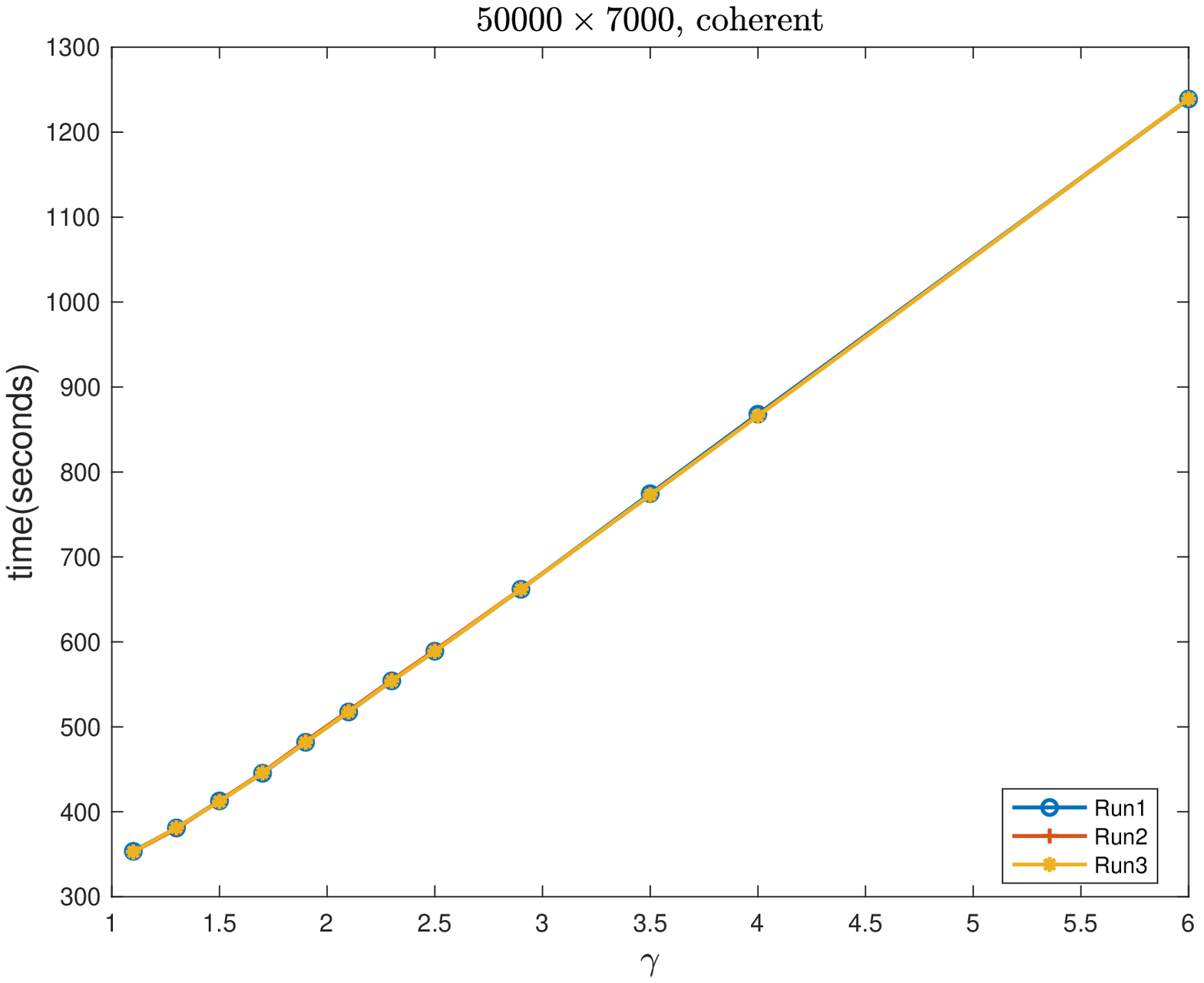}
{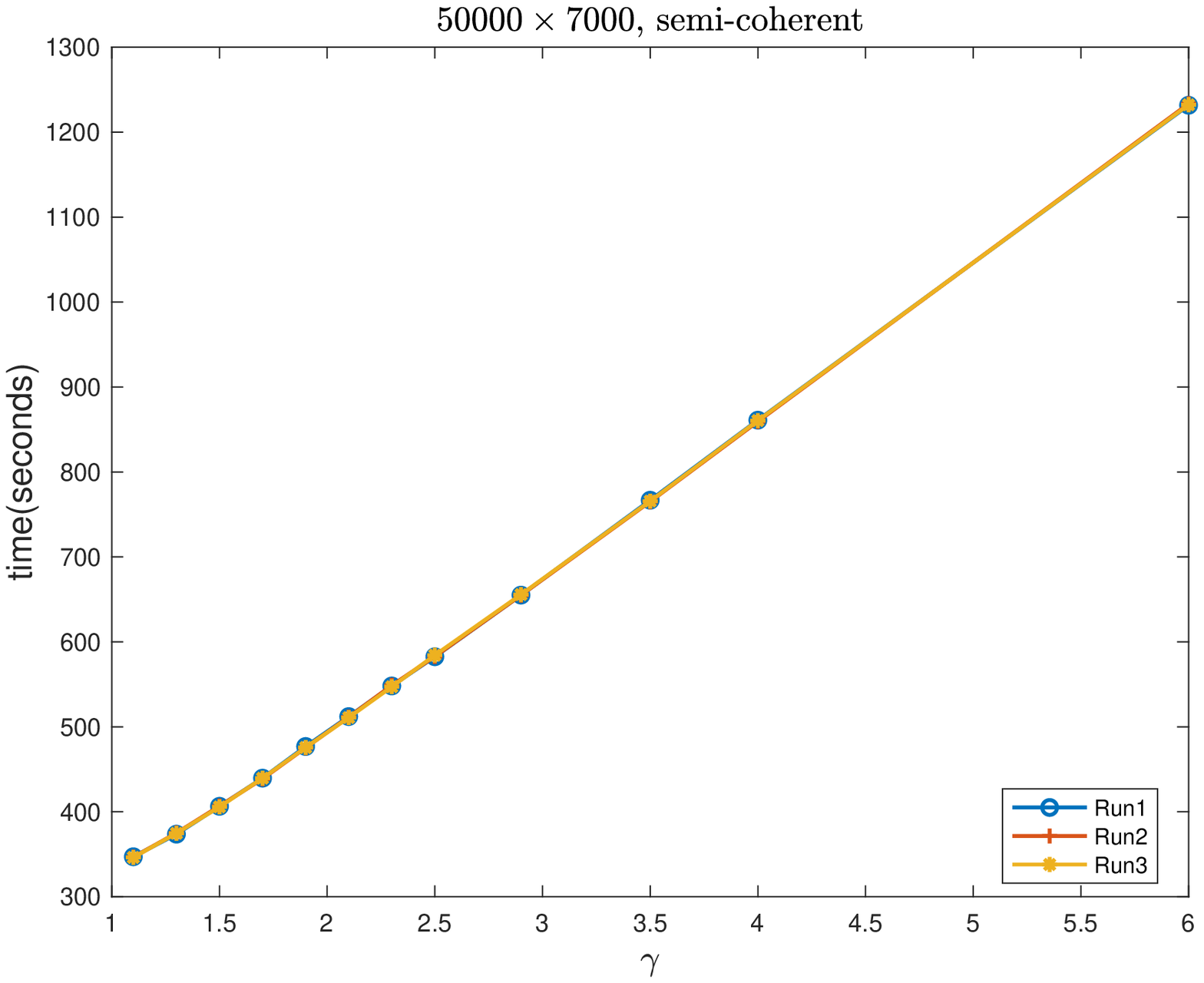}
{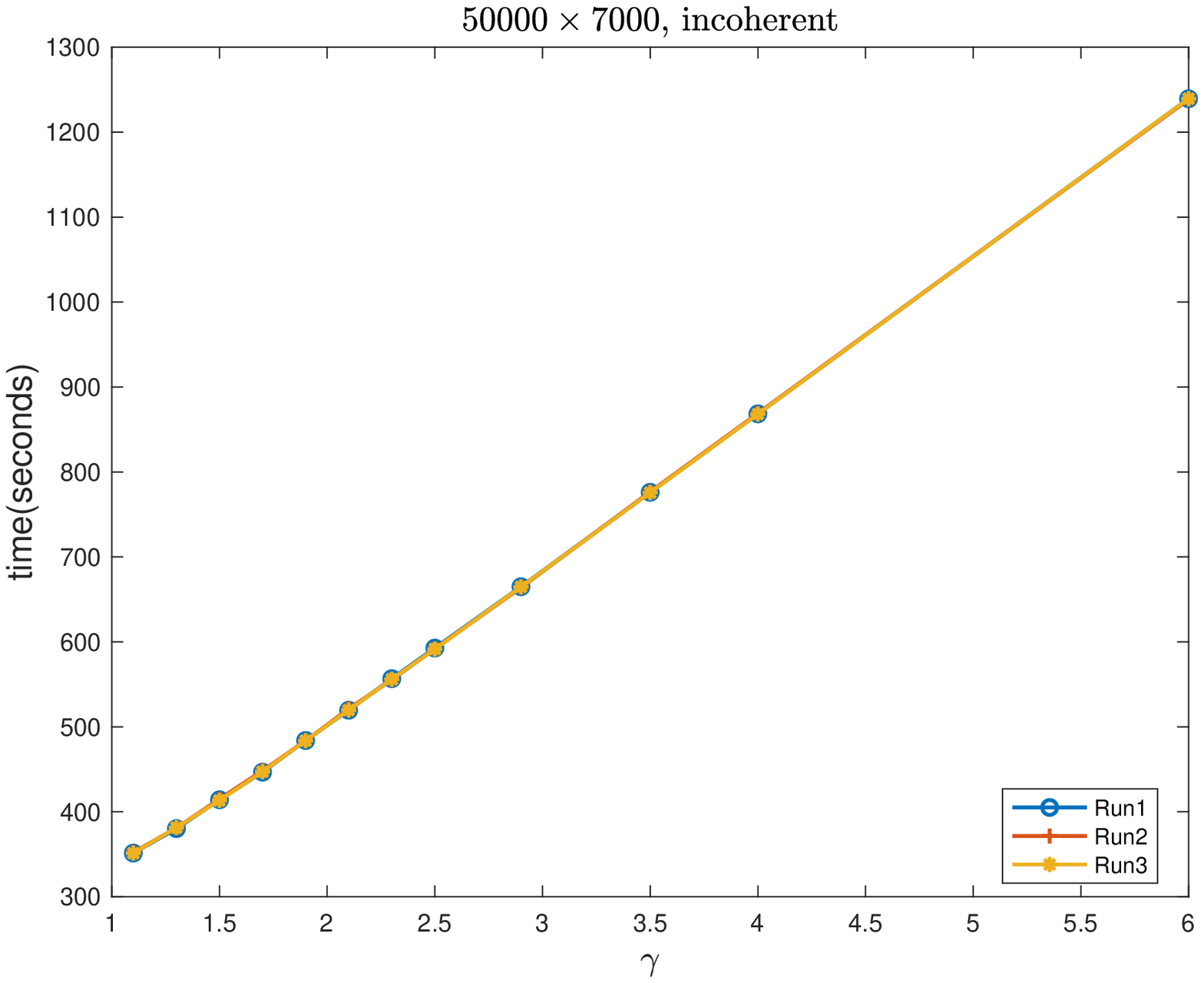}
{\calibrationDenseCaptionSentenceOne{LSRN}. \calibrationDenseCaptionSentenceTwo{LSRN}. Note that LSRN runs more than 5 times slower comparing to Blendenpik or \solverName{} in the serial testing environment, due to the use of SVD and Gaussian sketching. \calibrationDenseCaptionSentenceThree{$\gamma=1.1$}.} 
{fig::LSRN_engineering}

\paragraph{Calibration for \solverNameSparse{}}

In In Figures \ref{fig::Ls_qr_engineering_time} and \autoref{fig::Ls_qr_engineering_residual} 
\autoref{fig::Ls_lsrn_engineering_time} and \autoref{fig::Ls_lsrn_engineering_residual}, we tested \solverNameSparse{} and LSRN on the same calibration set and choose the optimal parameters from them for fair comparison.
Note that in the below calibration, $\tau_r = 10^{-4}$ is used instead of the default value of \solverNameSparse{}. There is no $\tau_a$ because the solver at that time has not implemented Step 3 of \refAlgOne. The SPQR ordering used is the SuiteSparse default instead of \solverName{} default (AMD). The other parameters, $it_{max}, rcond_{thres}, perturb$ are the same as the default. 

\calibrationSixFigures{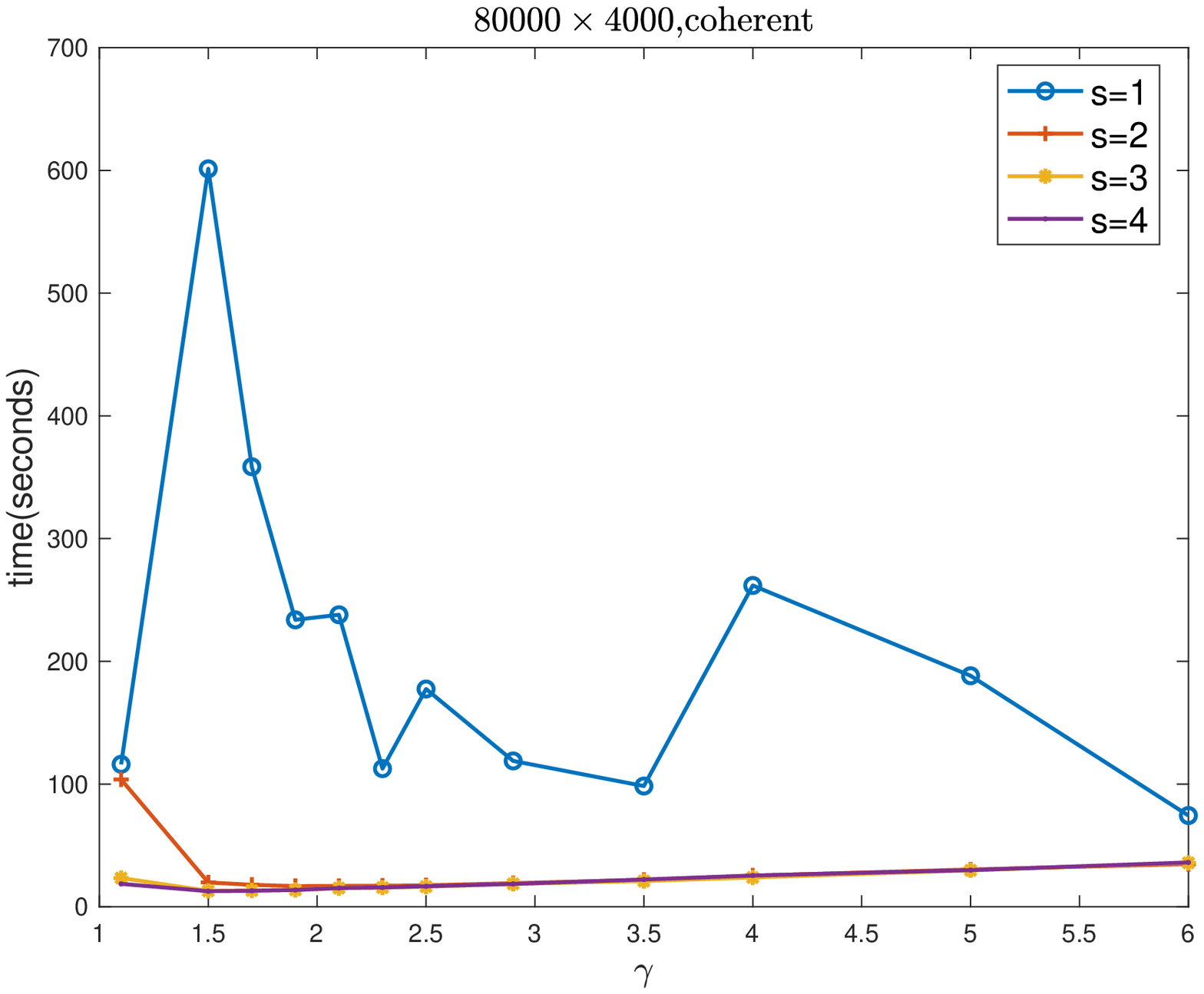}
{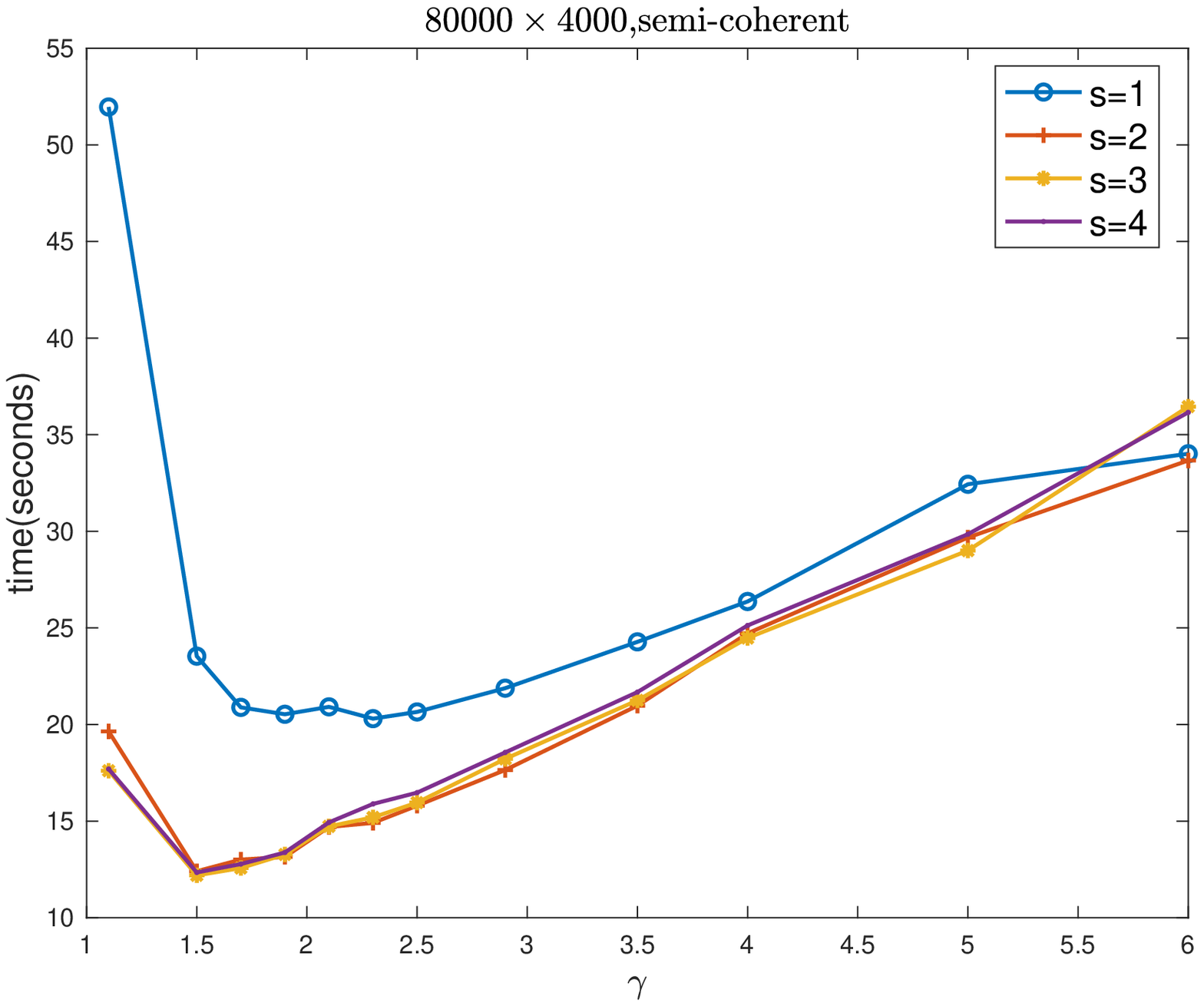}
{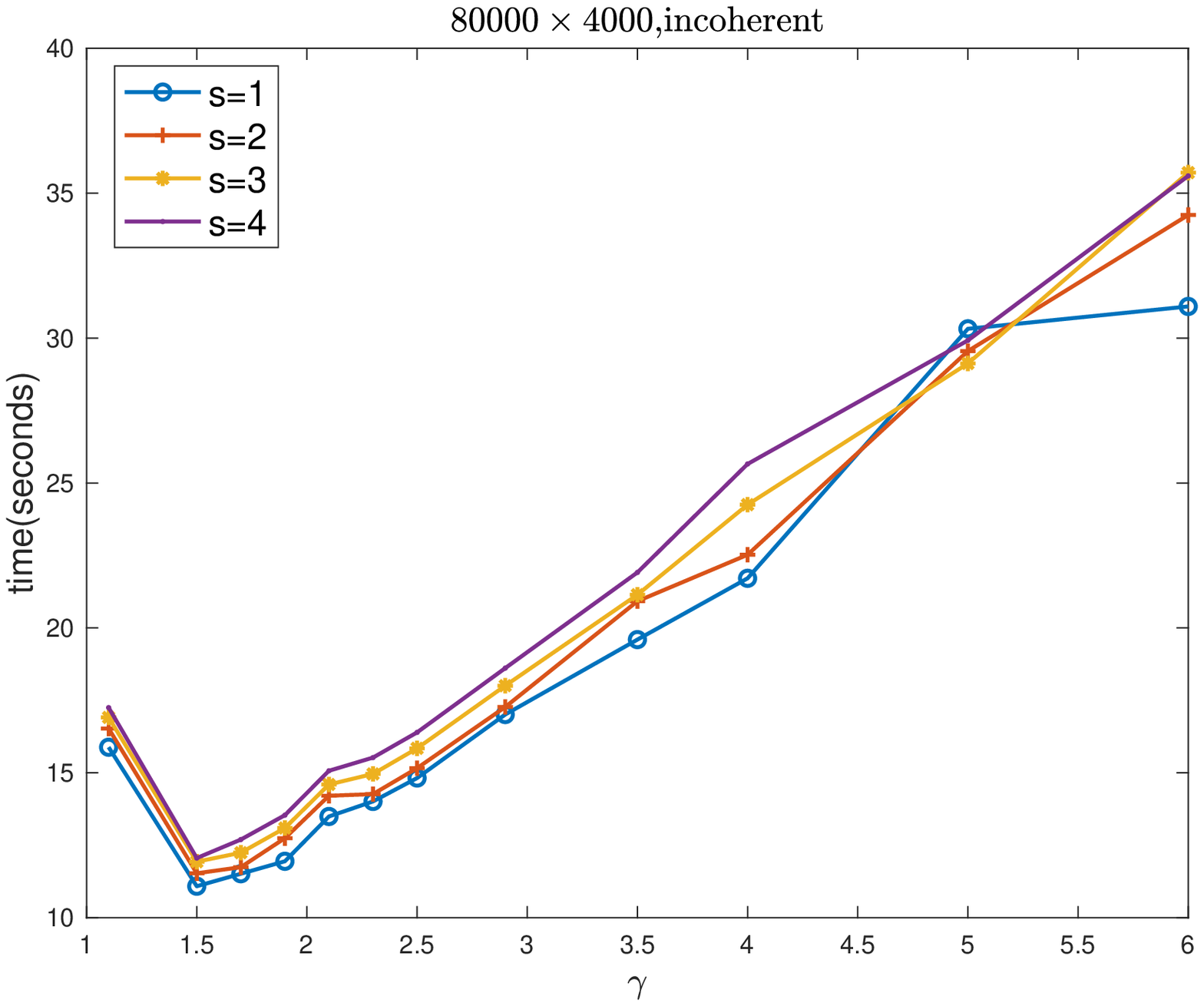}
{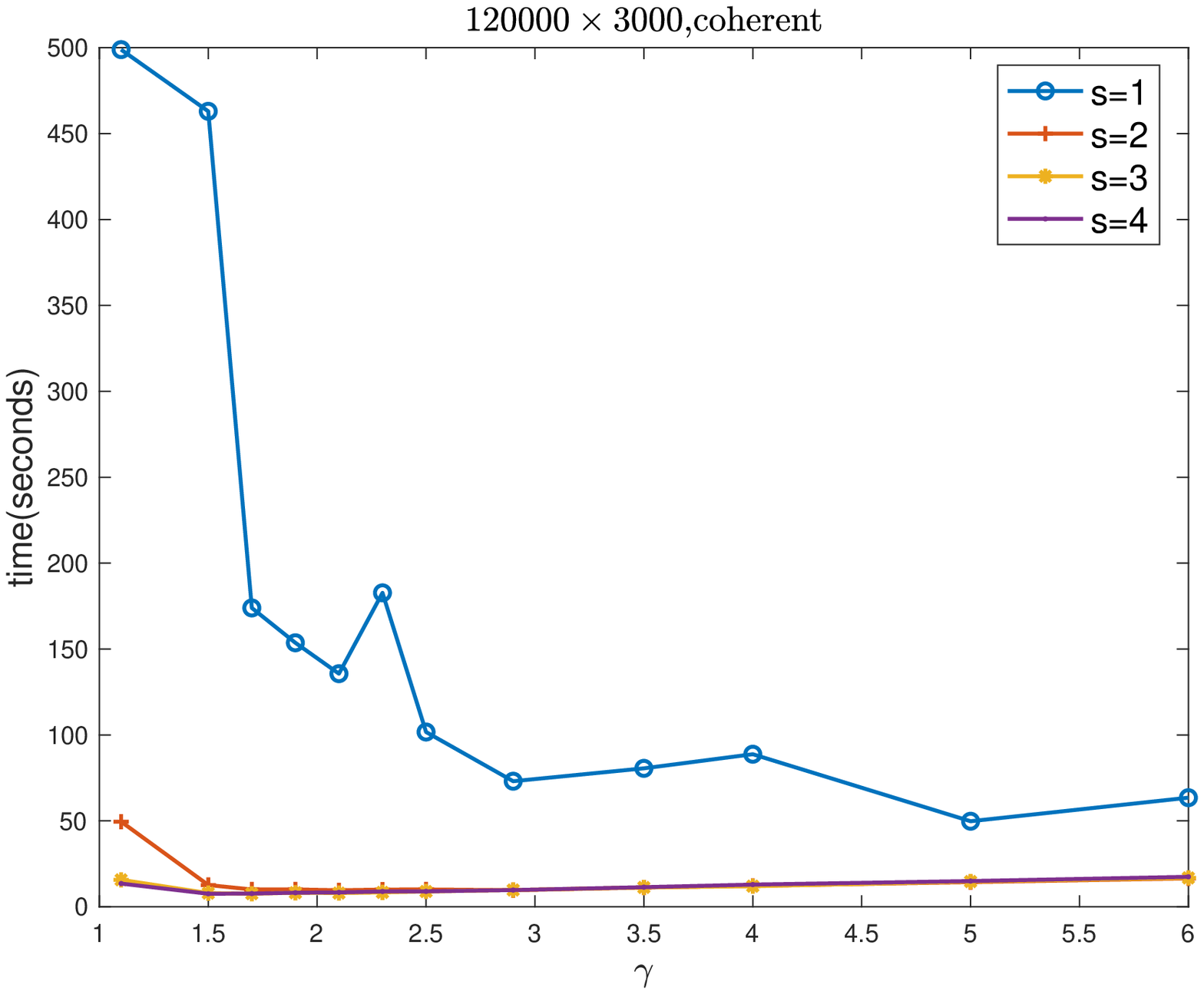}
{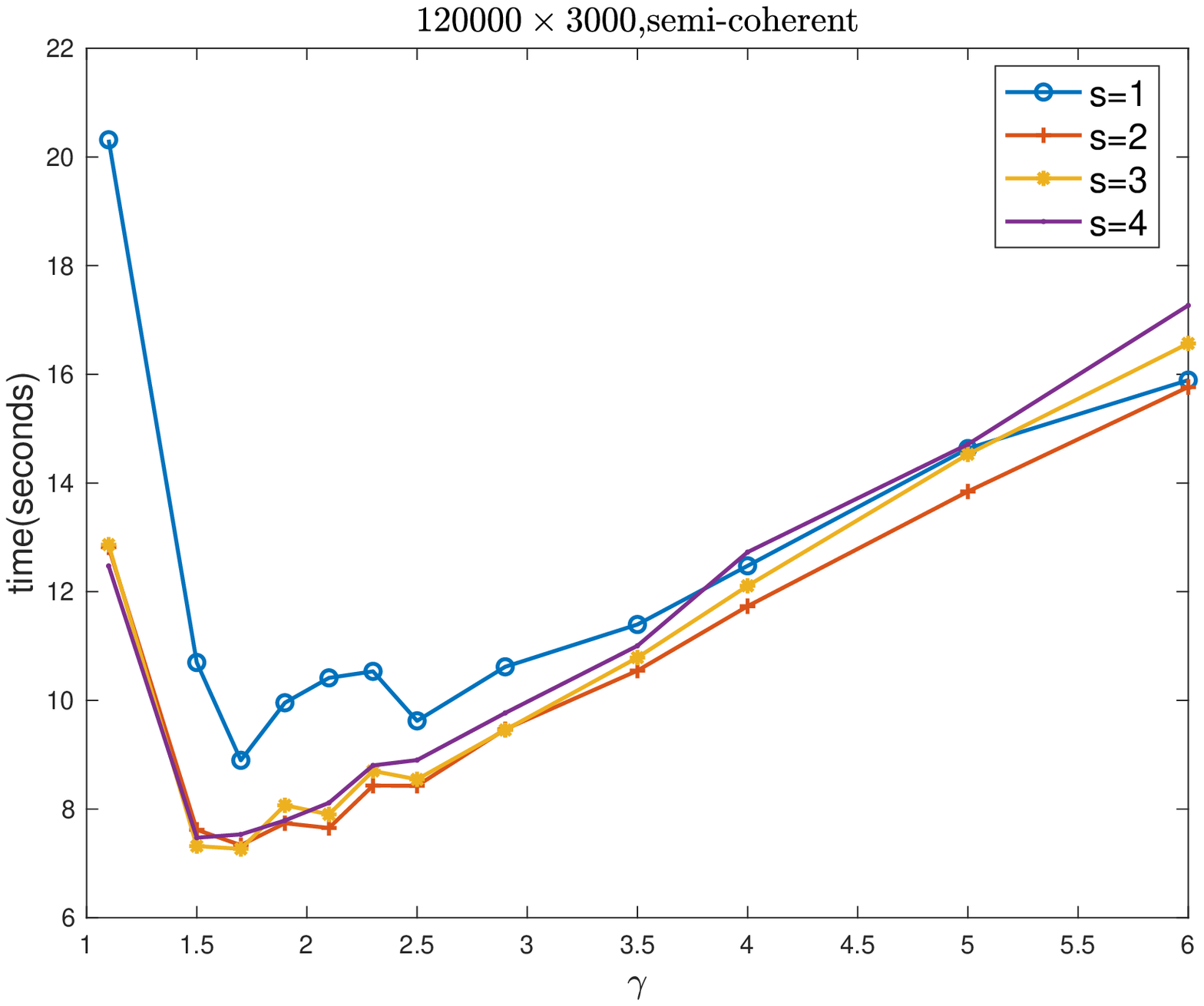}
{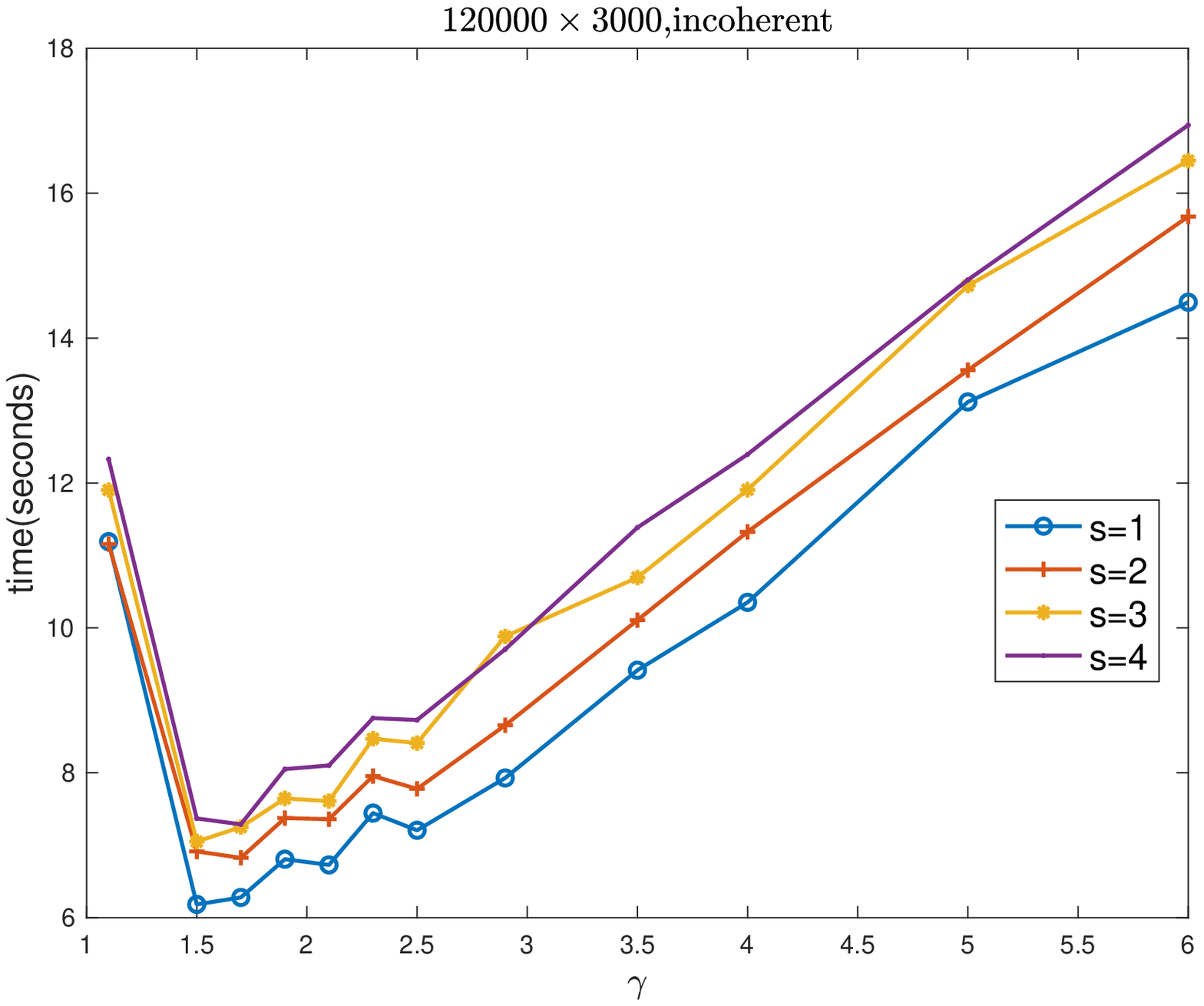}
{Running time of \solverNameSparse{} on sparse matrices $A\in \R^{m \times d}$ from Test Set 2 with $n = 80000, d=4000$ and $n = 120000, d=3000$ using different values of $s$ and $\gamma = m/d$. We choose $m = 1.4d, s=2$ in consideration of the above plot and the residual accuracy in \autoref{fig::Ls_qr_engineering_residual} but also taking into account some experiments of \solverNameSparse{} we have done on the Florida matrix collection.} 
{fig::Ls_qr_engineering_time}

\calibrationSixFigures{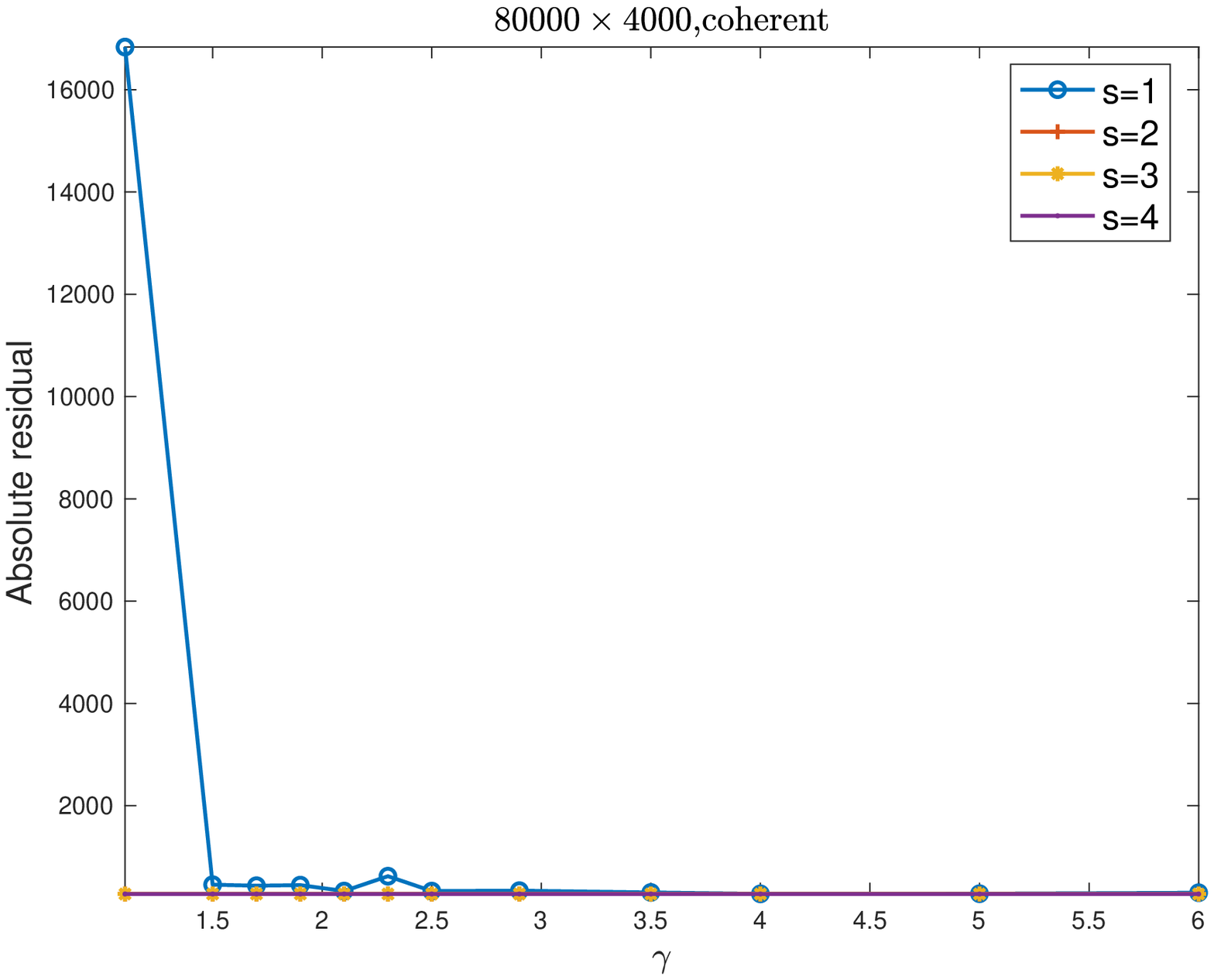}
{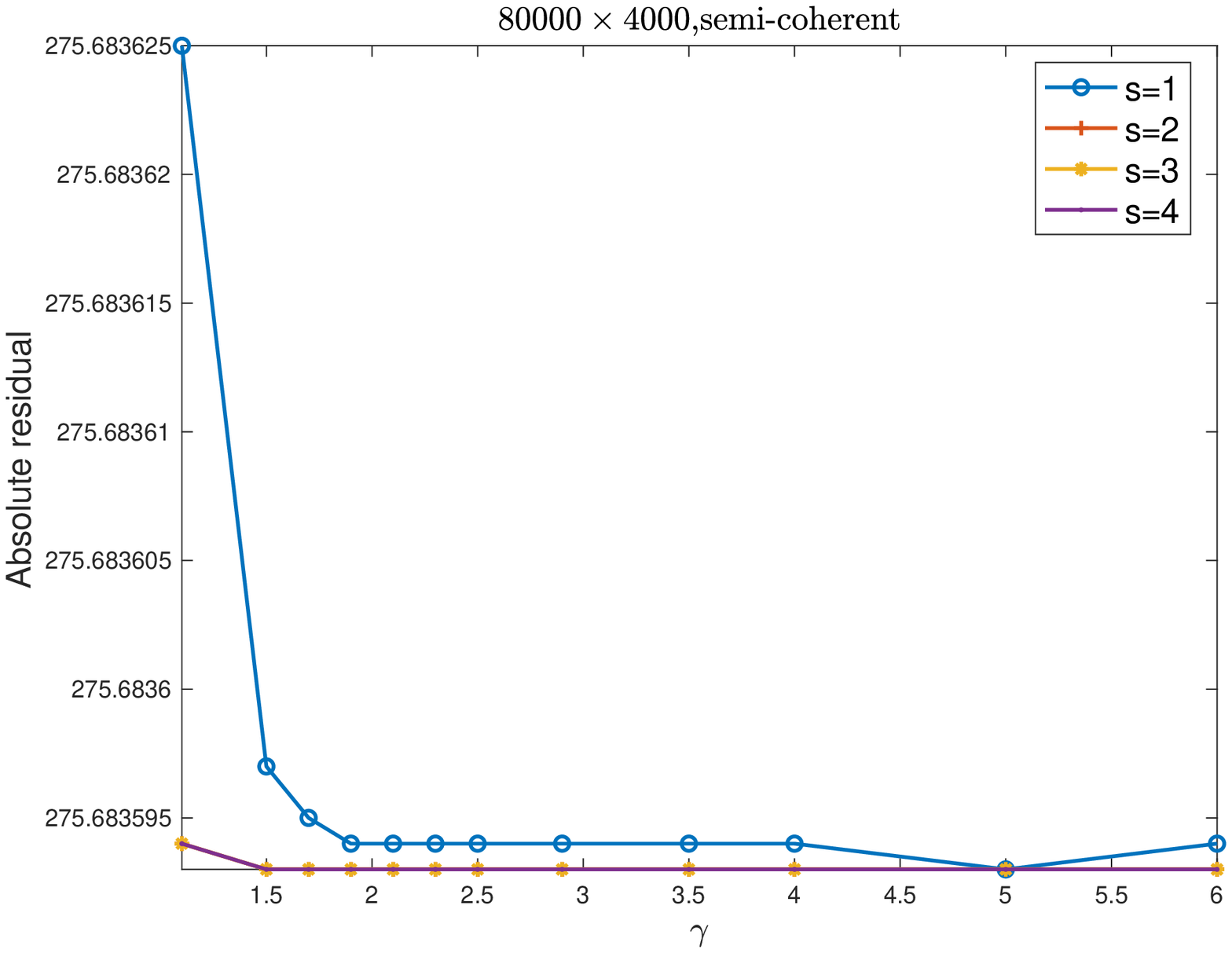}
{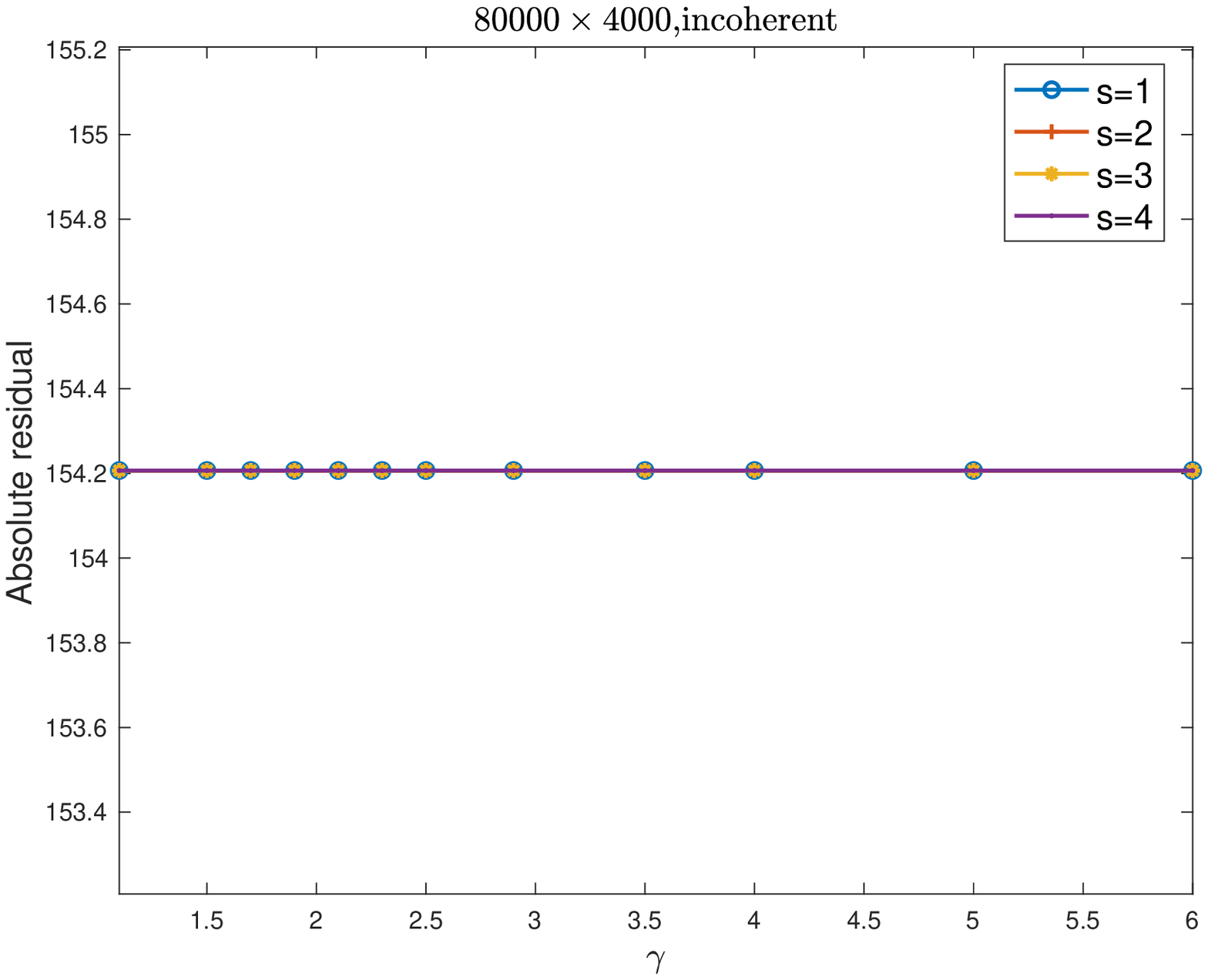}
{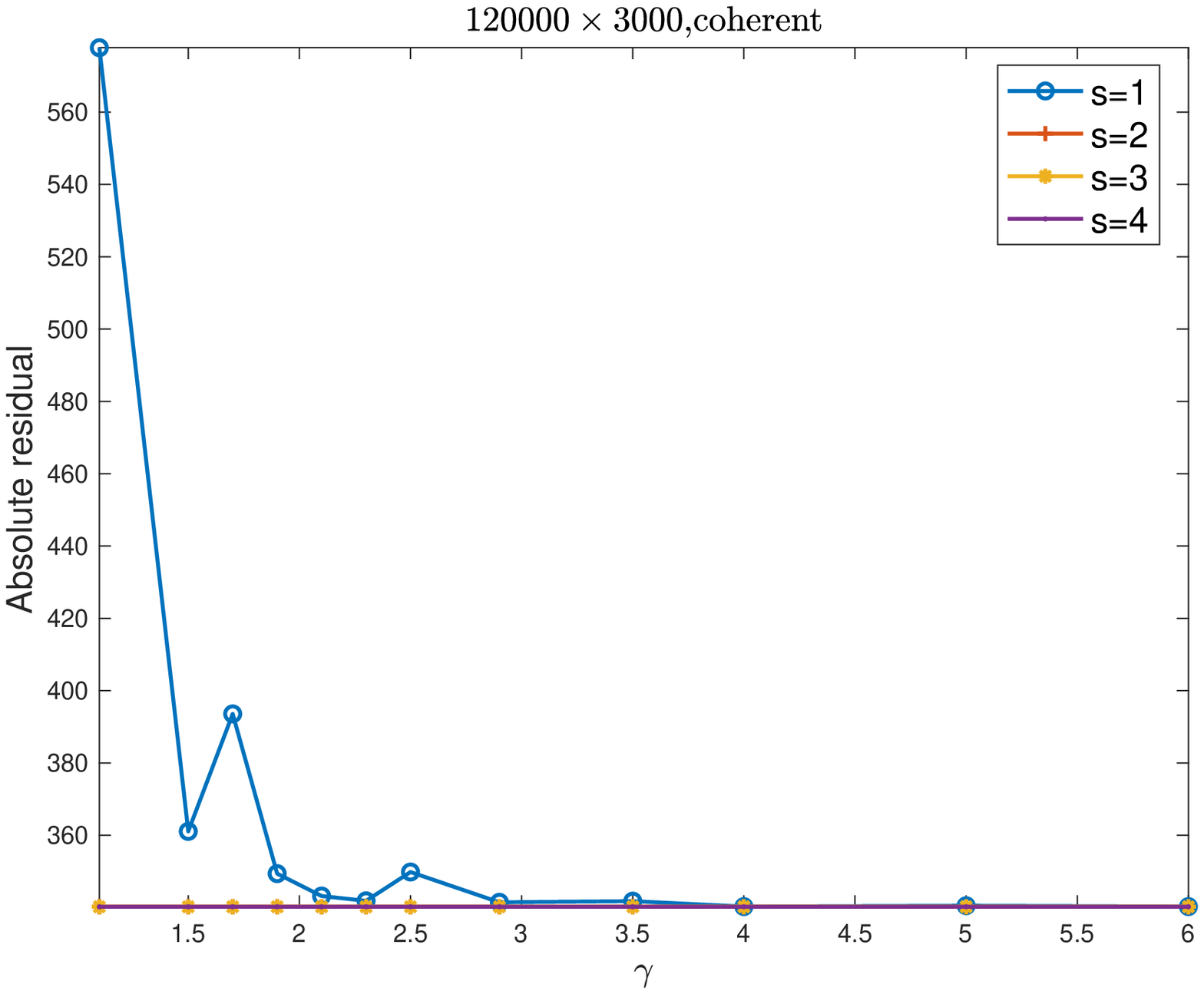}
{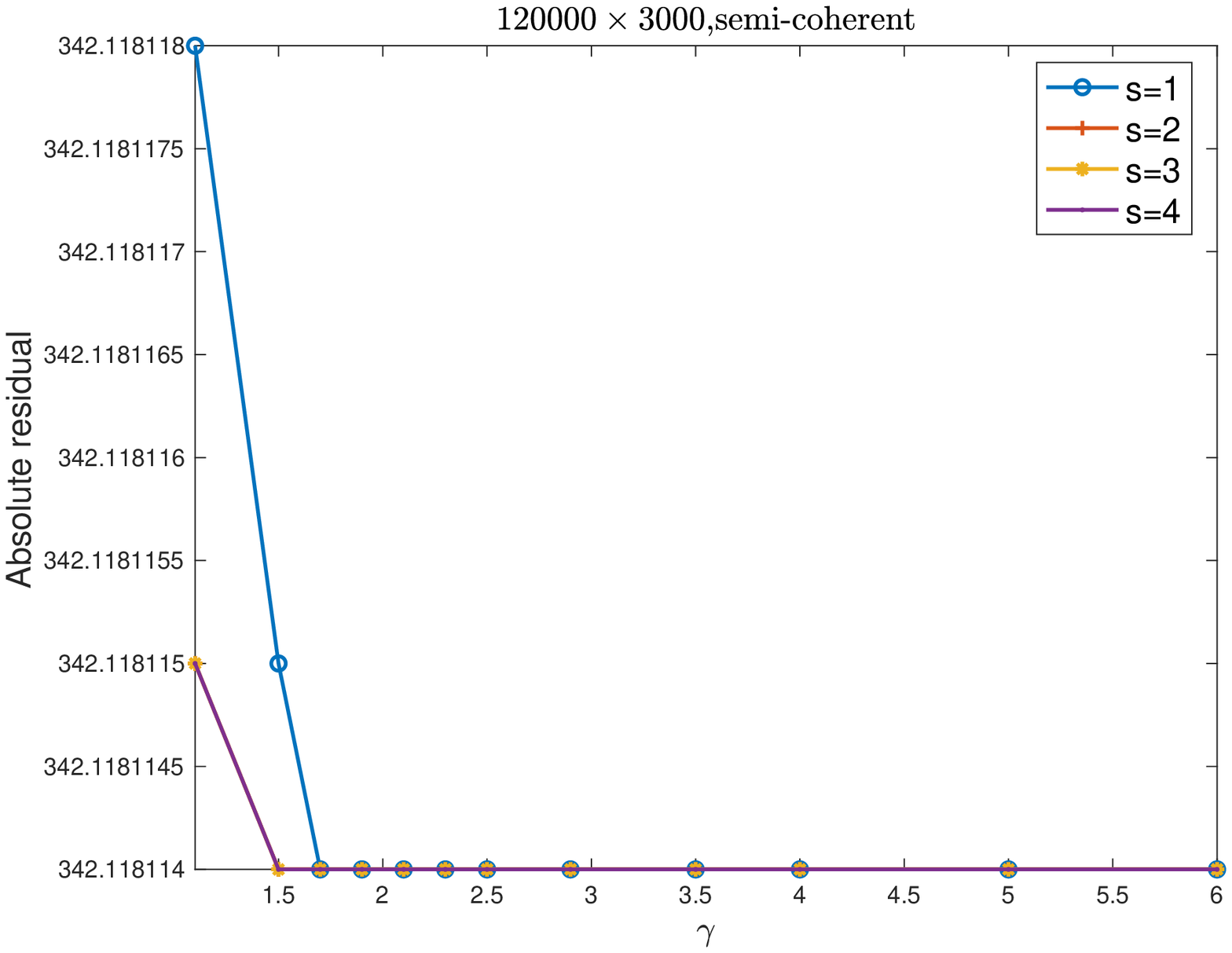}
{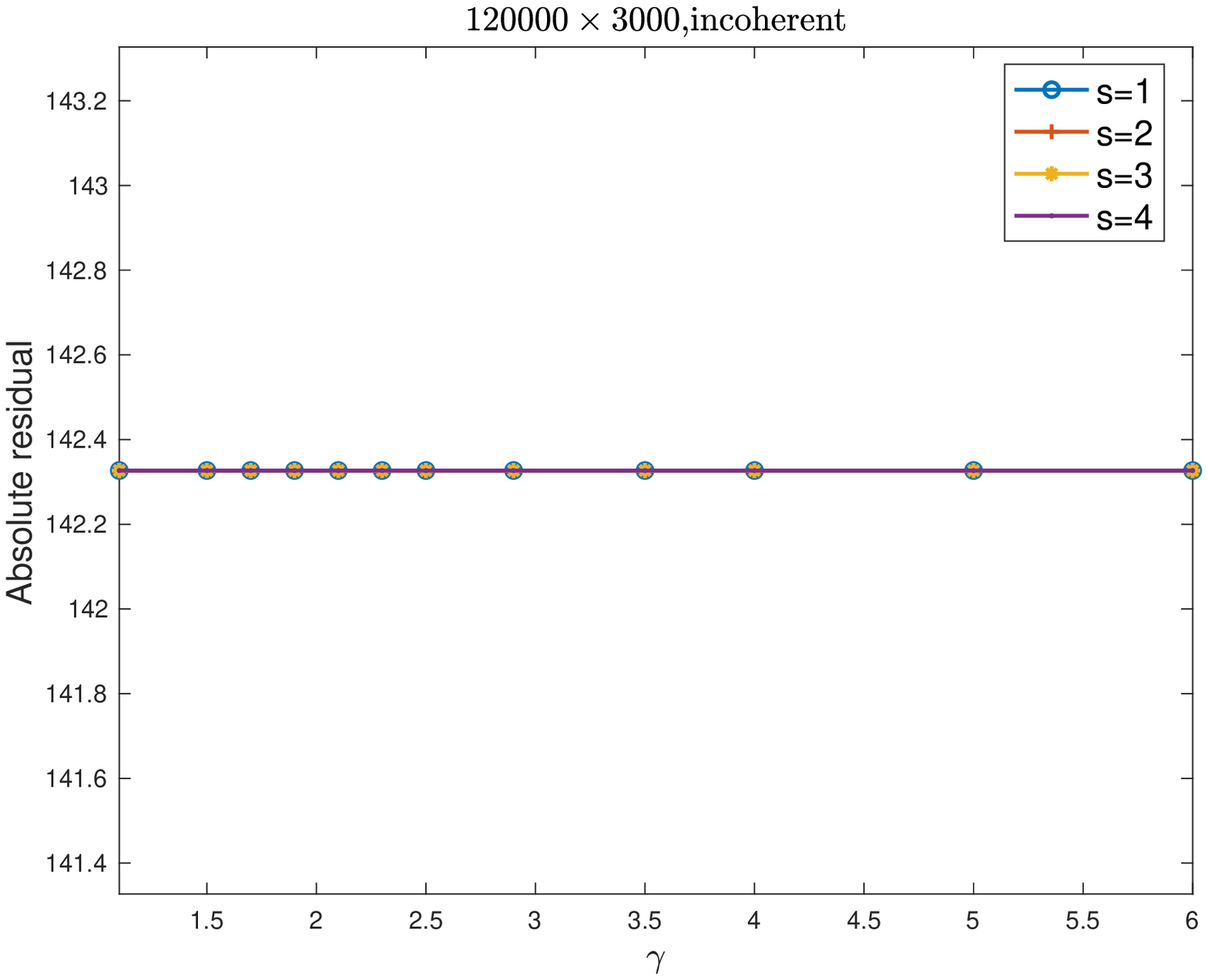}
{Corresponding residual on sparse matrices $A\in \R^{m \times d}$ from Test Set 2 with $n = 80000, d=4000$ and $n = 120000, d=3000$ using different values of $s$ and $\gamma = m/d$. Note that using $1$-hashing ($s=1$) results in inaccurate solutions. } 
{fig::Ls_qr_engineering_residual}

\calibrationSixFigures{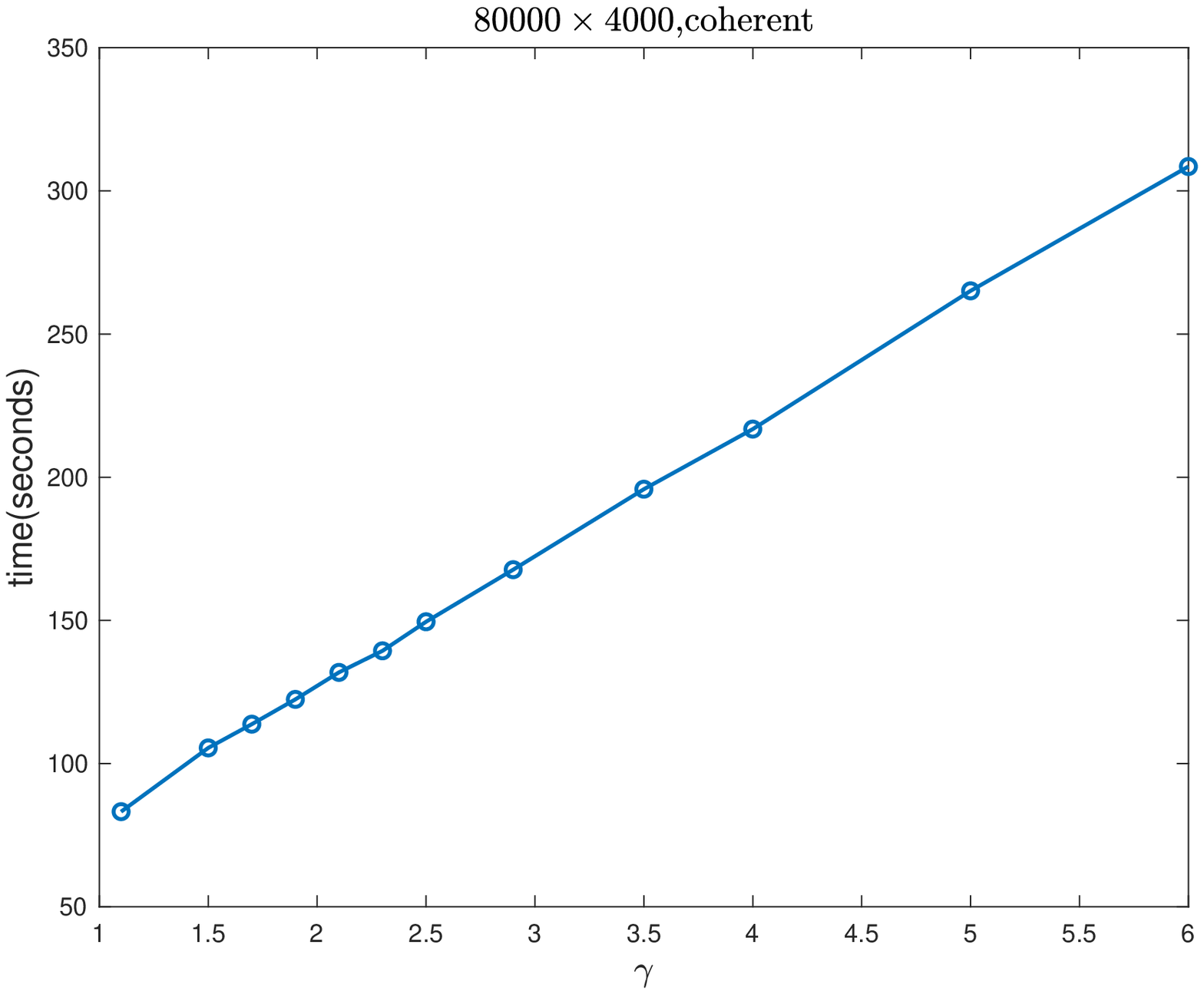}
{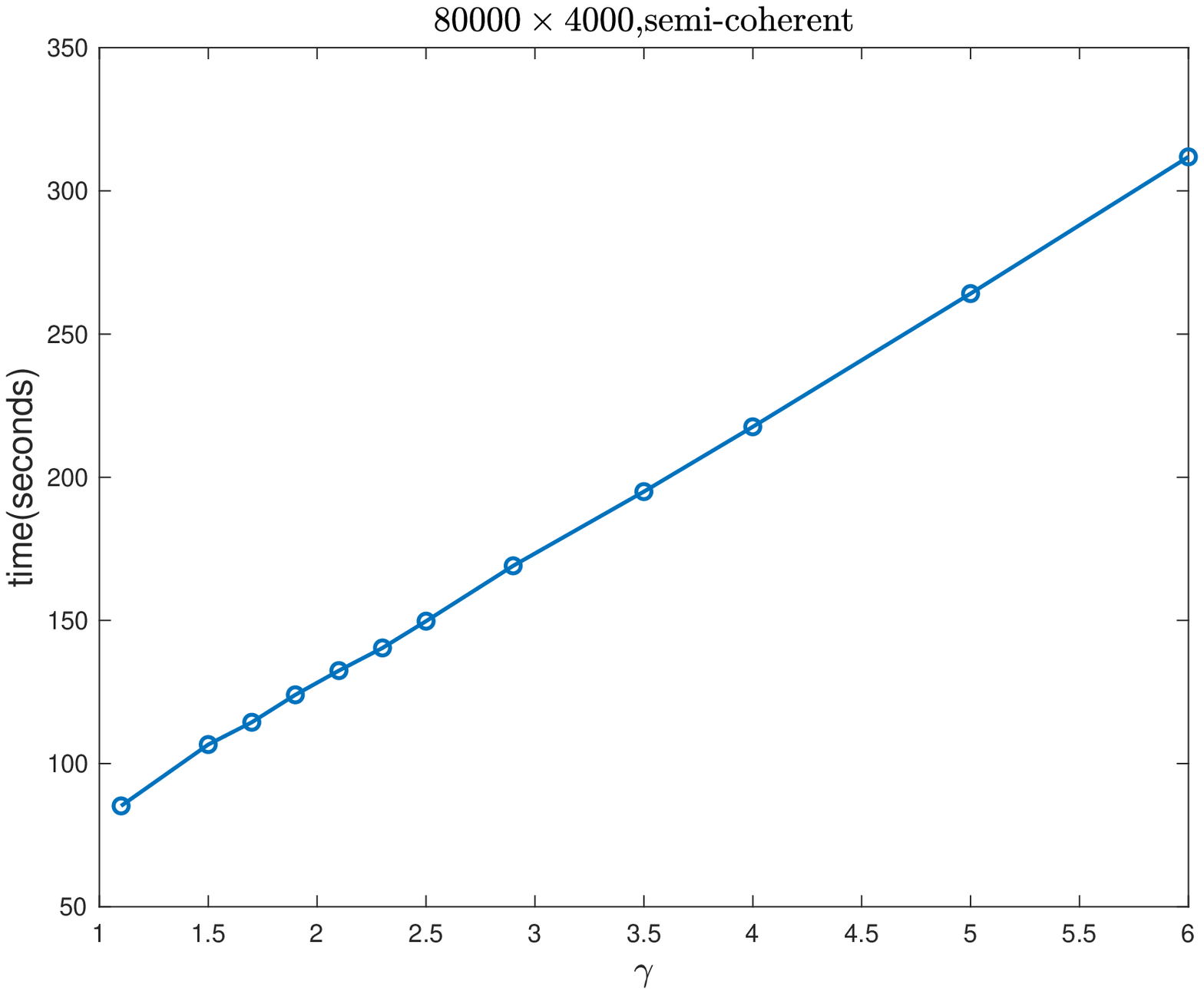}
{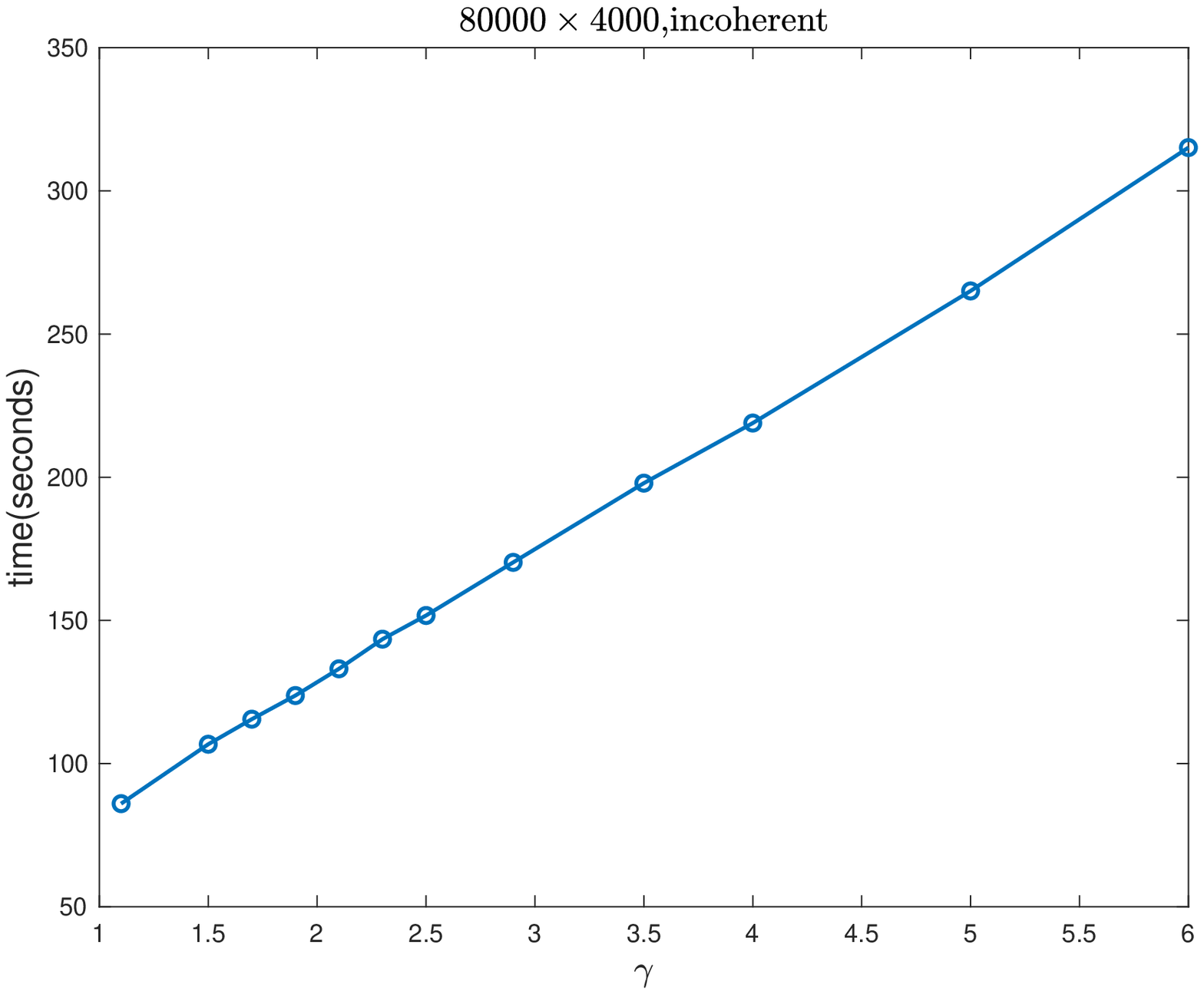}
{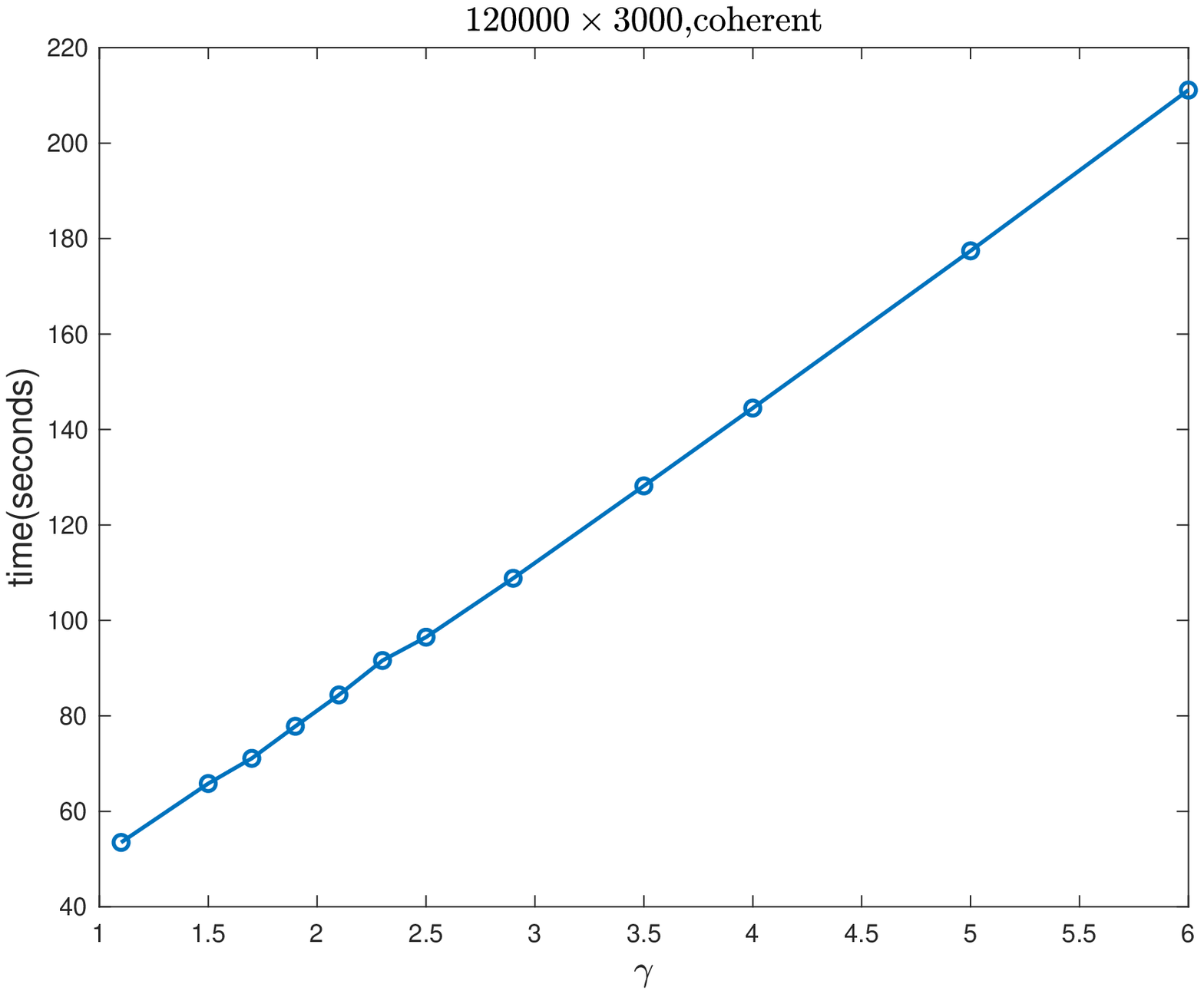}
{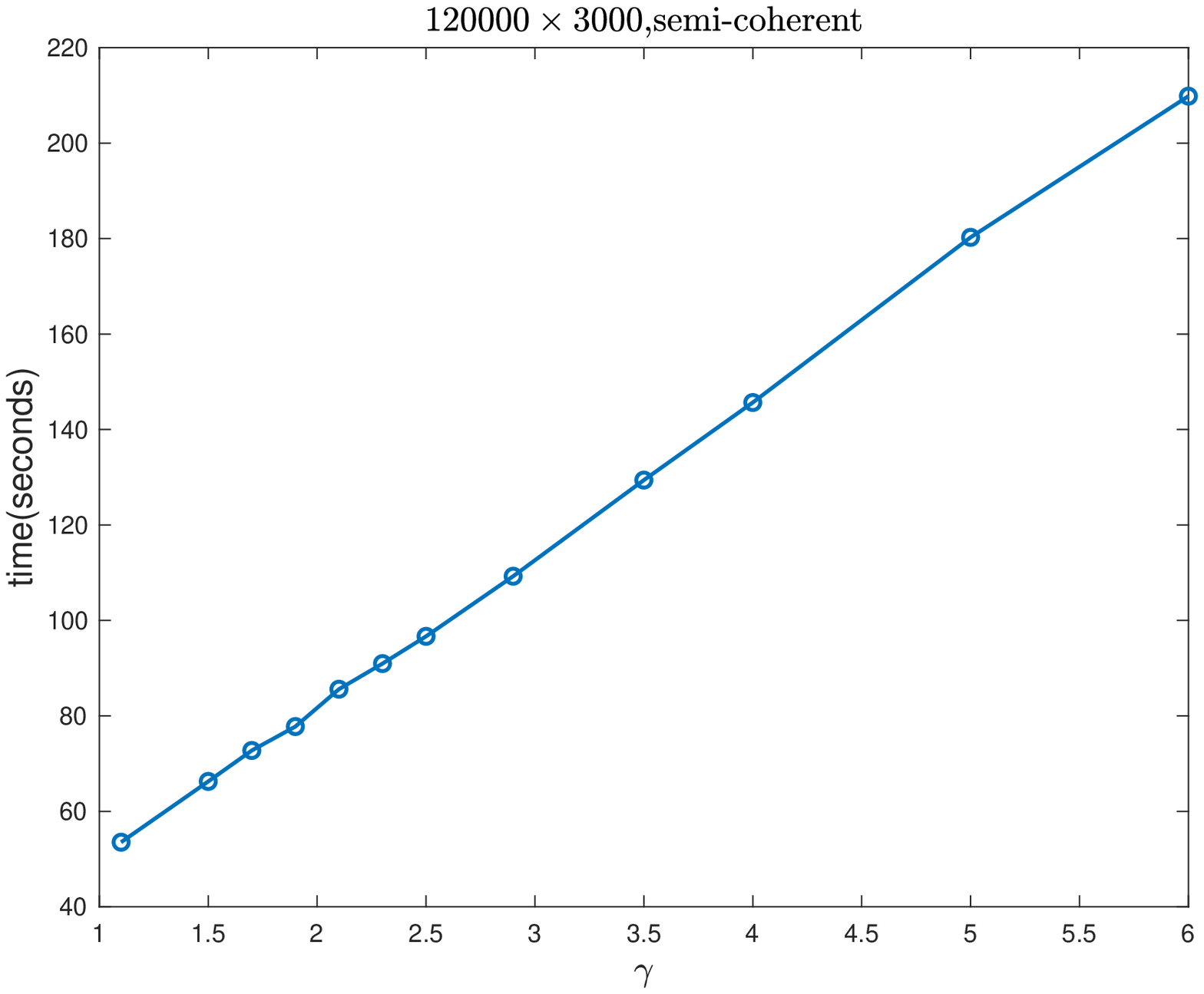}
{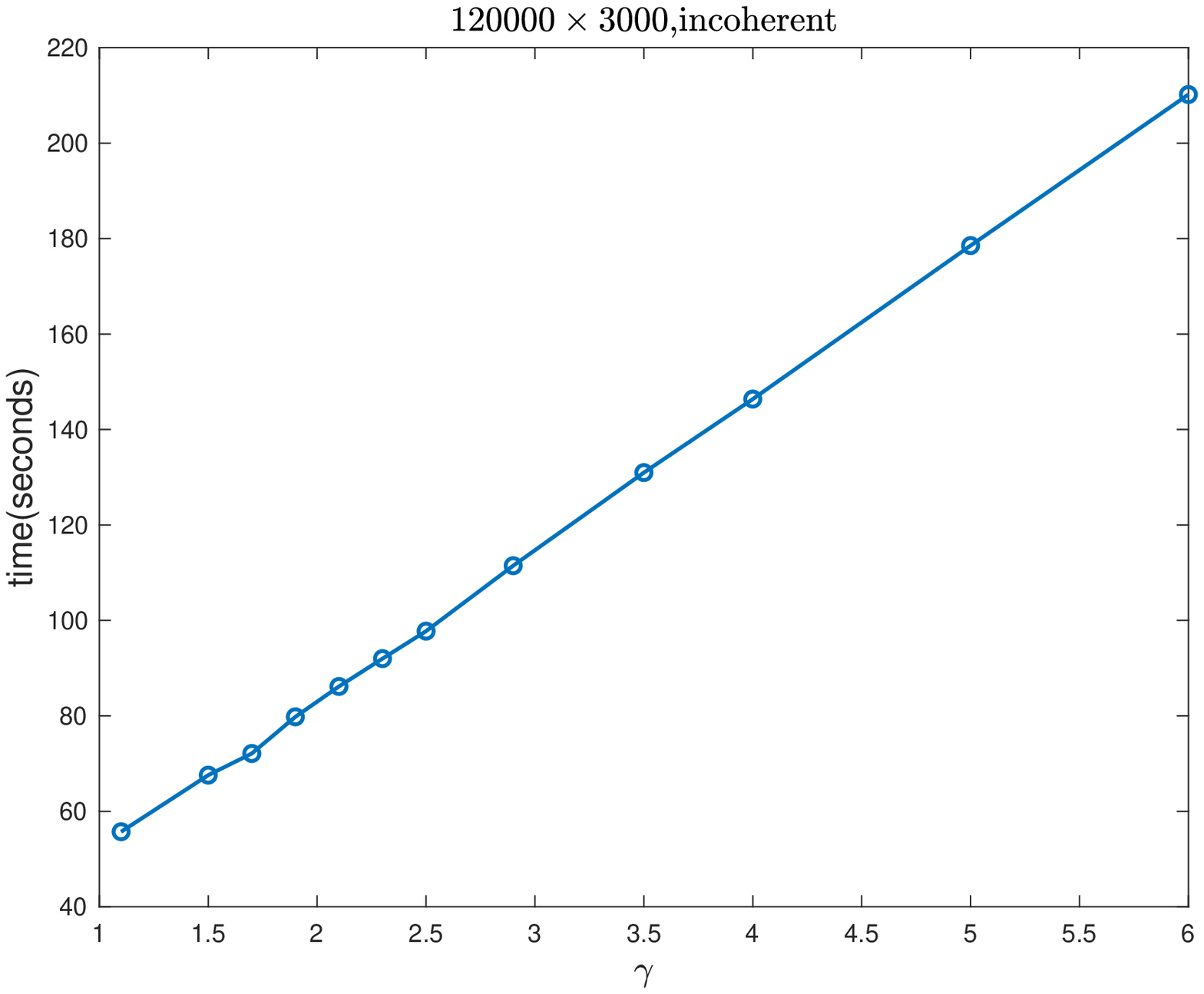}
{Runtime for LSRN  on sparse matrices $A\in \R^{m \times d}$ from Test Set 2 with $n = 80000, d=4000$ and $n = 120000, d=3000$ using different values of $\gamma = m/d$. We choose $m = 1.1d$ in consideration of the above plot and the residual accuracy in \autoref{fig::Ls_lsrn_engineering_residual} but also taking into account some experiments of \solverNameSparse{} we have done on the Florida matrix collection. }
{fig::Ls_lsrn_engineering_time}

\calibrationSixFigures{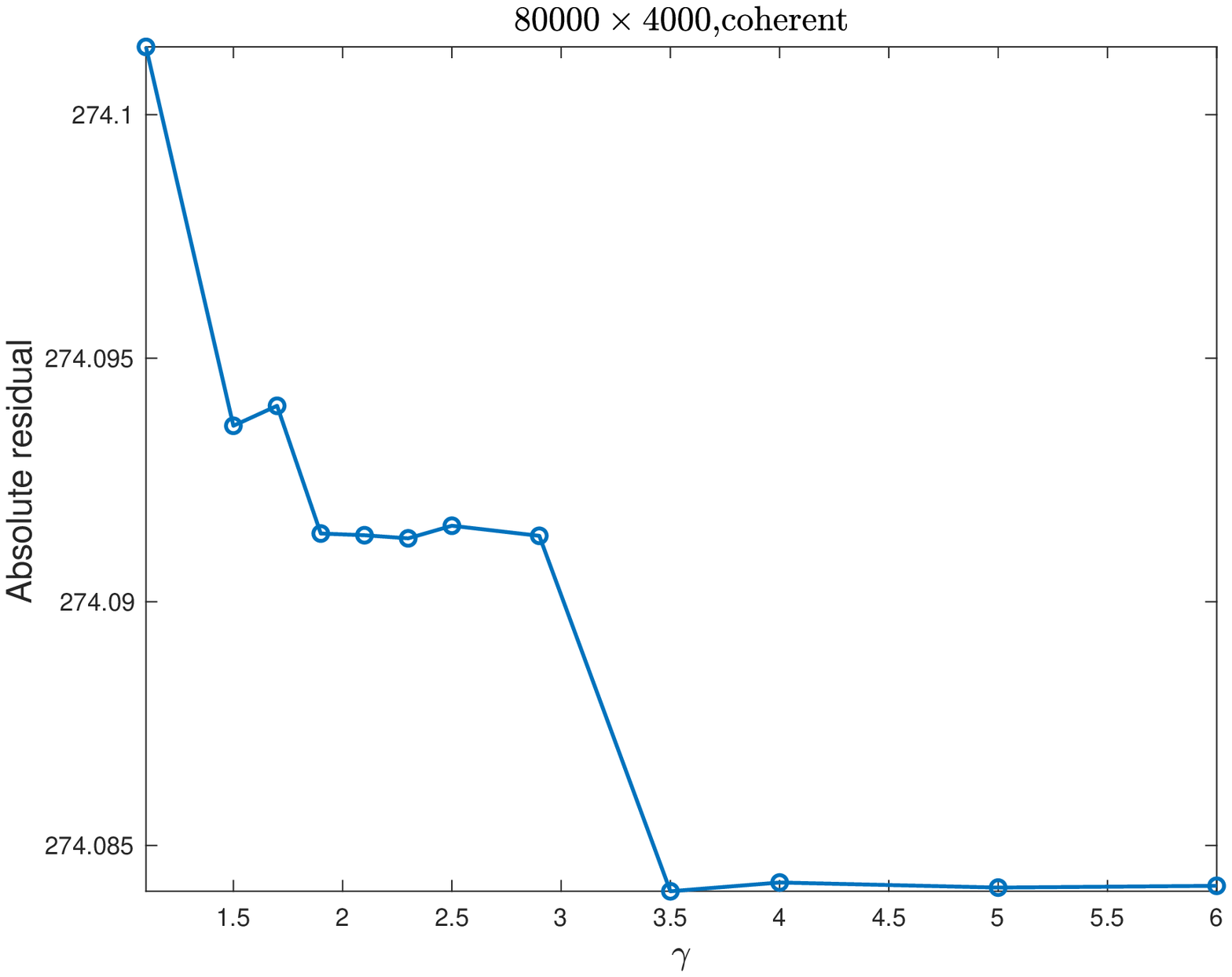}
{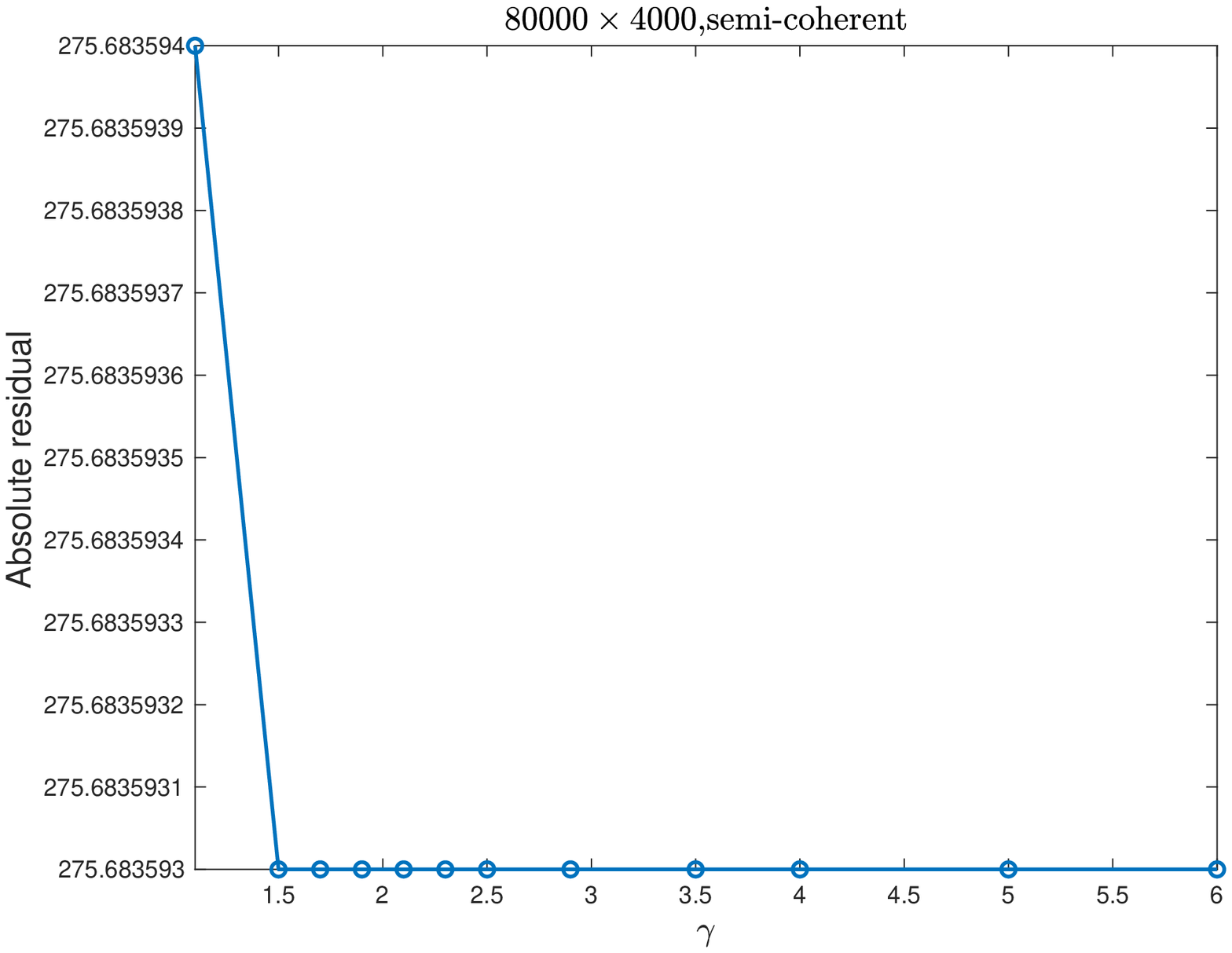}
{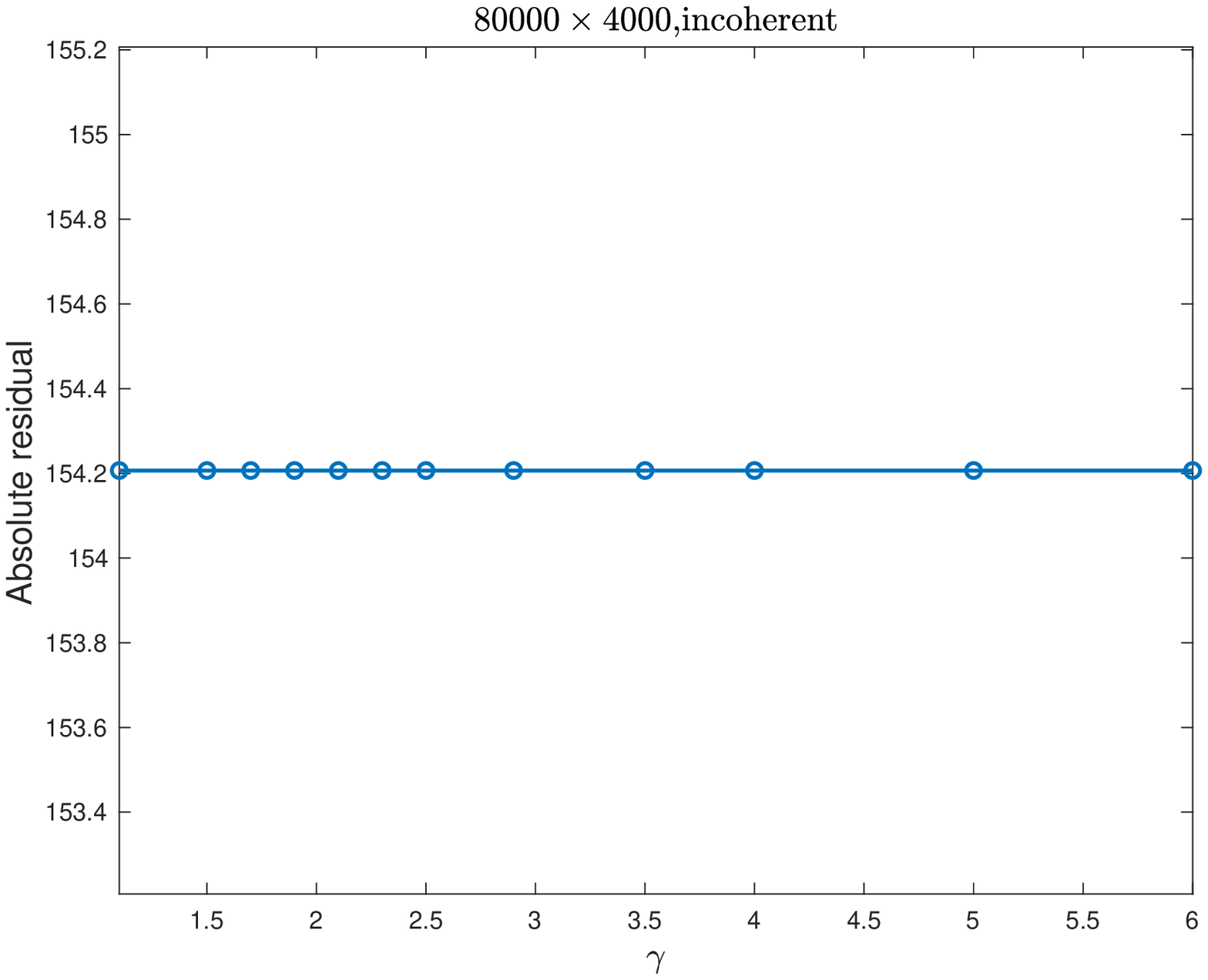}
{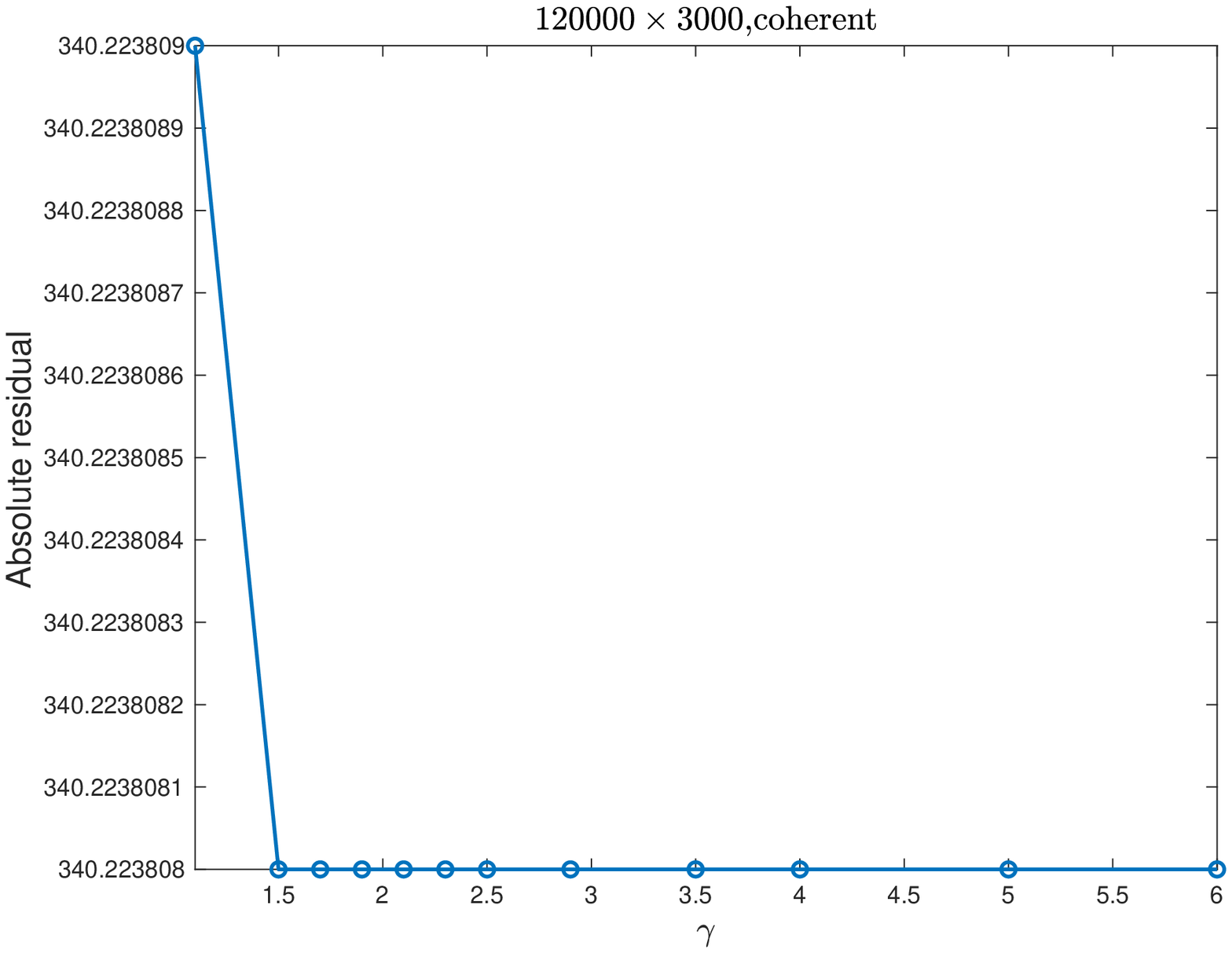}
{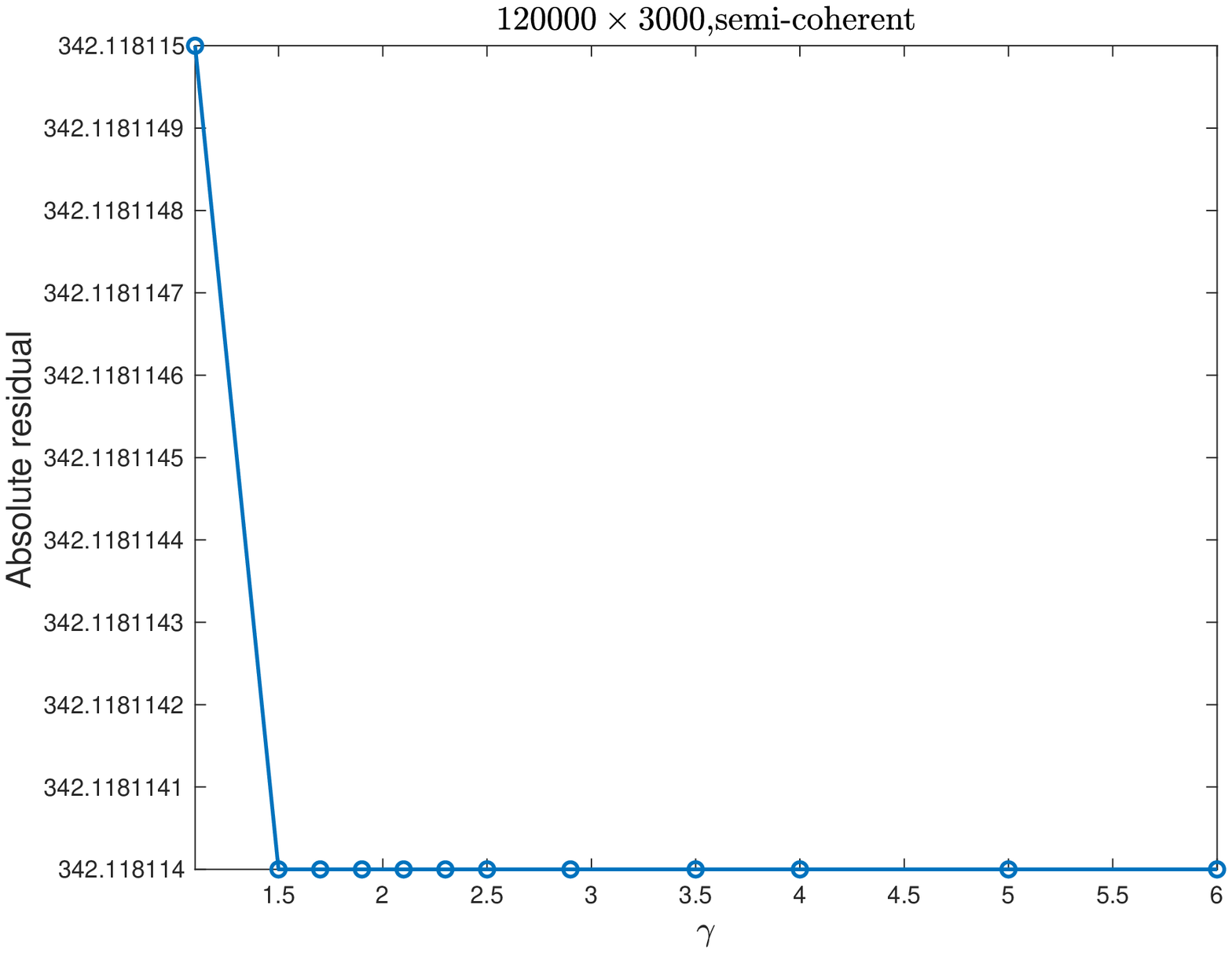}
{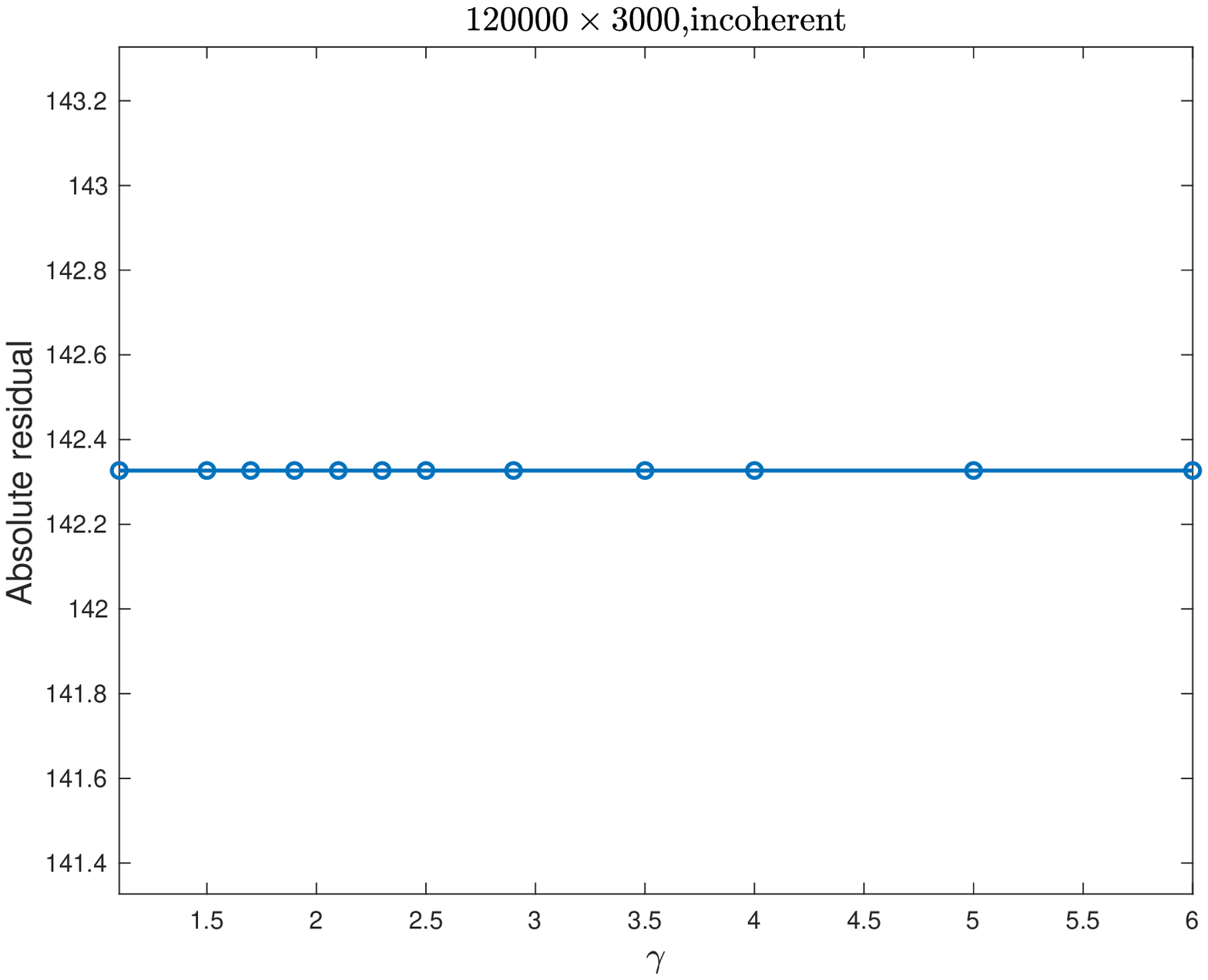}
{Residual values obtained by LSRN on
the same sparse problems as in Figure \ref{fig::Ls_lsrn_engineering_time}.} 
{fig::Ls_lsrn_engineering_residual}

\subsection{Residual accuracy of \solverName{}}
Since our theory in Section \ref{sec:algo_analysis} is only an approximation of the practical implementation as discussed in Section \ref{subsection:alg_implementation_discussion}, we numerically test \solverName{}'s accuracy of solving \eqref{LLS-statement}. We choose 14 matrices $A$ in the Florida matrix collection with different dimensions and rank-deficiencies (Table \ref{tab::rank_def_accuracy_dim}). We use LAPACK's SVD-based linear least squares solver (SVD), LSRN, Blendenpik, \solverNameDense{} and \solverNameSparse{} on these problems with the residual shown in Table \ref{tab::rank_def_accuracy}.\footnote{These problems are given in a sparse format. We convert them to a dense format for dense solvers. Thus the dense solvers cannot assume any entry is a priori zero. }

We see both of \solverNameDense{} and \solverNameSparse{} have excellent residual accuracy comparing to SVD-based LAPACK solver. The result also shows that Blendenpik fails to accurately solve rank-deficient \eqref{LLS-statement}.  

In our large scale numerical study with the Florida matrix collection, the residuals are also compared and the solution of \solverNameSparse{} is no-less accurate than the state-of-the-art sparse solvers LS\_SPQR and LS\_HSL(see later sections).

\begin{table}
\scriptsize
\centering
\begin{tabular}{l|rrrrrrr}
               & lp\_ship12l                          & Franz1               & GL7d26                                & cis-n4c6-b2          & lp\_modszk1                           & rel5                                  & ch5-5-b1              \\ 
\hline
SVD            & 18.336                               & 26.503               & 50.875                                & 6.1E-14              & 33.236                                & 14.020                                & 7.3194                \\
LSRN           & 18.336                               & 26.503               & 50.875                                & 3.2E-14              & 33.236                                & 14.020                                & 7.3194                \\
Blendenpk      & \multicolumn{1}{l}{~~~~~~~~~NaN~~~~} & 9730.700             & \multicolumn{1}{l}{~~~~~~~~~~NaN~~~~} & 3.0E+02              & \multicolumn{1}{l}{~~~~~~~~~~NaN~~~~} & \multicolumn{1}{l}{~~~~~~~~~~NaN~~~~} & 340.9200              \\
Ski-LLS-dense  & 18.336                               & 26.503               & 50.875                                & 5.3E-14              & 33.236                                & 14.020                                & 7.3194                \\
Ski-LLS-sparse & 18.336                               & 26.503               & 50.875                                & 6.8E-14              & 33.236                                & 14.020                                & 7.3194                \\
               & \multicolumn{1}{l}{}                 & \multicolumn{1}{l}{} & \multicolumn{1}{l}{}                  & \multicolumn{1}{l}{} & \multicolumn{1}{l}{}                  & \multicolumn{1}{l}{}                  & \multicolumn{1}{l}{}  \\
               & n3c5-b2                              & ch4-4-b1             & n3c5-b1                               & n3c4-b1              & connectus                             & landmark                              & cis-n4c6-b3           \\ 
\hline
SVD            & 9.0E-15                              & 4.2328               & 3.4641                                & 1.8257               & 282.67                                & 1.1E-05                               & 30.996                \\
LSRN           & 6.7E-15                              & 4.2328               & 3.4641                                & 1.8257               & 282.67                                & 1.1E-05                               & 30.996                \\
Blendenpk      & 1.3E+02                              & 66.9330              & 409.8000                              & 8.9443               & \multicolumn{1}{l}{~~~~~~~~~~NaN~~~~} & \multicolumn{1}{l}{~~~~~~~~~~NaN~~~~} & 3756.200              \\
Ski-LLS-dense  & 5.2E-15                              & 4.2328               & 3.4641                                & 1.8257               & 282.67                                & 1.1E-05                               & 30.996                \\
Ski-LLS-sparse & 6.9E-15                              & 4.2328               & 3.4641                                & 1.8257               & 282.67                                & 1.1E-05                               & 30.996               
\end{tabular}
\caption{Residuals of solvers for a range of rank-deficient problems taken from the Florida matrix collection \cite{10.1145/2049662.2049663}. The matrices are all sparse but we convert them into dense format before applying a dense solver such as Blendenpik. We see both \solverNameDense{} and \solverNameSparse{} achieve excellent residual accuracy for rank-deficient problems, as well as LSRN. Blendenpik is not designed for rank-deficient problems and either returns a large residual or encounters numerical issues, returning NaN. }
\label{tab::rank_def_accuracy}
\end{table}

\begin{table}
\centering
\scriptsize
\begin{tabular}{l|rrr}
\multicolumn{1}{r|}{} & nrow~  & ncol & rank            \\ 
\hline
lp\_ship12l           & 5533   & 1151 & 1042            \\
Franz1                & 2240   & 768  & 755             \\
GL7d26                & 2798   & 305  & 273             \\
cis-n4c6-b2           & 1330   & 210  & 190             \\
lp\_modszk1           & 1620   & 687  & 686             \\
rel5                  & 240    & 35   & 24              \\
ch5-5-b1              & 200    & 25   & 24              \\
n3c5-b2               & 120    & 45   & 36              \\
ch4-4-b1              & 72     & 16   & 15              \\
n3c5-b1               & 45     & 10   & 9               \\
n3c4-b1               & 15     & 6    & 5               \\
connectus             & 394792 & 512  & \textless{}458  \\
landmark              & 71952  & 2704 & 2671            \\
cis-n4c6-b3           & 5940   & 1330 & 1140           
\end{tabular}
\caption{Dimensions for the problems tested in \autoref{tab::rank_def_accuracy}}
\label{tab::rank_def_accuracy_dim}
\end{table}
    
    \section{Numerical performance}

\subsection{Solvers compared and their parameters}

Recall \solverName{} treats dense and sparse $A$ differently in \eqref{LLS-statement}. For dense $A$, we compare to the state-of-the-art sketching solver Blendenpik, that has been shown to be four times faster than LAPACK on dense, large scale and moderately over-determined full rank problems \cite{doi:10.1137/090767911} \footnote{Available at https://github.com/haimav/Blendenpik. For the sake of fair comparison, we wrote a C interface and uses the same LSQR routine as \solverName{}.}. The parameters for Blendenpik are $m = 2.2d$ \footnote{Chosen by calibration, see Appendix B.}, $\tau_r = 10^{-6}$, $it_{max} = 10^4$ and $wisdom=1$. The same wisdom data file as \solverName{} is used. 

For sparse $A$, we compare to the following solvers
\begin{enumerate}
	\item HSL\_MI35 (LS\_HSL), that uses an incomplete Cholesky factorization of $A$ to compute a preconditioner of the problem \eqref{LLS-statement}, before using LSQR. \footnote{See http://www.hsl.rl.ac.uk/specs/hsl_mi35.pdf for a full specification. For the sake of fair comparison, we wrote a C interface and uses the same LSQR routine as \solverName{}. We also disable the pre-processing of the data as it was not done for the other solvers. We then found using no scaling and no ordering was more effective than the default scaling and ordering. Hence we chose no scaling and no ordering in the comparisons. As a result, the performance of HSL may improve, however \cite{10.1145/3014057} experimented with the use of different scaling and ordering, providing some evidence that the improvement will not be significant.} The solver has been shown to be competitive for sparse problems $\eqref{LLS-statement}$ \cite{10.1145/3014057,Gould:2016vg}. We use $\tau_r = 10^{-6}$ and $it_{max} = 10^4$. 

	\item SPQR\_SOLVE (LS\_SPQR), that uses SPQR from Suitesparse to compute a sparse QR factorization of $A$, which is exploited to solve \eqref{LLS-statement} directly. \footnote{Available at https://people.engr.tamu.edu/davis/suitesparse.html.} The solver has been shown to be competitive for sparse problems \cite{10.1145/2049662.2049670}. 

	\item LSRN, that uses the framework of \refAlgOne, with $S$ having i.i.d. $N\bracket{0, 1/\sqrt{m}}$ entries in Step 1; SVD factorization from Intel LAPACK of the matrix $SA$ in Step 2; the same LSQR routine as \solverName{} in Step 4. \footnote{Note that LSRN does not contain the Step 3.} LSRN has been shown to be an effective solver for possibly rank-deficient dense and sparse \eqref{LLS-statement} under parallel computing environment \cite{Meng:2014ib}. However, parallel computing is outside the scope of this study and we therefore run LSRN in a serial environment. Hence the performance of LSRN may improve under parallelism. The default parameters are chosen to be $m = 1.1d$, $\tau_r = 10^{-6}$, $it_{max}=10^4$. 
\end{enumerate}

\subsection{Running time performance on randomly generated full rank dense A}
Our first experiment compares of \solverName{} for dense \eqref{LLS-statement} with Blendenpik. For each matrix of different size shown in the x-axis of Figure \ref{fig::compare_blen_coherent}, \ref{fig::compare_blen_semi-coherent}, \ref{fig::compare_blen_incoherent},  we generate a coherent, semi-coherent and incoherent dense matrix as defined in \defDenseTestMatrices. Blendenpik, \solverName{} (dense version) and \solverName{} without R-CPQR are to solve \eqref{LLS-statement} with $b$ being a vector of all ones. The running time $t$ with the residual $\|Ax-b\|_2$ are recorded, where $x$ is the solution returned by the solvers. The residuals are all the same up to six significant figures, indicating all three solvers give an accurate solution of \eqref{LLS-statement}. 

We see that using hashing instead of sampling yields faster solvers by comparing \solverName{} without R-CPQR to Blendenpik, especially when the matrix $A$ is of the form \eqref{eq::A_co_dense}. We also see that \solverName{} with R-CPQR is as fast as Blendenpik on full rank dense problems while being able to solve rank-deficient problems (Table \ref{tab::rank_def_accuracy}). 

The default parameters for \solverName{} is used, and the parameters for Blendenpik is mentioned before. 
\renewcommand{\mysize}{0.45}
\threeFigures
{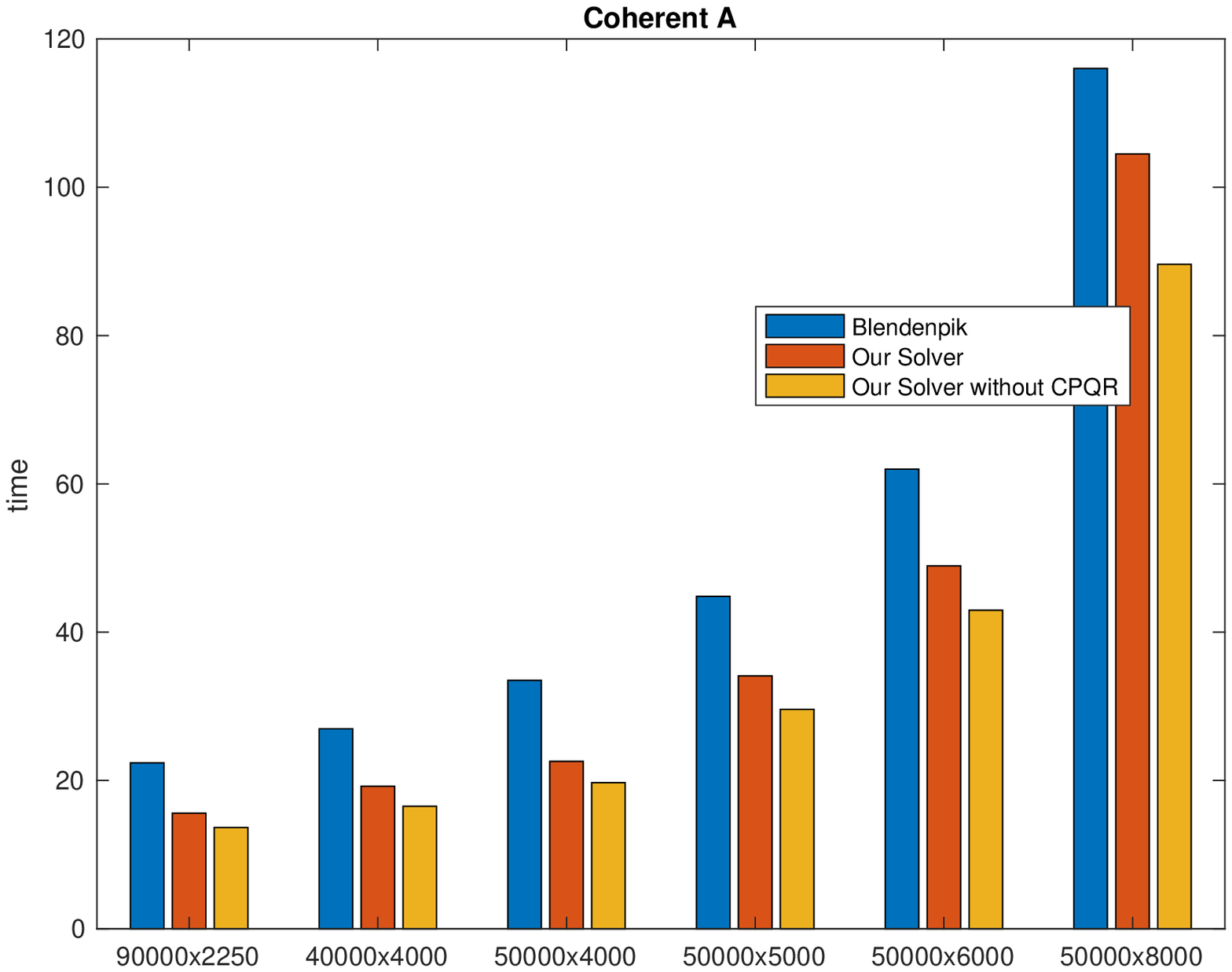}
{Time taken by solvers to compute the solution of problem (\ref{LLS-statement}) for $A$ being coherent dense matrices of various sizes (x-axis)}
{fig::compare_blen_coherent}
{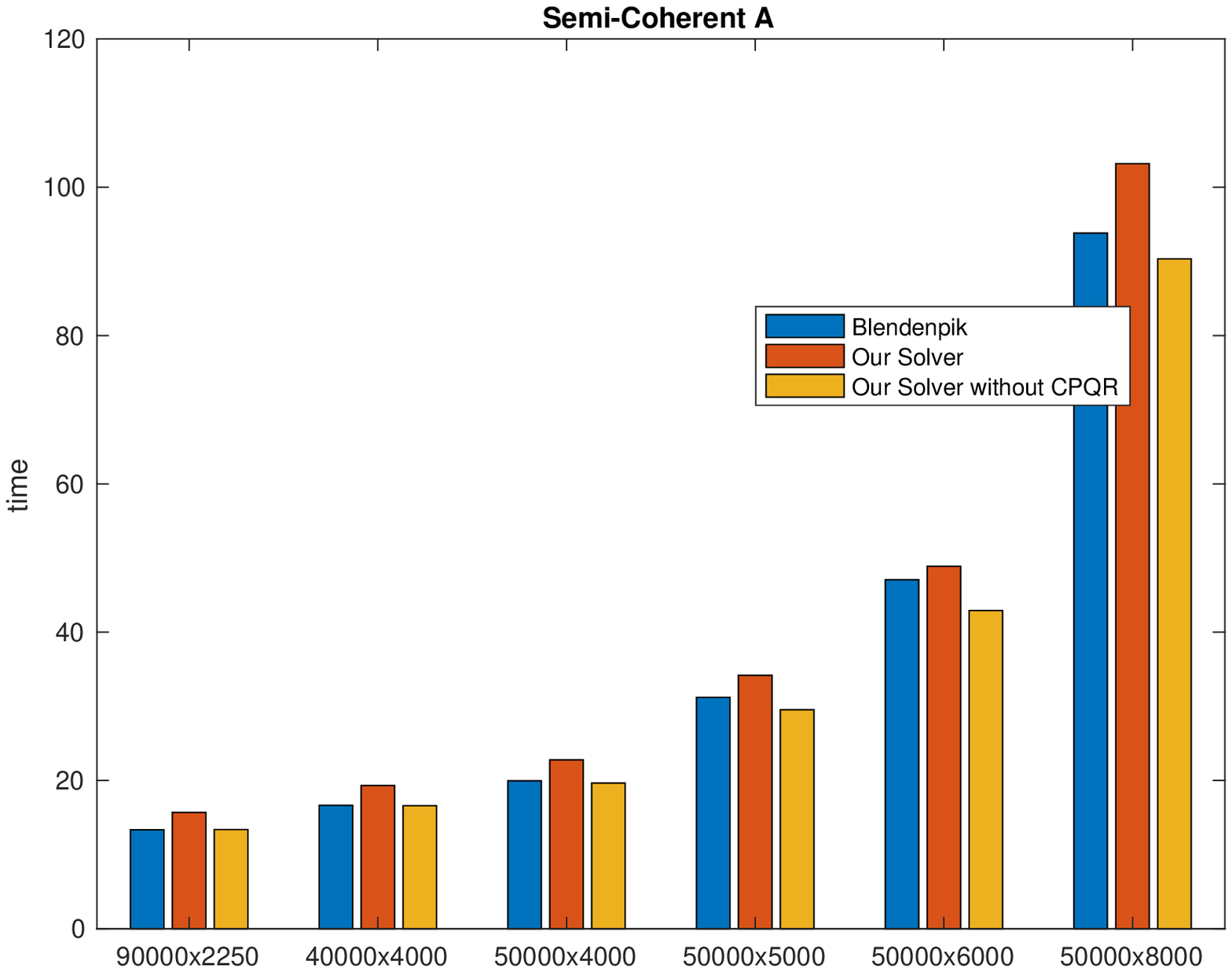}
{Time taken by solvers to compute the solution of problem (\ref{LLS-statement}) for $A$ being semi-coherent dense matrices of various sizes (x-axis)}
{fig::compare_blen_semi-coherent}
{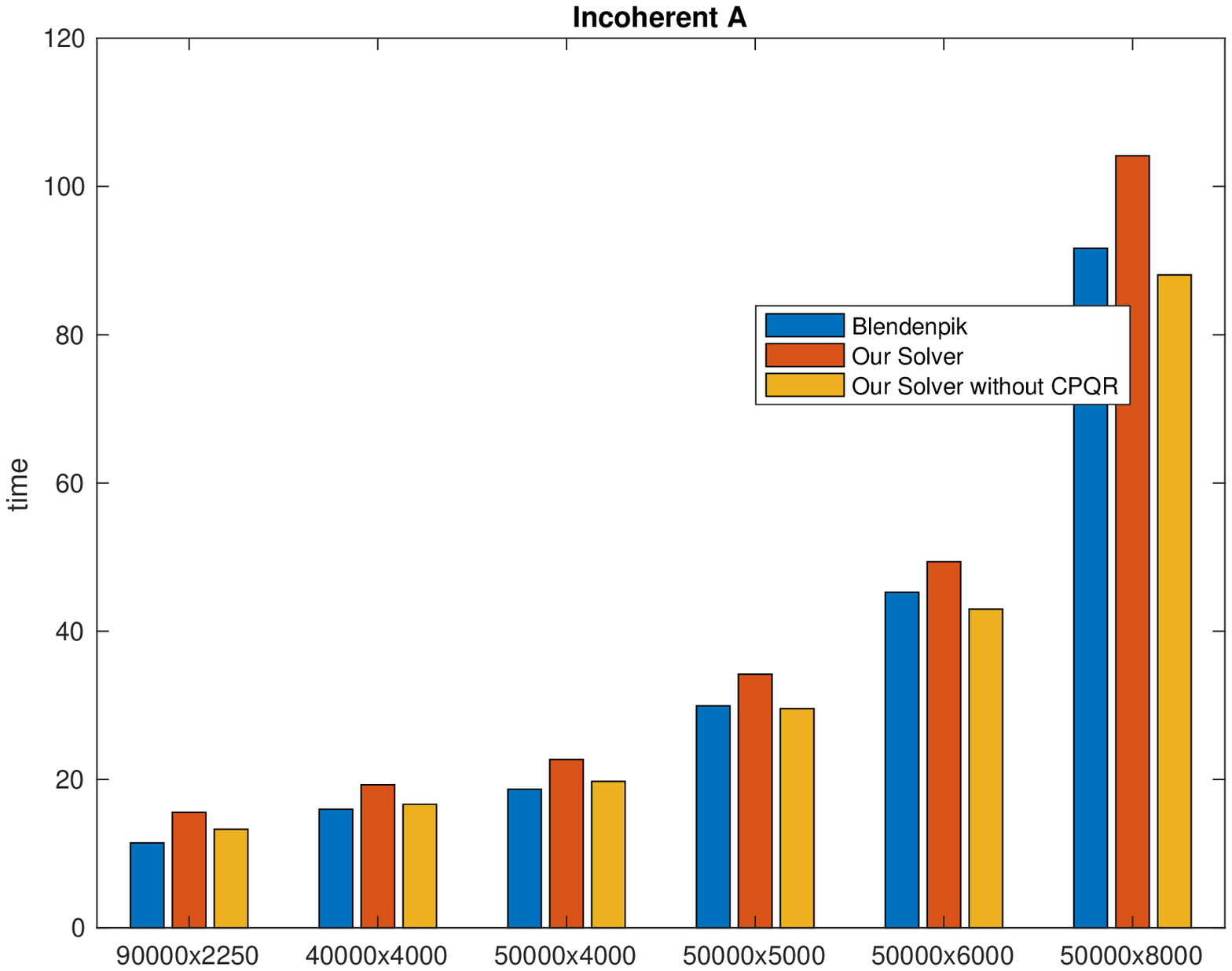}
{Time taken by solvers to compute the solution of problem (\ref{LLS-statement}) for $A$ being incoherent dense matrices of various sizes (x-axis)}
{fig::compare_blen_incoherent}

\subsection{Running time performance on randomly generated full rank sparse A}
\paragraph{Results, Sparse random matrices}
\autoref{fig::sparse_rand_inco}, \ref{fig::sparse_rand_semi_co}, \ref{fig::sparse_rand_co} show the performance of \solverName{} compared to LS\_HSL and LS\_SPQR on sparse random matrices of different types and sizes. We see \solverName{} can be up to 10 times faster on this class of data matrix $A$. We also tested on LSRN but do not report the result because LSRN takes much longer than the other solvers for this class of data in the serial running environment.

Note that in this experiment, the solvers are compiled and run in a different machine then mentioned in Section \ref{subsection:alg_implementation_discussion}, but all solvers are run on this machine in this experiment. The machine has 2.9 GHz Quad-Core Intel Core i7 CPU and 16MB RAM. Moreover, the parameter $s=3$ is chosen for \solverName{} \footnote{According to Appendix C, the running time of \solverName{} with $s=2$ and $s=3$ is similar.}. Furthermore, our solver was an old version without Step 3 implemented and uses the SPQR default ordering. Otherwise the settings are the same as the default settings for \solverName{}, LS\_HSL, LS\_SPQR and LSRN.

\newcommand{\runningTimeSparseRandomCaption}[1]{
Running time comparison of \solverName{} with LS\_HSL and LS\_SPQR for randomly generated #1 sparse matrices of different sizes.}

\threeFigures{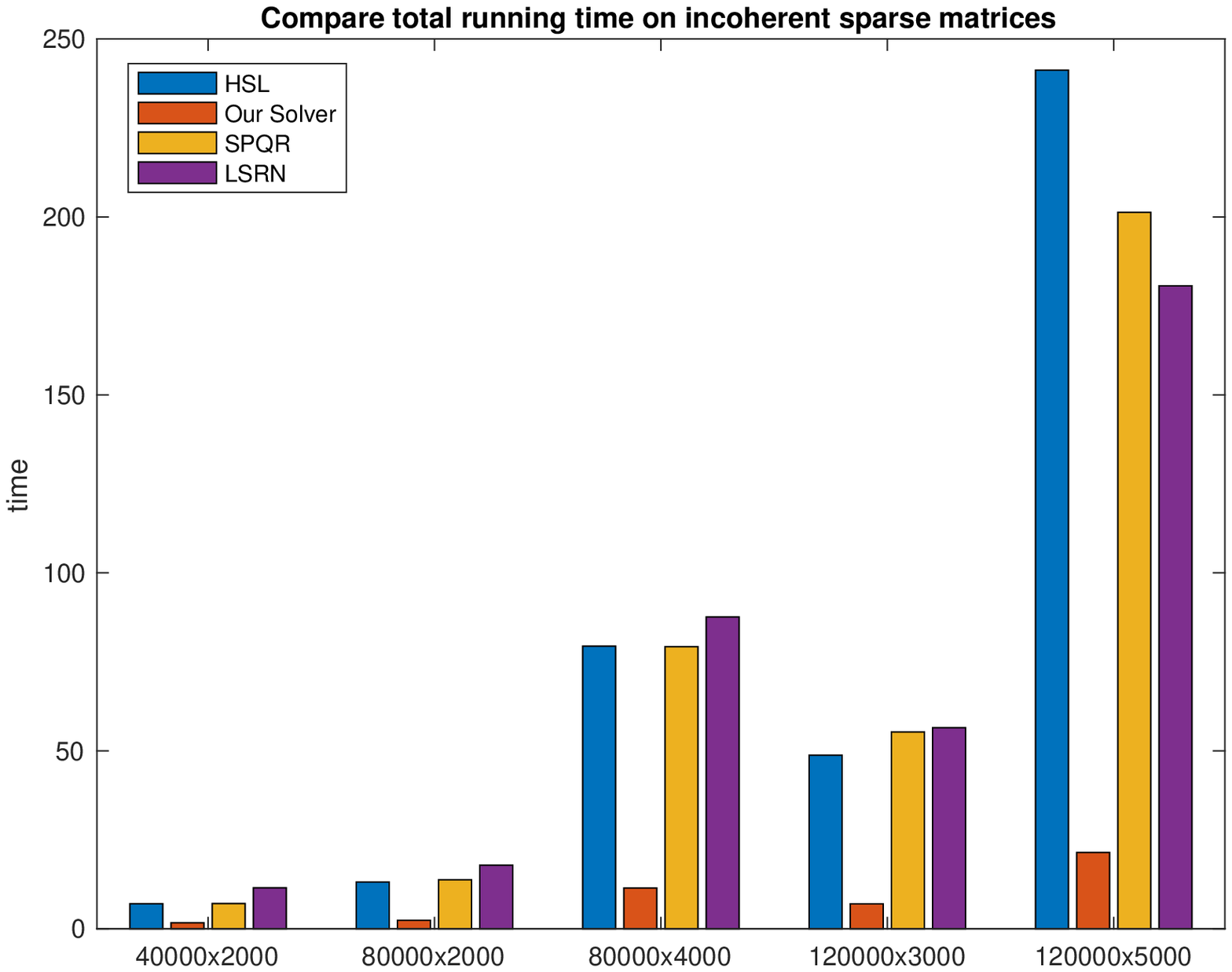}{\runningTimeSparseRandomCaption{incoherent}}{fig::sparse_rand_inco}
{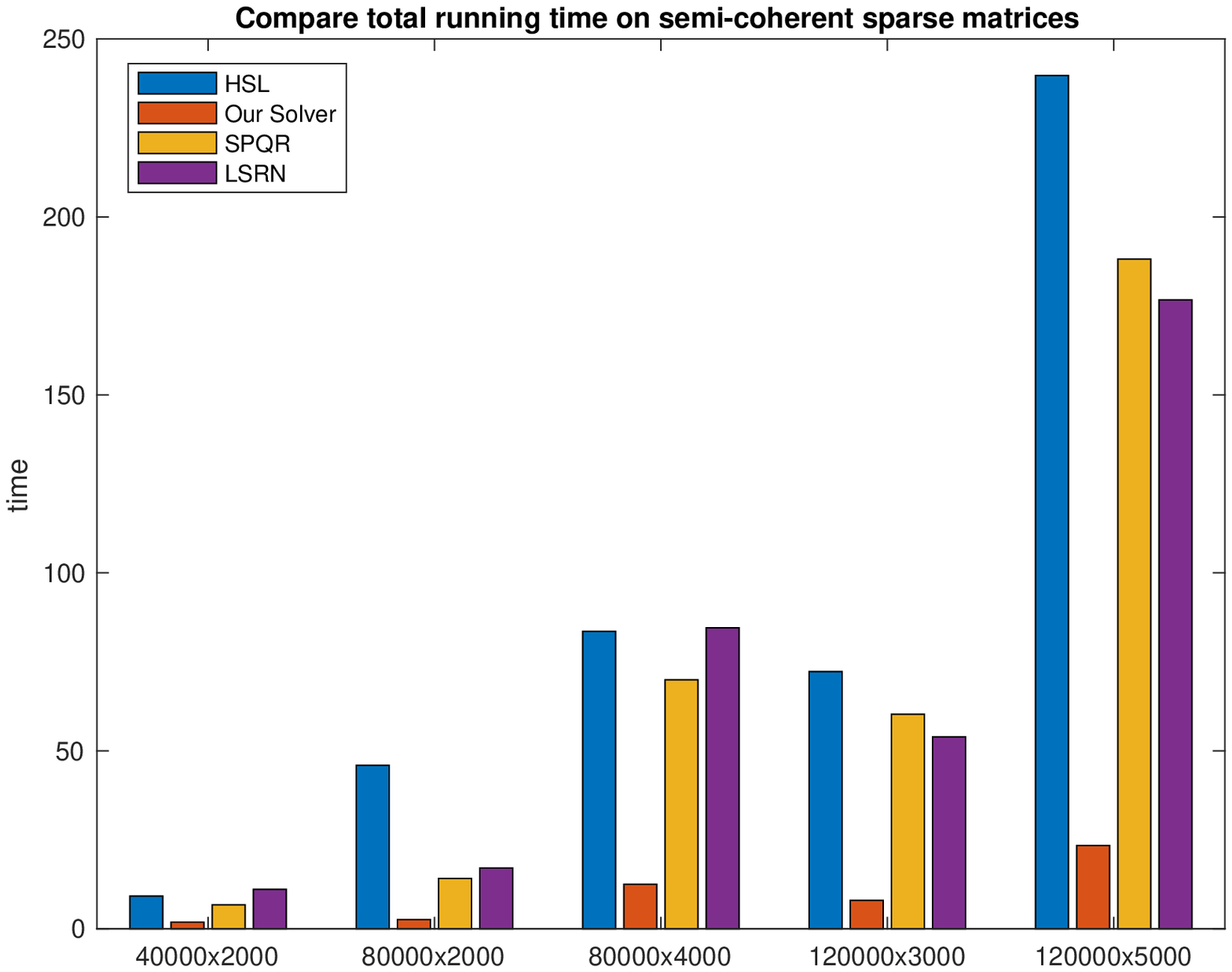}{\runningTimeSparseRandomCaption{semi-coherent}}{fig::sparse_rand_semi_co}
{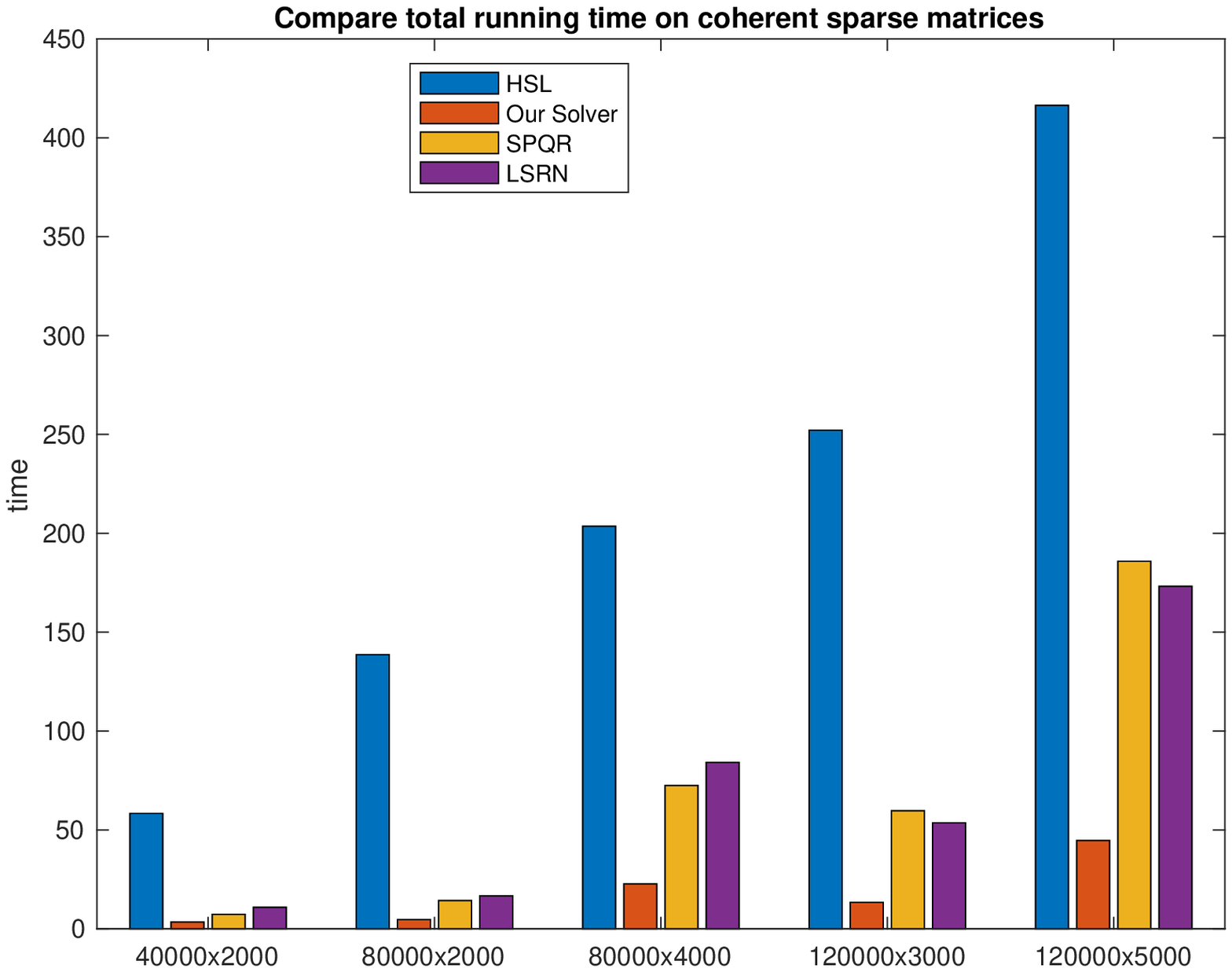}{\runningTimeSparseRandomCaption{coherent}}{fig::sparse_rand_co}

\subsection{Large scale benchmark of \solverNameSparse{} on the Florida Matrix Collection}

\paragraph{Performance profile}

Performance profile \cite{dolan2002benchmarking} has in recent years have become a popular and widely used tool for providing objective information when benchmarking software. In a typical performance profile here, we have the running time ratio against the fastest solver on the x-axis, reported in $log2$ scale. For each running time ratio $a$, we have the ratio of problems in the test set $b$ on the y-axis such that for a particular solver, the running time ratio against the best solver is within $a$ for $b$ percent of the problems in the test set. For example, the intersect between the performance curve and the y-axis gives the ratio of the problems in the test set such that a particular solver is the fastest.

\paragraph{Running and testing environment specific to the benchmark experiment}
Given a problem $A$ from the test set, let $(r_1, r_2, r_3, r_4)$ be the residuals of solutions computed by the four solvers compared. And let $r = \min r_i$. A solver is declared as failed on this particular problem if one of the following two conditions holds

\begin{enumerate}
    \item $r_i > (1+\tau_r)r$ and $r_i > r + \tau_a$. So that the residual of the solution computed is neither relatively close nor close in absolute value to the residual of the best solution
    \footnote{Note that the residual of the best solution is in general only an upper bound of the minimal residual. However since it is too computational intensive to compute the minimal residual solution of all the problems in the Florida matrix collection, we use the residual of the best solution as a proxy.}.
    
    \item The solver takes more than 800 wall clock seconds to compute a solution. 
\end{enumerate}
In the case that a solver is declared as failed, we set the running time of the solver to be 9999 seconds on the corresponding problem. This is because we want to include all the problems such that at least one of the solvers compared succeeded. As a result, a very large running time (ratio) could be due to either an inaccuracy of the solver or an inefficiency of the solver. We note that for all successful solvers, the running time is bounded above by 800 seconds so that there will be no confusion of whether a solver is successful or not. 

The default parameters (as described in Section \ref{subsection:alg_implementation_discussion}) for all solvers are used.

\paragraph{Results, highly over-determined matrices in the Florida Matrix Collection}

\autoref{fig::all_solver_30} shows \solverName{} is the fastest in 75\% of problems in the Florida matrix collection with $n\geq 30d$.

\paragraph{What happens when the problem is moderately over-determined?}
\autoref{fig::all_solver_10} shows LS\_HSL is the fastest for the largest percentage of problems in the Florida matrix collection with $n\geq 10d$. \solverName{} is still competitive and noticeably faster than LSRN.

\paragraph{Effect of condition number}
Many of the matrices in the Florida matrix condition has low condition numbers so that unpreconditioned LSQR converges in a few iterations. In those cases, it is disadvantageous to compare \solverName{} to LS\_HSL because we compute a better quality preconditioner through an complete orthogonal factorization. 

\autoref{fig::all_solver_10_LSQR_5} show \solverName{} is fastest in more than 50\% of the moderately over-determined ($n\geq 10d$) problems if we only consider problems such that it takes LSQR more than 5 seconds to solve.

\paragraph{Effect of sparsity}
Figure \ref{fig::density001} shows \solverName{} is extremely competitive, being the fastest in all but one moderately over-determined problems with moderate sparsity ($\text{nnz}(A) \geq 0.01nd)$.

\renewcommand{\mysize}{0.42}
\begin{figure}
    \centering
    \begin{minipage}{\mysize\textwidth}
        \centering
        \includegraphics[width=\textwidth]{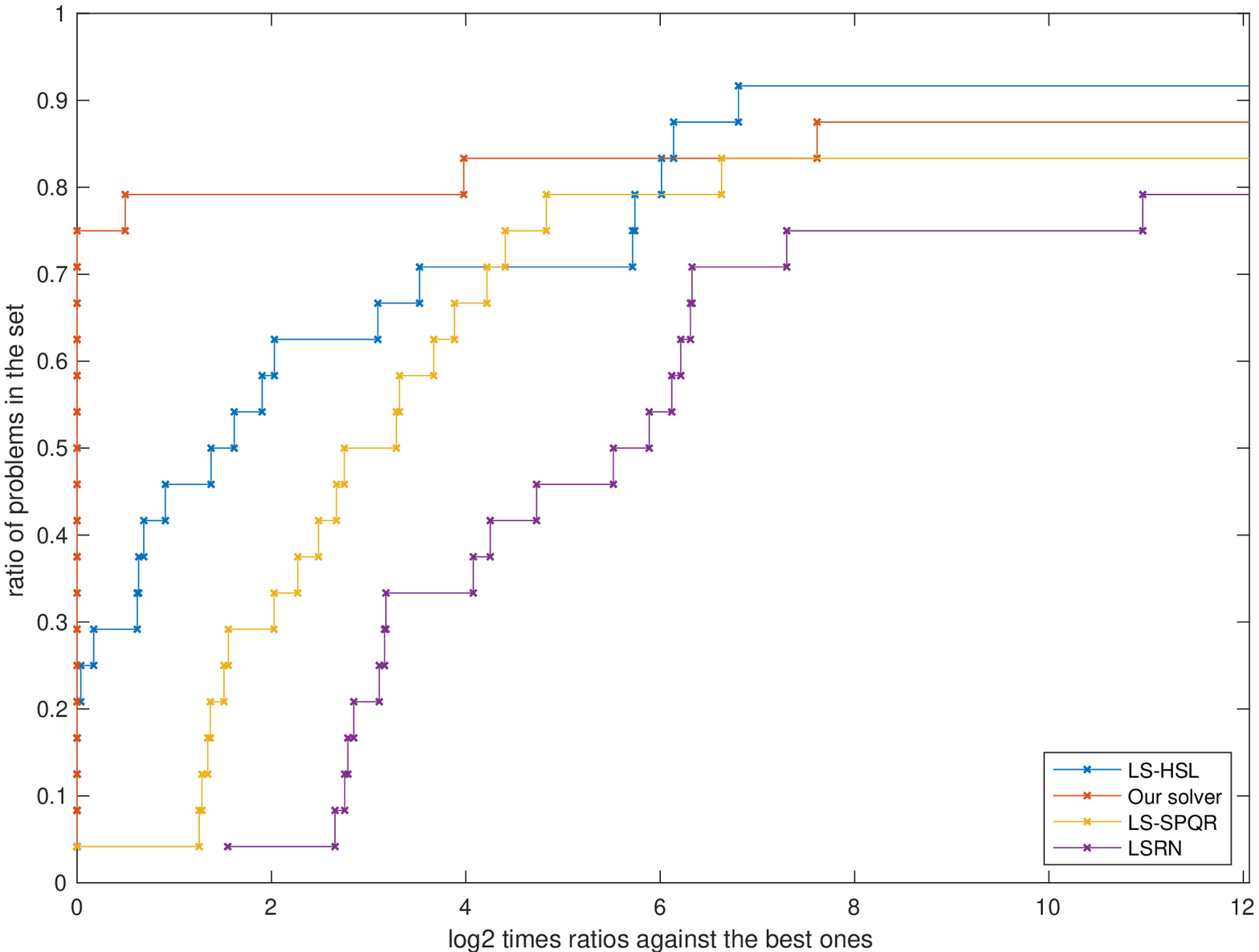} 
        \caption{\performanceProfileCaption{$n\geq 30d$}. }
        \label{fig::all_solver_30}
    \end{minipage}\hfill  
    \centering
    \begin{minipage}{\mysize\textwidth}
        \centering
        \includegraphics[width=\textwidth]{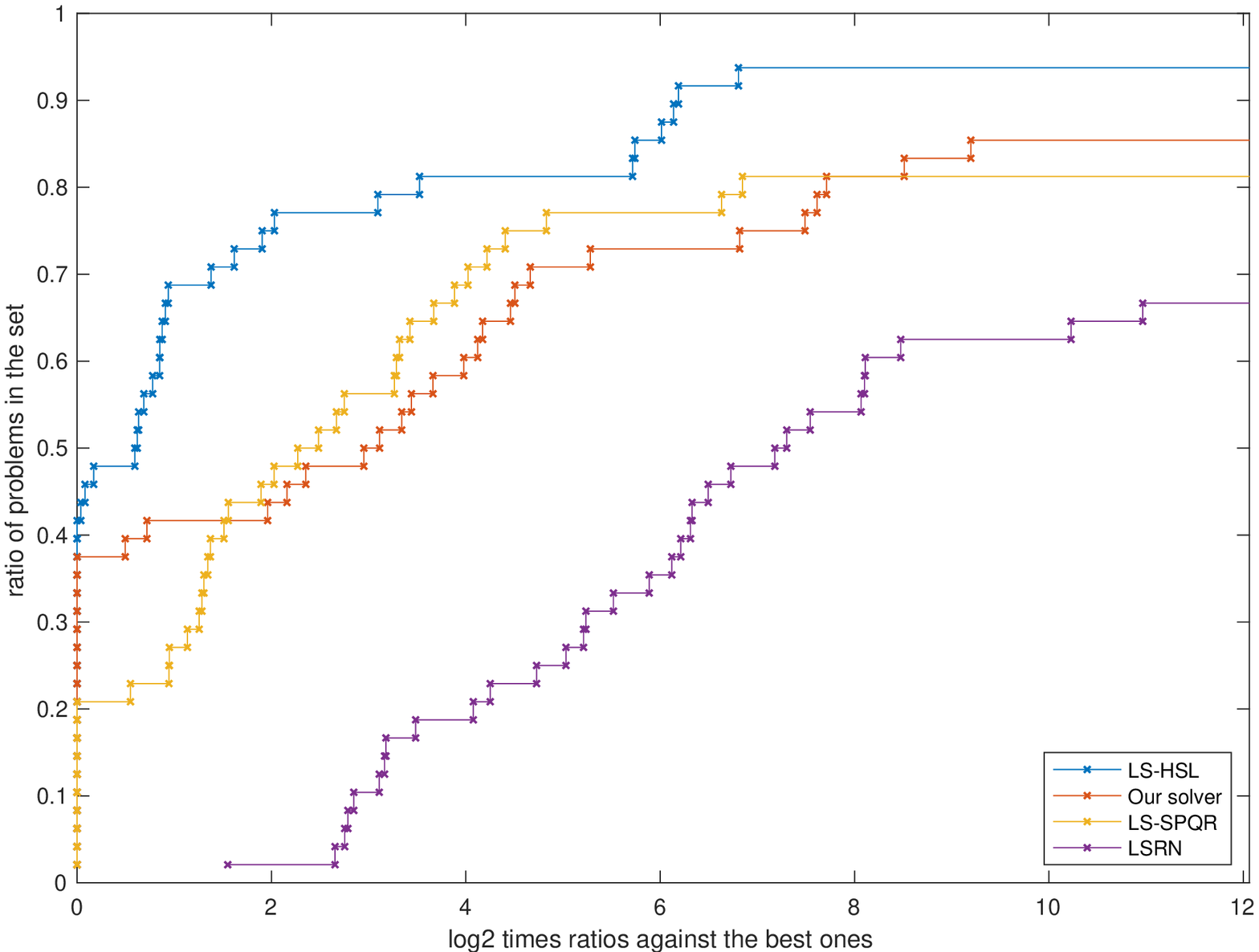} 
        \caption{\performanceProfileCaption{$n\geq 10d$}.}
        \label{fig::all_solver_10}
    \end{minipage}\hfill  
\end{figure}

\begin{figure}
    \centering
    \begin{minipage}{\mysize\textwidth}
        \centering
        \includegraphics[width=\textwidth]{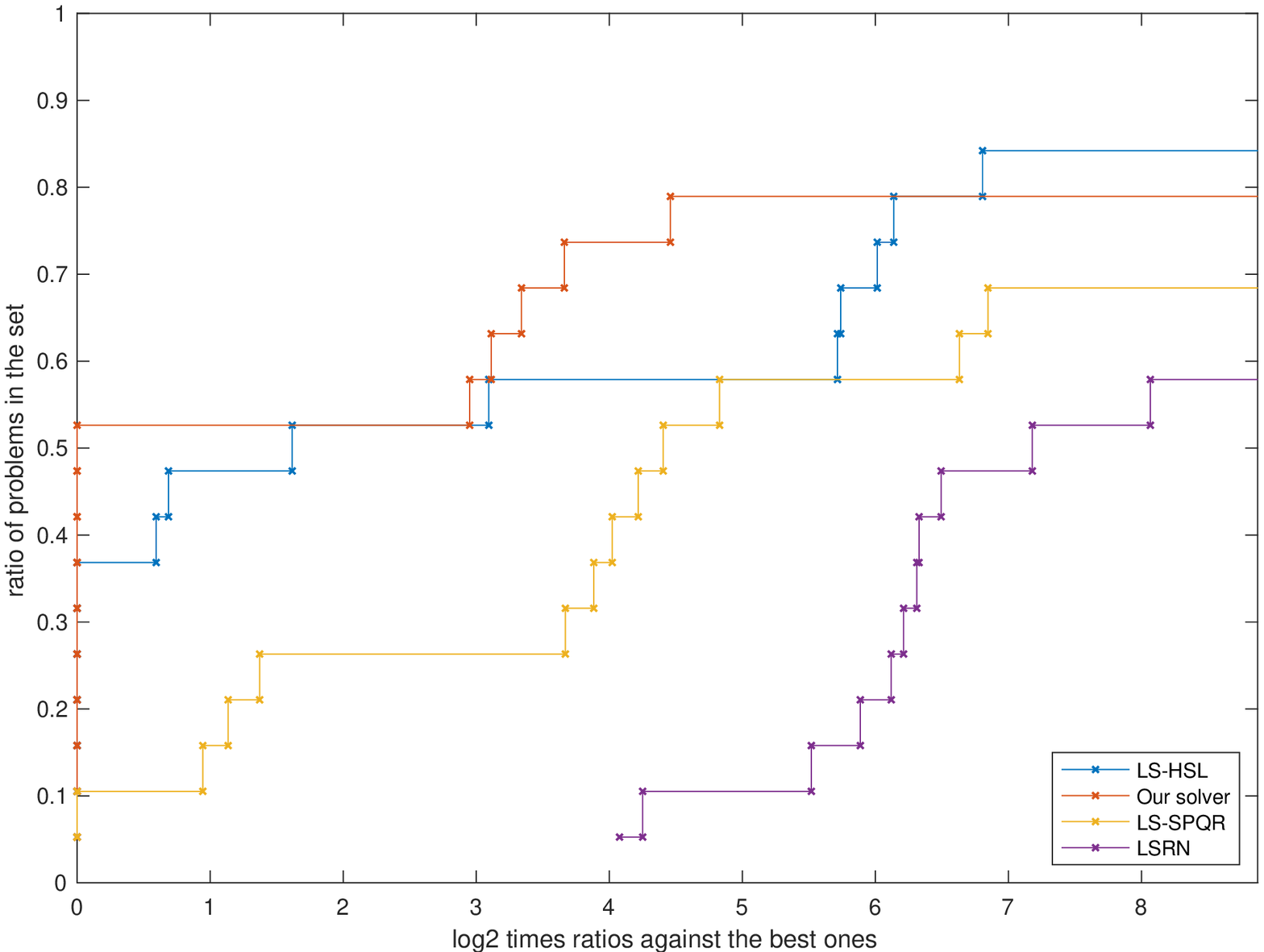} 
        \caption{\performanceProfileCaption{with $n\geq 10d$ and the unpreconditioned LSQR takes more than 5 seconds to solve}.}
        \label{fig::all_solver_10_LSQR_5}
    \end{minipage}\hfill  
    \begin{minipage}{\mysize\textwidth}
    \centering
    \includegraphics[width=\textwidth]{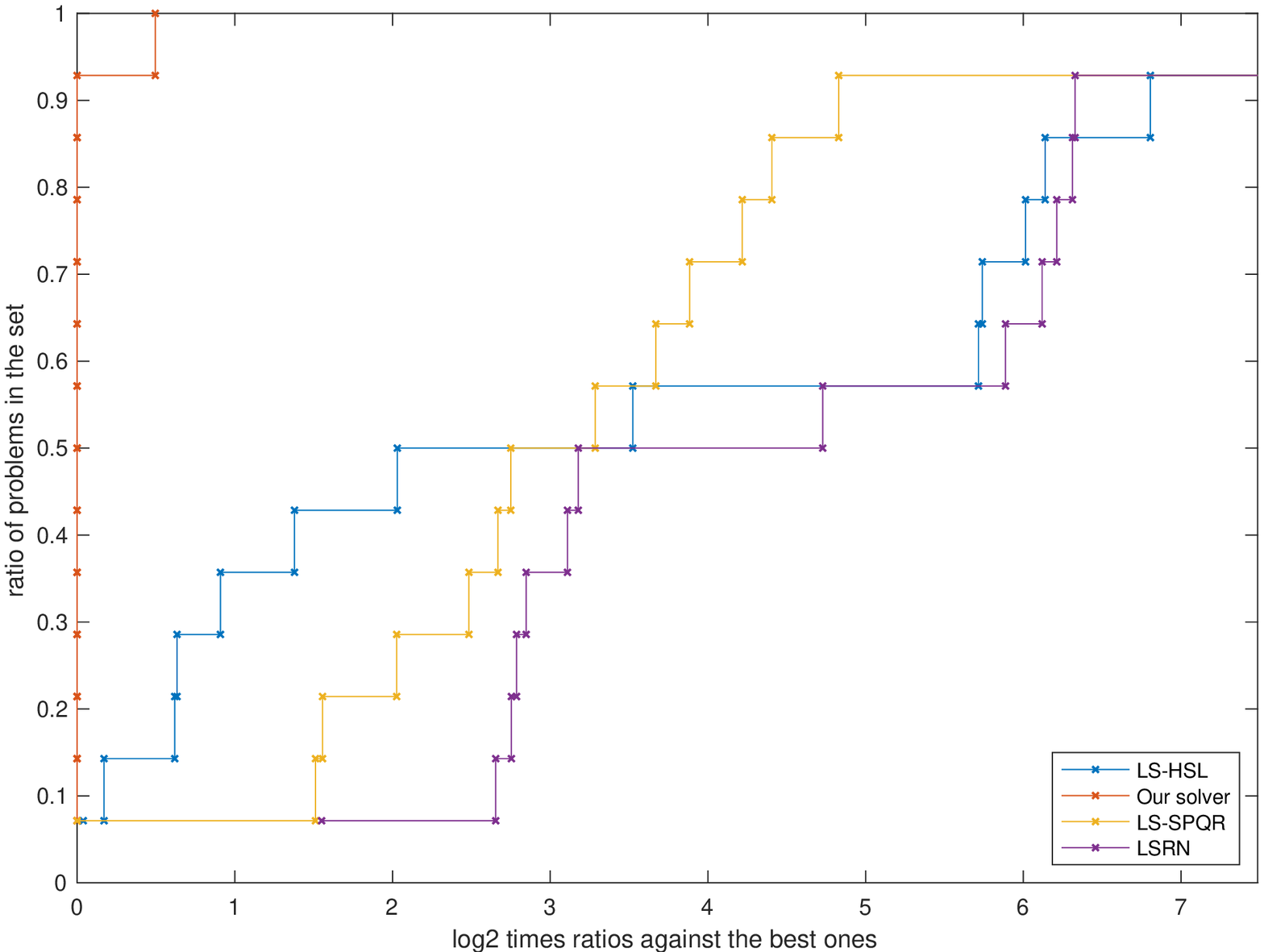}
    \caption{\performanceProfileCaption{$n \geq 10d$ and $\text{nnz}(A) \geq 0.01nd$}. }
    \label{fig::density001}
    \end{minipage}

\end{figure}

\chapter{First order subspace method for general objectives}
    
    \section{Introduction}
    This chapter expands the materials in \cite{zhen:icml_BCGN, BCGNPaper}. In Section \ref{BCGN:sec2}, \reply{we first describe} \autoref{alg:generic}, a generic algorithmic framework for solving \eqref{general_objective_statement} by taking successive steps computed from approximately minimising (random) reduced models. Our main result \autoref{thm2} provides complexity bound on the total number of iterations before \autoref{alg:generic} drives the gradient of objective below $\epsilon$, with high probability. Deducing from \autoref{thm2}, we also show that the quantity $\minGradientN$ converges to zero with high probability, and the convergence of $\expectation{\minGradientN}$. The rate of these convergences depends on a function in one of the assumptions required by the framework. 

In Section \ref{BCGN:sec2.5}, we prove \autoref{thm2}. The proof carefully counts the different types of iterations and uses a conditional form of the Chernoff bound whose proof we provide for completeness.  

In Section \ref{BCGN:sec3}, we describe \autoref{alg:sketching} that particularises \autoref{alg:generic} by using random matrices to build random reduced models. We show how using random matrices that are oblivious JL embeddings satisfies the assumptions required for the convergence result in \autoref{thm2}. 

In Section \ref{BCGN:sec4}, \autoref{alg:sketching}
\reply{is further particularised} to a quadratic-regularisation and a trust-region variant, depending on how the minimisation of the random reduced model is specified. Section \ref{BCGN:sec4} then uses \autoref{thm2} to show that \reply{both variants drive} the full objective gradient $\gradFK$ below $\epsilon$ in $\mathO{\frac{1}{\epsilon^2} }$ iterations with high probability, matching the deterministic methods' iteration complexity. 

In Section \ref{BCGN:sec5}, we introduce non-linear least squares, a particular type of non-convex optimisation problems \eqref{general_objective_statement}. We show how \autoref{alg:sketching} safe-guarded with trust-region leads to a subspace version of the well-known Gauss-Newton method, Randomised-Subspace Gauss Newton (R-SGN), with a convergence guarantee. We numerically illustrate the performance of R-SGN on non-linear least squares and logistic regression problems.  

\paragraph{Related literature}
\cite{Cartis:2017fa} proposes a generic algorithmic framework based on probabilistic models with an expected iteration complexity bound  to generate a sufficiently small  (true) gradient. Various strategies are discussed in \cite{Cartis:2017fa} to generate such models both for derivative-based and derivative-free methods; however, subspace methods cannot be easily captured within the conditions and models used in \cite{Cartis:2017fa}. In
\cite{Gratton:2017kz}, a trust-region based method with probabilistically accurate models is proposed, with an iteration complexity bound of $\mathO{\epsilon^{-2}}$  for the algorithm to drive the full gradient below $\epsilon$, with high probability. 
\cite{kozak2019stochastic} analyses a random subspace method with constant step size, where the sketching matrix $S_k$ satisfies $\expectation{S_k^T S_k} = \mathbb{I}_d$ and $S_k S_k^T = \frac{d}{l} \mathbb{I}_d$. However, their convergence result \reply{requires the objective to be convex, or to satisfy} the Polyak-Lojasiewicz inequality. Independently from our work/at the same time, \cite{RePEc:spr:coopap:v:79:y:2021:i:2:d:10.1007_s10589-021-00271-w} proposes a random subspace gradient descent method with linesearch; it uses Johnson-Lindenstrauss embedding properties  in the analysis, similarly to our framework, but fewer ensembles are considered. However, their analysis only applies under various convexity assumptions of the objective.

    \section{General algorithmic framework and its convergence result}
    \label{BCGN:sec2}
    We consider the unconstrained optimisation problem
\begin{equation}
    f^* = \min_{x \in \R^d} f(x).\label{eqn::fStar} 
\end{equation}

\subsection{Generic algorithmic framework and assumptions}

We first describe a generic algorithmic framework that encompasses the
main components of the unconstrained optimization schemes we analyse
in this chapter. The scheme relies on building a local, reduced model of the objective
function at each iteration, 
minimizing this model or reducing it in a sufficient manner and
considering the step which is dependent on a stepsize parameter and
which provides the model reduction (the stepsize parameter may be
present in the model or independent of it). 
This step determines a new candidate point. The function value is then
computed (accurately) at the new candidate point. 
If the function reduction provided by the candidate point is deemed
sufficient, then the iteration is declared successful, the 
candidate point becomes the new iterate and the step size parameter is
increased. Otherwise, the iteration is 
unsuccessful, the iterate is not updated and the step size parameter is reduced.

We summarize the main steps of the generic framework below.

\begin{algorithm}[H]
\begin{description}
\item[Initialization] \ \\
Choose a class of  (possibly random) models $m_k\left(w_k(\sHat)\right) = \mKHat{\sHat}$, where $\sHat \in \R^l$ with $l\leq d$ is the step parameter and $w_k: \R^l \to \R^d$ is the prolongation function. 
Choose constants $\gamma_1\in (0,1)$, $\gamma_2 = \gammaOne^{-c}$, for some \reply{$c \in \N^+$ ($\N^+$ refers to the set of positive natural numbers)}, $\theta \in (0,1)$ and $\alpha_{\max}>0$.
Initialize the algorithm by setting $x_0 \in \R^d$, $\alpha_0 = \alphaMax \gamma_1^p$ for some $p \in \N^+$ and $k=0$.

 \item[1. Compute a reduced model and a step] \ \\
Compute a local (possibly random) reduced model $\mKHat{\sHat}$ of $f$ around $x_k$ with $\mKHat{0} = f(x_k)$. \\
Compute a step parameter $\sHat_k(\alpha_k)$, where the parameter $\alpha_k$ is present in the reduced model or the step parameter computation.\\
Compute a potential step $s_k = w_k(\sHat_k)$.

\item[2. Check sufficient decrease]\ \\  
Compute $f(x_k + s_k)$ and check if sufficient decrease (parameterized by $\theta$) is achieved in $f$ with respect to $\mKHat{0} - \mKHat{\hat{\sK}(\alpha_k)}$.

\item[3, Update the parameter $\alphaK$ and possibly take the potential step $\sK$]\ \\
If sufficient decrease is achieved, set $\xKPlusOne = \xK + \sK$ and $\alphaKPlusOne = \min \set{\alphaMax, \gammaTwo\alphaK}$ [this is referred to as a successful iteration].\\
Otherwise set $\xKPlusOne = \xK$ and $\alphaKPlusOne = \gammaOne \alphaK$ [this is an unsuccessful iteration].\\
Increase the iteration count by setting $k=k+1$ in both cases. 

\caption{\bf{Generic optimization framework based on  randomly generated reduced models}} \label{alg:generic} 

\end{description}
\end{algorithm}

The generic framework and its assumptions we present here is similar to the framework presented in \cite{Cartis:2017fa}. We extended their framework so that the proportionality constants for increase/decrease of the step size parameter are not required to be reciprocal, but reciprocal up to an integer power (see \autoref{AA3}). 
Even though the framework and its assumptions are similar, our analysis and result are different and qualitatively improve upon their result (in the way that \autoref{thm2} implies their main result Theorem 2.1).
Moreover, we show how to use random-embedding based sketching to build the reduced model in Section \ref{BCGN:sec3}.

The connection between the generic framework and classical optimisation literature is detailed in \cite{Cartis:2017fa}. Here we give a simple example. If we let $l=d$ and $w_k$ be the identity function in \autoref{alg:generic} so that $m_k(s)=\mKHat{\sHat}$; and if we let 
\begin{equation}
    m_k(s) = f(x_k) + \gradFK^T s + \frac{1}{2}s^T B_k s, \notag
\end{equation}
where $B_k$ is a Hessian approximation, and compute the step parameter (or in this case, since $w_k$ is the identify function, the step) by seeking a solution of the problem
\begin{equation}
    \min_{s\in\R^d} m_k(s) \texteq{such that $\normTwo{s}\leq \alphaK$}; \notag
\end{equation}
and if we define the sufficient decrease by 
\begin{equation}
    f(x_k) - f(x_k + s_k) \geq \theta \squareBracket{\mK{0}-\mK{\sK}}, \notag
\end{equation}
then \autoref{alg:generic} reduces to the (deterministic) trust-region method, see \cite{Nocedal:2006uv}.

Because the model at each iteration is (possibly) random, $\xK, \sK, \alphaK$ are in general random variables. We will use $\barXK, \bar{s}_k, \bar{\alpha}_k$ to denote their realizations. We define convergence in terms of a random variable $\nEps$, that can be a function of a positive scalar(s) $\epsilon$, as well as the sequences $\set{f(x_k)}, \set{\gradFK}, \set{\hessFK}$. For example, 

\begin{equation}
    \nEps = \min \set{k: \normTwo{\gradFK}\leq \epsilon} \label{eqn::nEps}
\end{equation}

will be used to represent convergence to a first-order local stationary point, as in \cite{Cartis:2017fa}. We say that \autoref{alg:generic} has not converged if $k \leq \nEps$, and has converged otherwise. Furthermore, let us suppose that there is a class of iterations, hereafter will be referred to as \textbf{true iterations} such that \autoref{alg:generic} satisfies the some conditions. 

The convergence of \autoref{alg:generic} relies on the following four assumptions. 
The first assumption states that given the current iterate (at any value), an iteration $k$ is true at least with a fixed probability, and is independent of the truth values of all previous iterations.

\begin{assumption}\label{AA2}
There exists $\delta_S \in (0,1)$ such that for any $\barXK \in \R^d$ and $k=1, 2, \dots$
\begin{equation}
    \probabilityGivenXK{T_k} \geq 1-\delta_S, \notag
\end{equation}
where $T_k$ is defined as
        \begin{equation}
            T_k = 
        \twoCases{1}{\text{if iteration $k$ is true}}{0}{\text{otherwise}}
        \label{eqn::t_k}
        \end{equation}

Moreover, $\probability{T_0} \geq 1-\delta_S$; and $T_k$ is conditionally independent of $T_0, T_1, \dots, T_{k-1}$ given $x_k = \barXK$. 
\end{assumption}

The next assumption says that for $\alphaK$ small enough, any true iteration before convergence is guaranteed to be successful.

\begin{assumption}\label{AA3}
For any $\epsilon >0$, there exists an iteration-independent constant $\alphaLow>0$ (that may depend on $\epsilon$ and the problem and algorithm parameters) such that 
if iteration $k$ is true, $k< \nEps$, and $\alpha_k \leq \alphaLow$ then iteration $k$ is successful. 

\end{assumption}

The next assumption says that before convergence, true and successful iterations result in an objective decrease lower bounded by an (iteration-independent) function $h$, which is monotonically increasing in its two arguments, $\epsilon$ and $\alphaK$. 

\begin{assumption}\label{AA4}
There exists a non-negative, non-decreasing function $h(z_1,z_2)$ such that,
for any $\epsilon>0$, 
if iteration $k$ is true and successful with $k < \nEps$, then 

	\begin{equation}
	f(\xK) - f(\xK + \sK) \geq h(\epsilon, \alpha_k), \label{eqn::generic_model_decrease}
	\end{equation}
where $s_k$ is computed in step 1 of \autoref{alg:generic}. Moreover, $h(z_1, z_2)>0$ if both $z_1>0$ and $z_2>0$.
\end{assumption}

The final assumption requires that the function values at successive iterations must form a non-increasing sequence throughout the algorithm. 

\begin{assumption}\label{AA5}
For any $k \in \N$, we have 
\begin{equation}
    f(\xK) \geq f(\xKPlusOne). \label{eqn::fKNonIncreasing}
\end{equation}
\end{assumption}

The following Lemma is a simple consequence of \autoref{AA3}. 

\begin{lemma}
\label{lem::alphaMin}
Let $\epsilon>0$ and \autoref{AA3} hold with $\alphaLow>0$. Then there exists $\newL \in \N^+$, and $\alphaMin >0$ such that 
\begin{align}
&\alphaMin = \alphaZero \gammaOne^\newL, \label{eqn::alphaMin}\\
    &\alphaMin \leq \alphaLow, \notag \\
    &\alphaMin \leq \frac{\alphaZero}{\gammaTwo}, \label{eqn::alphaMinUpperByGammaTwoOverAlphaZero}
\end{align}
where $\gammaOne, \gammaTwo, \alphaZero$ are defined in \autoref{alg:generic}. 
\end{lemma}
\begin{proof}
Let 
\begin{align}
    &\newL = \ceil{\logBaseGammaOne{ \minMe{\frac{\alphaLow}{\alphaZero}}{\frac{1}{\gammaTwo}}}}, \label{eqn:tauLDef}\\
    & \alphaMin = \alphaZero \gammaOne^\newL. \label{eqn:alphaMinDef}
\end{align}
We have that $\alphaMin \leq \alphaZero \gammaOne^{\logBaseGammaOne{ \frac{\alphaLow}{\alphaZero}}} = \alphaLow$. Therefore by \autoref{AA3}, if iteration $k$ is true, $k< \nEps$, and $\alpha_k \leq \alphaMin$ then iteration $k$ is successful. Moreover, $\alphaMin \leq \alphaZero \gammaOne^{\logBaseGammaOne{\frac{1}{\gammaTwo}}} = \frac{\alphaZero}{\gammaTwo} = \alphaZero \gammaOneC$. It follows from $\alphaMin = \alphaZero \gammaOne^\newL$ that $\newL \geq c$. Since $c \in \N^+$, we have $\newL \in \N^+$ as well.

\end{proof}

\subsection{A probabilistic convergence result}

\autoref{thm2} is our main result for \autoref{alg:generic}. It states a probabilistic bound on the total number of iterations $\nEps$ needed to converge to $\epsilon$-accuracy for the generic framework.

\begin{theorem}
\label{thm2}
Let \autoref{AA2}, \autoref{AA3}, \autoref{AA4} and \autoref{AA5} hold with $\delta_S \in (0,1), \alphaLow>0$, $h: \R^2 \to \R$ and $\alphaMin = \alphaZero \gammaOne^\newL$ associated with $\alphaLow$, for some $\newL \in \N^+$. Let $\epsilon>0$, $f^*$ defined in \eqref{eqn::fStar}. Run \autoref{alg:generic} for $N$ iterations.
Suppose 
\begin{equation}
    \deltaS < \frac{c}{(c+1)^2} \label{eqn::deltaSConditionThmTwo}
\end{equation}
Then for any $\delta_1 \in (0,1)$ such that 
\begin{equation}
    \gDeltaSDeltaOne >0, \label{eqn:tmp32}
\end{equation}
where 
\begin{equation}
    g(\deltaS, \deltaOne) = \nPreFactorTR. \label{eqn:gDeltaSDeltaOneDef}
\end{equation}
If $N$ satisfies 
\begin{equation}
    N \geq \gDeltaSDeltaOne \squareBracket{
         \fZeroMinusfStarOverH
         + \frac{\newL}{1+c}}, \label{eqn::n_upper_2}
\end{equation}
we have that
\begin{equation}
    \probability{N \geq \nEps} \geq 1 - \chernoffLowerExponential. \footnote{For the sake of clarity, we stress that $N$ is a deterministic constant, namely, the total number of iterations that we run \autoref{alg:generic}. $\nEps$, the number of iterations needed before convergence, is a random variable.} \label{eqn:tmp33}
\end{equation}
\end{theorem}

\begin{remark}
Note that $\frac{c}{(c+1)^2} \in (0, \frac{1}{4}]$ for $c \in \N^+$. Therefore \eqref{eqn::deltaSConditionThmTwo} and \eqref{eqn:tmp32} can only be satisfied for some $c, \deltaOne$ given that $\deltaS < \frac{1}{4}$. Thus our theory requires an iteration is true with probability at least $\frac{3}{4}$. Compared to the analysis in \cite{Cartis:2017fa}, which requires an iteration is true with probability at least $\frac{1}{2}$, our condition is stronger. This is due to the high probability nature of our result, while their convergence result is in expectation. Furthermore, we will see in \autoref{lem::deduceAA2} that we are able to impose arbitrarily small value of $\delta_S$, thus satisfying the requirement, by choosing an appropriate dimension of the local reduced model $\mKHat{\sHat}$.
\end{remark}

We show how our result leads to Theorem 2.1 in \cite{Cartis:2017fa}, which concerns $\expectation{\nEps}$. We have, with $N_0$ defined as the RHS of \eqref{eqn::n_upper_2}, 
\begin{align}
    \expectation{\nEps} &= \integral{0}{\infty}{\probNEpsGrN} \notag \\
    &= \integral{0}{N_0}{\probNEpsGrN} + \integral{N_0}{\infty}{\probNEpsGrN} \notag \\
    & \leq N_0 + \integral{N_0}{\infty}{\probNEpsGrN} \notag \\
    & \leq N_0 + \integral{N_0}{\infty}{e^{-\frac{\delta_1^2}{2} (1-\delta_S) M} dM} \notag \\
    & = N_0 + \frac{2}{\delta_1^2 (1-\deltaS)} e^{-\frac{\delta_1^2}{2} (1-\delta_S) N_0}, \notag
\end{align}
where we used \autoref{thm2} to derive the last inequality. The result in \cite{Cartis:2017fa} is in the form of $\expectation{\nEps} \leq N_0$. Note that the discrepancy term is exponentially small in terms of $N_0$ and therefore the implication of our result is asymptotically the same as that in \cite{Cartis:2017fa}.

\subsection{Corollaries of \autoref{thm2}}
Before we begin the proof, we state and prove three implications of \autoref{thm2}, provided some mild assumptions on $h$ with $\nEps$. These results show different flavours of \autoref{thm2}.

The following expressions will be used, along with $\gDeltaSDeltaOne$ is defined in \eqref{eqn:tmp32}
\begin{align}
    &q(\epsilon) = h(\epsilon, \gammaOneC\alphaMin), \label{eqn:qEps} \\
    &D_1 = \gDeltaSDeltaOne (f(x_0) - f^*), \label{eqn::D_1} \\
    &D_2 = \gDeltaSDeltaOne \frac{\newL}{1+c}, \label{eqn::D_2}\\
    &D_3 = \frac{\delta_1^2}{2}(1-\delta_S). \label{eqn::D_3}
\end{align}. 

From \eqref{eqn:qEps}, \eqref{eqn::D_1}, \eqref{eqn::D_2}, \eqref{eqn::D_3}, a sufficient condition for \eqref{eqn::n_upper_2} to hold is
\begin{align}
    N &\geq \gDeltaSDeltaOne \squareBracket{\frac{f(x_0)-f^*}{h(\epsilon, \gammaOneC\alphaMin)} + \frac{\tCOne}{1+\tCTwo}} \notag \\
    & = \frac{D_1}{q(\epsilon)} + D_2; \notag
\end{align}
and \eqref{eqn:tmp33} can be restated as 
\begin{equation}
    \probability{N > \nEps} \geq 1 - e^{-D_3 N}. \notag
\end{equation}

The first corollary gives the rate of change of $\minGradientN$ as $N \to \infty$. It will yield a rate of convergence by substituting in a specific expression of $h$ (and hence $q^{-1}$).
\begin{corollary}
\label{Cor_1}
Let \autoref{AA2}, \autoref{AA3}, \autoref{AA4}, \autoref{AA5} hold. \\ Let $f^*, q, D_1, D_2, D_3$ be defined in \eqref{eqn::fStar}, \eqref{eqn:qEps}, \eqref{eqn::D_1}, \eqref{eqn::D_2} and \eqref{eqn::D_3}. \CorTwoDeltaOne. 
Then for any $N \in \N$ such that $\qInverseOfDOneOverNMinusDTwo$ exists, we have

\begin{equation}
    \probability{\minGradient{N} \leq \qInverseOfDOneOverNMinusDTwo} \geq 1 - e^{-D_3 N}. \label{eqn:tmp34}
\end{equation}
\end{corollary}

\begin{proof}
Let $N \in \N$ such that $\qInverseOfDOneOverNMinusDTwo$ exists and let $\epsilon = \qInverseOfDOneOverNMinusDTwo$. Then we have
\begin{equation}
    \probability{\minGradientN \leq \qInverseOfDOneOverNMinusDTwo}
    = \probability{\minGradientN \leq \epsilon}
    \geq \probability{N > \nEps}, \label{eq::tmp9}
\end{equation}
where the inequality follows from the fact that $N \geq \nEps$ implies $\minGradientN \leq \epsilon$. 
On the other hand, we have
\begin{align}
    N &= \frac{D_1}{\frac{D_1}{N-D_2}} + D_2 \notag \\
    &= \frac{D_1}{q(\epsilon)}+D_2. \notag
\end{align}
Therefore \eqref{eqn::n_upper_2} holds; and applying \autoref{thm2}, we have that $\probability{N \geq \nEps} \geq 1 - e^{-D_3 N}$. Hence \eqref{eq::tmp9} gives the desired result. 

\end{proof}

The next Corollary restates \autoref{thm2} for a fixed arbitrarily high success probability. 
\begin{corollary}
\label{Cor_2}
Let \autoref{AA2}, \autoref{AA3}, \autoref{AA4}, \autoref{AA5} hold.
\CorTwoDeltaOne. Then for any $\delta \in (0,1)$, suppose
\begin{equation}
    N \geq \max \set{\DOneOverQPlusDTwo, \frac{\logOneOverDelta}{D_3}}, \label{eqn::tmp10}
\end{equation}
where $D_1, D_2, D_3, q$ are defined in \eqref{eqn::D_1}, \eqref{eqn::D_2}, \eqref{eqn::D_3} and \eqref{eqn:qEps}. Then
\begin{equation}
    \probability{\minGradientN < \epsilon} \geq 1-\delta.  \notag
\end{equation}

\end{corollary}

\begin{proof}
We have 
\begin{equation}
    \probability{\minGradientN \leq \epsilon}
    \geq \probability{N \geq \nEps}
    \geq 1- e^{-D_3 N}
    \geq 1-\delta, \notag
\end{equation}
where the first inequality follows from definition of $\nEps$ in \eqref{eqn::nEps}, the second inequality follows from \autoref{thm2} (note that \eqref{eqn::tmp10} implies \eqref{eqn::n_upper_2}) and the last inequality follows from \eqref{eqn::tmp10}.
\end{proof}

The next Corollary gives the rate of change of the expected value of $\minGradientN$ as $N$ increases.

\begin{corollary}
\label{Cor_3}
Let \autoref{AA2}, \autoref{AA3}, \autoref{AA4}, \autoref{AA5} hold.
\CorTwoDeltaOne. Then for any $N \in \N$ such that $\qInverseOfDOneOverNMinusDTwo$ exists, where $q, D_1, D_2$ are defined in \eqref{eqn:qEps}, \eqref{eqn::D_1}, \eqref{eqn::D_2}, we have
\begin{equation}
    \expectation{\minGradientN} \leq \qInverseOfDOneOverNMinusDTwo + \normTwo{\grad f(x_0)} \eToMinusDThreeN, \notag
\end{equation}
where $D_3$ is defined in \eqref{eqn::D_3} and $x_0$ is chosen in \autoref{alg:generic}.
\end{corollary}

\begin{proof}
We have
\begin{equation}
    \begin{split}
    & \expectation{\minGradientN} \notag \\
     &\leq \probability{\minGradientN \leq \qInverseOfDOneOverNMinusDTwo} \qInverseOfDOneOverNMinusDTwo \notag \\
    &+ \probability{\minGradientN > \qInverseOfDOneOverNMinusDTwo}\normTwo{\grad f(x_0)} \notag \\ 
    & \leq \qInverseOfDOneOverNMinusDTwo +  \eToMinusDThreeN \normTwo{\grad f(x_0)},  \notag
    \end{split}
\end{equation}
where for the first inequality, we split the integral in the definition of expectation
\begin{align*}
 & \expectation{ \minGradientN}
 = \integral{0}{\infty}{\probability{\minGradientN = x} x dx} \\
 & = \integral{0}{\qInverseOfDOneOverNMinusDTwo}{\probability{\minGradientN = x} x dx} + \integral{\qInverseOfDOneOverNMinusDTwo}{\infty}{\probability{\minGradientN = x} x dx}   
\end{align*}
; and used $\probability{\minGradientN = x} = 0$ for $x> \normTwo{\grad f(x_0)}$ which follows from $\minGradientN \leq \normTwo{\grad f(x_0)}$. For the second inequality,  \\
we used $\probability{\minGradientN \leq \qInverseOfDOneOverNMinusDTwo}\leq 1$ and $\probability{\minGradientN >\qInverseOfDOneOverNMinusDTwo} \leq \eToMinusDThreeN$ by \eqref{eqn:tmp34}.
\end{proof}

    \section{Proof of \autoref{thm2}} \label{BCGN:sec2.5}
    
The proof of \autoref{thm2} involves a technical analysis of different types of iterations. An iteration can be true/false using \autoref{def::true_iters}, successful/unsuccessful (Step 3 of \autoref{alg:generic}) and with an $\alphaK$ above/below a certain value. The parameter $\alphaK$ is important due to \autoref{AA3} and \autoref{AA4} (that is, it influences the success of an iteration; and also the objective decrease in true and successful iterations).

Given that \autoref{alg:generic} runs for $N$ iterations, we use $N$ with different subscripts to denote the total number of different types of iterations, detailed in \autoref{tab:it::count}.
We note that they are all random variables because $\alphaK$, and whether an iteration is true/false, successful/unsuccessful all depend on the random model in Step 1 of \autoref{alg:generic} and the previous (random) steps.


\begin{table}[ht]
\begin{tabular}{ll}
Symbol            & Definition            \\                                    
$\nt$             & Number of true iterations                                         \\
$\nf$             & Number of false iterations                                                                \\
$\nts$            & Number of true and successful iterations                                                  \\
$\ns$             & Number of successful iterations                                                           \\
$\nuMe$             & Number of unsuccessful iterations                                                         \\
$\ntu$            & Number of true and unsuccessful iterations                                                \\
$\ntAlphaUpper$   & Number of true iterations such that $\alpha_k \leq \alphaLowOne$                             \\
$\nsAlphaUpper$   & Number of successful iterations such that $\alpha_k \leq \alphaLowOne$                       \\
$\ntAlphaLower$   & Number of true iterations such that $\alpha_k > \alphaLowOne$                             \\
$\ntsAlphaLower$  & Number of true and successful iterations such that $\alpha_k > \alphaLowOne$              \\
$\ntuAlphaLower$  & Number of true and unsuccessful iterations such that $\alpha_k > \alphaLowOne$            \\
$\nuAlphaLower$   & Number of unsuccessful iterations such that $\alpha_k > \alphaLowOne$                     \\
$\nsGammaAlpha$   & Number of successful iterations such that $\alpha_k > \gamma_1^{c}\alphaLowOne$           \\
$\ntsGammaCAlpha$ & Number of true and successful iterations such that $\alpha_k > \gamma_1^{c}\alphaLowOne$  \\
$\nfsGammaAlpha$  & Number of false and successful iterations such that $\alpha_k > \gamma_1^{c}\alphaLowOne$
\end{tabular}
\caption{List of random variables representing iteration counts given that \autoref{alg:generic} has run for $N$ iterations}
\label{tab:it::count}
\end{table}

The proof of \autoref{thm2} relies on the following three results relating the total number of different types of iterations. 

\paragraph{The relationship between the total number of true iterations and the total number of iterations}

\autoref{lem::Chernoff} shows that with high probability, a constant fraction of iterations of \autoref{alg:generic} are true. This result is a conditional variant of the Chernoff bound \cite{MR57518}.  

\begin{lemma}
\label{lem::Chernoff}
Let \autoref{AA2} hold with $\delta_S \in (0,1)$. Let \autoref{alg:generic} run for $N$ iterations. Then for any given $\delta_1 \in (0,1)$,
\begin{equation}
\P \left( \nt \leq (1-\delta_S)(1-\delta_1)N \right) \leq \chernoffLowerExponential, \label{Markov}
\end{equation}
where $N_T$ is defined in \autoref{tab:it::count}.
\end{lemma}

The proof of \autoref{lem::Chernoff} relies on the below technical result.

\begin{lemma}
    \label{lem::ChernoffHelper}
    Let \autoref{AA2} hold with $\delta_S \in (0,1)$. Let $T_k$ be defined in \eqref{eqn::t_k}.
    Then for any $\lambda>0$ and $N\in \N$, we have
        \begin{equation}
            \expectation{\eToMinusLambdaSumTk} \leq \exponentialMomentUpperTotal. \notag
        \end{equation}
\end{lemma}

\begin{proof}
Let $\lambda >0$. We use induction on $N$. For $N=1$, we want to show
    \begin{equation}
        \expectation{e^{-\lambda T_0}} \leq \exponentialMomentUpper. \notag
    \end{equation}
    
    Let $g(x) = e^{-\lambda x}$. Note that 
    \begin{equation}
        g(x) \leq g(0) + \left[ g(1) - g(0) \right]x, \text{ for any $x\in [0,1]$}, \label{eqn::convex}
    \end{equation}
    because $g(x)$ is convex. Substituting $x=T_0$, we have
    \begin{equation}
        \eToMinusLambda{T_0} \leq 1 + (\eToMinusLambdaMinusOne) T_0. \notag
    \end{equation}
    
    Taking expectation, we have that
    \begin{equation}
        \expectation{\eToMinusLambda{T_0}}  
        \leq 1 + (\eToMinusLambdaMinusOne)\expectation{T_0} . \label{eqn::inductionZeroUpper}
    \end{equation}
    
    Moreover, we have
    \begin{equation}
        \expectation{T_0} \geq \probability{T_0=1} \geq 1-\delta_S, \notag
    \end{equation}
    where the first inequality comes from $T_0\geq 0$ and the second inequality comes from \autoref{AA2}.
    
    Therefore, noting that $\eToMinusLambdaMinusOne <0$, \eqref{eqn::inductionZeroUpper} gives
    \begin{equation}
        \label{eqn::inductionFirstConclusion}
        \expectation{\eToMinusLambda{T_0}} 
        \leq 1 + (\eToMinusLambdaMinusOne) (1-\delta_S) \leq \exponentialMomentUpper, 
    \end{equation}
    where the last inequality comes from $1+y \leq e^y$ for $y \in \R$.
    
    Having completed the initial step for the induction, let us assume
    \begin{equation}
        \label{eqn::inductionAssumption}
        \expectation{  \eToMinusLambdaSumTkNMinusTwo } \leq \left[ \exponentialMomentUpper\right]^{N-1}.
    \end{equation}
    
    We have 
    \begin{align}
        & \expectation{\eToMinusLambdaSumTk} \notag \\
        &= \expectation{\conditionalE{\eToMinusLambdaSumTk}{\SZeroDotsToXNMinusOne}}\notag \\
        &= \expectation{\eToMinusLambda{\sumTkToNMinusTwo} \conditionalE{\eToMinusLambda{T_{N-1}}}{\SZeroDotsToXNMinusOne}} \notag \\
        &= \expectation{ \eToMinusLambda{ \sumTkToNMinusTwo } \conditionalE{ \eToMinusLambda{ T_{N-1} }}{ x_{N-1}}}, \label{eqn::inductionTotal}
    \end{align}
    where the first equality comes from the Tower property and the last equality follows from $T_{N-1}$ is conditionally independent of $T_0, T_1, \dots, T_{N-2}$ given $x_{N-1}$ (see \autoref{AA2}).
    
    Substituting $x=T_{N-1}$ in \eqref{eqn::convex}, and taking conditional expectation, we have that
    \begin{equation}
        \conditionalE{\eToMinusLambda{T_{N-1}}}{x_{N-1}}
        \leq 1 + (\eToMinusLambdaMinusOne) \conditionalE{T_{N-1}}{x_{N-1}}.\notag
    \end{equation}
    
    On the other hand, we have that $\conditionalE{T_{N-1}}{x_{N-1}} \geq \conditionalP{T_{N-1}=1}{x_{N-1}} \geq 1-\delta_S$, where we used $T_{N-1} \geq 0$ to derive the first inequality and $\conditionalP{T_{N-1}=1}{x_{N-1} = \bar{x}_{N-1}} \geq 1 - \delta$ for any $\bar{x}_{N-1}$ (see \autoref{AA2}) to derive the second inequality. 
    Hence
    \begin{equation}
        \conditionalE{\eToMinusLambda{T_{N-1}}}{x_{N-1}}\leq \exponentialMomentUpper, \notag
    \end{equation}
    as in \eqref{eqn::inductionFirstConclusion}.
    
    It then follows from \eqref{eqn::inductionTotal} that
    \begin{equation}
        \expectation{\eToMinusLambdaSumTk} 
        \leq \exponentialMomentUpper \expectation{ \eToMinusLambda{ \sumTkToNMinusTwo }} 
        \leq \exponentialMomentUpperTotal, \notag
    \end{equation}
    where we used \eqref{eqn::inductionAssumption} for the last inequality.
\end{proof}

\begin{proof}[Proof of \autoref{lem::Chernoff}]
Note that with $N$ being the total number of iterations, we have $N_T = \sumTk$, where $T_k$ is defined in \eqref{eqn::t_k}. Applying Markov inequality, we have that for any $\lambda >0$,
\begin{align}
    &\probability{N_T \leq (1-\delta_S)(1-\delta_1) N}
    = \probability{\eToMinusLambda{N_T} \geq \eToMinusLambda{ (1-\delta_S) (1-\delta_1) N }}  \notag \\
    & \leq \expectation{\eToMinusLambda{N_T}} e^{\lambda(1-\delta_S) (1-\delta_1) N}\notag\\
    &= \expectation{ \eToMinusLambdaSumTk} e^{\lambda(1-\delta_S) (1-\delta_1) N}
    \notag \\
    & \leq e^{ N(\eToMinusLambdaMinusOne)(1-\delta_S)+\lambda(1-\delta_S) (1-\delta_1) N }, \label{eqn::Chernoff_proof1}
\end{align}
where we used Lemma \ref{lem::ChernoffHelper} to derive the last inequality. 

Choosing $\lambda = -\log(1-\delta_1) >0$, we have from \eqref{eqn::Chernoff_proof1}
\begin{align}
    \probability{N_T \leq (1-\delta_S)(1-\delta_1) N } 
    &\leq e^{N (1-\delta_S) \left[ \deltaOneComplicated \right]} \notag \\
    &\leq \chernoffLowerExponential, \notag
\end{align}
where we used $\deltaOneComplicated \leq -\delta_1^2/2$ for $\delta_1 \in (0,1)$. 
\end{proof}



\paragraph{The relationship between the total number of true iterations with $\alphaK \leq \alphaMin$ and the total number of iterations}

The next Lemma shows that we can have at most a constant fraction of iterations of \autoref{alg:generic} that are true with $\alphaK \leq \alphaMin$.

\begin{lemma} \label{lm::Gratton}
    Let \autoref{AA3} hold with $\alphaLow>0$ and $c\in \N^+$ and let $\alphaMin$ associated with $\alphaLow$ be defined in \eqref{eqn::alphaMin} with $\newL \in \N^+$. Let $\epsilon>0$, $N \in \N$ be the total number of iterations; and $\ntAlphaUpper$ be defined in \autoref{tab:it::count}. Suppose $N\leq \nEps$. Then 
	\begin{equation}
	    \ntAlphaUpper \leq \frac{N}{c+1}. \label{eqn::tmp8}
	\end{equation}
\end{lemma}

\begin{proof}
Let $k\leq N-1$
\footnote{Note that $k=0,1, \dots, N-1$ if the total number of iterations is $N$.}.
It follows from $N\leq \nEps$ that $k < \nEps$ and by definition of
$\alphaMin$ (\autoref{lem::alphaMin}), iteration $k$ is true with $\alpha_k \leq \alphaMin   $ implies that iteration $k$ is successful (with $\alpha_k \leq \alphaMin$). 

Therefore we have 
\begin{equation}
    \ntAlphaUpper \leq \nsAlphaUpper. \label{eqn::tmp6}
\end{equation}
If $\nsAlphaUpper =0$, then $\ntAlphaUpper=0$ and \eqref{eqn::tmp8} holds. Otherwise let 
\begin{equation}
    \kBar = \max \set{k\leq N-1: \text{iteration $k$ is successful and $\alpha_k \leq \alphaLowOne$}}.\label{eqn::tmp5}
\end{equation}
Then for each $k \in \set{0, 1, \dots, \kBar}$, we have that either iteration $k$ is successful and $\alpha_k \leq \alphaLowOne$, in which case $\alpha_{k+1} = \gamma_2 \alpha_k$ (note that \eqref{eqn::alphaMinUpperByGammaTwoOverAlphaZero} and $\alphaK\leq \alphaMin$ ensure $\max \set{\gammaTwo \alphaK, \alphaMax} = \gammaTwo\alphaK$); or otherwise $\alpha_{k+1} \geq \gamma_1 \alpha_k$ (which is true for any iteration of $\autoref{alg:generic}$). Hence after $\kBar+1$ iterations, we have
\begin{align}
    \alpha_{\kBar+1} \geq \alpha_0 \gamma_2^{\nsAlphaUpper} \gamma_1^{\kBar+1 -\nsAlphaUpper} 
    &= \alpha_0 \left( \frac{\gamma_2}{\gamma_1}\right)^{\nsAlphaUpper}\gamma_1^{\kBar+1} \notag \\
    &\geq \alpha_0 \left( \frac{\gamma_2}{\gamma_1}\right)^{\nsAlphaUpper}\gamma_1^{N}, \label{eqn:tmp29}
\end{align}
where we used $\kBar+1\leq N$ for the last inequality. 
On the other hand, we have 
\begin{equation}
    \alpha_{\kBar+1} = \gamma_2 \alpha_{\kBar} \leq \gamma_2 \alphaLowOne, \notag
\end{equation}
where we used iteration $\kBar$ is successful and $\alpha_{\kBar}\leq \alphaLowOne$ from \eqref{eqn::tmp5}.
Therefore, combining the last displayed equation with \eqref{eqn:tmp29}, we have $\gamma_2\alphaLowOne \geq \alpha_{\kBar+1} \geq \alpha_0 \left( \frac{\gamma_2}{\gamma_1}\right)^{\nsAlphaUpper}\gamma_1^{N}$. Taking logarithm on both sides, we have
\begin{equation}
    \log(\gamma_2 \alphaLowOne) \geq \log(\alpha_0) + \nsAlphaUpper\log(\frac{\gamma_2}{\gamma_1}) + N \log(\gamma_1). \notag
\end{equation}
Rearranging, we have
\begin{equation}
    \nsAlphaUpper \leq p_0 N + p_1, \notag
\end{equation}
with $p_0 = \frac{\logOneOverGammaOne}{\logGammaTwoOverGammaOne} = \frac{1}{c+1}$ and $p_1 = \pOneDef = \frac{c-\newL}{c+1} \leq 0$ as $\newL \geq c>0$. Therefore we have $\nsAlphaUpper \leq \frac{N}{c+1}$ and \eqref{eqn::tmp6} then gives the desired result. 
\end{proof}


\paragraph{The relationship between the number of unsuccessful iterations and the number of successful iterations}

The next Lemma formalises the intuition that one cannot have too many unsuccessful iterations with $\alphaK > \alphaMin$ compared to successful iterations with $\alphaK > \gammaOneC \alphaMin$, because unsuccessful iterations reduce $\alphaK$ and only successful iterations with $\alphaK > \gammaOneC \alphaMin$ may compensate for these decreases. The conditions that $\alphaMin = \alphaZero\gammaOne^\newL$, $\gammaTwo = \frac{1}{\gammaOneC}$ and $\alphaMax = \alphaZero \gammaOne^p$ for some $\newL, c, p \in \N^+$ are crucial in the (technical) proof. 

\begin{lemma}\label{lm::Katya}
	Let \autoref{AA3} hold with $\alphaLow>0$. Let $\alphaMin$ associated with $\alphaLow$ be defined in \eqref{eqn::alphaMin} with $\newL\in \N^+$. Let $N \in \N$ be the total number of iterations of \autoref{alg:generic} and $\nuAlphaLower$, $\nsGammaAlpha$ be defined in \autoref{tab:it::count}. Then
	\begin{equation}
	    \nuAlphaLower \leq \newL + c\nsGammaAlpha. \notag
	\end{equation}

\end{lemma}

\begin{proof}
    Define 
    \begin{equation}
    \betaK = \logBaseGammaOne{\frac{\alphaK}{\alphaZero}}. \label{eqn::betaKDef}
    \end{equation}
    Note that since $\alpha_{k+1} = \gamma_1 \alpha_k$ if iteration $k$ is successful and $\alpha_{k+1} = \minMe{\alphaMax}{\gamma_2 \alphaK}$ otherwise, $\gamma_2 = \gammaTwoExpression$ and $\alpha_{max} = \alpha_0 \gamma_1^p$ with $c,p \in \N^+$, we have that $\betaK \in \Z$. Moreover, we have that $\alphaK=\alphaZero$ corresponds to $\betaK = 0$, $\alphaK=\alphaMin$ corresponds to $\betaK = \newL$ and $\alphaK = \gamma^c \alphaMin$ corresponds to $\betaK= \newL+c$. Note also that on successful iterations, we have $\alphaKPlusOne \leq \gammaTwo \alphaK = \gammaOne^{-c} \alphaK$ (as $\alphaKPlusOne = \minMe{\alphaMax}{\gammaTwo\alphaK}$) so that $\betaKPlusOne \geq \betaK -c$; and on unsuccessful iterations, we have $\betaKPlusOne = \betaK +1$.
	
	Let $\kStartOne=-1$; and define the following sets.
	\begin{align}
	    &\aOne = \set{k \in  \openClosedInterval{\kStartOne}{N-1}\intersect \N: \betaK=\newL}. \label{eqn::tmp11} \\
	    & \kEndOne = \twoCases{\inf \aOne}{\text{if $\aOne \neq \emptyset$}}{N}{\text{otherwise}}  \notag \\
	    & \mOneOne = \set{k \in \openInterval{\kStartOne}{\kEndOne}:  \text{iteration $k$ is unsuccessful with $\betaK < \newL$}} \notag \\
	    & \mTwoOne = \set{k \in \openInterval{\kStartOne}{\kEndOne}:  \text{iteration $k$ is successful with $\betaK < \newL+c $ }} \label{eqn::tmp12}.
	\end{align}
	Let $\nOneOne = |\mOneOne|$ and $\nTwoOne=| \mTwoOne|$, where $|.|$ denotes the cardinality of a set. 
	
	If $\kEndOne <N$, we have that $\kEndOne$ is the first time $\betaK$ reaches $\newL$. Because $\betaK$ starts at $0<\newL$ when $k=0$; $\betaK$ increases by one on unsuccessful iterations and decreases by an integer on successful iterations (so that $\betaK$ remains an integer). So for $\kInStartEndOne$, all iterates have $\betaK < \newL < \newL+c$. It follows then the number of successful/unsuccessful iterations for $\kInStartEndOne$ are precisely $\nOneOne$ and $\nOneTwo$ respectively. Because $\betaK$ decreases by at most $c$ on successful iterations, increases by one on unsuccessful iterations, starts at zero and $\beta_{\kEndOne}\leq \newL$, we have $0 + \nOneOne - c\nTwoOne \leq \newL$ (using $\beta_{\kEndOne} \geq \beta_{\kStartI+1} + \nOneOne - c\nTwoOne$). Rearranging gives
	\begin{equation}
	    \nOneOne \leq c \nTwoOne + \newL. \label{eqn::tmp3}
	\end{equation}
	
	If $\kEndOne = N$, then we have that $\betaK <\newL$ for all $k \leq N-1$ and so $\beta_{\kEndOne} \leq \newL$. In this case we can derive \eqref{eqn::tmp3} using the same argument. Moreover, since $\kEndOne=N$, we have that 
	\begin{align}
	    \nOneOne &= \nuAlphaLower, \label{eqn:tmp30} \\
	    \nOneTwo &= \nsGammaAlpha. \label{eqn:tmp31}
	\end{align}
	The desired result then follows. 
	
	Hence we only need to continue in the case where $\kEndOne < N$, in which case
	let
	\begin{align}
	    \bOne & = \set{k \in \closedInterval{\kEndOne}{N-1}: \text{iteration $k$ is successful with $\betaK < \newL+c $ } } \notag \\
	    \kStartTwo &= \twoCases{\inf \bOne}{\text{if $\bOne \neq \emptyset$}}{N}{\text{otherwise}}. \notag
	\end{align}

	Note that there is no contribution to $\nsGammaAlpha$ or $\nuAlphaLower$ for $k \in \closedOpenInterval{\kEndOne}{\kStartTwo}$. 
	There is no contribution to $\nsGammaAlpha$ because $\kStartTwo$ is the first iteration (if any) that would make this contribution. 
	Moreover, since $\beta_{\kEndOne}=\newL$ by definition of $\kEndOne$, the first iteration with $\betaK<\newL$ for $k\geq \kEndOne$ must be proceeded by a successful iteration with $\betaK< \newL+c$ (note that in particular, since $\kStartTwo$ is the first such iteration, we have $\beta_{\kStartTwo}\geq \newL$). 
	Therefore there is no contribution to $\nuAlphaLower$ either for $k \in \closedOpenInterval{\kEndOne}{\kStartTwo}$. Hence if $\kStartTwo = N$, we have \eqref{eqn:tmp30}, \eqref{eqn:tmp31} and \eqref{eqn::tmp3} gives the desired result. 
	
	Otherwise similarly to \eqref{eqn::tmp11}--\eqref{eqn::tmp12},
	let 
	\begin{align}
	    &\aTwo = \set{k \in \openClosedInterval{\kStartTwo}{N-1} \intersect \N: \betaK=\newL}. \notag \\
	    & \kEndTwo = \twoCases{\inf \aTwo }{\text{if $\aTwo \neq \emptyset$}}{N}{\text{otherwise}} \notag \\
	    & \mOneTwo = \set{k \in \openInterval{\kStartTwo}{\kEndTwo}: \text{iteration $k$ is unsuccessful with $\betaK < \newL$}} \notag \\
	    & \mTwoTwo = \set{k \in \openInterval{\kStartTwo}{\kEndTwo}: \text{iteration $k$ is successful with $\betaK < \newL+c $ }}. \notag
	\end{align}	
	And let $\nOneTwo = |\mOneTwo|$ and $\nTwoTwo=| \mTwoTwo|$. Note that for $k\in \openInterval{\kStartTwo}{\kEndTwo}$, we have $\newL-c \leq \beta_{\kStartTwo+1}$ and $\beta_{\kEndTwo} \leq \newL$ (the former is true as $\beta_{\kStartTwo\geq l}$ and iteration $\kStartTwo$ is successful). Therefore we have
	\begin{equation}
	    \newL-c + \nOneTwo-c\nTwoTwo 
	    \leq \beta_{\kStartTwo+1} + \nOneTwo - c\nTwoTwo \leq \beta_{\kEndTwo} \leq \newL.  \notag
	\end{equation}
	Rearranging gives
	\begin{equation}
	    \nOneTwo \leq c \nTwoTwo+ \newL - [\newL-c] = c\nTwoTwo +c, \label{eqn::tmp4}
	\end{equation}
	
	Let $\hatNOneOne$ be the total number of iterations contributing to $\nuAlphaLower$ with $k \in \closedInterval{\kEndOne}{\kStartTwo}$; and $\hatNOneTwo$ be the total number of iterations contributing to $\nsGammaAlpha$ with $k \in \closedInterval{\kEndOne}{\kStartTwo}$. Since there is no contribution to either for $k\in \closedOpenInterval{\kEndOne}{\kStartTwo}$ as argued before, and iteration $\kStartTwo$ by definition contributes to $\nsGammaAlpha$ by one, we have
	\begin{align}
	    \hatNOneOne = 0, \label{eqn::tmp13}\\
	    \hatNOneTwo = 1.\label{eqn::tmp14}
	\end{align}
	
	Using \eqref{eqn::tmp3}, \eqref{eqn::tmp4}, \eqref{eqn::tmp13} and \eqref{eqn::tmp14},we have
	\begin{equation}
	    \nOneOne+ \hatNOneOne + \nOneTwo \leq c \left( \nTwoOne + \hatNOneTwo + \nTwoTwo \right) +\newL. \label{eqn::tmp15}
	\end{equation}
	If $\kEndTwo = N$ the desired result follows. 
	Otherwise define $\bTwo$ in terms of $\kEndTwo$, 
	and $\kStartThree$ in terms of $\bTwo$ similarly as before. 
	If $\kStartThree=N$, 
	then we have the desired result as before. 
	Otherwise repeat what we have done (define $A^{(3)}$,
	$k_{end}^{(3)}$, $M_1^{(3)}$, $M_2^{(3)}$ etc). 
	Note that we will reach either $\kEndI = N$ 
	for some $i\in \N$ 
	or $\kStartI = N$ for some $i \in \N$, 
	because if $\kEndI<N$ and $\kStartI <N$ for all $i$, 
	we have that $\kStartI < \kEndI \leq \kStartIPlusOne$ 
	by definitions. So $\kStartI$ is strictly increasing,
	contradicting $\kStartI <N$ for all $i$. 
	In the case wither $\kEndI=N$ or $\kStartI=N$, 
	the desired result will follow using our previous argument.
\end{proof}

\paragraph{An intermediate result bounding the total number of iterations}
With \autoref{lem::Chernoff}, \autoref{lm::Gratton}, \autoref{lm::Katya}, we show a bound on the total number of iterations of \autoref{alg:generic} in terms of the number of true and successful iterations with $\alphaK$ above a certain constant.

\begin{lemma}\label{lem:AssOneTwo}
    Let \autoref{AA2} and \autoref{AA3} hold with 
    $ \deltaS \in (0,1)$, $c, \newL \in \N^+$. 
    Let $N$ be the total number of iterations. 
    Then for any $\deltaOne \in (0,1)$ such that
    $\gDeltaSDeltaOne >0 $, we have that 
    $
        \probability{N < \gDeltaSDeltaOne \squareBracket{
         \ntsAlphaZeroGammaOneCL
         + \frac{\newL}{1+c}}} \geq 
         1- \chernoffLowerExponential
    $
    where $\gDeltaSDeltaOne$ is defined in \eqref{eqn:gDeltaSDeltaOneDef}.
\end{lemma}

\begin{proof}
We decompose the number of true iterations as 
\begin{equation}
    \nt = \ntAlphaUpper + \ntAlphaLower = \ntAlphaUpper + \ntsAlphaLower + \ntuAlphaLower \leq \ntAlphaUpper + \ntsAlphaLower + \nuAlphaLower, \label{eqn::tmp7}
\end{equation}
where $\nt, \ntAlphaUpper, \ntAlphaLower, \ntsAlphaLower, \ntuAlphaLower, \nuAlphaLower$ are defined in \autoref{tab:it::count}.

From \autoref{lm::Katya}, we have
\begin{align}
    \nuAlphaLower & \leq \tCOne + \tCTwo \nsGammaAlpha \notag \\
    & = \tCOne + \tCTwo\ntsGammaCAlpha + \tCTwo\nfsGammaAlpha \notag \\
    & \leq \tCOne +\tCTwo \ntsGammaCAlpha + \tCTwo\nf \notag \\
    & \leq \tCOne + \tCTwo \ntsGammaCAlpha + \tCTwo(N-\nt), \notag
\end{align}
It then follows from \eqref{eqn::tmp7} that
\begin{equation}
    \nt \leq \ntAlphaUpper + \ntsAlphaLower + \tCOne + \tCTwo \ntsGammaCAlpha + \tCTwo(N-\nt). \notag
\end{equation}
Rearranging, we have
\begin{equation}
    \nt \leq \frac{\ntAlphaUpper}{\onePlusTCTwo}+ \frac{1}{\onePlusTCTwo}\squareBracket{\ntsAlphaLower + \tCTwo\ntsGammaCAlpha} + \frac{\tCOne+\tCTwo N}{\onePlusTCTwo}. \notag
\end{equation}
Using \autoref{lm::Gratton} to bound $\ntAlphaUpper$;  $\ntsAlphaLower \leq \ntsGammaCAlpha$; and $ \alphaMin = \alphaZero \gammaOne^\newL$ gives 

\begin{equation}
    N_T \leq 
    \squareBracket{1 - \frac{c}{(c+1)^2}}N
    + N_{TS, \underline{\alphaZero \gammaOne^{c+\newL}}}
    + \frac{\newL}{1+c}. \label{tmp:2021-12-31-2}
\end{equation}

Combining with \autoref{lem::Chernoff}; and
rearranging gives the result. 
\end{proof}



\paragraph{The bound on true and successful iterations}

The next lemma bounds the total number of true and successful iterations with $\alphaK > \alphaZero \gammaOne^{c+\newL}$.

\begin{lemma} \label{lm::bound_T_S_with_artificial_alpha_low}
Let \autoref{AA4} and \autoref{AA5} hold.
Let $\epsilon>0$ and $N \in \N$ 
be defined in \autoref{tab:it::count}. Suppose $N \leq \nEps$. Then 

\begin{align}
 N_{TS, \underline{\alphaZero \gammaOne^{c+\newL}}} \leq \frac{f(x_0) - f^*}{
h(\epsilon, \alphaZero \gammaOne^{c+\newL})} \notag
\end{align}
where $f^*$ is defined in \eqref{eqn::fStar}, and $x_0$ is chosen in the initialization of \autoref{alg:generic}.

\end{lemma}

\begin{proof}
We have, using \autoref{AA5} and \autoref{AA4} respectively for the two inequalities
\begin{align}
    f(x_0) - f(x_{N})
    &= \sum_{k=0}^{N-1} f(\xK) - f(\xKPlusOne) \notag \\
    &\geq \sum_{\IterKTrueandSuccssfulWithAlphaKGeqAlphaMin} f(x_k) - f(\xKPlusOne)  \notag \\
    &\geq \sum_{\IterKTrueandSuccssfulWithAlphaKGeqAlphaMin} h(\epsilon, \alphaZero \gammaOne^{c+\newL}) \notag \\
    &= N_{TS, \underline{\alphaZero \gammaOne^{c+\newL}}} h(\epsilon, \alphaZero \gammaOne^{c+\newL}).\label{eqn::tmp1}
\end{align}
Noting $f(x_{N}) \geq f^*$ and $h(\epsilon, \alphaZero \gammaOne^{c+\newL})>0$ 
by \autoref{AA4}, rearranging \eqref{eqn::tmp1} gives the result.
\end{proof}

\paragraph{The final proof}

We are ready to prove \autoref{thm2} using \autoref{lem:AssOneTwo} and \autoref{lm::bound_T_S_with_artificial_alpha_low}.

\begin{proof}[Proof of \autoref{thm2}]
We have
\begin{align}
    \nEps \geq N & 
    \implies
    \fZeroMinusfStarOverH
    \geq \ntsAlphaZeroGammaOneCL
    \texteq{by 
    \autoref{lm::bound_T_S_with_artificial_alpha_low} }\\
    & \impliesSince{\eqref{eqn::n_upper_2}}
    N \geq \gDeltaSDeltaOne \squareBracket{
         \ntsAlphaZeroGammaOneCL
         + \frac{\newL}{1+c}}.
\end{align}
Therefore by \autoref{lem:AssOneTwo}, we have
$\probability{\nEps \geq N} \leq 
\probability{N \geq \gDeltaSDeltaOne \squareBracket{
         \ntsAlphaZeroGammaOneCL
         + \frac{\newL}{1+c}}}
         \leq \chernoffLowerExponential$.

\end{proof}

    \section{An algorithmic framework based on sketching}
    \label{BCGN:sec3}
    
\subsection{A generic random subspace method based on sketching}
\autoref{alg:sketching} particularises \autoref{alg:generic} by specifying the local reduced model as one generated by sketching using a random matrix; the step transformation function; and the criterion for sufficient decrease. We leave specification of the computation of the step parameter to the next section. 

\begin{algorithm}[H]
\begin{description}

 \item[Initialization] \ \\
 Choose a matrix distribution $\cal{S}$ of matrices $S\in \rLTimesD$. Let $\gamma_1, \gamma_2, \theta, \alphaMax, x_0, \alpha_0$ be defined in \autoref{alg:generic} with $\mKHat{\sHat}$ and $w_k$ specified below in \eqref{eqn::mKHatSpec} and \eqref{eqn::wKSpec}. 

 \item[1. Compute a reduced model and a step] \ \\
 In Step 1 of \autoref{alg:generic}, draw a random matrix $S_k \in \R^{l \times d}$ from $\cal{S}$, and let
 \begin{align}
     &\mKHat{\sHat} = \fK + \innerProduct{\sKGradFK}{\sHat} + \frac{1}{2} \innerProduct{\sHat}{\sKBKSKT \sHat}; \label{eqn::mKHatSpec} \\
     &w_k(\sHat_k) = S_k^T \sHat_k, \label{eqn::wKSpec}
 \end{align}
 where $B_k \in \R^{d\times d}$ is a user provided matrix. 
 
 Compute $\sKHat$ by approximately minimising $\mKHat{\sHat}$ such that at least $\mKHat{\sKHat} \leq \mKHat{0}$\footnote{what exactly this approximate minimisation entails will be the subject of the next section} where $\alphaK$ appears as a parameter,
 and set $\sK = w_k(\sKHat)$ as in \autoref{alg:generic}.

\item[2. Check sufficient decrease]\ \\  
In Step 2 of \autoref{alg:generic}, let sufficient decrease be defined by the condition
\begin{equation}
    \fK - \fKPlusOne \geq \theta \squareBracket{\mKHat{0} - \mKHat{\hat{\sK}(\alpha_k)}}. \label{eqn::sufDecreaseSpec}
\end{equation}

\item[3. Update the parameter $\alphaK$ and possibly take the potential step $\sK$]\ \\
Follow Step 3 of \autoref{alg:generic}.

\caption{\bf{A generic random subspace method based on sketching}} \label{alg:sketching} 
\end{description}
\end{algorithm}

With the concrete criterion for sufficient decrease, we have that \autoref{AA5} is satisfied by \autoref{alg:sketching}.

\begin{lemma} \label{lem:AA8impliesAA5}
\autoref{alg:sketching} satisfies \autoref{AA5}.
\end{lemma}
\begin{proof}
If iteration $k$ is successful, \eqref{eqn::sufDecreaseSpec} with $\theta \geq 0$ and $\mKHat{\sKHat} \leq \mKHat{0}$ (specified in \autoref{alg:sketching}) give $\fK-\fKPlusOne \geq 0$. If iteration $k$ is unsuccessful, we have $s_k=0$ and therefore $\fK - \fKPlusOne = 0$. 
\end{proof}

Next, we define what a true iteration is for \autoref{alg:sketching} and show \autoref{AA2} is satisfied with $\cal{S}$ being a variety of random ensembles. 


\begin{definition} \label{def::true_iters}
Iteration $k$ is a true iteration if
\begin{align}
    &\normTwo{\sKGradFK}^2 \geq (1- \epS)\normTwo{\gradFK}^2, \label{eqn::JL} \\
    & \normTwo{S_k}\leq \sMax, \label{eqn::sMax}
\end{align}
where $S_k \in \R^{l \times d}$ is the random matrix drawn in Step 1 of \autoref{alg:sketching}, and $\epS \in (0,1), \sMax>0$ are iteration-independent constants.
\end{definition}

\begin{remark}
In \cite{Cartis:2017fa}, true iterations are required to satisfy 
\begin{equation}
    \normTwo{\grad m_k(0) - \gradFK} \leq \kappa \alphaK \normTwo{\grad m_k(0)}, \notag
\end{equation}
where $\kappa>0$ is a constant and $\alphaK$ in their algorithm is bounded by $\alphaMax$. The above equation implies
\begin{equation}
    \normTwo{\grad m_k(0) } \geq \frac{\normTwo{\gradFK}}{1+\kappa \alphaMax}, \notag
\end{equation}
which implies \eqref{eqn::JL} with $1-\epS = \frac{1}{1 + \kappa\alphaMax}$ and $\deltaSOne = p$. Since \autoref{AA7} is easily satisfied for a variety of random matrix distributions $\cal{S}$ we see that their requirement is stronger than our (main) requirement for true iterations.
\end{remark}

We first show that with this definition of the true iterations, \autoref{AA2} holds if the following two conditions on the random matrix distribution $\cal{S}$ are met. 

\begin{assumption} \label{AA6}
There exists $\epS, \deltaSOne \in (0,1)$ such that for a(ny) fixed
$y \in \set{\grad f(x): x \in \R^d}$, $S_k$ drawn from $\cal{S}$ satisfies
\begin{equation}
    \probability{\normTwo{S_k y}^2 \geq (1-\epS)\normTwo{y}^2} 
    \geq 1-\deltaSOne. \label{eqn:tmp36}
\end{equation}
\end{assumption}

\begin{assumption} \label{AA7}
There exists $\deltaSTwo\in [0,1), \sMax>0$ such that for $S_k$ randomly 
drawn from $\cal{S}$, we have
\begin{equation}
    \probability{\normTwo{S_k} \leq \sMax} \geq 1-\deltaSTwo. \notag
\end{equation}
\end{assumption}

\begin{lemma} \label{lem::deduceAA2}
Let \autoref{AA6} and \autoref{AA7} hold with $\epS,\deltaSTwo \in (0,1), \deltaSOne\in [0,1), \sMax >0$. Suppose that $\deltaSOne+\deltaSTwo<1$. Let true iterations be defined in \autoref{def::true_iters}. Then \autoref{alg:sketching} satisfies \autoref{AA2} with $\delta_S = \deltaSOne + \deltaSTwo$.
\end{lemma}

The proof of \autoref{lem::deduceAA2} makes use of the following elementary result in probability theory, whose proof is included for completeness.

\begin{lemma}
\label{lem:union_bound}
Let $ n \in \N^+$ and $A_1, A_2 \dots, A_n$ be events. Then we have 
\begin{equation}
    \probability{A_1 \intersect A_2 \dots \intersect A_n} = 1- \probability{A_1^c} - \probability{ A_2^c - \dots - \probability{A_n^c}}. \notag
\end{equation}
\end{lemma}

\begin{proof}
We have 
\begin{align*}
    \probability{A_1 \intersect A_2 \dots \intersect A_n} 
    &\reply{ = } 1 - \probability{ \complement{A_1 \intersect \dots \intersect A_n}} \\
    &= 1 - \probability{A_1^c \union{} \dots \union{} A_n^c} \\
    & \geq 1 - \sum_{k=1}^n \probability{A_k^c}.
\end{align*}
\end{proof}

\begin{proof}[Proof of \autoref{lem::deduceAA2}]
Let $\barXK \in \R^d$ be given. Note that this determines $\grad f(\barXK) \in \R^d$. Let $\aKOne$ be the event that \eqref{eqn::JL} hold and $\aKTwo$ be the event that \eqref{eqn::sMax} hold. Thus $T_k = \aKOne \intersect \aKTwo$. Note that given $x_k = \barXK$, $T_k$ only depends on $S_k$, which is independent of all previous iterations. Hence $T_k$ is conditionally independent of $T_0, T_1 \dots, T_{k-1}$ given $x_k = \barXK$. 

Next, we have for $k \geq 1$, 
\begin{equation}
    \probabilityGivenXK{\aKOne \intersect \aKTwo } \geq 1 - \probabilityGivenXK{\complement{\aKOne}} - \probabilityGivenXK{\complement{\aKTwo}}, \label{eqn::tmp17}
\end{equation}
by \autoref{lem:union_bound}.

Note that 
\begin{align}
    \probabilityGivenXK{\aKOne} &= \conditionalP{\aKOne}{x_k = \barXK, \gradFK = \grad f(\barXK)} \notag \\
    & = \conditionalP{\aKOne}{\gradFK = \grad f(\barXK)} \notag \\
    & \geq 1-\deltaSOne,\label{eqn::tmp18}
\end{align}
where the first equality follows from the fact that $\xK = \barXK$ implies $\gradFK = \grad f(\barXK)$; the second equality follows from the fact that given $\gradFK = \grad f(\barXK)$, $\aKOne$ is independent of $\xK$; and the inequality follows from applying \autoref{AA6} with $y = \grad f(\barXK)$.

On the other hand, because $\aKTwo$ is independent of $x_k$, we have that
\begin{equation}
    \probabilityGivenXK{\aKTwo} = \probability{\aKTwo} \geq 1-\deltaSTwo,\label{eqn::tmp19}
\end{equation}
where the inequality follows from \autoref{AA7}.
It follows from \eqref{eqn::tmp17} using \eqref{eqn::tmp18} and \eqref{eqn::tmp19} that for $k\geq 1$, 
\begin{equation}
    \probabilityGivenXK{\aKOneIntersectaKTwo} \geq 1-\deltaSOne - \deltaSTwo = 1-\delta_S. \notag
\end{equation}
For $k=0$, we have $\probability{\aZeroOne} \geq 1-\deltaSOne$ by \autoref{AA6} with $y=\grad f(x_0)$ and $ \probability{\aZeroTwo} \geq 1-\deltaSTwo $ by \autoref{AA7}. So $\probability{\aZeroOne \intersect \aZeroTwo} \geq 1-\delta_S$ by \autoref{lem:union_bound}. 

\end{proof}

Next, we give four distributions $\cal{S}$ that satisfy \autoref{AA6} and \autoref{AA7}, thus
satisfying \autoref{AA2} and can be used in \autoref{alg:sketching}. Other random ensembles are possible, for example, Subsampled Randomised Hadamard Transform (\autoref{def::SRHT}), Hashed Randomised Hadamard Transform (\autoref{def::HRHT}), and many more (see discussion of random ensembles in Chapter 2).

\subsection{The random matrix distribution $\cal{S}$ in \autoref{alg:sketching}}\label{BCGN:randomMatrixDistr}

\subsubsection{Gaussian sketching matrices}

(Scaled) Gaussian matrices have independent and identically distributed normal entries (see \autoref{def:Gaussian}). The next result, which is a consequence of the scaled Gaussian matrices being an oblivious JL embedding (\autoref{Oblivious_embedding}), shows that using scaled Gaussian matrices with \autoref{alg:sketching} satisfies \autoref{AA6}. The proof is included for completeness but can also be found in \cite{MR1943859}. 

\begin{lemma} \label{lem:GaussJLEmbedding}
Let $S\in \R^{l \times d}$ be a scaled Gaussian matrix so that each entry is $N(0, l^{-1})$. Then $S$ satisfies \autoref{AA6} with any $\epS \in (0,1) $ and $\deltaSOne = e^{-\epS^2 l /4}$. 
\end{lemma}

\begin{proof}

Since \eqref{eqn:tmp36} is invariant to the scaling of $y$ and is trivial for $y=0$, 
we may assume without loss of generality that $\normTwo{y}=1$. 

Let $R = \sqrt{l}S$, so that each entry of $R$ is distributed independently as $N(0,1)$. 
Then because the sum of independent Gaussian random variables is distributed as a Gaussian random variable; $\normTwo{y}=1$; and the fact that rows of $S$ are independent; we have that the entries of $Ry$, denoted by $z_i$ for $i\in [l]$, are independent $N(0,1)$ random variables. Therefore, for any $-\infty < q < \frac{1}{2}$, we have that
\begin{equation}
    \expectation{e^{q \normTwo{Ry}^2}} = \expectation{e^{q \sum_{i=1}^l z_i^2}} = \prod_{i=1}^l \expectation{e^{qz_i^2}} = (1-2q)^{-l/2}, \label{eqn:tmp21}
\end{equation}
where we used $\expectation{e^{qz_i^2}} = \frac{1}{1-2q}$ for $z_i \in N(0,1)$ and $-\infty < q < \frac{1}{2}$.

Hence, by Markov inequality, we have that, for $q <0$, 
\begin{equation}
    \probability{\normTwo{Ry}^2 \leq l (1-\epS)} = \probability{e^{q \normTwo{Ry}^2}\geq e^{ql(1-\epS)}} 
    \leq \frac{\expectation{\eToQRySqaured}}{\eToQLOneMinusEps} = (1-2q)^{-l/2} \eToMinusQLOneMinusEps, \label{eqn:tmp23}
\end{equation}
where the last inequality comes from \eqref{eqn:tmp21}.

Noting that 
\begin{equation}
    (1-2q)^{-l/2} \eToMinusQLOneMinusEps = \exp \squareBracket{-l \bracket{\frac{1}{2} \log(1-2q)+ q(1-\epS)}}, \label{eqn:tmp22}
\end{equation}
which is minimised at $q_0 = -\frac{\epS}{2(1-\epS)}<0$, we choose $q=q_0$ and 
the right hand side of \eqref{eqn:tmp23} becomes 
\begin{equation}
    e^{\frac{1}{2}l\squareBracket{\epS + \log(1-\epS)}} \leq e^{-\frac{1}{4}l\epS^2}, \label{eqn:tmp24}
\end{equation}
where we used $\log(1-x) \leq -x -x^2/2$, valid for all $x \in [0,1)$.

Hence we have
\begin{align}
      &\probability{\normTwo{Sy}^2 
    \leq (1-\epS) \normTwo{y}^2} \notag \\
    &= \probability{\normTwo{Sy}^2 \leq (1-\epS)} \texteq{by $\normTwo{y}=1$} \notag \\
    &= \probability{\normTwo{Ry}^2 \leq l(1-\epS)} \texteq{by $S = \frac{1}{\sqrt{l}}R$} \notag \\
    & \leq e^{-\frac{l\epS^2}{4}} \texteq{by \eqref{eqn:tmp24} and \eqref{eqn:tmp23}}. \notag
\end{align}

\end{proof}

In order to show using scaled Gaussian matrices satisfies \autoref{AA7}, we make use of the following bound on the maximal singular value of scaled Gaussian matrices.




\begin{lemma}[\cite{MR1863696} Theorem 2.13] \label{lem:Davidson}
Given $l,d \in \N$ with $l \leq d$, consider the $d\times l$ matrix $\Gamma$ whose entries are independent $N(0, {d}^{-1})$. Then for any $\delta >0$,\footnote{We set $t = \sqrt{\frac{2 \logOneOverDelta}{d} }$ in the original theorem statement.}
\begin{equation}
    \probability{\sigma_{max} \bracket{\Gamma} \geq 1 + \sqrt{\frac{l}{d}} + \sqrt{\frac{2 \logOneOverDelta}{d}} } < \delta, \label{eqn:Davidson_upper_Gaussian}
\end{equation}
where $\sigma_{max}(.)$ denotes the largest singular value of its matrix argument.
\end{lemma}

The next lemma shows that using scaled Gaussian matrices satisfies \autoref{AA7}.
\begin{lemma}\label{Lem:GaussSMax}
Let $S \in \R^{l\times d}$ be a scaled Gaussian matrix. Then $S$ satisfies \autoref{AA7} with any $\deltaSTwo \in (0,1)$ and 
\begin{equation}
    \sMax = 1 + \sqrt{\frac{d}{l}} + \sqrt{\frac{2\logOneOverDeltaSTwo}{l}}. \notag
\end{equation}
\end{lemma}

\begin{proof}
We have $\normTwo{S} = \normTwo{S^T} = \sqrt{\frac{d}{l}} \normTwo{\sqrt{\frac{l}{d}}S^T}$. Applying \autoref{lem:Davidson} with $\Gamma = \sqrt{\frac{l}{d}}S^T$, we have that 
\begin{equation*}
    \probability{\sigma_{max} \bracket{ \sqrt{\frac{l}{d}}S^T} \geq 1 + \sqrt{\frac{l}{d}} + \sqrt{\frac{2 \logOneOverDeltaSTwo}{d}} }
    < \deltaSTwo.
\end{equation*}
Noting that $\normTwo{S} = \sqrtFrac{d}{l} \sigma_{max} \bracket{\Gamma}$, and taking the event complement gives the result.
\end{proof}


\subsubsection{\texorpdfstring{$s$-hashing}{TEXT} matrices}
Comparing to Gaussian matrices, $s$-hashing matrices, including the $s=1$ case, (defined in \autoref{def::sampling_and_hashing}) 
are sparse so that it preserves the sparsity (if any) of the vector/matrix it acts on; 
and the corresponding linear algebra computation is faster.
The next two lemmas show that using $s$-hashing matrices satisfies \autoref{AA6} and \autoref{AA7}.



\begin{lemma}[\cite{MR3167920} Theorem 13, also see \cite{MR3773205} Theorem 5 for a simpler proof] \label{thm:s_hashing}
Let $S \in \rLTimesD$ be an $s$-hashing matrix. 
Then $S$ satisfies \autoref{AA6} for any $\epS \in (0,1)$ 
and 
$\deltaSOne = e^{-\frac{l\epS^2}{C_1}}$ given that $s = C_2 \epS l$.
where $C_1, C_2$ are problem-independent constants. 
\end{lemma}

\begin{lemma}
Let $S \in \rLTimesD$ be an $s$-hashing matrix. Then $S$ satisfies \autoref{AA7} with $\deltaSTwo = 0$
and $\sMax = \sqrtFrac{d}{s}$.
\end{lemma}
\begin{proof}
Note that for any matrix $A\in\rLTimesD$, $\normTwo{A} \leq \sqrt{d} \normInf{A}$; and $\normInf{S} = \frac{1}{\sqrt{s}}$. The result follows from combining these two facts. 
\end{proof}



\subsubsection{(Stable) $1$-hashing matrices}
In \cite{CHEN2020105639}, a variant of $1$-hashing matrix is proposed that satisfies \autoref{AA6} but with better $\sMax$ bound. The construction is given as follows.

\begin{definition}\label{def::stable-1-hashing}
\reply{Let $l<d \in \N^+$. A stable $1$-hashing matrix $S \in \R^{l \times d}$ has one non-zero per column, whose value is $\pm 1$ with equal probability, with the row indices of the non-zeros given by the sequence $I$ constructed as the following.  Repeat $[l]$ (that is, the set $\set{1,2, \dots, l}$) for $\ceil{d/l}$ times to obtain a set $D$. Then randomly sample $d$ elements from $D$ without replacement to construct sequence $I$. \footnote{One may also conceptually think $S$ as being constructed from taking the first $d$ columns of a random column permutation of the matrix $T = \squareBracket{I_{l\times l}, I_{l \times l}, \dots, I_{l \times l}}$ where the identify matrix $I_{l \times l}$ is concatenated by columns $\ceil{d/l}$ times.}}
\end{definition}
\begin{remark}
Comparing to a $1$-hashing matrix, a stable $1$-hashing matrix still has $1$ non-zero per column. However its construction guarantees that each row has at most $\ceil{d/l}$ non-zeros because the set $D$ has at most $\ceil{d/l}$ repeated row indices and the sampling is done without replacement. 
\end{remark}

In order to show using stable $1$-hashing matrices satisfies \autoref{AA6}, we need to following result from \cite{CHEN2020105639}.
\begin{lemma}[Theorem 5.3 in \cite{CHEN2020105639}] \label{tmp-2021-12-31-3}
The matrix $S \in \R^{l\times d}$ defined in \autoref{def::stable-1-hashing} satisfies the following: given 
$0< \epsilon, \delta < 1/2$, there exists 
$l = \mathO{\frac{\log(1/\delta)}{\epsilon^2}}$ 
such that for any $x\in \R^d$, we have that 
$\probability{ \|Sx\|_2 \geq (1 - \epsilon) \|x\|_2} > 1-\delta$.
\end{lemma}

\begin{lemma}\label{lem:tmp:2022-1-13-1}
Let $S \in \R^{l \times d}$ be a stable $1$-hashing matrix. Let $\epS \in (0,3/4)$ and suppose that $e^{-\frac{l(\epS-1/4)^2}{C_3}} \in (0, 1/2)$, where $C_3$ is a problem-independent constant. Then $S$ satisfies \autoref{AA6} with $
\deltaSOne = e^{-\frac{l(\epS-1/4)^2}{C_3}}$. 
\end{lemma}
\begin{proof}
Let $\Bepsilon = \epS - 1/4 \in (0, 1/2)$.
From \autoref{tmp-2021-12-31-3}, we have that there exists $C_3>0$ such that with $\deltaSOne = e^{-\frac{l(\epS-1/4)^2}{C_3}}$, $S$ satisfies
$\probability{ \|Sx\|_2 \geq (1 - \Bepsilon) \|x\|_2} > 1-\deltaSOne.$ 
Note that $\|Sx\|_2 \geq (1-\Bepsilon) \|x\|_2$ implies 
$\|Sx\|_2^2 \geq (1 - 2\Bepsilon + \Bepsilon^2) \|x\|_2$, which implies
$\|Sx\|_2^2 \geq (1-\Bepsilon-1/4)\|x\|_2^2$ because $\Bepsilon^2-\Bepsilon \geq -1/4$ for $\Bepsilon \in (0,1/2)$. The desired result follows.
\end{proof}



The next lemma shows that using stable $1$-hashing matrices satisfies \autoref{AA7}. Note that the bound $\sMax$ is smaller than that for $1$-hashing matrices; and, assuming $l>s$, smaller than that for $s$-hashing matrices as well. 
\begin{lemma}\label{lem:stable-1-hashing-SMax}
Let $S \in \R^{l\times d}$ be a stable $1$-hashing matrix. Then $S$ satisfies \autoref{AA7} with $\deltaSTwo=0$ and $\sMax = \sqrt{\ceil{d/l}}$.
\end{lemma}
\begin{proof}
Let $D$ be defined in \autoref{def::stable-1-hashing}. we have that

\begin{align}
\normTwo{Sx}^2 
&= (\sum_{1\leq j \leq d, I(j)=1} \pm x_j)^2 
+ (\sum_{1\leq j \leq d, I(j)=2} \pm x_j)^2
+ \dots + (\sum_{1\leq j \leq d, I(j)=l} \pm x_j)^2 \\
& \leq 
(\sum_{1\leq j \leq d, I(j)=1}  |x_j|)^2 
+ (\sum_{1\leq j \leq d, I(j)=2} |x_j|)^2
+ \dots + (\sum_{1\leq j \leq d, I(j)=l} |x_j|)^2 \\
&\leq 
\ceil{d/l} \bracket{
\sum_{1\leq j \leq d, I(j)=1}  x_j^2
+ \sum_{1\leq j \leq d, I(j)=2}  x_j^2
+ \dots + \sum_{1\leq j \leq d, I(j)=l}  x_j^2
}  \\
& = \ceil{d/l} \|x\|_2,
\end{align}
where the $\pm$ on the first line results from the non-zero entries of $S$ having random signs, and the last inequality is because for any vector $v \in \R^n$, 
$\|v\|_1^2\leq  n \|v\|_2^2$; and $I(j) = k$ is true for at most $\ceil{d/l}$ indices $j$.
\end{proof}

\subsubsection{Sampling matrices}\label{sampling_mat_paragraph}

(Scaled) Sampling matrices $S\in\R^{l\times d}$ (defined in \autoref{def:sampling}) randomly select rows of vector/matrix it acts on (and scale it). 
Next we show that sampling matrices satisfy \autoref{AA6}. The following expression that represents the maximum non-uniformity (see \autoref{def:non-uniform-vector}) of the objective gradient will be used
\begin{equation}
    \nu = \max \set{\frac{\normInf{y}}{\normTwo{y}}, y=\grad f(x) \text{ for some } x\in \R^d}. \label{eq:nu_def}
\end{equation}

The following concentration result will be useful.

\begin{lemma}[\cite{MR2946459}]\label{lem:Tropp_matrix_chernoff}
Consider a finite sequence of independent random numbers $\set{X_k}$ that 
satisfies $X_k \geq 0$ and $\abs{X_k} \leq P$ almost surely. 
Let $\mu = \sum_k \expectation{X_k}  $, then $\probability{\sum_k X_k \leq (1-\epsilon) \mu} \leq 
e^{-\frac{\epsilon^2 \mu}{2P}}$. 
\end{lemma}

\begin{lemma} \label{lem:sampling:non-uniformity-BCGN}
Let $S \in \rLTimesD$ be a scaled sampling matrix. 
Let $\nu$ be defined in \eqref{eq:nu_def}.
Then $S$ satisfies \autoref{AA6} for any $\epS \in (0,1)$
with $\deltaSOne = e^{- \frac{\epS^2 l}{2d\nu^2}}$.
\end{lemma}

\begin{proof}
Note that \eqref{eqn:tmp36} is invariant to scaling of $y$
and trivial for $y=0$. Therefore we may assume $\normTwo{y}=1$ 
without loss of generality. 

We have $\normTwo{Sy} = \frac{l}{d} \sum_{k=1}^l
\squareBracket{\bracket{Ry}_k}^2 $
where $R \in \rLTimesD$ is an (un-scaled) sampling matrix \footnote{I.e. each row of $R$ has a one at a random column.} and $(Ry)_k$ denotes 
the $k^{th}$ entry of $Ry$.
Let $X_k = \squareBracket{\bracket{Ry}_k}^2$.
Note that because the rows of $R$ are independent, 
$X_k$ are independent. 
Moreover, because $\bracket{Ry}_k$ equals to some entry of $y$,
and $\normInf{y}\leq \nu$ by definition of $\nu$ and 
$\normTwo{y} = 1$; we have $\squareBracket{\bracket{Ry}_k}^2 \leq \nu^2$.
Finally, note that $\expectation{X_k} = \frac{1}{d}\normTwo{y}^2 = 
\frac{1}{d}$; so that $\sum_k \expectation{X_k} = \frac{l}{d}$.

Therefore applying \autoref{lem:Tropp_matrix_chernoff} 
with $\epsilon = \epS$ we have
\begin{equation}
    \probability{\sum_{k=1}^l \squareBracket{\bracket{Ry}_k}^2
    \leq (1-\epS) \frac{l}{d}} \geq e^{- \frac{\epS^2 l}{2d\nu^2}}. \notag
\end{equation}
Using $\normTwo{Sy}^2 =\frac{l}{d} \sum_{k=1}^l 
\squareBracket{\bracket{Ry}_k}^2$ gives the result. 
\end{proof}


We note that the theoretical property for scaled sampling matrices is different to Gaussian/$s$-hashing matrices in the sense that the required value of $l$ depends on $\nu$. Note that $\frac{1}{d} \leq \nu^2 \leq 1$ with both bounds attainable. Therefore in the worst case, for fixed value of $\epS, \deltaSOne$, $l$ is required to be $\mathO{d}$ and no dimensionality reduction is achieved by sketching. This is not surprising given that sampling based random methods often require adaptively increasing the sampling size for convergence (reference). However note that for \singleQuote{nice} objective functions such that $\nu^2 = \mathO{\frac{1}{d}}$, sampling matrices have the same theoretical property as Gaussian/$s$-hashing matrices. The attractiveness of sampling lies in the fact that only a subset of entries of the gradient need to be evaluated.


Sampling matrices also have bounded 2-norms, thus \autoref{AA7} is satisfied.  
\begin{lemma}\label{tmp-2022-1-14-12}
Let $S \in \rLTimesD$ be a scaled sampling matrix. Then \autoref{AA7} is satisfied with $\deltaSTwo=0$ and $\sMax=\sqrt{\frac{d}{l}}$.
\end{lemma}
\begin{proof}
We have that $\normTwo{Sx}^2 \leq \frac{d}{l}\normTwo{x}^2$ for any $x\in \R^d$.
\end{proof}

We summarises this section in \autoref{tab:alg:sketching}, where we also give $l$ in terms of $\epS$ and $\deltaSOne$ by rearranging the expressions for $\deltaSOne$. Note that for $s$-hashing matrices, $s$ is required to be $C_2 \epS l$ (see \autoref{thm:s_hashing}), while for scaled sampling matrices, $\nu$ is defined in \eqref{eq:nu_def}. \reply{One may be concerned about the exponential increase of the embedding dimension $l$ as $\epS$ goes to zero. However, $\epS$ may in fact be taken as some $\mathO{1}$ constant that is smaller than $1$ (or $3/4$ in the case of stable 1-hashing). The reason being that the iterative nature of \autoref{alg:sketching} mitigates the inaccuracies of the embedding. See, e.g., the complexity bound in \autoref{thm:complexity-QR-Gaussian}. }

\begin{table}[h]
\small
\begin{tabular}{|l|l|l|l|l|l|}
\hline
                                         & $\epS$                              & $\deltaSOne$                                      & $l$          & $\deltaSTwo$ & $\sMax$              \\ \hline
Scaled Gaussian  & $(0,1)$               & $e^{-\frac{\epS^2 l}{4}}$           & $4\epS^{-2} \log(\frac{1}{\deltaSOne})$           & $(0,1)$      &  $1 + \sqrt{\frac{d}{l}} + \sqrt{\frac{2\logOneOverDeltaSTwo}{l}}$ \\ \hline
$s$-hashing         & $(0,1)$            & $e^{-\frac{\epS^2 l}{C_1}}$         & $C_1 \epS^{-2} \log(\frac{1}{\deltaSOne})$        & $0$          & $\sqrtFrac{d}{s}$    \\ \hline
Stable $1$-hashing  & $(0,\frac{3}{4})$ & $e^{-\frac{l(\epS-1/4)^2}{C_3}}$     & $C_3 (\epS-1/4)^{-2} \log(\frac{1}{\deltaSOne}) $ & $0$          & $\sqrt{\ceil{\frac{d}{l}}}$  \\ \hline
Scaled sampling    & $(0,1)$            & $e^{- \frac{\epS^2 l}{2d\nu^2}}$ & $2d\nu^2 \epS^{-2} \log(\frac{1}{\deltaSOne})$    & $0$          & $\sqrt{\frac{d}{l}}$ \\ \hline
\end{tabular}
\caption{Summary of theoretical properties of using different random ensembles with \autoref{alg:sketching}.}
\label{tab:alg:sketching}
\end{table}

    \section{Random subspace quadratic regularisation and 
    subspace trust region methods}
    \label{BCGN:sec4}
    In this section, we analyse two methods for computing the trial step $\sKHat$ given the sketching based model in \autoref{alg:sketching}. We show that using both methods: quadratic regularisation and trust-region, satisfy \autoref{AA3} and \autoref{AA4}. Using \autoref{thm2}, we show that the iteration complexity for both methods is $\mathO{\epsilon^{-2}}$ to bring the objective's gradient below $\epsilon$.


First we show that \autoref{AA4} holds for \autoref{alg:sketching} if the following model reduction condition is met.

\begin{assumption} \label{AA8}
There exists a non-negative, non-decreasing function $\hBar: \R^2 \to \R$ such that on each true iteration $k$ of \autoref{alg:sketching} we have
\begin{equation}
    \mKHat{0} - \mKHat{\sHat_k(\alpha_k)} \geq \hBar\bracket{\normTwo{\sKGradFK}, \alphaK}, \notag
\end{equation}
where $S_k, \hat{m}_k, \alphaK, \sHat_k$ are defined in \autoref{alg:sketching}.
\end{assumption}

\begin{lemma} \label{lem:AA8impliesAA4}
Let \autoref{AA8} hold with $\hBar$ and true iterations defined in \autoref{def::true_iters}. Then \autoref{alg:sketching} satisfies \autoref{AA4} with $h(\epsilon, \alphaK) = \theta \hBar \bracket{ (1-\epS)^{1/2} \epsilon, \alphaK}$, where $\epS$ is defined in \eqref{eqn::JL}.
\end{lemma}

\begin{proof}
Let $k$ be a true and successful iteration with $k < \nEps$ for some $\epsilon>0$ where $\nEps$ is defined in \eqref{eqn::nEps}. Then, using the fact that the iteration is true, successful, \autoref{AA8} and $k<\nEps$, we have
\begin{align}
    \fK-\fKPlusOne 
    & \geq \theta \squareBracket{\mKHat{0} - \mKHat{\sHat_k(\alphaK)}}\notag\\
    &\geq \theta \hBar( \normTwo{\sKGradFK}, \alphaK) \notag\\
    & \geq \theta \hBar( \oneMiusEpsSToHalf \normTwo{\gradFK}, \alphaK) \notag\\
    & \geq \theta \hBar( \oneMiusEpsSToHalf \epsilon, \alphaK). \notag
\end{align}
\end{proof}

The next Lemma is a standard result and we include its proof for completeness. It is needed to show random subspace quadratic regularisation and trust region methods satisfy \autoref{AA3}.

\begin{lemma} \label{lem:Taylor}
In \autoref{alg:sketching}, suppose that $\normTwo{B_k} \leq \BMax$ for all $k$ where $\BMax$ is independent of $k$, and $f$ is continuously differentiable with $L$-Lipschitz continuous gradient. Then for any $\sKHat \in \R^l$ and $S_k\in \R^{l \times d}$, let $s_k = S_k^T \sKHat \in \R^d$. We have that

\begin{equation}
    | f(x_k + s_k) - \mkHatSkHat| \leq \bracket{\frac{L+\BMax}{2}} \normTwo{S_k^T \hat{\sK}}^2. \label{eqn::Alg2Lipschitz}
\end{equation}
\end{lemma}

\begin{proof}
As $f$ is continuously differentiable with L-Lipschitz gradient, we have from Corollary 8.4 in \cite{CoraBook} that
\begin{equation}
    \abs{f(x_k + S_k \sKHat) - \innerProduct{S_k\gradFK}{\sKHat}}
    \leq \frac{L}{2}\normTwo{S_k^T \sKHat}^2.
\end{equation}
The above equation and triangle inequality implies
\begin{align}
     \abs{f(x_k + s_k) - \mkHatSkHat} 
     = & \abs{f(x_k + s_k) - f(x_k) - \innerProduct{\sKGradFK}{\sKHat}
    - \frac{1}{2}\innerProduct{\SKTransposedsKHat}{B_k \SKTransposedsKHat}}
    \notag\\
    & \leq \bracket{\frac{L}{2} + \frac{1}{2}\normTwo{B_k}}
        \normTwo{S_k^T \sKHat}^2 \notag \\
    & \leq \frac{L + \BMax}{2} \normTwo{S_k^2 \sKHat},
\end{align}
where we used $\normTwo{B_k} \leq \BMax$ to derive the last inequality.

\end{proof}


\subsection{Random subspace quadratic regularisation with sketching}
Here we present \autoref{alg:sketching_QR}, a generic random subspace quadratic regularisation method with sketching, which is a particular form of \autoref{alg:sketching} where the step is computed using a quadratic regularisation framework (see Page \pageref{Intro_QR} in Chapter \ref{Ch1}). We show that in addition to \autoref{AA5} which is satisfied by \autoref{alg:sketching}, \autoref{alg:sketching_QR} satisfies \autoref{AA3} and \autoref{AA4}.

\begin{algorithm}[H]
\begin{description}

 \item[Initialization] \ \\
 Choose a matrix distribution $\cal{S}$ of matrices $S \in \R^{l\times d}$. 
 Choose constants $\gamma_1\in (0,1)$, $\gamma_2 = \gammaOne^{-c}$, for some $c \in \N^+$, $l \in \N^+$, $\theta \in (0,1)$ and $\alpha_{\max}, \BMax>0$.
 Initialize the algorithm by setting $x_0 \in \R^d$, $\alpha_0 = \alphaMax \gamma_1^p$ for some $p \in \N^+$ and $k=0$.

 \item[1. Compute a reduced model and a step] \ \\
 Draw a random matrix $S_k \in \R^{l \times d}$ from $\cal{S}$, and let
 \begin{align}
     \mkHatSHat = \fK + \innerProduct{\sKGradFK}{\sHat} + \frac{1}{2} \innerProduct{\sHat}{\sKBKSKT \sHat}
 \end{align}
 where $B_k \in \R^{d\times d}$ is a positive-semi-definite user provided matrix with $\normTwo{B_k} \leq \BMax$.
 
 Compute $\sKHat$ by approximately minimising $\lkHatSHat = \mkHatSHat + \frac{1}{2\alphaK}\normTwo{S_k^T\sK}^2$ 
 such that the following two conditions hold
\begin{align}
    \normTwo{\grad \lkHat{\sKHat}} \leq \kappaT \normTwo{\SKTransposed \sKHat}, \label{eqn:QRskHatCond1}\\
    \lkHat{\sKHat} \leq \lkHat{0}, \label{eqn:QRsKHatCond2}
\end{align}
where $\kappaT\geq 0$ is a user chosen constant.
 And set $\sK = S_k^T \sKHat$.

\item[2. Check sufficient decrease]\ \\  
Let sufficient decrease be defined by the condition
\begin{equation}
    \fK - \fKPlusOne \geq \theta \squareBracket{\mKHat{0} - \mKHat{\hat{\sK}(\alpha_k)}}. \notag
\end{equation}

\item[3, Update the parameter $\alphaK$ and possibly take the potential step $\sK$]\ \\
If sufficient decrease is achieved, set $\xKPlusOne = \xK + \sK$ and $\alphaKPlusOne = \min \set{\alphaMax, \gammaTwo\alphaK}$ [a successful iteration]. \\
Otherwise set $\xKPlusOne = \xK$ and $\alphaKPlusOne = \gammaOne \alphaK$ [an unsuccesful iteration].\\
Increase the iteration count by setting $k=k+1$ in both cases.

\caption{\bf{A generic random subspace quadratic regularisation method with sketching}} \label{alg:sketching_QR} 
\end{description}
\end{algorithm}

We note that
\begin{equation}
    \mKHat{0} - \mKHat{\sKHat} = \lkHat{0} - \lkHat{\sKHat} + \oneOverTwoAlphaK \normTwo{\SKTransposed\sKHat}^2 \geq \oneOverTwoAlphaK \normTwo{\SKTransposed\sKHat}^2, \label{eqn:QRModelDecreaseLowerByStep}
\end{equation}
where we have used \eqref{eqn:QRsKHatCond2}.
\autoref{lem:QRAlpahLow} shows \autoref{alg:sketching_QR} satisfies \autoref{AA3}. 

\begin{lemma} \label{lem:QRAlpahLow}
Let $f$ be continuously differentiable with $L$-Lipschitz continuous gradient. 
Then \autoref{alg:sketching_QR} satisfies \autoref{AA3} with 
\begin{equation}
    \alphaLow = \frac{1-\theta}{L + \BMax}. \notag
\end{equation}
\end{lemma}

\begin{proof}
Let $\epsilon>0$ and $k < \nEps$, and assume iteration $k$ is true with $\alphaK \leq \alphaLow$, define
\begin{equation}
    \rho_k = \frac{f(x_k) - f(x_k + s_k)}{\mKHat{0} - \mKHat{\sKHat}}. \notag
\end{equation}


We have
\begin{align}
    \abs{1- \rho_k} 
    &= \frac{\abs{f(\xK + \sK) - \mKHat{\sK}}}{\abs{\mKHat{0}-\mKHat{\sKHat}}} \texteq{by $\mKHat{0} = \fK$} \notag\\
    &\leq \frac{\bracket{\frac{L+\BMax}{2}} \normTwo{\SKTransposedsKHat}^2}{\frac{1}{2\alpha_k} \normTwo{\SKTransposedsKHat}}
    \notag\\
    & \leq 1-\theta, \notag
\end{align}
where the first inequality follows from \autoref{lem:Taylor} and \eqref{eqn:QRModelDecreaseLowerByStep}.
The above equation implies that $\rho_k \geq \theta$ and therefore iteration $k$ is successful.\footnote{For $\rho_k$ to be well-defined, we need the denominator to be strictly positive. But this is shown in \eqref{eqn:QRModelDecreaseLowerBySketchedGradient}.} 
\end{proof}

The next Lemma shows \autoref{alg:sketching_QR} satisfies \autoref{AA8}, thus satisfying \autoref{AA4} by \autoref{lem:AA8impliesAA4}.

\begin{lemma} \label{lem:qr:hbar_bound}
\autoref{alg:sketching_QR} satisfies \autoref{AA8} with 
\begin{equation}
    \barH{z_1}{z_2} = \frac{z_1^2}{2\alphaMax \bracket{\sMax \bracket{\BMax + z_2^{-1}} + \kappaT}^2}, \label{eqn:QRHBarEq}
\end{equation}
where $\sMax$ is defined in \eqref{eqn::sMax}.

\end{lemma}

\begin{proof}
Let iteration $k$ be true. Using the definition of $\hat{l}_k$, we have 
\begin{equation}
    \grad \lkHat{\sKHat} = S_k \gradFK + S_k B_k S_k^T \sKHat + \oneOverAlphaK S_k S_k^T \sKHat, \notag
\end{equation}
It follows that
\begin{align}
    \normTwo{S_k \gradFK} &= \normTwo{- S_k\bracket{ B_k + \oneOverAlphaK} S_k^T \sKHat + \grad \lkHat{\sKHat}} \notag\\
    & \leq \bracket{\sMax \bracket{\BMax + \oneOverAlphaK}}\normTwo{S_k^T \sKHat} + \normTwo{\grad \lkHat{\sKHat}} \notag\\
    & \leq \bracket{\sMax \bracket{\BMax + \oneOverAlphaK} + \kappaT} \normTwo{S_k^T \sKHat}, \label{eqn:QRSketchedGradientUpperBySketchedStep}
\end{align}
where we used $\normTwo{S_k}\leq \sMax$ on true iterations and $\normTwo{B_k} \leq \BMax$ to derive the first inequality and \eqref{eqn:QRskHatCond1} to derive the last inequality.

Therefore, using \eqref{eqn:QRModelDecreaseLowerByStep} and \eqref{eqn:QRSketchedGradientUpperBySketchedStep}, we have 
\begin{align}
     \mKHat{0} - \mKHat{\sKHat} &\geq \oneOverTwoAlphaK \normTwo{\SKTransposed\sKHat}^2 \notag\\
     &\geq \oneOverTwoAlphaK \bracket{\frac{1}{\sMax \bracket{\BMax + \oneOverAlphaK} + \kappaT}}^2 \normTwo{S_k \gradFK}^2 \notag\\
     &\geq \frac{1}{2\alphaMax} \bracket{\frac{1}{\sMax \bracket{\BMax + \oneOverAlphaK} + \kappaT}}^2 \normTwo{S_k \gradFK}^2,
     \label{eqn:QRModelDecreaseLowerBySketchedGradient}
\end{align}
satisfying \autoref{AA8}.

\end{proof}

\subsection{Iteration complexity of random subspace quadratic regularisation methods}\label{iter_complexity_QR}
Here we derive complexity results for three concrete implementations of \autoref{alg:sketching_QR} that use different random ensembles. Many other random ensembles are possible. 
As a reminder, the below expression, introduced earlier in this chapter, will be needed. 
\begin{align}
    & \newL = \ceil{\logBaseGammaOne{ \minMe{\frac{\alphaLow}{\alphaZero}}{\frac{1}{\gammaTwo}}}} \label{tmp-2021-12-31-8} 
\end{align}
Applying \autoref{lem:AA8impliesAA4}, \autoref{lem:QRAlpahLow}, \autoref{lem:qr:hbar_bound} for \autoref{alg:sketching_QR}, we have that \autoref{AA3} and \autoref{AA4} are satisfied with
\begin{align}
    & \alphaLow = \frac{1-\theta}{L + \BMax} \notag \\
    & h(\epsilon, \alphaZero\gammaOne^{c+\newL}) = \theta \barH{(1 - \epS)^{1/2}\epsilon}{\alphaZero\gammaOne^{c+\newL}} \notag \\
    & = \frac{\theta   (1-\epS)\epsilon^2}{2\alphaMax \bracket{\sMax \bracket{\BMax + \alphaZero^{-1}\gammaOne^{-c-\newL}} + \kappaT}^2} 
    \label{tmp-2022-1-13-2} 
\end{align}
Moreover, \autoref{AA5} for \autoref{alg:sketching_QR} is satisfied by applying \autoref{lem:AA8impliesAA5}.
The following three subsections give complexity results of \autoref{alg:sketching_QR} with different random ensembles. We suggest the reader to refer back to \autoref{tab:alg:sketching} for a summary of their theoretical properties. 

\subsubsection{Using scaled Gaussian matrices}
\autoref{alg:sketching_QR} with scaled Gaussian matrices have a (high-probability) iteration complexity of $\mathO{\frac{d}{l}\epsilon^{-2}}$ to drive $\gradFK$ below $\epsilon$ and $l$ can be chosen as a (problem dimension-independent) constant (see \autoref{tab:alg:sketching}). 
\begin{theorem}\label{thm:complexity-QR-Gaussian}
Suppose $f$ is continuously differentiable with $L$-Lipschitz continuous gradient.
Let $\deltaSTwo, \epS, \deltaOne>0$, $l\in \N^+$ such that
\begin{equation*}
    \deltaS < \frac{c}{(c+1)^2}, \quad \nPreFactorTR >0,
\end{equation*}
where $\deltaS = e^{-l\epS^2/4} + \deltaSTwo$.
Run \autoref{alg:sketching_QR} with $\cal{S}$ being the distribution of scaled Gaussian matrices, for $N$ iterations with
\begin{equation}
    N \geq  \nPreFactorTR \squareBracket{
         \fZeroMinusfStarOverH
         + \frac{\newL}{1+c}}, \notag
\end{equation}
where 
\begin{equation}
    h(\epsilon, \alphaZero\gammaOne^{c+\newL}) = \frac{\theta   (1-\epS)\epsilon^2}{ 2\alphaMax \bracket{\squareBracket{1 + \sqrt{\frac{d}{l}} + \sqrt{\frac{2\logOneOverDeltaSTwo}{l}}}\bracket{\BMax + \alphaZero^{-1}\gammaOne^{-c-\newL}} + \kappaT}^2 } \notag
\end{equation}
and $\newL$ is given in \eqref{tmp-2021-12-31-8}.
Then, we have 
\begin{equation}
    \probability{N \geq \nEps} \geq 1 - \chernoffLowerExponential, \notag
\end{equation}
where $\nEps$ is defined in \eqref{eqn::nEps}.
\end{theorem}

\begin{proof}
We note that \autoref{alg:sketching_QR} is a particular form of \autoref{alg:generic}, therefore \autoref{thm2} applies. Moreover \autoref{AA3}, \autoref{AA4} and \autoref{AA5} are satisfied. Applying \autoref{lem::deduceAA2}, \autoref{lem:GaussJLEmbedding} and \autoref{Lem:GaussSMax} for scaled Gaussian matrices, \autoref{AA2} is satisfied with 
\begin{align*}
    & \sMax = 1 + \sqrt{\frac{d}{l}} + \sqrt{\frac{2\logOneOverDeltaSTwo}{l}} \\
    & \deltaS =  e^{-\epS^2 l/ 4} + \deltaSTwo.
\end{align*}
Applying \autoref{thm2} and substituting the expression of $\sMax$ above in \eqref{tmp-2022-1-13-2} gives the desired result.
\end{proof}

\subsubsection{Using stable $1$-hashing matrices}
\autoref{alg:sketching_QR} with stable $1$-hashing matrices have a (high-probability) iteration complexity of $\mathO{\frac{d}{l}\epsilon^{-2}}$ to drive $\gradFK$ below $\epsilon$ and $l$ can be chosen as a (problem dimension-independent) constant (see \autoref{tab:alg:sketching}). 

\begin{theorem}\label{thm:complexity-QR-stable-1-hashing}
Suppose $f$ is continuously differentiable with $L$-Lipschitz continuous gradient.
Let $\deltaOne>0$, $\epS \in (0,3/4)$, $l\in \N^+$ such that
\begin{equation*}
    \deltaS < \frac{c}{(c+1)^2}, \quad \nPreFactorTR >0,
\end{equation*}
where $\deltaS = e^{-\frac{l(\epS-1/4)^2}{C_3}}$ and $C_3$ is defined in \autoref{lem:tmp:2022-1-13-1}.
Run \autoref{alg:sketching_QR} with $\cal{S}$ being the distribution of stable 1-hashing matrices, for $N$ iterations with
\begin{equation}
    N \geq  \nPreFactorTR \squareBracket{
         \fZeroMinusfStarOverH
         + \frac{\newL}{1+c}}, \notag
\end{equation}
where 
\begin{equation}
    h(\epsilon, \alphaZero\gammaOne^{c+\newL}) = \frac{\theta   (1-\epS)\epsilon^2}{ 2\alphaMax \bracket{\sqrt{\ceil{d/l}}\bracket{\BMax + \alphaZero^{-1}\gammaOne^{-c-\newL}} + \kappaT}^2 } \notag
\end{equation}
and $\newL$ is given in \eqref{tmp-2021-12-31-8}.
Then, we have 
\begin{equation}
    \probability{N \geq \nEps} \geq 1 - \chernoffLowerExponential, \notag
\end{equation}
where $\nEps$ is defined in \eqref{eqn::nEps}.
\end{theorem}

\begin{proof}
Applying \autoref{lem::deduceAA2}, \autoref{lem:tmp:2022-1-13-1} and \autoref{lem:stable-1-hashing-SMax} for stable 1-hashing matrices, \autoref{AA2} is satisfied with 
\begin{align*}
    & \sMax = \sqrt{\ceil{d/l}} \\
    & \deltaS =  e^{-\frac{l(\epS-1/4)^2}{C_3}}.
\end{align*}
Applying \autoref{thm2} and substituting the expression of $\sMax$ above in \eqref{tmp-2022-1-13-2} gives the desired result. 
\end{proof}

\subsubsection{Using sampling matrices}
\autoref{alg:sketching_QR} with scaled sampling matrices have a (high-probability) iteration complexity of $\mathO{\frac{d}{l}\epsilon^{-2}}$ to drive $\gradFK$ below $\epsilon$. However, unlike in the previous two cases, here $l$ depends on the problem dimension $d$ and a problem specific constant $\nu$ (see \autoref{tab:alg:sketching}). If $\nu = \mathO{1/d}$, then $l$ can be chosen as a problem dimension-independent constant.

\begin{theorem}\label{thm:complexity-QR-sampling}
Suppose $f$ is continuously differentiable with $L$-Lipschitz continuous gradient.
Let $\deltaOne>0$, $\epS \in (0,1)$, $l\in \N^+$ such that
\begin{equation*}
    \deltaS < \frac{c}{(c+1)^2}, \quad \nPreFactorTR >0,
\end{equation*}
where $\deltaS =e^{- \frac{\epS^2 l}{2d\nu^2}}$ and $\nu$ is defined in \eqref{eq:nu_def}.
Run \autoref{alg:sketching_QR} with $\cal{S}$ being the distribution of scaled sampling matrices, for $N$ iterations with
\begin{equation}
    N \geq  \nPreFactorTR \squareBracket{
         \fZeroMinusfStarOverH
         + \frac{\newL}{1+c}}, \notag
\end{equation}
where 
\begin{equation}
    h(\epsilon, \alphaZero\gammaOne^{c+\newL}) = \frac{\theta   (1-\epS)\epsilon^2}{ 2\alphaMax \bracket{\sqrt{d/l}\bracket{\BMax + \alphaZero^{-1}\gammaOne^{-c-\newL}} + \kappaT}^2 } \notag
\end{equation}
and $\newL$ is given in \eqref{tmp-2021-12-31-8}.
Then, we have 
\begin{equation}
    \probability{N \geq \nEps} \geq 1 - \chernoffLowerExponential, \notag
\end{equation}
where $\nEps$ is defined in \eqref{eqn::nEps}.
\end{theorem}

\begin{proof}
Applying \autoref{lem::deduceAA2}, \autoref{lem:sampling:non-uniformity-BCGN} and \autoref{tmp-2022-1-14-12} for scaled sampling matrices, \autoref{AA2} is satisfied with 
\begin{align*}
    & \sMax = \sqrt{d/l} \\
    & \deltaS = e^{- \frac{\epsilon^2 l}{2d\nu^2}}.
\end{align*}
Applying \autoref{thm2} and substituting the expression of $\sMax$ above in \eqref{tmp-2022-1-13-2} gives the desired result.
\end{proof}

\begin{remark}
The dependency on $\epsilon$ in the iteration complexity matches that for the full-space quadratic regularisation method (Page \pageref{Intro_QR}). Note that for each ensemble considered, there is dimension-dependence in the bound of the form $\frac{d}{l}$. 
We may eliminate the dependence on $d$ in the iteration complexity by fixing the ratio $\frac{d}{l}$ to be a constant. 
\end{remark}

\subsection{Random subspace trust region methods with sketching}
Here we present a generic random subspace trust region method with sketching, \autoref{alg:sketching_TR}, which is a particular form of \autoref{alg:sketching} where the step is computed using a trust region framework (see Page \pageref{Intro_trust_region} in Chapter \ref{Ch1}). 

\begin{algorithm}[H]
\begin{description}

 \item[Initialization] \ \\
 Choose a matrix distribution $\cal{S}$ of matrices $S \in \R^{l\times d}$. 
 Choose constants $\gamma_1\in (0,1)$, $\gamma_2 = \gammaOne^{-c}$, for some $c \in \N^+$, $l \in \N^+$, $\theta \in (0,1)$ and $\alpha_{\max}, \BMax>0$.
 Initialize the algorithm by setting $x_0 \in \R^d$, $\alpha_0 = \alphaMax \gamma_1^p$ for some $p \in \N^+$ and $k=0$.

 \item[1. Compute a reduced model and a step] \ \\
 Draw a random matrix $S_k \in \R^{l \times d}$ from $\cal{S}$, and let
 \begin{align}
     \mkHatSHat =\fK + \innerProduct{\sKGradFK}{\sHat} + \frac{1}{2} \innerProduct{\sHat}{\sKBKSKT \sHat}
 \end{align}
 where $B_k \in \R^{d\times d}$ is a user provided matrix with $\normTwo{B_k} \leq \BMax$.
 
 Compute $\sKHat$ by approximately minimising $\mkHatSHat$ such that for some $C_7>0$,
\begin{align}
    &\normTwo{\sKHat} \leq \alphaK \label{tmp-2021-12-31-4} \\
    &\mKHat{0} - \mKHat{\sKHat} \geq C_7 \normTwo{S_k \gradFK} \min \set{\alphaK, \frac{\normTwo{S_k \gradFK}}{\normTwo{B_k}}}. \label{eqn::TRStepComputationCOndition}
\end{align}

\item[2. Check sufficient decrease]\ \\  
Let sufficient decrease be defined by the condition
\begin{equation}
    \fK - \fKPlusOne \geq \theta \squareBracket{\mKHat{0} - \mKHat{\hat{\sK}(\alpha_k)}}. \notag
\end{equation}

\item[3, Update the parameter $\alphaK$ and possibly take the potential step $\sK$]\ \\
If sufficient decrease is achieved, set $\xKPlusOne = \xK + \sK$ and $\alphaKPlusOne = \min \set{\alphaMax, \gammaTwo\alphaK}$ [successful iteration].\\
Otherwise set $\xKPlusOne = \xK$ and $\alphaKPlusOne = \gammaOne \alphaK$. [unsuccessful iteration].\\

Increase the iteration count by setting $k=k+1$ in both cases.

\caption{\bf{A generic random subspace trust region method with sketching}} \label{alg:sketching_TR} 
\end{description}
\end{algorithm}

\begin{remark}
Lemma 4.3 in \cite{Nocedal:2006uv} shows there always exists $\hat{s_k} \in \R^l$ such that \eqref{eqn::TRStepComputationCOndition} holds. Specifically, define $\gK = \sKGradFK$. 
If $\gK=0$, one may take $\hat{s_k} = 0$; and otherwise one may take $\sKHat$ to be the Cauchy point (that is, the point where the model $\hat{m}_k$ is minimised in the negative model gradient direction within the trust region), which can be computed by $\hat{s}_k^c = -\tau_k \frac{\alphaK}{\normTwo{\gK}} \gK$, where $\tau_k=1$ if $\gK^T B_k \gK \leq 0$; and $\tau_k = \min \bracket{ \frac{\normTwo{\gK}^3}{ \gK^T B_k \gK \alphaK}, 1}$ otherwise. 
\end{remark}

\autoref{tmp-2021-12-31-5} shows that \autoref{alg:sketching_TR} satisfies \autoref{AA3}.

\begin{lemma} \label{tmp-2021-12-31-5}
Suppose $f$ is continuously differentiable with $L$-Lipschitz continuous gradient. Then \autoref{alg:sketching_TR} satisfies \autoref{AA3} with
\begin{equation}
    \alphaLow = \oneMiusEpsSToHalf \epsilon \min \bracket{ \frac{C_7 (1-\theta)}{(\lPlusHalfBmax)\sMax^2}, \frac{1}{\BMax} }.\label{eqn::TRalphaLowExpression}
\end{equation}
\end{lemma}

\begin{proof}
Let $\epsilon>0$ and $k < \nEps$, and assume iteration $k$ is true with $\alphaK \leq \alphaLow$, define
\begin{equation}
    \rho_k = \frac{f(x_k) - f(x_k + s_k)}{\mKHat{0} - \mKHat{\sKHat}}. \notag
\end{equation} 

Then we have
\begin{align}
    \abs{1-\rho_k} &= \abs{\frac{f(x_k + s_k) - \mkHatSkHat}{\mKHat{0}-\mkHatSkHat}} \notag \\
    & \leq \frac{(\lPlusHalfBmax)\normTwo{\SKTransposedsKHat}^2}{C_7 \normTwo{\sKGradFK}\min \bracket{\alphaK, \frac{\normTwo{\sKGradFK}}{\normTwo{B_k}}}}
    \notag\\
    & \leq \frac{(\lPlusHalfBmax) \sMax^2 \alphaK^2}{C_7 \normTwo{\sKGradFK}\min \bracket{\alphaK, \frac{\normTwo{\sKGradFK}}{\normTwo{B_k}}}} 
    \notag\\
    & \leq \frac{(\lPlusHalfBmax)\sMax^2\alphaK^2}{C_7 \oneMiusEpsSToHalf\epsilon \min \bracket{\alphaK, \frac{\oneMiusEpsSToHalf\epsilon}{\BMax}}} 
    \notag\\
    & \leq 1 - \theta,\notag
\end{align}
where the first inequality follows from \eqref{eqn::TRStepComputationCOndition} and \autoref{lem:Taylor}, the second inequality follows from \eqref{eqn::sMax} and $\normTwo{\sKHat} \leq \alphaK$, the third inequality follows from \eqref{eqn::JL} and the fact that $\gradFK>\epsilon$ for $k<\nEps$, the last inequality follows from $\alphaK \leq \alphaLow$ and \eqref{eqn::TRalphaLowExpression}. It follows then $\rho_k \geq \theta$ and iteration $k$ is successful. \footnote{Note that for $k$ being a true iteration with $k<\nEps$, \eqref{eqn::TRStepComputationCOndition} along with \eqref{eqn::JL}, $\alphaK>0$ gives $\mKHat{0}-\mKHat{\sKHat}>0$ so that $\rho_k$ is well defined.} 
\end{proof}

The next lemma shows that \autoref{alg:sketching_TR} satisfies \autoref{AA8}, thus satisfying \autoref{AA4}.

\begin{lemma}\label{tmp-2022-1-13-5}
\autoref{alg:sketching_TR} satisfies \autoref{AA8} with 
\begin{equation}
    \barH{z_1}{z_2} = C_7\min \bracket{z_1 z_2, z_1^2/\BMax}. \notag
\end{equation}
\end{lemma}
\begin{proof}
Use \eqref{eqn::TRStepComputationCOndition} with $\normTwo{B_k} \leq \BMax$. 
\end{proof}

\subsection{Iteration complexity of random subspace trust region methods}

Here we derive complexity results for three concrete implementations of \autoref{alg:sketching_TR} that use different random ensembles. The exposition follows closely Section \ref{iter_complexity_QR}. And the complexity results are in the same order in $\epsilon, \frac{d}{l}$ but have different constants. 

Applying \autoref{lem:AA8impliesAA4}, \autoref{tmp-2021-12-31-5}, \autoref{tmp-2022-1-13-5} for \autoref{alg:sketching_TR}, we have that \autoref{AA3} and \autoref{AA4} are satisfied with
\begin{align}
    & \alphaLow = \oneMiusEpsSToHalf \epsilon \min \bracket{ \frac{C_7 (1-\theta)}{(\lPlusHalfBmax)\sMax^2}, \frac{1}{\BMax} } \notag \\
    & h(\epsilon, \alphaZero\gammaOne^{c+\newL}) = \theta \barH{(1 - \epS)^{1/2}\epsilon}{\alphaZero\gammaOne^{c+\newL}} \notag \\
    & = \theta C_7  \minMe{\bracket{1-\epS}^{1/2} \epsilon \alphaZero \gammaOne^{c + \newL}}{\bracket{1-\epS} \epsilon^2/\BMax }
    \label{tmp-2022-1-13-6} 
\end{align}
Here, unlike in the analysis of \autoref{alg:sketching_QR}, $\alphaLow$ (and consequently $\newL$) depends on $\epsilon$. We make this dependency on $\epsilon$ explicit. 
Using the definition of $\newL$ in \eqref{eqn:tauLDef} and substituting in the expression for $\alphaLow$, we have
\begin{align*}
    \alphaZero \gammaOne^{c+\newL} 
    & = \alphaZero\gammaOne^c
        \gammaOne^{
            \ceil{
                \log_\gammaOne \bracket{
                    \minMe{
                        \oneMiusEpsSToHalf \epsilon \min \bracket{ \frac{C_7 (1-\theta)}{(\lPlusHalfBmax)\sMax^2}, \frac{1}{\BMax} } \alphaZero^{-1}
                    }
                    {
                        \gammaTwo^{-1}
                    }
                }
            }
        } \notag \\
    & \geq \alphaZero\gammaOne^c \gammaOne
        \minMe{
                \oneMiusEpsSToHalf \epsilon \min \bracket{ \frac{C_7 (1-\theta)}{(\lPlusHalfBmax)\sMax^2}, \frac{1}{\BMax} } \alphaZero^{-1}
                }
                {
                    \gammaTwo^{-1}
                } \notag \\
    & = \gammaOne^{c+1}
        \minMe{
                \oneMiusEpsSToHalf \epsilon \min \bracket{ \frac{C_7 (1-\theta)}{(\lPlusHalfBmax)\sMax^2}, \frac{1}{\BMax} }
                }
                {
                    \alphaZero \gammaTwo^{-1}
                }, \notag 
\end{align*}
where we used $\ceil{y} \leq y+1$ to derive the inequality.
Therefore, \eqref{tmp-2022-1-13-6} implies 
\begin{align}
    &h(\epsilon, \alphaZero\gammaOne^{c+\newL}) \notag\\
    &\geq 
    \theta C_7 
    \minMe{
        \gammaOne^{c+1}
        \minMe{
                \bracket{1-\epS} \epsilon^2 \min \bracket{ \frac{C_7 (1-\theta)}{(\lPlusHalfBmax)\sMax^2}, \frac{1}{\BMax} }
                }
                {
                    \bracket{1-\epS}^{1/2} \epsilon \alphaZero \gammaTwo^{-1}
                }        
    }
    {
        \frac{\bracket{1-\epS}\epsilon^2}{\BMax}
    } \notag\\ 
    & = 
    \theta C_7 \bracket{1-\epS}\epsilon^2
    \minMe{
        \gammaOne^{c+1}
        \minMe{
                \min \bracket{ \frac{C_7 (1-\theta)}{(\lPlusHalfBmax)\sMax^2}, \frac{1}{\BMax} }
                }
                {
                    \frac{\alphaZero}{\bracket{1-\epS}^{1/2}\epsilon \gammaTwo}
                }        
    }
    {
        \frac{1}{\BMax}
    } \notag\\
    & = 
    \theta C_7 \bracket{1-\epS}\epsilon^2
        \gammaOne^{c+1}
        \minMe{
                \min \bracket{ \frac{C_7 (1-\theta)}{(\lPlusHalfBmax)\sMax^2}, \frac{1}{\BMax} }
                }
                {
                    \frac{\alphaZero}{\bracket{1-\epS}^{1/2}\epsilon \gammaTwo}
                }        
    \label{tmp-2022-1-13-9}.
\end{align}
where the last equality follows from $\gammaOne^{c+1} < 1$.
Moreover, \autoref{AA5} for \autoref{alg:sketching_TR} is satisfied by applying \autoref{lem:AA8impliesAA5}.
The following three subsections give complexity results of \autoref{alg:sketching_TR} using different random ensembles with \autoref{alg:sketching_TR}. Again, we suggest the reader to refer back to \autoref{tab:alg:sketching} for a summary of their theoretical properties. 

\subsubsection{Using scaled Gaussian matrices}
\autoref{alg:sketching_TR} with scaled Gaussian matrices have a (high-probability) iteration complexity of $\mathO{\frac{d}{l}\epsilon^{-2}}$ to drive $\gradFK$ below $\epsilon$ and $l$ can be chosen as a (problem dimension-independent) constant (see \autoref{tab:alg:sketching}). 
\begin{theorem}\label{thm:complexity-TR-Gaussian}
Suppose $f$ is continuously differentiable with $L$-Lipschitz continuous gradient.
Let $\deltaSTwo, \epS, \deltaOne>0$, $l\in \N^+$ such that
\begin{equation*}
    \deltaS < \frac{c}{(c+1)^2}, \quad \nPreFactorTR >0,
\end{equation*}
where $\deltaS = e^{-l\epS^2/4} + \deltaSTwo$.
Run \autoref{alg:sketching_TR} with $\cal{S}$ being the distribution of scaled Gaussian matrices, for $N$ iterations with
\begin{equation}
    N \geq  \nPreFactorTR \squareBracket{
         \fZeroMinusfStarOverH
         + \frac{\newL}{1+c}}, \notag
\end{equation}
where 
\begin{align}
    &h(\epsilon, \alphaZero\gammaOne^{c+\newL}) \\
    & = \theta C_7 \bracket{1-\epS}\epsilon^2
        \gammaOne^{c+1}
        \minMe{
                \min \bracket{ \frac{C_7 (1-\theta)}{(\lPlusHalfBmax)
                    \squareBracket{1 + \sqrt{\frac{d}{l}} + \sqrt{\frac{2\logOneOverDeltaSTwo}{l}}}^2
                    }, \frac{1}{\BMax} }
                }
                {
                    \frac{\alphaZero}{\bracket{1-\epS}^{1/2}\epsilon \gammaTwo}
                }  
\end{align}
Then, we have 
\begin{equation}
    \probability{N \geq \nEps} \geq 1 - \chernoffLowerExponential, \notag
\end{equation}
where $\nEps$ is defined in \eqref{eqn::nEps}.
\end{theorem}

\begin{proof}
We note that \autoref{alg:sketching_TR} is a particular version of \autoref{alg:generic} therefore \autoref{thm2} applies. Applying \autoref{lem::deduceAA2}, \autoref{lem:GaussJLEmbedding} and \autoref{Lem:GaussSMax} for scaled Gaussian matrices, \autoref{AA2} is satisfied with 
\begin{align*}
    & \sMax = 1 + \sqrt{\frac{d}{l}} + \sqrt{\frac{2\logOneOverDeltaSTwo}{l}} \\
    & \deltaS =  e^{-\epS^2 l/ 4} + \deltaSTwo.
\end{align*}
Applying \autoref{thm2} and substituting the expression of $\sMax$ in \eqref{tmp-2022-1-13-9} gives the desired result.
\end{proof}

\subsubsection{Using stable $1$-hashing matrices}
\autoref{alg:sketching_TR} with stable $1$-hashing matrices have a (high-probability) iteration complexity of $\mathO{\frac{d}{l}\epsilon^{-2}}$ to drive $\gradFK$ below $\epsilon$ and $l$ can be chosen as a (problem dimension-independent) constant (see \autoref{tab:alg:sketching}). 

\begin{theorem}\label{thm:complexity-TR-stable-1-hashing}
Suppose $f$ is continuously differentiable with $L$-Lipschitz continuous gradient.
Let $\deltaOne>0$, $\epS \in (0,3/4)$, $l\in \N^+$ such that
\begin{equation*}
    \deltaS < \frac{c}{(c+1)^2}, \quad \nPreFactorTR >0,
\end{equation*}
where $\deltaS = e^{-\frac{l(\epS-1/4)^2}{C_3}}$ and $C_3$ is defined in \autoref{lem:tmp:2022-1-13-1}.
Run \autoref{alg:sketching_TR} with $\cal{S}$ being the distribution of stable 1-hashing matrices, for $N$ iterations with
\begin{equation}
    N \geq  \nPreFactorTR \squareBracket{
         \fZeroMinusfStarOverH
         + \frac{\newL}{1+c}}, \notag
\end{equation}
where 
\begin{equation}
    h(\epsilon, \alphaZero\gammaOne^{c+\newL}) = 
        \theta C_7 \bracket{1-\epS}\epsilon^2
        \gammaOne^{c+1}
        \minMe{
                \min \bracket{ \frac{C_7 (1-\theta)}{(\lPlusHalfBmax)\ceil{d/l}}, \frac{1}{\BMax} }
                }
                {
                    \frac{\alphaZero}{\bracket{1-\epS}^{1/2}\epsilon \gammaTwo}
                }\notag
\end{equation}
Then, we have 
\begin{equation}
    \probability{N \geq \nEps} \geq 1 - \chernoffLowerExponential, \notag
\end{equation}
where $\nEps$ is defined in \eqref{eqn::nEps}.
\end{theorem}

\begin{proof}
Applying \autoref{lem::deduceAA2}, \autoref{lem:tmp:2022-1-13-1} and \autoref{lem:stable-1-hashing-SMax} for stable 1-hashing matrices, \autoref{AA2} is satisfied with 
\begin{align*}
    & \sMax = \sqrt{\ceil{d/l}} \\
    & \deltaS =  e^{-\frac{l(\epS-1/4)^2}{C_3}}.
\end{align*}
Applying \autoref{thm2} and substituting the expression of $\sMax$ in \eqref{tmp-2022-1-13-9} gives the desired result.
\end{proof}

\subsubsection{Using sampling matrices}
\autoref{alg:sketching_TR} with scaled sampling matrices have a (high-probability) iteration complexity of $\mathO{\frac{d}{l}\epsilon^{-2}}$ to drive $\gradFK$ below $\epsilon$. Similar to \autoref{alg:sketching_QR} with scaled sampling matrices, here $l$ depends on the problem dimension $d$ and a problem specific constant $\nu$ (see \autoref{tab:alg:sketching}). If $\nu = \mathO{1/d}$, then $l$ can be chosen as a problem dimension-independent constant.

\begin{theorem}\label{thm:complexity-TR-sampling}
Suppose $f$ is continuously differentiable with $L$-Lipschitz continuous gradient.
Let $\deltaOne>0$, $\epS \in (0,1)$, $l\in \N^+$ such that
\begin{equation*}
    \deltaS < \frac{c}{(c+1)^2}, \quad \nPreFactorTR >0,
\end{equation*}
where $\deltaS =e^{- \frac{\epsilon^2 l}{2d\nu^2}}$ and $\nu$ is defined in \eqref{eq:nu_def}.
Run \autoref{alg:sketching_TR} with $\cal{S}$ being the distribution of scaled sampling matrices, for $N$ iterations with
\begin{equation}
    N \geq  \nPreFactorTR \squareBracket{
         \fZeroMinusfStarOverH
         + \frac{\newL}{1+c}}, \notag
\end{equation}
where 
\begin{equation}
    h(\epsilon, \alphaZero\gammaOne^{c+\newL}) = 
        \theta C_7 \bracket{1-\epS}\epsilon^2
        \gammaOne^{c+1}
        \minMe{
                \min \bracket{ \frac{C_7 (1-\theta)}{(\lPlusHalfBmax)d/l}, \frac{1}{\BMax} }
                }
                {
                    \frac{\alphaZero}{\bracket{1-\epS}^{1/2}\epsilon \gammaTwo}
                }\notag
\end{equation}
Then, we have 
\begin{equation}
    \probability{N \geq \nEps} \geq 1 - \chernoffLowerExponential, \notag
\end{equation}
where $\nEps$ is defined in \eqref{eqn::nEps}.
\end{theorem}

\begin{proof}
Applying \autoref{lem::deduceAA2}, \autoref{lem:sampling:non-uniformity-BCGN} and \autoref{tmp-2022-1-14-12} for scaled sampling matrices, \autoref{AA2} is satisfied with 
\begin{align*}
    & \sMax = \sqrt{d/l} \\
    & \deltaS = e^{- \frac{\epsilon^2 l}{2d\nu^2}}.
\end{align*}
Applying \autoref{thm2} and substituting the expression of $\sMax$ in \eqref{tmp-2022-1-13-9} gives the desired result.
\end{proof}

\begin{remark}
Similar to \autoref{alg:sketching_QR}, \autoref{alg:sketching_TR} matches the iteration complexity of the corresponding (full-space) trust region method; and the $l/d$ dependency can be eliminated by setting $l$ to be a constant fraction of $d$.
\end{remark}

\begin{remark}
Although \autoref{alg:sketching_QR} and \autoref{alg:sketching_TR} with a(any) of the above three random ensembles only require $l$ directional derivative evaluations of $f$ per iteration, instead of $d$ derivative evaluations required by the (full-space) methods, the iteration complexities are increased by a factor of $d/l$. Therefore, theoretically, \autoref{alg:sketching_QR} and \autoref{alg:sketching_TR} do not reduce the total number of gradient evaluations. However, the computational cost of the step $\sKHat$ is typically reduced from being proportional to $d^2$ to being proportional to $l^2$ (for example, if we are solving a non-linear least squares problem and choose $B_k = J_k^T J_k$) thus we still gain in having a smaller computational complexity. In practice, our theoretical analysis may not be tight and therefore we could gain in having both a smaller gradient evaluation complexity and a smaller computational complexity. See numerical illustrations in Section \ref{BCGN:sec5}.
\reply{In addition, by reducing the number of variables from $d$ (which can be arbitrarily large) to $l$ (which can be set as a constant, see \autoref{tab:alg:sketching}), \autoref{alg:sketching_QR} and \autoref{alg:sketching_TR} reduce the memory requirement of the computation of the sub-problem at each iteration, comparing to the corresponding full-space methods.  }
\end{remark}

    \section{Randomised Subspace Gauss-Newton (R-SGN) for non-linear least squares}
    \label{BCGN:sec5}
    We consider the nonlinear least-squares problem (defined in \eqref{NLS})
\begin{align}
\min_{x \in \R^d} f(x) = \frac{1}{2}\sum_{i=1}^n \norm{ r_i(x) }_2^2=\frac{1}{2}\norm{r(x)}_2^2  \notag
\end{align}
where $r = (r_1, \dots, r_n): \R^d \to \R^n$ is a smooth vector of nonlinear (possibly nonconvex) residual functions. We define the Jacobian (matrix of first order derivatives) as
\begin{align*}
J(x) = \left(\pdv{r_i(x)}{x_j}\right)_{ij} \in \R^{n \times d}
\end{align*}
and can then compactly write the gradient as $\grad f(x) = J(x)^Tr(x)$.
It can be shown e.g. in \cite{Nocedal:2006uv}, that the gradient and Hessian of $f(x)$ is then given by
\begin{align*}
\grad f(x) &= J(x)^Tr(x), \\
\grad^2 f(x) &= J(x)^TJ(x) + \sum_{i=1}^n r_i(x) \grad^2 r_i(x)  
\end{align*}
The classical Gauss-Newton (GN) algorithm applies Newton's method to minimising $f$ with only the first-order $J(x)^TJ(x)$ term in the Hessian, dropping the second-order terms involving the Hessians of the residuals $r_i$. 
This is equivalent to linearising the residuals in \eqref{NLS} so that 
\begin{align*}
r(x+s) \approx r(x) + J(x) s,
\end{align*}
and minimising the resulting model in the step $s \in \R^d$. 
Thus, at every iterate $x_k$, Gauss-Newton approximately minimises the following convex quadratic local model
\begin{align*}
f(x_k) + \dotp{ J(x_k)^Tr(x_k)}{ s } + \frac{1}{2} \dotp{ s }{ J(x_k)^T J(x_k) s }  
\end{align*}
over $s \in \R^d$. In our approach, which we call Random Subspace Gauss-Newton (R-SGN), we reduce the dimensionality of this model by minimising in an $l$-dimensional randomised subspace $\mathcal{L} \subset \R^d$, with $l\ll d$, by approximately minimising the following reduced model 
\begin{align}
f(x_k) + \dotp{ J_{\mathcal{S}}(x_k)^Tr(x_k) }{ \sHat } + \frac{1}{2} \dotp{ \sHat }{ J_{\mathcal{S}}(x_k)^T J_{\mathcal{S}}(x_k) \sHat } 
\label{m_k_definition_TR}
\end{align}
over $\sHat \in \R^l$, where  $J_{\mathcal{S}}(x_k) = J(x_k)S_k^T \in \R^{n\times l}$ denotes the reduced Jacobian for $S_k \in \R^{l\times d}$ being a randomly generated sketching matrix. Note that with $B_k = J_k^T J_k$, \autoref{alg:sketching} framework can be applied directly to 
this subspace Gauss-Newton method; guaranteeing its convergence 
under assumptions of model minimisation and sketching matrices.
Compared to the classical Gauss-Newton model, in addition to the speed-up gained due to the model dimension being reduced from $d$ to $l$, this reduced model also offers the computational advantage that it only needs to evaluate $l$ Jacobian actions, giving $J_{\mathcal{S}}(x_k)$, instead of the full Jacobian matrix $J(x_k)$. 

In its simplest form, when $S_k$ is a scaled sampling matrix, $J_{\mathcal{S}}$ can be thought of as a random subselection of columns of the full Jacobian $J$, which leads to variants of our framework that are Block-Coordinate Gauss-Newton (BC-GN) methods. In this case, for example, if the Jacobian were being calculated by finite-differences of the residual $r$, only a small number of evaluations of $r$ along coordinate directions would be needed;  such a BC-GN variant has already been used for parameter estimation in climate modelling \cite{tett2017calibrating}. Note that theoretically, the convergence of BC-GN method requires an upper bound on $\frac{\normInf{\gradFK}}{\normTwo{\gradFK}}$ for all $k\in \N$ (for more details, see the discussion of sampling matrices on page \pageref{sampling_mat_paragraph}) and \autoref{thm:complexity-TR-sampling}.

More generally, $S_k$ can be generated from any matrix distribution that satisfies \autoref{AA6}, \autoref{AA7}, e.g/ scaled Gaussian matrices or $s$-hashing matrices. In our work jointly done with Jaroslav Fowkes \cite{zhen:icml_BCGN, BCGNPaper}, we showcase the numerical performance of R-SGN methods with different sketching matrices. In this thesis we provide some numerical illustrations; the code used to produce these illustrations is written by Jaroslav Fowkes, and the results below appear in \cite{zhen:icml_BCGN, BCGNPaper}.
    \paragraph{Large-scale CUTEst problems}
We look at the behaviour of R-SGN on three large-scale ($d\approx5,000$ to $10,000$) non-linear least squares problems from the CUTEst collection \cite{gould2015cutest}. The three problems are given in \autoref{tab:cutestind}. we run R-SGN five times (and take the average performance) on each problem until we achieve a $10^{-1}$ decrease in the objective, or failing that, for a maximum of 20 iterations. Furthermore, we plot the objective decrease against cumulative Jacobian action evaluations \footnote{\reply{The total number of evaluations of Jacobian-vector product used by the algorithm.}} for each random run with subspace-sizes of 1\%, 5\%, 10\%, 50\%, 100\% of the full-space-sizes. 

\begin{table}[!h]
\centering
\begin{tabular}{lrrlrrlrr}  
\toprule
Name & $d$ & $n$ & Name & $d$ & $n$ & Name & $d$ & $n$ \\
\midrule
ARTIF & 5,000 & 5,000 & BRATU2D & 4,900 & 4,900 & OSCIGRNE & 10,000 & 10,000 \\
\bottomrule
\end{tabular}
\caption{The 3 large-scale CUTEst test problems.}\label{tab:cutestind}
\end{table}

Let us start by looking at the performance of R-SGN with scaled sampling matrices $S_k$. In \autoref{fig:large_coordinate}, we see that the objective decrease against cumulative Jacobian action evaluations for ARTIF, BRATU2D and OSCIGRNE. On ARTIF, we see that while R-SGN exhibits comparable performance, Gauss-Newton is clearly superior from a Jacobian action budget perspective. On BRATU2D we see that R-SGN really struggles to achieve any meaningful decrease, as does Gauss-Newton initially but then switches to a quadratic regime and quickly converges. On OSCIGRNE, we see that R-SGN with subspace sizes of $0.05d,0.1d,0.5d$ sometimes performs very well (outperforming Gauss-Newton) but sometimes struggles, and on averages Gauss-Newton performs better.

Next, we compare the performance of R-SGN with scaled Gaussian sketching matrices $S_k$. In \autoref{fig:large_gaussian}, we can see the objective decrease against cumulative Jacobian action evaluations for ARTIF, BRATU2D and OSCIGRNE. On ARTIF, we see that R-SGN with a subspace size of $0.5d$ outperforms Gauss-Newton initially before stagnating. On BRATU2D we once again see that R-SGN struggles to achieve any meaningful decrease, as does Gauss-Newton initially but then switches to a quadratic regime and quickly converges. However, on OSCIGRNE we see that R-SGN with a subspace size of $0.5d$ consistently outperforms Gauss-Newton.

Finally, we compare the performance of R-SGN with 3-hashing sketching matrices $S_k$. In \autoref{fig:large_hashing}, we can see the objective decrease against cumulative Jacobian action evaluations for ARTIF, BRATU2D and OSCIGRNE. On ARTIF, we again see that R-SGN with a subspace size of $0.5d$ outperforms Gauss-Newton initially before stagnating. On BRATU2D we once again see that R-SGN struggles to achieve any meaningful decrease, as does Gauss-Newton initially but then switches to a quadratic regime and quickly converges. However, on OSCIGRNE we again see that R-SGN with a subspace size of $0.5d$ consistently outperforms Gauss-Newton. 

\begin{figure}[!h]
\begin{subfigure}{0.32\textwidth}
\includegraphics[width=\textwidth]{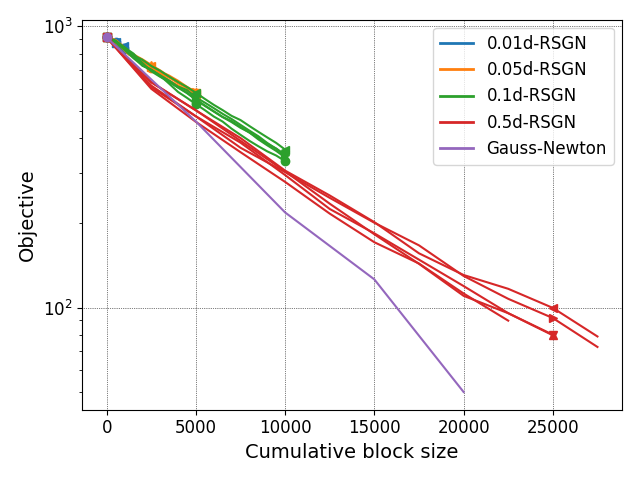}
\end{subfigure}
\begin{subfigure}{0.32\textwidth}
\includegraphics[width=\textwidth]{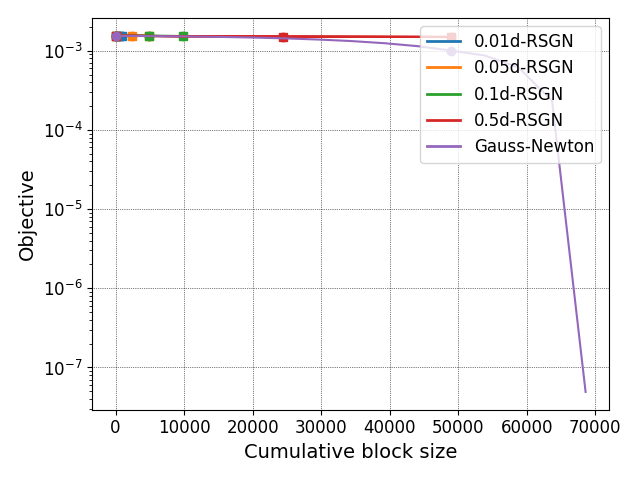}
\end{subfigure}
\begin{subfigure}{0.32\textwidth}
\includegraphics[width=\textwidth]{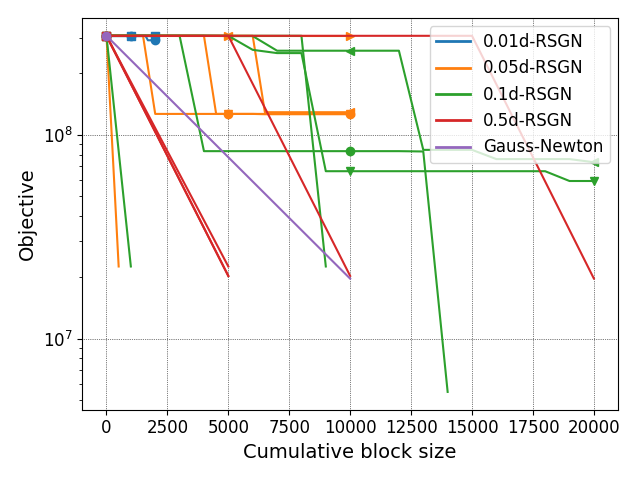}
\end{subfigure}
\caption{ARTIF (left), BRATU2D (middle) and OSCIGRNE (right) objective value against cumulative Jacobian action size for R-SGN with coordinate sampling.}
\label{fig:large_coordinate}
\end{figure}

\begin{figure}[!h]
\begin{subfigure}{0.32\textwidth}
\includegraphics[width=\textwidth]{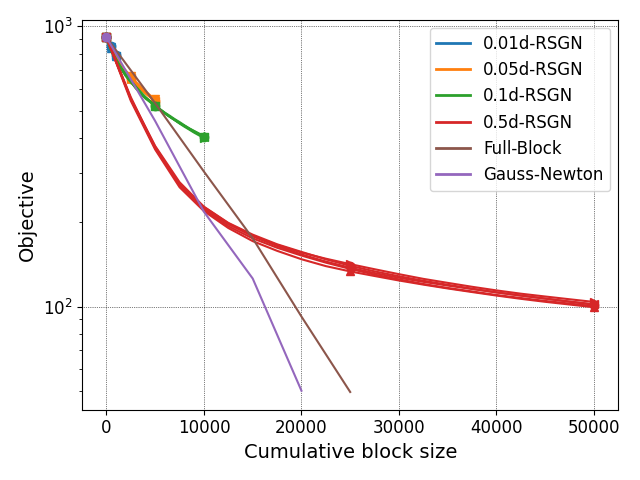}
\end{subfigure}
\begin{subfigure}{0.32\textwidth}
\includegraphics[width=\textwidth]{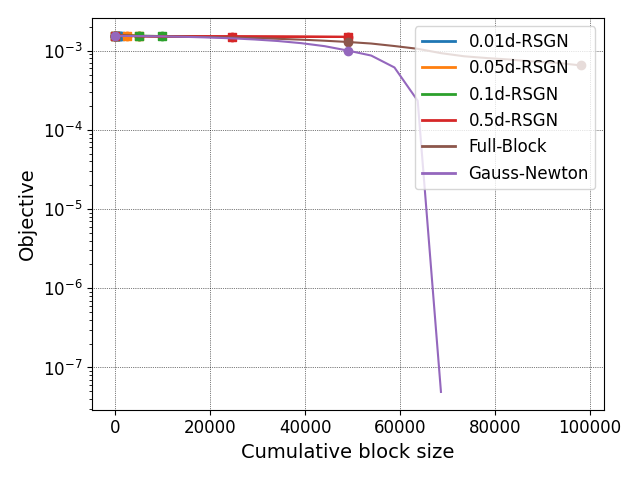}
\end{subfigure}
\begin{subfigure}{0.32\textwidth}
\includegraphics[width=\textwidth]{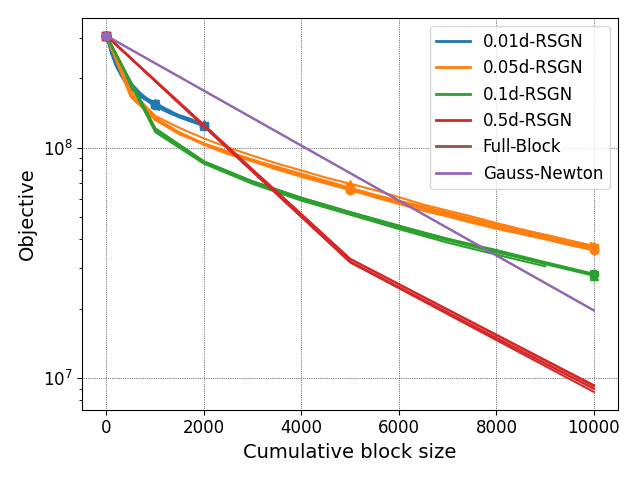}
\end{subfigure}
\caption{ARTIF (left), BRATU2D (middle) and OSCIGRNE (right) objective value against cumulative Jacobian action size for R-SGN with Gaussian sketching.}
\label{fig:large_gaussian}
\end{figure}

\begin{figure}[!h]
\begin{subfigure}{0.32\textwidth}
\includegraphics[width=\textwidth]{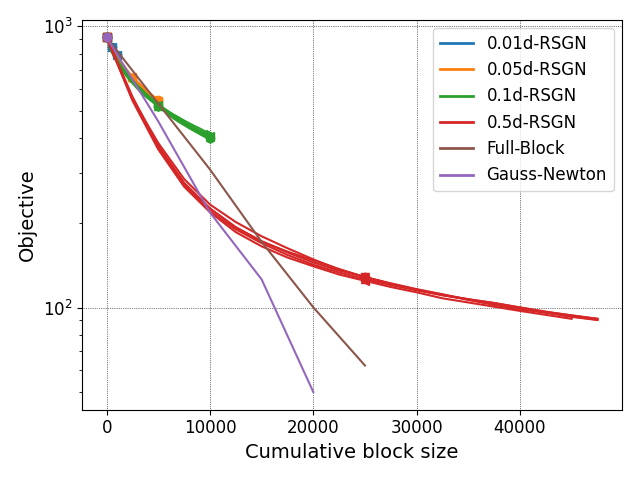}
\end{subfigure}
\begin{subfigure}{0.32\textwidth}
\includegraphics[width=\textwidth]{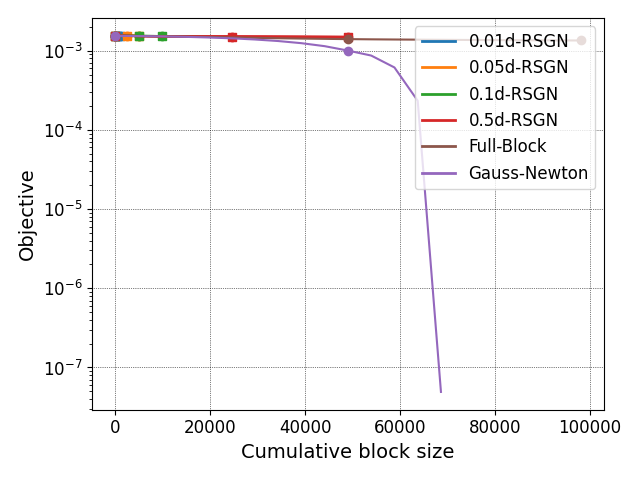}
\end{subfigure}
\begin{subfigure}{0.32\textwidth}
\includegraphics[width=\textwidth]{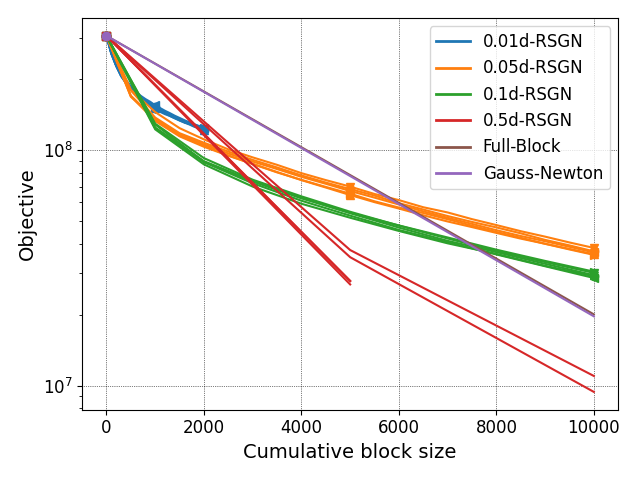}
\end{subfigure}
\caption{ARTIF (left), BRATU2D (middle) and OSCIGRNE (right) objective value against cumulative Jacobian action size for R-SGN with $3$-hashing sketching.}
\label{fig:large_hashing}
\end{figure}

\paragraph{Large scale machine learning problems}
Here we only use scaled sampling sketching matrices.
We consider logistic regressions \footnote{Here in order to fit in the non-linear least squares framework, we square the logistic losses $r_i$ in the objective}, written in the form \eqref{NLS}, by letting
$r_i(x) = \ln(1 + \exp(-y_i a_i^T x))$,
where $a_i \in \R^d$ are the observations and $y_i \in \{-1,1\}$ are the class labels; we also include a quadratic regularization term $\lambda\|x\|_2^2$ by treating it as
an additional residual.

\begin{figure}[!t]
\begin{subfigure}{\columnwidth}
\includegraphics[width=\columnwidth]{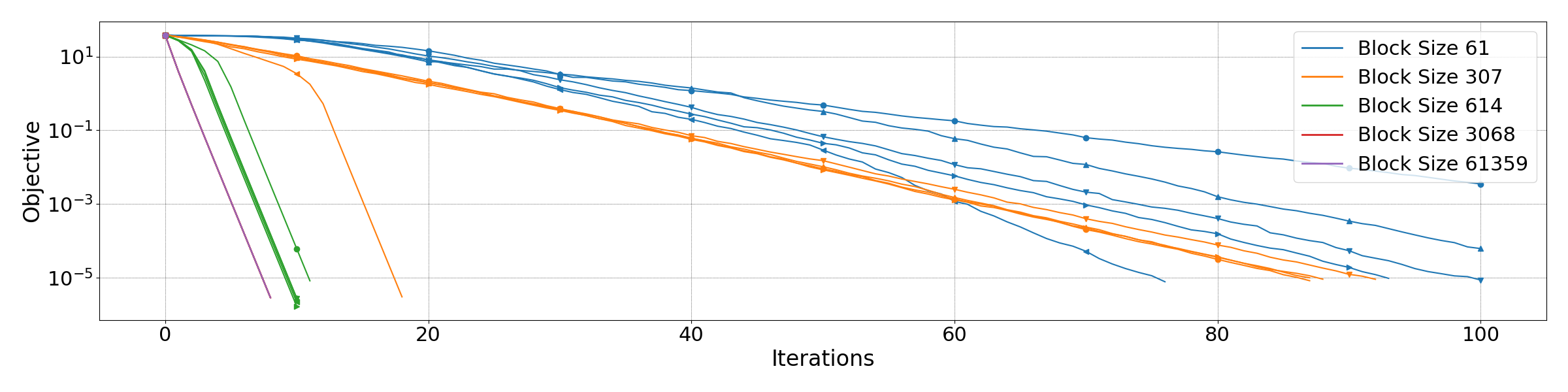}
\end{subfigure}
\begin{subfigure}{\columnwidth}
\includegraphics[width=\columnwidth]{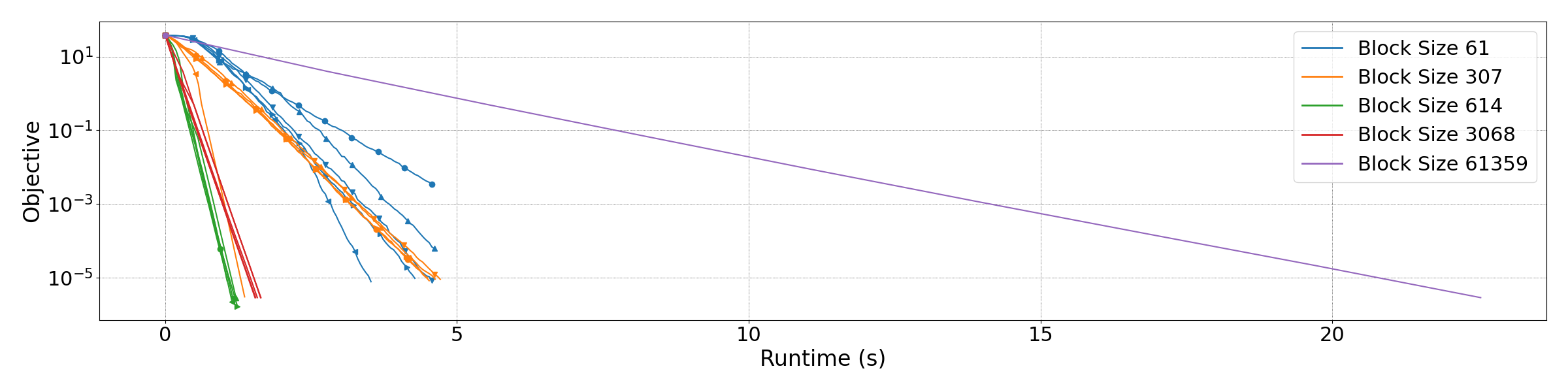}
\end{subfigure}
\caption{R-SGN on the \textsc{chemotherapy} dataset}\label{fig:chemo}
\end{figure}

\begin{figure}[!t]
\begin{subfigure}{\columnwidth}
\includegraphics[width=\columnwidth]{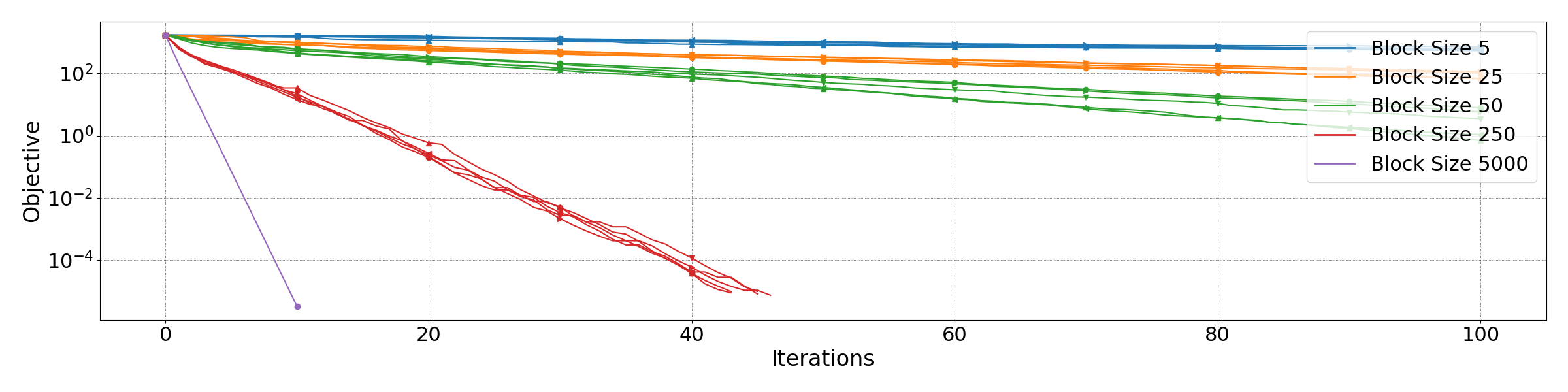}
\end{subfigure}
\begin{subfigure}{\columnwidth}
\includegraphics[width=\columnwidth]{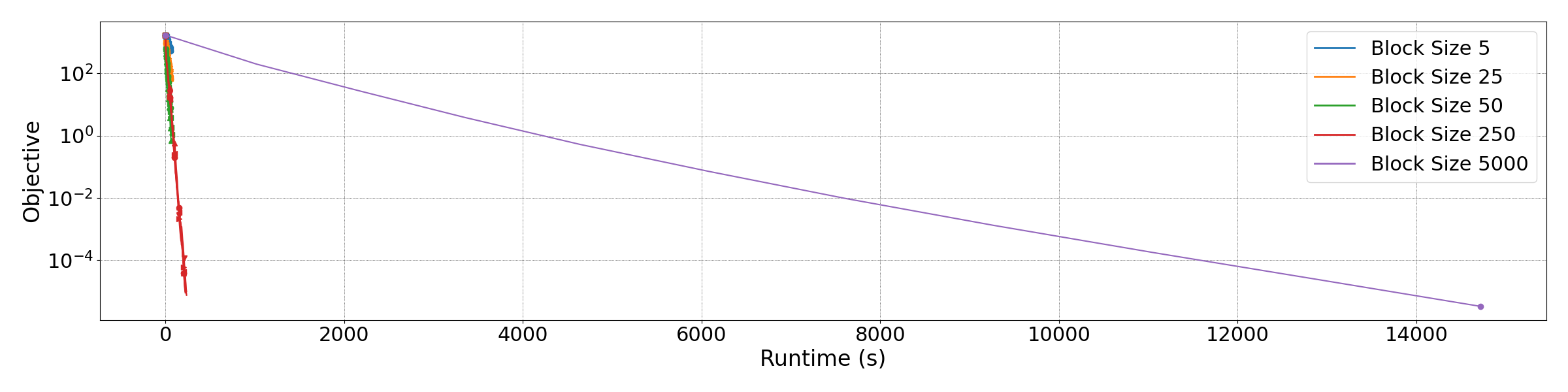}
\end{subfigure}
\caption{R-SGN on the \textsc{gisette} dataset}\label{fig:gisette}
\end{figure}

We test on the \textsc{chemotherapy} and \textsc{gisette} datasets from OpenML \cite{OpenML2013} for $100$ iterations with $\lambda=10^{-10}$, using subspace-sizes (or block-sizes, as here we are using the BC-GN variant by using the sampling sketching) of 0.1\%, 0.5\%, 1\%, 5\% and 100\% of the full-space-sizes for the $61,359$ dimensional \textsc{chemotherapy} dataset and the $5,000$ dimensional \textsc{gisette} dataset; in a similar testing setup to \cite{gower2019rsn}. We perform five runs of the algorithm for each block size starting at $x_0=0$ (and take the average performance). We terminate once the objective $f(x_k)$ goes below $10^{-5}$ and plot $f(x_k)$ against iterations and runtime in each Figure. On the \textsc{chemotherapy} dataset, we see from \autoref{fig:chemo} that we are able to get comparable performance to full Gauss-Newton ($d=61,359$ in purple) using only $1\%$ of the original block size ($l=614$ in green) at $1/20$th of the runtime.  For the \textsc{gisette} dataset, we see from \autoref{fig:gisette} that similarly, we are able to get good performance compared to GN ($d=5,000$ in purple) using $5\%$ of the original block size ($l=250$ in red) at $1/60$th of the runtime.



\renewcommand{\zK}{z_k} 
\chapter{Second order subspace methods for general objectives}
    \section{Introduction}
In this chapter we continue our investigation on subspace methods for the minimisation of general objectives. In the last chapter we saw that if the sketching matrix $S_k$ stays bounded and is sufficiently accurate to capture the gradient of the objective at the current iterate with positive probability, convergence of the subspace methods occurs at essentially the same rate as classical full-space first order methods.
It is known that if second order information of the objective function is available, cubic regularisation full-space methods achieve faster convergence rates for general non-convex objective \cite{CoraBook}. In this chapter, we first show that the same can be obtained with subspace methods. Namely, when the sketching matrix $S_k$ captures sufficiently accurate second order information, essentially the same faster rate of convergence can be achieved. 
We then show that this faster rate of convergence can be achieved also in the case of sparse second derivatives, without requiring the low rank/subspace embedding condition. 
Next, we show that a class of subspace methods converge to an approximate second order minimum in the subspace, in the sense that the subspace Hessian at the limit point is almost positive-semi-definite. 
Finally, we show that the second order subspace method with Gaussian sketching converges to an approximate second order minimum in the full space.

\section{R-ARC: random subspace adaptive cubic regularisation method}
First, we describe the random  subspace cubic regularisation algorithm (R-ARC). 

\begin{algorithm}[H]
\begin{description}

\item[Initialization] \ \\
 Choose a matrix distribution $\cal{S}$ of matrices $S\in \rLTimesD$. 
 Choose constants $\gamma_1\in (0,1)$, $\gamma_2 >1$, $\theta \in (0,1)$, $\kappaT, \kappaS \geq 0$
 and $\alpha_{\max}>0$ such that $
        \gamma_2 = \oneOverGammaOneC,
        $
    for some $c\in \N^+$.
 Initialize the algorithm by setting 
 $x_0 \in \R^d$, 
 $\alpha_0 = \alphaMax \gamma^p$
 for some $p\in \N^+$ and $k=0$.

\item[1. Compute a reduced model and a trial step] \ \\
 In Step 1 of \autoref{alg:generic}, draw a random matrix $S_k \in \R^{l \times d}$ from $\cal{S}$, and let
 \begin{align}
     \mKHat{\sHat} 
     & = \fK + \innerProduct{S_k\gradFK}{ \sHat} + 
     \frac{1}{2} \innerProduct{\sHat}{S_k\hK S_k^T\sHat}
     + \frac{1}{3\alphaK}\normTwo{S_k^T \sHat}^3 \notag \\
     & = \hat{q}_k(\hat{s}) + \frac{1}{3\alphaK}\normTwo{S_k^T \sHat}^3, \label{eqn::mKHatSpec_CB} 
 \end{align}
 where $\hat{q}_k(\hat{s})$ is the second order Taylor series of $f(x_k + S_k^T \sKHat)$ around $x_k$;
 
 Compute $\sKHat$ by approximately minimising \eqref{eqn::mKHatSpec_CB} such that
 \begin{align}
     \mKHat{\sKHat} \leq \mKHat{0} \label{tmp:CBGN:3}\\
     \normTwo{\grad \mKHat{\sKHat}} \leq 
     \kappa_T \normTwo{S_k^T \sKHat}^2 \label{tmp:CBGN:5} \\
     \grad^2 \mKHat{\sKHat} \succeq -\kappaS \normTwo{S_k^T \sKHat}, \label{hessMKGeq0}
 \end{align}
 where we may drop \eqref{hessMKGeq0} if only convergence to a first order critical point is desired. 
 
 Compute a trial step 
 \begin{align}
     s_k = w_k(\sHat_k) = S_k^T \sHat_k, \label{eqn::wKSpec_CB}
 \end{align}

\item[2. Check sufficient decrease]\ \\  
In Step 2 of \autoref{alg:generic}, check sufficient decrease as defined by the condition
\begin{equation}
    \fK - \fKPlusOne \geq \theta \squareBracket{\hat{q}_k(0) - 
    \hat{q}_k(\hat{s})
    }, \label{eq:cubic:sufficient_decrease}
\end{equation}

\item[3, Update the parameter $\alphaK$ and possibly take the trial step $\sK$]\ \\
If \eqref{eq:cubic:sufficient_decrease} holds, set $\xKPlusOne = \xK + \sK$ and $\alphaKPlusOne = \min \set{\alphaMax, \gammaTwo\alphaK}$ [successful iteration]. 

Otherwise set $\xKPlusOne = \xK$ and $\alphaKPlusOne = \gammaOne \alphaK$ [unsuccessful iteration]

Increase the iteration count by setting $k=k+1$ in both cases.

\caption{\bf{Random subspace cubic regularisation algorithm (R-ARC) }} \label{alg:CBGN} 
\end{description}
\end{algorithm}

Here note that \autoref{alg:CBGN} is a specific form of \autoref{alg:generic}. Therefore the convergence result in \autoref{thm2} can be applied, provided that the four assumptions of the theorem can be shown to hold here. In the remaining sections of this chapter, we give different definitions of the two key terms in the convergence result \autoref{thm2}: $\nEps$ and true iterations. These lead to different requirements for the matrix distribution $\cal{S}$, and iteration complexities to drive $\normTwo{\gradFK} < \epsilon$, $\lambdaMin{S_k \hessFK S_k^T} > -\epH $ and/or $\lambdaMin{\hessFK} > -\epH$

Compared to \autoref{alg:sketching}, \autoref{alg:CBGN} lets $B_k$ be the Hessian at the iterate $\hessFK$, although we only need it in the form of $S_k \hessFK S_k^T$ so that the full Hessian never needs to be computed. 
Furthermore, we let $\sKHat$, the reduced step, be computed by minimising a cubically regularised subspace model, corresponding to the classical approaches in \cite{Cartis:2009fq}, also see Section 1.4.2 on page \pageref{page:second:order:cubic:classical}.
As in the corresponding full-space method, the combination of the availability of second order information and the cubic regularisation term leads to improved iteration complexity for our subspace methods to drive $\gradFK < \epsilon$ and convergence to a second order critical point. 

\reply{
\begin{remark}
Two strategies for computing $\sKHat$ by minimising \eqref{eqn::mKHatSpec_CB} are given in \cite{Cartis:2009fq}, either requiring a factorisation of $S_k \hessFK S_k^T$ (in a Newton-like algorithm), or repeated matrix-vector products involving $S_k \hessFK S_k^T$ (in a Lanczos-based algorithm). Although we note that the iteration complexity, and the evaluation complexity of $f$ and its derivatives (which are the focuses of this chapter) are unaffected by the computation complexity of calculating $\sKHat$, \autoref{alg:CBGN} significantly reduces the computation of this inner problem by reducing the dimension of the Hessian from $d \times d$ to $l \times l$ comparing to the full-space counterpart. (In addition to reducing the gradient and Hessian evaluation complexity per iteration.) 
\end{remark}
}

    \section{Fast convergence rate assuming subspace embedding of the Hessian matrix} \label{CBGN:sec2}
    Our first convergence result shows \autoref{alg:CBGN} drives $\normTwo{\gradFK}$ below $
\epsilon$ in $\mathO{\epsilon^{-3/2}}$ iterations, given that $\cal{S}$ has an embedding property (a necessary condition of which is that $\cal{S}$ is an oblivious subspace embedding for matrices of rank $r+1$, where $r$ is the maximum rank of $\hessFK$ across all iterations).

\paragraph{Define $\nEps$ and true iterations based on (one-sided) subspace embedding}
In order to prove convergence of \autoref{alg:CBGN}, we show that \autoref{AA2}, \autoref{AA3}, \autoref{AA4}, \autoref{AA5} that are needed for \autoref{thm2} to hold are satisfied. To this end, we first define $\nEps$, the criterion for convergence, as $\min \{k: \normTwo{\grad f(x_{k+1})} \leq \epsilon \} $.
Then, we define the true iterations based on achieving an embedding of the Hessian and the gradient.  

\begin{definition} \label{def:true:CBGN}
Let $\epSTwo \in (0,1)$, $\sMax >0$. Iteration $k$ is ($\epSTwo, \sMax)$-true if
\begin{align}
    &\normTwo{S_k \MK \zK}^2 \geq (1-\epSTwo) \normTwo{\MK \zK}^2, \texteq{for all $z_k \in \R^{d+1}$} \label{tmp:CBGN:1}\\
    &\normTwo{S_k} \leq \sMax, \label{tmp:CBGN:10}
\end{align}

where $\MK = \squareBracket{\gradFK \quad \hessFK} \in \R^{d \times (d+1)}$. Note that all vectors are column vectors.
\end{definition}

\begin{remark}
\eqref{tmp:CBGN:1} implies 
\begin{equation}
    \normTwo{S_k \gradFK}^2 \geq (1-\epSTwo) \normTwo{\gradFK}^2, \label{tmp:CBGN:2}
\end{equation}

by taking $\zK = [1, 0, \dots, 0]^T$. Thus \autoref{def:true:CBGN} is a stronger condition than \autoref{def::true_iters}, the definition of true iterations for our convergence result of first order subspace methods.
\end{remark}

\subsection{Auxiliary results}

In this subsection we provide some useful results needed to prove our assumptions in \autoref{thm2}.


\begin{lemma} \label{succStepDecrease}
In \autoref{alg:CBGN}, if iteration $k$ is successful, then
\[f(x_{k+1}) \leq \fun{f}{x_k} -\frac{\theta}{3\alphaK} \normTwo{S_k^T \sKHat}^3 \].
\end{lemma}

\begin{proof}
From the definition of successful iterations and \eqref{eq:cubic:sufficient_decrease}
\begin{align}
    f(x_{k+1}) 
    &= f(x_k + s_k) \notag  \\
    & \leq \fK - \theta \squareBracket{\mKHat{0} - \mKHat{\sKHat}} - \frac{\theta}{3\alphaK} \normTwo{S_k^T \sKHat}^3 \notag \\
    & \leq \fK - \frac{\theta}{3\alphaK}\normTwo{S_k^T \sKHat}^3, \label{tmp:CBGN:4}
\end{align}
where in the last inequality, we used \eqref{tmp:CBGN:3}.
\end{proof}

The gradient of the model has the expression
\begin{equation}
    \grad \mKHat{\sKHat} = S_k \gradFK + S_k \hessFK S_k^T \sKHat + 
    \frac{1}{\alphaK} S_k S_k^T \sKHat \normTwo{\SK^T \sKHat}. 
    \label{mkGradCubic}
\end{equation}

The following lemma bounds the size of the step at true iterations.
\begin{lemma} \label{stepSizeSubEmbed}
Assume $f$ is twice continuosly differentiable with $\LH$-Lipschitz Hessian $\grad^2 f$ 
and $k < \nEps$.
Suppose that iteration $k$ is ($\epSTwo, \sMax)$-true. We have
\begin{equation}
\NsKTSKHat^2 \geq \frac{\epsilon}{2}
\min\set{\frac{2}{\LH}, 
\bracket{\frac{1}{\alphaK}\sMax + \kappaT}^{-1}
\sqrt{1 - \epSTwo}} \label{Lemma5.2.2_eqn}
\end{equation}

\end{lemma}

\begin{proof}
\eqref{mkGradCubic} and the triangle inequality give
\begin{align}
    \normTwo{\SK \gradFK + \SK \hessFK \SK^T \sKHat}  \notag
    &= \normTwo{\frac{1}{\alphaK}\SK \SK^T \sKHat \normTwo{\SK^T \sKHat} - \grad\mKHat{\sKHat}} \notag \\
    & \leq \frac{1}{\alphaK}\normTwo{\SK}\normTwo{\SK^T\sKHat}^2 + \normTwo{\grad \mKHat{\sKHat}}\notag \\
    & \leq \bracket{\frac{1}{\alphaK}\normTwo{\SK}+\kappaT} \normTwo{\SK^T \sKHat^2}^2 \texteq{by \eqref{tmp:CBGN:5}} \\
    & \leq \bracket{\frac{1}{\alphaK}\sMax+\kappaT} \normTwo{\SK^T \sKHat^2}^2,
    \label{tmp:CBGN:9}
\end{align}
where we used \eqref{tmp:CBGN:10}.
On the other hand, we have that
\begin{align}
    & \normTwo{\SK \gradFK + \SK \hessFK \SK^T \sKHat} \notag \\
    & = \normTwo{\SK\MK \squareBracket{1, 
    (\SK^T \sKHat)^T}^T} \notag \\
    & \geq \sqrt{\oneMinusEpSTwo} 
    \normTwo{\gradFK + \hessFK s_k} 
    \texteq{by \eqref{tmp:CBGN:1} with $z_k
    = \squareBracket{1, (\SK^T \sKHat)^T}^T$} 
    \notag\\
    & = \sqrt{\oneMinusEpSTwo} \normTwo{ \grad f(x_{k+1}) - \squareBracket{\grad f(x_{k+1}) - \gradFK - \hessFK s_k}} \\
    & \geq \sqrt{\oneMinusEpSTwo} 
        \left| 
            \normTwo{ \grad f(x_{k+1})} - 
            \normTwo{\squareBracket{\grad f(x_{k+1}) -     \gradFK - \hessFK s_k}} 
        \right|
    \label{tmp:CBGN:8}
\end{align}
Note that by Taylor's Theorem, because $f$ is twice continuously differentiable with $\LH$-Lipschitz $\grad^2 f$, we have that $\grad f (x_k + s_k) = \gradFK + \int_0^1 \grad^2 f(x_k + ts_k) s_k dt$. 
Therefore, we have 
\begin{align}
    \normTwo{\grad f(x_{k+1}) - \gradFK - \hessFK \sK} & 
    = \normTwo{\int_0^1 \squareBracket{\grad^2 f(x_k + ts_k) - \hessFK}s_k dt} \\
    & \leq \int_0^1 \normTwo{s_k} \normTwo{\grad^2 f(x_k + ts_k) - \hessFK} dt \\
    & \leq \normTwo{s_k} \int_0^1 \LH t \normTwo{s_k} dt \\
    & = \frac{1}{2}\LH\normTwo{s_k}^2
\end{align}
by Lipschitz continuity of $\grad^2 f$.
Next we discuss two cases,
\begin{enumerate}
    \item If $\LH \normTwo{\sK}^2 > \epsilon$, 
    then we have the desired result in \eqref{Lemma5.2.2_eqn}.
     
    \item If ${\LH} \normTwo{\sK}^2 \leq \epsilon$, then \eqref{tmp:CBGN:8}, 
    and the fact that $\normTwo{\grad f(x_{k+1})} \geq \epsilon$ 
    by $k < \nEps$, imply that 
        \[
            \normTwo{\SK\gradFK + \SK \hessFK \SK^T \sKHat} 
            \geq \sqrt{\oneMinusEpSTwo}\frac{\epsilon}{2}.
        \]
    Then \eqref{tmp:CBGN:9} implies
        \begin{equation}
            \normTwo{\SK^T \sK}^2 \geq \bracket{\frac{1}{\alphaK}\sMax + \kappaT}^{-1} 
            \sqrt{1 - \epSTwo}\frac{\epsilon}{2}.
            \notag
        \end{equation}
    This again gives the desired result.
\end{enumerate}

\end{proof}

\subsection{Satisfying the assumptions of \autoref{thm2}}
Here we only address the case where $\cal{S}$ is the distribution of scaled Gaussian matrices. But $\cal{S}$ could also be the distribution of scaled sampling matrices, $s$-hashing matrices, SRHT matrices and HRHT matrices because those distributions also satisfy similar properties detailed below, namely, having a bounded two-norm with high probability (\autoref{lem:GaussSMax:CubicSubspace}), and having a one-sided subspace embedding property (\autoref{lem:Gauss_embedding}).

Concerning scaled Gaussian matrices, we have the following results.
\begin{lemma}[\autoref{Lem:GaussSMax}] \label{lem:GaussSMax:CubicSubspace}
Let $S \in \R^{l\times d}$ be a scaled Gaussian matrix (\autoref{def:Gaussian}). Then for any $\deltaSTwo > 0$, 
$S$ satisfies \eqref{tmp:CBGN:10} with probability $1-\deltaSTwo$ and 

\begin{equation}
    \sMax = 1 + \sqrt{\frac{d}{l}} + \sqrt{\frac{2\logOneOverDeltaSTwo}{l}}. \notag
\end{equation}

\end{lemma}

\begin{lemma}[Theorem 2.3 in \cite{10.1561/0400000060}] \label{lem:Gauss_embedding}
Let $\epSTwo \in (0,1)$ and $S \in \R^{l \times d}$ 
be a scaled Gaussian matrix. 
Then for any fixed $d \times (d+1)$ matrix $M$ with rank at most $r+1$, 
with probability $1-\deltaSThree$ we have that simultaneously for all 
$z \in \R^{d+1}$, $\normTwo{SMz}^2 \geq (1-\epSTwo) \normTwo{Mz}^2$, 
where

\begin{equation}
    \deltaSThree = \deltaSThreeExpression \label{deltaSThree:eq}
\end{equation}

and $C_l$ is an absolute constant. 
\end{lemma}

\paragraph{Satisfying \autoref{AA2} (page \pageref{AA2})}
\begin{lemma}
    Suppose that $\hessFK$ has rank at most $r\leq d$ for all $k$; $S \in \R^{l \times d}$ is drawn as a scaled Gaussian matrix. Let $\epSTwo, \deltaSTwo \in (0,1)$ such that $\deltaSTwo + \deltaSThree < 1$ where $\deltaSThree$ is defined in \eqref{deltaSThree:eq}.
    Then \autoref{alg:CBGN} satisfies \autoref{AA2} with $\deltaS = \deltaSTwo + \deltaSThree$ and $S_{max} = 1 + \sqrt{\frac{d}{l}} + \sqrt{\frac{2\logOneOverDeltaSTwo}{l}}$, with true iterations defined in \autoref{def:true:CBGN}.
\end{lemma}

\begin{proof}
Let $x_k = \barXK \in \R^d$ be given. This determines $\gradFK, \hessFK$ and hence $M_k$. As $\hessFK$ has rank at most $r$, $M_k$ has rank at most $r+1$. 
Consider the events
\begin{align*}
    &\aKOne = \set{\normTwo{S_k M_k z}^2 \geq (1-\epSTwo) \normTwo{\MK z}^2, \quad \forall z\in \R^{d+1}} \\
    &\aKTwo = \set{\normTwo{S_k} \leq S_{max}}.
\end{align*}
Note that iteration $k$ is true if and only if $\aKOne$ and $\aKTwo$ occur. It follows from \autoref{lem:Gauss_embedding} that $\probabilityGivenXK{\aKOne} \geq 1-\deltaSThree$; and from \autoref{lem:GaussSMax:CubicSubspace} that $\probability{\aKTwo} \geq 1-\deltaSTwo$. Since $\aKTwo$ is independent of $x_k$, we have $\probabilityGivenXK{\aKTwo} = \probability{\aKTwo} \geq 1-\deltaSTwo$. 

Hence, we have $\probabilityGivenXK{{\aKOne}\intersect{\aKTwo}} \geq 1 - \probabilityGivenXK{\complement{\aKOne}} - \probabilityGivenXK{\complement{\aKTwo}} \geq 1- \deltaSTwo - \deltaSThree$. A similar argument shows that $\probability{{\aZeroOne}\intersect{\aZeroTwo}} \geq 1 - \deltaSTwo - \deltaSThree$, as $x_0$ is fixed. 

Moreover, given $x_k = \barXK$, $\aKOne$ and $\aKTwo$ only depend on $S_k$, which is drawn randomly at iteration $k$. Hence given $x_k = \barXK$, ${\aKOne}\intersect{\aKTwo}$ is independent of whether the previous iterations are true or not. Hence \autoref{AA2} is true.
\end{proof}

\paragraph{Satisfying \autoref{AA3} (page \pageref{AA3})}

\begin{lemma}\label{lem:cubic:sub:A2}
Let $f$ be twice continuously differentiable with $\LH$-Lipshitz continuous Hessian $\grad^2 f$. \autoref{alg:CBGN} satisfies \autoref{AA3} with
    \begin{equation}
        \alphaLow = \frac{2(1-\theta)}{\LH} \label{eq:alphaLow:CBGN}
    \end{equation}
\end{lemma}

\begin{proof}
From \eqref{tmp:CBGN:3}, we have that 
    \begin{equation}
        f(x_k) - \qKHat{\sKHat} \geq \frac{1}{3\alphaK} 
        \normTwo{\SK^T \sKHat}^3.
        \notag
    \end{equation}
Using \autoref{stepSizeSubEmbed}, in true iterations with $k < \nEps$, we have that 
$
    \normTwo{\SK^T \sKHat} >0.\notag
$
Therefore we can define \footnote{Note that \autoref{alg:CBGN} does not use the ratio $\rho_k$ in \eqref{eqn:rho_k:CBGN}, but uses \eqref{eq:cubic:sufficient_decrease}. This is because the denominator of \eqref{eqn:rho_k:CBGN} may be zero before termination, on account of sketching/subspace techniques being used.}
    \begin{equation}
        \rho_k = \frac{f(x_k) - f(x_k + s_k)}{f(x_k) - \qKHat{\sKHat}},
        \label{eqn:rho_k:CBGN}
    \end{equation}
with 
    \begin{equation}
        \abs{1 - \rho_k} = \frac{\abs{f(x_k + s_k) - \qKHat{\sKHat}}}
        {\abs{f(x_k) - \qKHat{\sKHat}}}. \notag
    \end{equation}
The numerator can be bounded by 
    \begin{align}
        \abs{f(x_k + s_k) - \qKHat{\sKHat}}  \leq \frac{1}{6}\LH \normTwo{\sK}^2,\notag
    \end{align}
    by Corollary A.8.4 in \cite{CoraBook}.
Therefore, we have 
\begin{equation}
    \abs{1 - \rho_k} \leq \frac{\frac{1}{6}\LH \normTwo{\sK}^3}{\frac{1}{3\alphaK}
    \normTwo{\sK}^3} = \frac{1}{2}\alphaK \LH \leq 1-\theta 
    \texteq{by \eqref{eq:alphaLow:CBGN} and $\alphaK \leq \alphaLow$}. 
\end{equation}

Thus $1-\rho_k \leq \abs{1-\rho_k} \leq 1-\theta$ so $\rho_k \geq \theta $ and iteration $k$ is successful.
\end{proof}

\paragraph{Satisfying \autoref{AA4} (page \pageref{AA4})}

\begin{lemma}
    Let $f$ be twice continuously differentiable with $\LH$-Lipschitz continuous Hessian. 
    \autoref{alg:CBGN} with true iterations defined in \autoref{def:true:CBGN} satisfies \autoref{AA4} with
    
    \begin{equation}
        h(\epsilon, \alphaK) = \frac{\theta}{3\alphaMax} 
        \bracket{\frac{\epsilon}{2}}^{3/2} \min \set{ \frac{2^{3/2}}{\LH^{3/2}}, \bracket{\frac{\sqrt{\oneMinusEpSTwo}}{\frac{1}{\alphaK}\sMax +\kappaT}}^{3/2}}.
        \label{eq:hEpsAlphaCB}
    \end{equation}
    
\end{lemma}

\begin{proof}
For true and successful iterations with $k < \nEps$, use \autoref{stepSizeSubEmbed} with \autoref{succStepDecrease} and $\alphaK \leq \alphaMax$.
\end{proof}

\paragraph{Satisfying \autoref{AA5} (page \pageref{AA5})}
The next lemma shows that the function value following \autoref{alg:CBGN} is non-increasing.

\begin{lemma}\label{tmp-2022-1-15-1}
\autoref{alg:CBGN} satisfies \autoref{AA5}.
\end{lemma}

\begin{proof}
In \autoref{alg:CBGN}, we either have $x_{k+1} = x_k$ when the step is unsuccessful, 
in which case $f(x_k) = f(x_{k+1})$; or the step is successful,
in which case we have $\fun{f}{x_{k+1}} - \fun{f}{x_k} \leq 0$ 
by \autoref{succStepDecrease}.
\end{proof}

\subsection{Iteration complexity of \autoref{alg:CBGN} to decrease $\gradFK$ below $\epsilon$}
We have shown that \autoref{alg:CBGN} satisfies \autoref{AA2}, \autoref{AA3}, \autoref{AA4} and \autoref{AA5}. Noting that \autoref{alg:CBGN} is a particular case of \autoref{alg:generic}, we apply \autoref{thm2} to arrive at the main result of this section.

\begin{theorem} \label{thm:CBGN_subspace_first}
    Let $\cal{S}$ be the distribution of scaled Gaussian matrices $S \in \R^{l \times d}$ defined in \autoref{def:Gaussian}. Suppose that $f$ is bounded below by $f^*$, twice continuously differentiable with $\LH$-Lipschitz $\grad^2 f$, $\hessFK$ has rank at most $r$ for all $k$ and let $\epsilon>0$.
    Choose $l = 4 C_l (\log16 + r + 1); 
    \epSTwo=\frac{1}{2}; 
    \deltaSTwo = \frac{1}{16};
    $ so that
    $\deltaSThree = \deltaSThreeExpression = \frac{1}{16};
    \deltaS = \frac{1}{8}; 
    \sMax = 1 + 
    \frac{ \sqrt{d} + \sqrt{2\log16} }
    {\sqrt{4 C_l \bracket{ \log 16 + r + 1 } }}$, where $C_l$ is defined in \eqref{deltaSThree:eq}. Run \autoref{alg:CBGN} for $N$ iterations. 
	Suppose that $ \delta_S <\frac{c}{(c+1)^2}$ (i.e. 
	$\frac{c}{(c+1)^2} > \frac{1}{8}$). 
	Then for any $\delta_1 \in (0,1)$ with 
    \begin{equation}
        g(\deltaOne) >0, \nonumber
    \end{equation}
    where 
    \begin{equation}
        g(\deltaOne) = \nPreFactorTRCubic, \nonumber
    \end{equation}
    if $N\in \N$ satisfies 
    \begin{equation}
        N \geq g(\deltaOne) \squareBracket{
             \fZeroMinusfStarOverH
             + \frac{4 C_l (\log16 + r + 1)}{1+c}}, \nonumber
    \end{equation}
    where $h(\epsilon, \alpha_k)$ is defined in \eqref{eq:hEpsAlphaCB} with $\epSTwo, \sMax$ defined in the theorem statement, $\alphaLow$ is given in \eqref{eq:alphaLow:CBGN} and
     $\alphaMin = \alphaZero \gammaOne^\newL$ associated with $\alphaLow$,  for some $\newL \in \N^+$.
    Then we have that
    \begin{equation}
        \probability{\min_{k\leq N} \{\normTwo{\grad f(x_{k+1})}\} \leq \epsilon } \geq 1 - e^{-\frac{7\delta_1^2}{16} N}. \nonumber
    \end{equation}
\end{theorem}

\subsubsection{Discussion}

\paragraph{Use other random ensembles than the scaled Gaussian matrices in \autoref{alg:CBGN}}
Although \autoref{thm:CBGN_subspace_first} requires $\cal{S}$ to be the distribution of scaled Gaussian matrices, qualitatively similar result, namely, convergence with a rate of $\mathO{\epsilon^{-3/2}}$ with exponentially high probability, can be established for $s$-hashing matrices (defined in \autoref{def::sampling_and_hashing}), Subsampled Randomised Hadamard Transforms (defined in \autoref{def::SRHT}) and Hashed Randomised Hadamard Transforms (defined in \autoref{def::HRHT}). The proof for satisfying \autoref{AA2} needs to be modified, using the upper bounds for $S_{max}$ and the subspace embedding properties of these ensembles instead. Consequently, the constants in \autoref{thm:CBGN_subspace_first} will change, but the convergence rate and the form of the result stays the same (as the results in Section \ref{iter_complexity_QR}).

\paragraph{Comparison with the adaptive cubic regularisation method with random models in \cite{Cartis:2017fa}}

We achieve the same $\mathO{\epsilon^{-3/2}}$ convergence rate as 
\cite{Cartis:2017fa}, which is optimal for non-convex optimisations using second order models \cite{CoraBook}, and the same for deterministic adaptive cubic regularisation method. 
One main difference between our work and \cite{Cartis:2017fa} is the definition of true iterations. Instead of \autoref{def:true:CBGN}, they define true iterations \reply{as those iterations that satisfy}
    \begin{align}
    & \normTwo{\gradFK - \grad m_k(\sK)} \leq \kappa_g \normTwo{\sK}^2 \label{Katya_second_order_cond1} \\
    & \normTwo{\hessFK - \grad^2 m_k(\sK)} \leq \kappa_H \normTwo{\sK},
    \end{align}
where $\kappa_g, \kappa_H >0$ are constants. 

This difference leads to different potential applications of the two frameworks. In their work, they proposed to use sampling with adaptive sample sizes for problems having the finite sum structure ($f = \sum_i f_i)$ to construct the model $m_k$, or to use finite differences in the context of derivative free optimisation to construct the model $m_k$. However, without other assumptions, even just in order to obtain condition \eqref{Katya_second_order_cond1}, one may need a sample size that may be impractically large. 
In contrast, in our framework, the sketching size is fixed and even then, true iterations happen sufficiently frequently for scaled Gaussian matrices (and indeed for other random embeddings, see remark above). However, since the subspace dimension $l$ is proportional to the rank of the Hessian matrix $r$, the Hessian matrix $\grad^2 f$ is assumed to have a lower rank $r$ than the full space dimension $l$, as otherwise \autoref{alg:CBGN} does not save computation/gradient/Hessian evaluations compared to the deterministic version. 
Another difference is that our convergence result \autoref{thm:CBGN_subspace_first} is expressed in the high probability form, while the result in \cite{Cartis:2017fa} is in expectation. Our result is stronger because it leads to an equivalent expectation result in \cite{Cartis:2017fa}, see \autoref{Cor_3} on Page \pageref{Cor_3}.

Inexact local models constructed by subsampling for sums of functions have also been proposed for cubic regularization and other Newton-type methods in \cite{KohlLucc17, XuRoosMaho17, XuRoosMaho18, YaoXuRoosMaho18}. Our emphasis here is related to reducing specifically the dimension of the variable domain (rather than the observational space).

    \section{Fast convergence rate assuming the sparsity of the Hessian matrix}
    This section is mostly conceptual and is an attempt to show the fast convergence rate of \autoref{alg:CBGN} can be achieved without assuming subspace embedding of the Hessian matrix. 
Here, we maintain $\nEps$ as $\min \{k: \normTwo{\grad f(x_{k+1})} \leq \epsilon \} $, similarly to the last section. 
However, in the definition of true iterations, we replace the condition \eqref{tmp:CBGN:1} on subspace embedding of the Hessian with the condition that the sketched Hessian $S_k \hessFK$ \reply{has a small norm}. This may be achieved when the Hessian matrix has sparse rows and we choose $S_k$ to be a scaled sampling matrix.
We show that this new definition of true iterations still allows the same $\mathO{\epsilon^{-3/2}}$ iteration complexity to drive the norm of the objective's gradient norm below $\epsilon$.
Specifically, true iterations are defined as follows.

\begin{definition} \label{def:true:CBGN:SparseHess}
Let $\epS \in (0,1)$, $\sMax >0$. Iteration $k$ is ($\epS, \sMax)$-true if 
\begin{align}
    &\normTwo{S_k \hessFK} \leq c_k \epsHalf, 
    \label{smallSketchedHessian} \\
    &\normTwo{S_k \gradFK}^2 \geq \bracket{1-\epS} \epsilon^2, 
    \label{skGradLowerBound_sparse_Hess} \\
    &\normTwo{S_k} \leq \sMax \label{S_k_norm_bound_Sparse_hess},
\end{align}
where $c_k = \sqrt{\frac{4\bracket{1-\epS}^{1/2} \sMax}{3\alphaMax}}$ and $\alphaMax$ is a user-chosen constant in \autoref{alg:CBGN}. 
\end{definition}

Note that the desired accuracy $\epsilon$ appears in this particular definition of true iterations. 

Consequently, the requirements on the objective and the sketching dimension $l$ may be stronger for smaller $\epsilon$. For simplicity, we assume $\kappa_T = 0$ (where $\kappaT$ is a user chosen parameter in \eqref{tmp:CBGN:5} in \autoref{alg:CBGN}) in this section, namely $\grad \mKHat{\sKHat} = 0$, and it follows from \eqref{tmp:CBGN:5} that
\begin{equation}
    S_k \gradFK = \frac{1}{\alphaK} S_k S_k^T 
    \sKHat \normTwo{S_k^T \sKHat} - S_k\hessFK S_k^T \sKHat.
    \label{exactModelGradZero}
\end{equation}


The proofs that \autoref{AA3} and \autoref{AA5} are satisfied are identical to the previous section, while the following technical lemma helps us to 
satisfy \autoref{AA4}. 

\begin{lemma}\label{spHessLem}
Let $\epsilon>0$.
Let $\epS \in (0,1)$, $\kappaT=0$. 
Suppose we have \eqref{smallSketchedHessian}, \eqref{skGradLowerBound_sparse_Hess} and
\eqref{S_k_norm_bound_Sparse_hess}.
Then 
\begin{equation}
    \normTwo{S_k^T \sKHat} \geq \alphaK
    \sqrtFrac{\bracket{1-\epS}^{1/2}\epsilon}{3\sMax\alphaMax}. \notag
\end{equation}
\end{lemma}

\begin{proof}
Let $b = \normTwo{S_k \hessFK}, x = \normTwo{S_k^T \sKHat}$,
then taking 2-norm of \eqref{exactModelGradZero} with
\eqref{skGradLowerBound_sparse_Hess}, $\normTwo{S_k}\leq \sMax$ and the triangle inequality
gives 
\begin{align}
    & \oneMinusEpSHalf \epsilon \leq \normTwo{S_k \gradFK}
    \leq \frac{\sMax}{\alphaK}x^2 + bx \notag\\
    \implies 
    &\frac{\sMax}{\alphaK} x^2 + bx - \oneMinusEpSHalf\epsilon \geq 0
    \notag \\
    \implies 
    & x^2 + \frac{\alphaK b}{\sMax}x - \frac{\oneMinusEpSHalf\epsilon
    \alphaK}{\sMax} \geq 0 \notag \\
    \implies
    & \bracket{x + \frac{\alphaK b}{2\sMax}}^2 
    \geq \frac{\oneMinusEpSHalf\epsilon\alphaK}{\sMax}
    + \frac{\alphaK^2 b^2}{4\sMax^2} \notag \\
    \impliesSince{x,b\geq 0} 
    & x \geq \sqrt{
    \frac{\oneMinusEpSHalf\epsilon\alphaK}{\sMax}
    + \frac{\alphaK^2 b^2}{4\sMax^2} } - \frac{\alphaK b}{2\sMax}.
    \notag
\end{align}

Introduce $a = \frac{\oneMinusEpSHalf\epsilon\alphaK}{\sMax}$ and the function $
y(b) = \frac{\alphaK b}{2\sMax}$, then the above gives
\begin{equation}
    x \geq \sqrt{a + y(b)^2} - y(b). \label{tmp-2022-1-14-9}
\end{equation}

We note, by taking derivative, that given $y(b) \geq 0$, the RHS of \eqref{tmp-2022-1-14-9} is monotonically decreasing
with $y(b)$. Therefore given $b\leq c_k 
\epsilon^{\frac{1}{2}}$ and thus $y(b) \leq y(c_k \epsilon^{\frac{1}{2}})$, we have
\begin{equation}
    x \geq \sqrt{a + y(c_k \pow{\epsilon}{\frac{1}{2}})^2}
    - y(c_k \pow{\epsilon}{\frac{1}{2}}).
    \notag
\end{equation}

The choice of $c_k = \sqrt{\frac{4\bracket{1-\epS}^{1/2} \sMax}{3\alphaMax}} \leq \sqrt{\frac{4\bracket{1-\epS}^{1/2} \sMax}{3\alpha_k}}$ gives 
$a \geq 3 y\bracket{c_k \epsilon^{\frac{1}{2}}}^2$.
And therefore we have $x \geq \fun{y}{c_k \pow{\epsilon}{\frac{1}{2}}}$. Noting that $x = \normTwo{S_k^T \sKHat}$ and substituting the expression for $y$ and $c_k$ gives the desired result. 
    
\end{proof}

\begin{lemma}
Following the framework of \autoref{alg:CBGN} 
with $\kappaT = 0$, let $\epsilon > 0.$ Define
true iterations as iterations
that satisfy \eqref{smallSketchedHessian}, \eqref{skGradLowerBound_sparse_Hess},
and 
\eqref{S_k_norm_bound_Sparse_hess}
. Then
\autoref{AA4} is satisfied with
\begin{equation}
    h(\epsilon, \alphaK) = 
    \frac{\theta \alphaK^2 \epsilon^{3/2}}{3}\squareBracket{
        \frac{\bracket{1-\epS}^{1/2}}{3\sMax\alphaMax}
    }^{3/2}. \label{tmp-2022-1-14-10}
\end{equation}

\end{lemma}

\begin{proof}
A true and successful iteration $k$ gives
\begin{equation}
    \fK - f(x_k +s_k) \geq \frac{\theta}{3\alphaK}\normTwo{S_k^T\sKHat}^3
    \notag
\end{equation}

by \autoref{succStepDecrease} and combining with
the conclusion of \autoref{spHessLem} gives the result. 
\end{proof}

With \autoref{AA3}, \autoref{AA4} and \autoref{AA5} satisfied, applying \autoref{thm2} gives the following result for \autoref{alg:CBGN}.

\begin{theorem}
    Let $f$ be bounded below by $f^*$ and twice continuously differentiable with $\LH$-Lipschitz continuous Hessian. Run \autoref{alg:CBGN} for $N$ iterations. Suppose \autoref{AA2} hold with $\deltaS\in (0,1)$ and true iterations defined in \autoref{def:true:CBGN:SparseHess}. Suppose $\deltaS < \frac{c}{(c+1)^2}$.
    
    Then for any $\deltaOne \in (0,1)$ such that $g(\deltaS, \deltaOne) >0$ where $g(\deltaS, \deltaOne)$ is defined in \eqref{eqn:gDeltaSDeltaOneDef}. If $N$ satisfies
    \begin{equation}
         N \geq \gDeltaSDeltaOne \squareBracket{
         \fZeroMinusfStarOverH
         + \frac{\newL}{1+c}},
    \end{equation}
    where $h$ is given in \eqref{tmp-2022-1-14-10}, $\alphaLow$ is given in \eqref{eq:alphaLow:CBGN} and $\alphaMin, \newL$ are given in \autoref{lem::alphaMin}; then we have
    \begin{equation}
        \probability{\min_{k\leq N} \{\normTwo{\grad f(x_{k+1})}\} \leq \epsilon } \geq 1 - e^{-\frac{\delta_1^2}{2} (1-\deltaS)N}. \nonumber
    \end{equation}
\end{theorem}

\begin{remark}
In order to satisfy \autoref{AA2}, we require that at each iteration, with positive probability, \eqref{smallSketchedHessian}, \eqref{skGradLowerBound_sparse_Hess} and \eqref{S_k_norm_bound_Sparse_hess} hold. This maybe achieved for objective functions whose Hessian only has a few non-zero rows, with $S$ being a scaled sampling matrix. Because if $\hessFK$ only has a few non-zero rows, we have that $S_k \hessFK=0$ with positive probability, thus satisfying  \eqref{smallSketchedHessian}. Scaled sampling matrices also satisfy \eqref{skGradLowerBound_sparse_Hess} and \eqref{S_k_norm_bound_Sparse_hess} (See \autoref{lem:sampling:non-uniformity-BCGN} and \autoref{tmp-2022-1-14-12}).
\end{remark}



    \section{Convergence to second order (subspace) critical points}
    In this section, we show that 
\autoref{alg:CBGN} converges to a (subspace) second order critical point of
$f(x)$.
Our convergence aim here is
\begin{equation}
    \nEps = \nTwoEpsH = \min \set{k: \lambdaMin{S_k \hessFK S_k^T} \geq -\epH}
    \label{tmp-2021-12-31-6}
\end{equation}
And we define $\sMax$-true iterations as

\begin{definition} \label{def:true:cubic:subspaceHess}
Let $\sMax>0$. An iteration $k$ is true if $\normTwo{S_k} \leq \sMax$.
\end{definition}

Compared to Section \ref{CBGN:sec2}, here we have a less restrictive definition of true iterations. Consequently it is easy to show \autoref{AA2} is true.

\paragraph{Satisfying \autoref{AA2}}
For $S$ being a scaled Gaussian matrix, \autoref{Lem:GaussSMax} gives that \autoref{alg:CBGN} satisfies \autoref{AA2} with 
with any $\deltaS \in (0,1)$ and 
\begin{equation}
    \sMax = 1 + \sqrt{\frac{d}{l}} + \sqrt{\frac{2\logOneOverDeltaS}{l}}. \notag
\end{equation} Results for other random ensembles can be found in Section \ref{BCGN:randomMatrixDistr}.

\paragraph{Satisfying \autoref{AA3}}

\begin{lemma}
Let $f$ be twice continuously differentiable with $\LH$-Lipschitz continuous Hessian. 
\autoref{alg:CBGN} satisfies \autoref{AA3} with
\begin{equation}
    \alphaLow = \frac{2(1-\theta)}{\LH} \label{eq:alphaLow:CBGN:secondOrder}
\end{equation}

\end{lemma}

The proof is similar to \autoref{lem:cubic:sub:A2}, where the condition $\normTwo{S_k^T \sKHat} > 0$ on true iterations before convergence is ensured by  \autoref{lem:SubHessNegImpStepLower}.

\paragraph{Satisfying \autoref{AA4}}
We can calculate
\begin{equation}
    \grad^2 \mKHat{\sHat} = 
    S_k \hessFK S_k^T
    + \frac{1}{\alphaK} \squareBracket{
    \normTwo{S_k^T \sHat}^{-1} \bracket{
    S_k S_k^T \sHat} \bracket{S_kS_k^T \sHat}^T
    + \normTwo{S_k^T \sHat} S_k S_k^T}.
    \label{hessMkExpr}
\end{equation}
Therefore for any $y \in \R^l$, we have
\begin{equation}
    y^T \grad^2 \mKHat{\sHat} y 
    = y^T S_k \hessFK S_k^T y
    + \frac{1}{\alphaK} \squareBracket{
    \normTwo{S_k^T \sHat}^{-1} 
    \squareBracket{\bracket{
    S_kS_k^T \sKHat}^T y}^2 +
    \normTwo{S_k^T \sHat} \bracket{
    S_k^T y}^2}. \label{hessMkExprWithY}
\end{equation}
The following Lemma says that if the subspace Hessian has
negative curvature, then the step size is bounded below 
by the size of the negative curvature. (But also depends on $\alphaK$.)

\begin{lemma} \label{lem:SubHessNegImpStepLower}
If $\lambdaMin{S_k \hessFK S_k^T } < -\epH$; 
and $\normTwo{S_k} \leq \sMax$, 
then 
\begin{equation*}
    \normTwo{S_k^T \sKHat} \geq \epH \squareBracket{\frac{2\sMax^2}{\alphaK} + \kappaS}^{-1}.
\end{equation*}
\end{lemma}

\begin{proof}
Let $y \in \R^l$. Using \eqref{hessMKGeq0} and
\eqref{hessMkExprWithY} we have that
\begin{equation}
    y^T S_k \hessFK S_k^T y \geq
    - \frac{1}{\alphaK} \squareBracket{
    \NsKTSKHat^{-1} 
    \squareBracket{\bracket{
    S_kS_k^T \sKHat}^T y}^2 +
    \NsKTSKHat \bracket{
    S_k^T y}^2} - \kappaS\normTwo{S_k^T \sKHat}.  \notag
\end{equation}

Given $\normTwo{S_k} \leq \sMax$, 
we have that $S_k^T y \leq \sMax \normTwo{y}$. 
So we have
\begin{equation}
    y^T S_k \hessFK S_k^T y \geq
    - \frac{1}{\alphaK} \squareBracket{
    \NsKTSKHat^{-1} 
    \squareBracket{\bracket{
    S_kS_k^T \sKHat}^T y}^2 +
    \NsKTSKHat \sMax^2 \normTwo{y}^2} - \kappaS\normTwo{S_k^T \sKHat}. \notag
\end{equation} 

Minimising over $\normTwo{y}=1$, noting that
$\max_{\normTwo{y}=1} \bracket{\bracket{
S_k S_k^T \sKHat}^T y}^2 = 
\normTwo{S_kS_k^T \sKHat}^2$, we have
\begin{align}
    -\epH > \lambdaMin{S_k \hessFK S_k^T }
    & \geq -\frac{1}{\alphaK}\squareBracket{
    \NsKTSKHat^{-1} \normTwo{
    S_k S_k^T \sKHat}^2 + \NsKTSKHat
    \sMax^2} - \kappaS\normTwo{S_k^T \sKHat} \notag \\
    & \geq -\frac{1}{\alphaK}\squareBracket{
    \NsKTSKHat^{-1} \sMax^2 \NsKTSKHat^2
    + \NsKTSKHat \sMax^2} - \kappaS\normTwo{S_k^T \sKHat} \notag\\
    &= -\frac{2\sMax^2}{\alphaK}\NsKTSKHat - \kappaS\normTwo{S_k^T \sKHat}. \notag
\end{align}

Rearranging gives the result.
\end{proof}

\begin{lemma}
\autoref{alg:CBGN} satisfies \autoref{AA4} with
\begin{equation}
    h(\epsilon_H, \alphaK) = \frac{\theta\epH^3}{3\alphaK}\squareBracket{\frac{2\sMax^2}{\alphaK} + \kappaS}^{-3} \label{h:cubic:secondOrder:subspace}
\end{equation}

\end{lemma}

\begin{proof}
Using \autoref{succStepDecrease}, on successful iterations, we have 
$\fK - \fKPlusOne \geq \frac{\theta}{3\alphaK} \normTwo{S_k^T \sKHat}^3$.
Consequently, $k \leq \nEps$ (note the definition \eqref{tmp-2021-12-31-6} of $\nEps$ in this section)
and \autoref{lem:SubHessNegImpStepLower} give the lower bound $h$ \reply{that} holds
in true, successful and $k \leq \nEps$ iterations.
\end{proof}

\paragraph{Satisfying \autoref{AA5}}

\begin{lemma}
\autoref{alg:CBGN} satisfies \autoref{AA5}.
\end{lemma}

The proof of this lemma is identical to \autoref{tmp-2022-1-15-1}.

\paragraph{Convergence result}

Applying \autoref{thm2}, we have a convergence result for \autoref{alg:CBGN} to a point where the subspace Hessian has approximately non
negative curvature. While the statement is for scaled Gaussian matrices,
it is clear from the above proof that similar results apply
to a wide range of sketching matrices.

\begin{theorem}
    Let $\epH > 0, l \in \N^+ $.
    Let $\cal{S}$ be the distribution of scaled Gaussian matrices $S \in \R^{l \times d}$. 
    Suppose that $f$ is bounded below by $f^*$, twice continuously differentiable with $\LH$-Lipschitz continuous Hessian $\grad^2 f$.
    Choose $
    \deltaS = \frac{1}{16};
    $ so that
    $
    \sMax = 1 + 
    \frac{ \sqrt{d} + \sqrt{2\log16} }
    { \sqrt{l} }$. 
    Let $h$ be defined in \eqref{h:cubic:secondOrder:subspace}, $\alphaLow$ be given in \eqref{eq:alphaLow:CBGN:secondOrder} and
     $\alphaMin = \alphaZero \gammaOne^\newL$ associated with $\alphaLow$, for some $\newL \in \N^+$. 
	Suppose that $ \delta_S < \frac{c}{(c+1)^2}$.
	
	Then for any $\delta_1 \in (0,1)$ with 
    \begin{equation}
        \gDeltaSDeltaOne >0, \nonumber
    \end{equation}
    where 
    \begin{equation}
        \gDeltaSDeltaOne = \nPreFactorTR \nonumber
    \end{equation};
    if $N\in \N$ satisfies 
    \begin{equation}
        N \geq \gDeltaSDeltaOne \squareBracket{
             \frac{f(x_0) - f^*}{h(\epsilon_H, \alphaZero\gammaOne^{c+\newL}
)}
             + \frac{\newL}{1+c}}, \nonumber
    \end{equation}
    we have that
    \begin{equation}
        \probability{
            \min \{
                k: \lambdaMin{S_k^T \hessFK S_k}  \geq -\epH 
            \} \leq N 
            } \geq 1 - \chernoffLowerExponential. \nonumber
    \end{equation}
\end{theorem}

\begin{remark}
We see that the convergence rate to a (subspace) second order critical point
is $\epH^{-3}$.
\end{remark}

    \section{Convergence to second order (full space)
    critical points}
    In this section, we show that, if $\cal{S}$ is the distribution of scaled Gaussian matrices, \autoref{alg:CBGN} will converge to a (full-space) second order critical point, with a rate matching the classical full space algorithm.

We define 
\begin{equation}
    \nEps = \nThreeEpH 
= \min \set{k: \lambdaMin{ \hessFK } \geq -\epH } \label{eq:nEps:secondOrderFull}
\end{equation}

The following definition of true iterations assumes
that $\hessFK$ has rank $r\leq d$.

\begin{definition}\label{true:cubic:full:secondOrder}
Let $\sMax> 0, \epsOne \in (0,1)$. 
An iteration $k$ is $(\epsOne, \sMax)$-true if the following two conditions hold
\begin{enumerate}
    \item $\normTwo{S_k} \leq \sMax$. 
    \item There exists an eigen-decomposition of $\hessFK 
    = \sum_{i=1}^r \lambda_i u_i u_i^T$ with $\lambda_1 
    \geq \lambda_2 \geq \dots \geq \lambda_r$, such that
    with $w_i = S_k u_i$,
    \begin{align}
        & 1-\epsOne \leq \normTwo{w_r}^2 \leq 1+ \epsOne, \label{gauss:norm:one}\\
        & (w_i^T w_r)^2 \leq 16l^{-1} (1+\epsOne)
        \texteq{ for all $i \neq r$}. \label{gauss:approx:orthogonal}
    \end{align}
\end{enumerate}

\end{definition}

Note that since $S_k \in \R^{l \times d}$ is a (scaled) Gaussian matrix, 
and the set $\set{u_i}$ is orthonormal, 
we have that $\set{w_i}$ are independent Gaussian vectors, 
with entries being $N(0, l^{-1} )$. \eqref{gauss:approx:orthogonal} simply
requires that those high-dimensional Gaussian vectors are approximately 
orthogonal, which is known to happen with high probability \cite{MR3837109}.
We proceed to show that the four assumptions needed for \autoref{thm2} hold, and then apply \autoref{thm2} for this particular
definition of $\nEps$. 

\paragraph{Satisfying \autoref{AA2}}
As before, we show each of the two conditions in \autoref{true:cubic:full:secondOrder}
hold with high probability, and then use the union bound (See the proof of \autoref{lem::deduceAA2})
to show \autoref{AA2} is true. Note that the conditional independence between
iterations is clear here because given $x_k$, whether the iteration is true
or not only depends on the random matrix $S_k$ and is independent of all the previous iterations.

For the first condition in \autoref{true:cubic:full:secondOrder}, We have that for $S$ being a scaled Gaussian matrix, \autoref{Lem:GaussSMax} gives that \autoref{alg:CBGN} satisfies \autoref{AA2} with 
with any $\deltaSTwo \in (0,1)$ and 
\begin{equation}
    \sMax = 1 + \sqrt{\frac{d}{l}} + \sqrt{\frac{2\logOneOverDeltaSTwo}{l}}. \label{tmp-2022-1-14-1}
\end{equation}





\autoref{lem:WrWr} shows that \eqref{gauss:norm:one} holds with high probability.
\begin{lemma} \label{lem:WrWr}
Let $w_i \in \R^l$ with $w_{ij}$ be independent $N(0, l^{-1})$.
Let $\epsOne \in (0,1)$. 
Then we have for some problem-independent constant $C$,
\begin{equation}
    \probability{ \abs{\normTwo{w_i}^2-1 } \leq \epsOne }
    \geq 1 - 2\pow{e}{-\frac{l\epsOne^2}{C}}. 
    \label{w_r}
\end{equation}

\end{lemma}

\begin{proof}
The proof is standard. One side of the bound is established 
in \autoref{lem:GaussJLEmbedding}. Also see \cite{MR1943859}. We note that
$C \approx 4$.
\end{proof}

Next, we show that conditioning on \eqref{gauss:norm:one} being true, \eqref{gauss:approx:orthogonal} holds with high probability. We first study the case for a single fixed $i$ instead of all $i$.

\begin{lemma} \label{lem:WiWr}
Let $\epsOne \in (0,1)$ and suppose $w_r$ satisfies \eqref{w_r}.
Then with (conditional) probability at least 0.9999, 
independent of $w_r$,
we have $(w_i^T w_r)^2 \leq 16l^{-1} (1+\epsOne)$.
\end{lemma}

\begin{proof}
We have $(w_i^T w_r)^2 = 
\normTwo{w_r}^2 \bracket{
w_i^T \frac{w_r}{\normTwo{w_r}} }^2
$. The term inside the 
bracket is an $N(0, l^{-1})$ random variable independent 
of $w_r$, because sum of independent normal random variables
is still normal. Note that for a normal random variable $N(0, \sigma^2)$, 
with probability at least $0.9999$, its absolute value lies within
$\pm 4 \sigma$. Therefore we have that
with probability at least $0.9999$, 
\begin{equation}
    \bracket{
w_i^T \frac{w_r}{\normTwo{w_r}} }^2 \leq 
16 l^{-1}. 
\end{equation}

Combining with \eqref{w_r} gives the result. 
\end{proof}
\autoref{lem:wIwRallIneqR} shows that conditioning on \eqref{gauss:norm:one} being true, \eqref{gauss:approx:orthogonal} is true with high probability.
\begin{corollary}\label{lem:wIwRallIneqR}
With (conditional) probability at least $0.9999^{(r-1)}$, 
we have that 
$(w_i^T w_r)^2 \leq (1+\epsOne)16l^{-1}$ for all $i \neq r$.
\end{corollary}

\begin{proof}
Note that conditioning on $\normTwo{w_r}^2$,  
$w_i^T w_r$ are independent events. 
Therefore we simply multiply the probability. 
\end{proof}

The following Lemma shows that the second condition in \autoref{true:cubic:full:secondOrder} is true with high probability.

\begin{lemma}\label{lem:wrAndAllwi}
Let $\epsOne > 0$.
Let $A_1 = \set{\abs{\normTwo{w_r}^2-1}\leq \epsOne}$, \\
and $A_2 = \set{\bracket{w_i^T w_r}^2\leq 
16l^{-1}(1+\epsOne)}, \quad
\forall i\neq r$. 
\\ Then with probability at least $(0.9999)^{r-1}
\bracket{1-2e^{-\frac{l\epsOne^2}{C}}}$, 
we have that $A_1$ and $A_2$ hold simultaneously.
\end{lemma}

\begin{proof}
We have 
$\probability{A_1 \intersect A_2} 
= \conditionalP{A_2}{A_1} \probability{A_1}$.
Using \autoref{lem:WrWr}
and \autoref{lem:wIwRallIneqR} 
gives the result.  
\end{proof}

Therefore, using \eqref{tmp-2022-1-14-1}, \autoref{lem:wrAndAllwi} and the union bound we have the following

\begin{lemma}\label{tmp-2022-1-14-5}
Let $\epsOne > 0$, $l \in \N^+$, $\deltaSTwo > 0$ such that 
\begin{equation}
    \deltaS 
= (0.9999)^{r-1} \bracket{1 - 2e^{-\frac{l \epsOne^2}{C} } } + \deltaSTwo
< 1. \label{tmp-2022-1-14-4}
\end{equation}
Then \autoref{alg:CBGN} with $(\sMax, \epsOne)$-true iterations
defined in \autoref{true:cubic:full:secondOrder} satisfies \autoref{AA2} where
$\sMax = 1 + \sqrt{\frac{d}{l}} + \sqrt{\frac{2\logOneOverDeltaSTwo}{l}}$.
\end{lemma}

\paragraph{Satisfying \autoref{AA3}}

\begin{lemma}\label{tmp-2022-1-14-6}
Let $f$ be twice continuously differentiable with $\LH$-Lipschitz continuous Hessian. Then
\autoref{alg:CBGN} with true iterations defined in 
\autoref{true:cubic:full:secondOrder} satisfies \autoref{AA3} with
\begin{equation}
    \alphaLow = \frac{2(1-\theta)}{\LH} \label{eq:alphaLow:CBGN:secondOrder:1}
\end{equation}

\end{lemma}

The proof is identical to \autoref{lem:cubic:sub:A2}. \footnote{Except that we need
$\normTwo{S_k^T \sKHat} > 0$ in true iterations before convergence. But this is shown in \eqref{tmptmp}.}

\paragraph{Satisfying \autoref{AA4}}





\autoref{lem:lambdaMinSketchedHess} is a key ingredient. It shows that in true iterations, the subspace Hessian's negative curvature ($\lambdaMin{S_k \hessFK S_k^T}$) is proportional to the full Hessian's negative curvature ($\lambdaMin{\hessFK}$). 
\begin{lemma} \label{lem:lambdaMinSketchedHess}
Suppose iteration $k$ is true with $\epSOne \in (0,1)$ and $k < \nEps$. 
Let $\kappa_H = \min\set{0, \lambda_1/\lambda_r}$. Suppose 
\begin{equation}
    1 - \epsOne + 16\frac{r-1}{l}\frac{1+\epsOne}{1-\epsOne}\frac{\lambda_1}{\lambda_r} \geq 0. \label{tmp-2022-1-14-3}
\end{equation}
Then 
we have that
\begin{equation}
    \lambdaMin{S_k \hessFK S_k^T} \leq
    -\epH m(\epsOne, r, l, \kappa_H), \notag
\end{equation}

where
\begin{equation}
    m(\epsOne, r, l, \kappa_H) = 
    \bracket{1-\epsOne
    + 16\frac{r-1}{l}\frac{1+\epsOne}{1-\epsOne}\kappa_H
    }. \label{eq:mDef}
\end{equation}

\end{lemma}

\begin{proof}
Using the eigen-decomposition of $\hessFK$, we have that
$S_k \hessFK S_k^T = \sum_{i=1}^r \lambda_i w_i w_i^T$.
Use the Rayleigh quotient expression of minimal 
eigenvalue (with $w_r$ being the trial vector):
\begin{equation}
    \lambdaMin{S_k \hessFK S_k^T} \leq  
\frac{\sum_{i=1}^r \lambda_i \bracket{w_i^T w_r}^2}{
w_r^T w_r} \notag
\end{equation}

We have
\begin{align}
    & \frac{\sum_{i=1}^r \lambda_i \bracket{w_i^T w_r}^2}{
    w_r^T w_r} \notag\\
    & = \bracket{w_r^T w_r }\lambda_r +
    \frac{\sum_{i=1}^{r-1} \lambda_i \bracket{w_i^T w_r}^2}{
    w_r^T w_r} \notag\\
    & \leq \bracket{1-\epsOne}\lambda_r + 
    \lambda_1 \frac{\sum_{i=1}^{r-1} \bracket{w_i^T w_r}^2}{
    w_r^T w_r} \notag\\
    & \leq \bracket{1-\epsOne}\lambda_r + 
    16\frac{r-1}{l}\frac{1+\epsOne}{1-\epsOne}\lambda_1
    ,\label{tmp-2022-1-14-2}
\end{align}
where the two inequalities follow from \eqref{gauss:norm:one}
    and \eqref{gauss:approx:orthogonal} because iteration $k$ is true.
Next we discuss two cases. 

\begin{enumerate}
    \item If $\lambda_1 < 0$, then $\kappa_H = 0$ because $\lambda_r < -\epH < 0$. Thus, $m(\epsOne, r, l, \kappa_H) = 1-\epsOne$. The desired result follows from \eqref{tmp-2022-1-14-2} by noting that the second term $16\frac{r-1}{l}\frac{1+\epsOne}{1-\epsOne}\lambda_1 < 0$ and $\lambda_r < -\epH$.
    \item If $\lambda_r \geq 0$, then $\kappa_H = \frac{\lambda_1}{\lambda_r}$ and from \eqref{tmp-2022-1-14-2}, we have
        \begin{align*}
            \bracket{1-\epsOne}\lambda_r + 
            16\frac{r-1}{l}\frac{1+\epsOne}{1-\epsOne}\lambda_1
            & = \lambda_r \bracket{1 - \epsOne + 16\frac{r-1}{l}\frac{1+\epsOne}{1-\epsOne}\frac{\lambda_1}{\lambda_r}} \\
            & \leq -\epH  m(\epsOne, r, l, \kappa_H),
        \end{align*}
\end{enumerate}
where we used $\eqref{tmp-2022-1-14-3}$ and $\lambda_r < -\epH$ to derive the inequality. 
And the desired result follows. 
\end{proof}

\begin{remark}
\eqref{tmp-2022-1-14-3} always holds if $\lambda_1 \leq 0$ (where recall that $\lambda_i$ are eigen-values of $\hessFK$ and $k < \nEps$ implies $\lambda_r<-\epH < 0$.). If  $\lambda_1 >0$, then \eqref{tmp-2022-1-14-3} holds if we have $\kappa\bracket{\hessFK} \frac{r-1}{l} \leq \frac{(1-\epsOne)^2}{16(1+\epsOne)}$ where $\kappa\bracket{\hessFK} = \abs{\frac{\lambda_1}{\lambda_r}}$ is the condition number of $\hessFK$.
\end{remark}



We conclude that \autoref{AA4} is satisfied. 
\begin{lemma} \label{tmp-2022-1-14-7}
\autoref{alg:CBGN} with $\cal{S}$ being the distribution of scaled Gaussian
matrices, true iteration defined in \autoref{true:cubic:full:secondOrder}
and $\nEps$ defined in \eqref{eq:nEps:secondOrderFull} satisfies
\autoref{AA4} with
\begin{equation}
    h(\epH, \alphaK) = 
    \frac{\theta \epH^3 m(\epsOne, r, l, \kappa_H)^3}{3\alphaK} 
\squareBracket{\frac{2\sMax^2}{\alphaK} + \kappaS}^{-3}. 
     \label{h:cubic:secondOrder:fullspace}
\end{equation}
\end{lemma}

\begin{proof}
Let iteration $k$ be true and successful with $k < \nEps$. 
\autoref{lem:lambdaMinSketchedHess} gives that 
\begin{equation}
    \lambdaMin{S_k \hessFK S_k^T} \leq
    -\epH m(\epsOne, r, l, \kappa_H). \notag
\end{equation}
Then we have 
\begin{equation}
    \normTwo{S_k^T \sKHat}
    \geq 
    \epH m(\epsOne, r, l, \kappa_H)
    \squareBracket{\frac{2\sMax^2}{\alphaK} + \kappaS}^{-1}, \label{tmptmp}
\end{equation}
by applying \autoref{lem:SubHessNegImpStepLower} with
$\epH = \epH m(\epsOne, r, l, \kappa_H)$.
The desired result follows by applying \autoref{succStepDecrease}.
\end{proof}

\paragraph{Satisfying \autoref{AA5}}
The identical proof as the last section applies because \autoref{AA5} is not affected by the change of definitions of $\nEps$ and true iterations.


\paragraph{Convergence of \autoref{alg:CBGN} to a second order (full-space) critical point}
Applying \autoref{thm2}, the next theorem shows that using \autoref{alg:CBGN} with scaled Gaussian matrices
achieves convergence to a second order critical point, with a rate matching the classical full-space method. 

\begin{theorem}
In \autoref{alg:CBGN}, let $\cal{S}$ be the distribution of scaled Gaussian matrices. 
Let $\epH > 0$ and $\nEps$ be defined in \eqref{eq:nEps:secondOrderFull}.
Define true iterations in \autoref{true:cubic:full:secondOrder}. Suppose $f$ is lower bounded by $f^*$ and twice continuously differentiable with $\LH$-Lipschitz continuous Hessian $\grad^2 f$. 
    
    Choose $
    \deltaSTwo = \frac{1}{16};
    $ so that
    $
    \sMax = 1 + 
    \frac{ \sqrt{d} + \sqrt{2\log16} }
    { \sqrt{l} }$. 
    Let $\deltaS$ be defined in \eqref{tmp-2022-1-14-4}.
    Let $h$ be defined in \eqref{h:cubic:secondOrder:fullspace}, $\alphaLow$ be given in \eqref{eq:alphaLow:CBGN:secondOrder:1} and
     $\alphaMin = \alphaZero \gammaOne^\newL$ associated with $\alphaLow$ (See \autoref{lem::alphaMin}).
	Suppose that $ \delta_S <\frac{c}{(c+1)^2}$.
	Then for any $\delta_1 \in (0,1)$ with 
    \begin{equation}
        \gDeltaSDeltaOne >0, \nonumber
    \end{equation}
    where 
    \begin{equation}
        \gDeltaSDeltaOne = \nPreFactorTR \nonumber;
    \end{equation}
    if $N$ satisfies 
    \begin{equation}
        N \geq \gDeltaSDeltaOne \squareBracket{
             \frac{f(x_0) - f^*}{h(\epsilon_H, \alphaZero\gammaOne^{c+\newL}
)}
             + \frac{\newL}{1+c}}, \nonumber
    \end{equation}
    we have that
    \begin{equation}
        \probability{
            \min \{
                k: \lambdaMin{ \hessFK }  \geq -\epH 
            \} \leq N 
            } \geq 1 - \chernoffLowerExponential. \nonumber
    \end{equation}
\end{theorem}
\begin{proof}
Applying \autoref{tmp-2022-1-14-5}, \autoref{tmp-2022-1-14-6}, \autoref{tmp-2022-1-14-7}, we have that the four assumptions in \autoref{thm2} are satisfied. Then applying \autoref{thm2} gives the desired result.
\end{proof}



\chapter{Conclusion and future directions}
    In this thesis, we studied random embeddings and their application to optimisation problems and algorithms in order to achieve faster and more scalable solutions. 

After introducing the necessary background related to random embeddings --- Johnson-Lindenstrauss lemma, subspace and oblivious embeddings, and commonly used random ensembles --- we analysed the subspace embedding property of hashing embeddings when the matrix whose column space is to be embedded has low coherence. We found that 1-hashing embeddings achieve the same theoretical dimensionality reduction property as the scaled Gaussian matrices if the coherence of the data is sufficiently low. This result motivated us to propose a new type of general random subspace embeddings -- where the typically-used subsampling is replaced by hashing when combined with coherence-reducing transformations; this is the case of Subsampled- versus the novel Hashed-Randomised Hadamard Transform. Some open questions remain, that would be worthwhile pursuing and that would further enrich our understanding of this fascinating area of random matrices. For example, in our Theorem 2.3.1, we showed that 1-hashing matrices provide an oblivious subspace embedding of optimal size $m=\mathO{d}$
provided the input coherence is sufficiently low, of order $1/d$. Though the former, size requirement, cannot be improved in order, 
the latter, coherence one, probably can be improved to allow a larger class of input matrices to be embedded.

In chapter 3, we cascade our findings about sparse random matrices to the development of efficient solvers for large-scale linear least-squares problems, building on the success of the randomised Blendenpik algorithm and state-of-the-art numerical linear algebra techniques. We additionally present comprehensive benchmarking results of our proposed solver Ski-LLS against both random embedding-based, and deterministic, solvers. We found that our solver, SKi-LLS, which is available as an open source C++ code, outperforms not only  sketching-based solvers but also state-of-the-art deterministic sparse solvers on certain subsets of the Florida collection of large-scale sparse matrices. Future development of our solver Ski-LLS may include incorporation and testing of other sparse ensembles such as the stable 1-hashing proposed in \cite{CHEN2020105639} (see also Section 4.4.2.3). 

After considering reducing the dimensionality of the observational space in the linear least squares, we next turned to applying random embeddings to reduce the dimensionality of the variable/parameter space, leading to random subspace algorithmic variants of standard optimization algorithms for nonconvex problems. We showed that the $\mathO{\epsilon^{-2}}$ convergence rate of first-order-type methods to obtain an approximately small gradient value, within $\epsilon$, can be preserved, with high probability, when the gradient and the search direction are sketched/randomly projected.  
Various sketching matrices are allowed, of dimension independent of problem size, and can be used in a generic algorithmic framework that incorporates quadratic regularization and trust region variants. 
A current direction here is to particularise our general random subspace framework to linesearch methods, which in light of \cite{Cartis:2017fa}, is clearly possible, with similar complexity bounds being obtained.

When the second order information is also available, we investigated in Chapter 5, a Random subspace variant of Adaptive Cubic Regularization (R-ARC). We found that when the Hessian information is low rank, and Gaussian sketching is used to generate the subspace, the optimal complexity of order $\mathO{\epsilon^{-3/2}}$ of the full-space algorithm is recovered with high probability, for generating a sufficiently small gradient. The complexity of achieving approximate second order criticality using R-ARC is also addressed, with similar outcomes in relation to the complexity of the full space variant provided again, that low-rank assumptions hold for the curvature information. Our focus in Chapters 4 and 5 was theoretical, but it has informed us about the potential and strength of fixed-size random projections to generating suitable subspaces for minimization, and strongly convergent ensuing algorithmic variants. 
Future work  would be to numerically implement and test the more general variants (not just Gauss-Newton type) on large-scale general objectives, which would likely involve further, careful algorithm development. 



Finally, we see potential in the techniques in this thesis to apply to other problem classes in numerical analysis, either more directly or with further development, such as to  low rank matrix approximation, linear programming and nonlinear sum of functions arising in machine learning. Could we solve more large scale numerical analysis problems faster with random embeddings? Or perhaps we can find domain-tailored random embeddings for specific problems in machine learning, finance and other applications?



\addcontentsline{toc}{chapter}{Bibliography}
\bibliography{Reference.bib}        

\begin{thebibliography}{100}

\bibitem{MR2005771}
D.~Achlioptas.
\newblock Database-friendly random projections.
\newblock In {\em Proceedings of the twentieth ACM SIGMOD-SIGACT-SIGART
  symposium on Principles of database systems}, pages 274--281, 2001.

\bibitem{10.1145/1132516.1132597}
N.~Ailon and B.~Chazelle.
\newblock Approximate nearest neighbors and the fast {J}ohnson-{L}indenstrauss
  transform.
\newblock In {\em S{TOC}'06: {P}roceedings of the 38th {A}nnual {ACM}
  {S}ymposium on {T}heory of {C}omputing}, pages 557--563. ACM, New York, 2006.

\bibitem{10.1145/2483699.2483701}
N.~Ailon and E.~Liberty.
\newblock An almost optimal unrestricted fast {J}ohnson-{L}indenstrauss
  transform.
\newblock {\em ACM Trans. Algorithms}, 9(3):Art. 21, 1--12, 2013.

\bibitem{laug}
E.~Anderson, Z.~Bai, C.~Bischof, L.~S. Blackford, J.~Demmel, J.~Dongarra,
  J.~Du~Croz, A.~Greenbaum, S.~Hammarling, A.~McKenney, and D.~Sorensen.
\newblock {\em LAPACK Users' Guide}.
\newblock Society for Industrial and Applied Mathematics, third edition, 1999.

\bibitem{doi:10.1137/090767911}
H.~Avron, P.~Maymounkov, and S.~Toledo.
\newblock Blendenpik: supercharging {L}apack's least-squares solver.
\newblock {\em SIAM J. Sci. Comput.}, 32(3):1217--1236, 2010.

\bibitem{Avron:2009aa}
H.~Avron, E.~Ng, and S.~Toledo.
\newblock Using perturbed {$QR$} factorizations to solve linear least-squares
  problems.
\newblock {\em SIAM J. Matrix Anal. Appl.}, 31(2):674--693, 2009.

\bibitem{Bach2011}
F.~Bach, R.~Jenatton, J.~Mairal, and G.~Obozinski.
\newblock {\em Optimization with Sparsity-Inducing Penalties}, volume 4:1.
\newblock Foundations and Trends in Machine Learning, 2011.

\bibitem{Berahas2020}
A.~S. Berahas, R.~Bollapragada, and J.~Nocedal.
\newblock An investigation of {N}ewton-{S}ketch and subsampled {N}ewton
  methods.
\newblock {\em Optimization Methods and Software}, 2020.

\bibitem{Bjorck:1996uz}
A.~Bj\"{o}rck.
\newblock {\em Numerical methods for least squares problems}.
\newblock Society for Industrial and Applied Mathematics (SIAM), Philadelphia,
  PA, 1996.

\bibitem{Bourgain:2015tc}
J.~Bourgain, S.~Dirksen, and J.~Nelson.
\newblock Toward a unified theory of sparse dimensionality reduction in
  {E}uclidean space.
\newblock {\em Geom. Funct. Anal.}, 25(4):1009--1088, 2015.

\bibitem{MR4163541_Carmon}
Y.~Carmon, J.~C. Duchi, O.~Hinder, and A.~Sidford.
\newblock Lower bounds for finding stationary points {I}.
\newblock {\em Math. Program.}, 184(1-2, Ser. A):71--120, 2020.

\bibitem{2021arXiv210511815C}
C.~{Cartis}, J.~{Fiala}, and Z.~{Shao}.
\newblock {Hashing embeddings of optimal dimension, with applications to linear
  least squares}.
\newblock {\em arXiv e-prints}, page arXiv:2105.11815, May 2021.

\bibitem{BCGNPaper}
C.~{Cartis}, J.~{Fiala}, and Z.~{Shao}.
\newblock {Randomised subspace methods for non-convex optimization, with
  applications to nonlinear least-squares}.
\newblock {\em arXiv e-prints, in preparation}, 2022.

\bibitem{zhen:icml_BCGN}
C.~Cartis, J.~Fowkes, and Z.~Shao.
\newblock A randomised subspace gauss-newton method for nonlinear
  least-squares.
\newblock In {\em Thirty-seventh International Conference on Machine Learning},
  2020.
\newblock In Workshop on Beyond First Order Methods in ML Systems.

\bibitem{Cartis:2009fq}
C.~Cartis, N.~I. Gould, and P.~L. Toint.
\newblock Adaptive cubic regularisation methods for unconstrained optimization.
  part i: motivation, convergence and numerical results.
\newblock {\em Mathematical Programming}, 127(2):245--295, 2011.

\bibitem{CoraBook}
C.~Cartis, N.~I.~M. Gould, and P.~L. Toint.
\newblock {\em Evaluation complexity of algorithms for nonconvex optimization}.
\newblock MOS-SIAM series on Optimization. Society for Industrial and Applied
  Mathematics (SIAM), 2022.

\bibitem{Cartis:2017fa}
C.~Cartis and K.~Scheinberg.
\newblock Global convergence rate analysis of unconstrained optimization
  methods based on probabilistic models.
\newblock {\em Mathematical Programming}, 169(2):337--375, 2018.

\bibitem{CERDAN2020112621}
J.~Cerd\'{a}n, D.~Guerrero, J.~Mar\'{\i}n, and J.~Mas.
\newblock Preconditioners for rank deficient least squares problems.
\newblock {\em J. Comput. Appl. Math.}, 372:112621, 2020.

\bibitem{CHEN2020105639}
L.~Chen, S.~Zhou, and J.~Ma.
\newblock Stable sparse subspace embedding for dimensionality reduction.
\newblock {\em Knowledge-Based Systems}, 195:105639, 2020.

\bibitem{MR57518}
H.~Chernoff.
\newblock A measure of asymptotic efficiency for tests of a hypothesis based on
  the sum of observations.
\newblock {\em Ann. Math. Statistics}, 23:493--507, 1952.

\bibitem{10.1145/3019134}
K.~L. Clarkson and D.~P. Woodruff.
\newblock Low-rank approximation and regression in input sparsity time.
\newblock {\em J. ACM}, 63(6):Art. 54, 1--45, 2017.

\bibitem{10.5555/2884435.2884456}
M.~B. Cohen.
\newblock Nearly tight oblivious subspace embeddings by trace inequalities.
\newblock In {\em Proceedings of the {T}wenty-{S}eventh {A}nnual {ACM}-{SIAM}
  {S}ymposium on {D}iscrete {A}lgorithms}, pages 278--287. ACM, New York, 2016.

\bibitem{MR3773205}
M.~B. Cohen, T.~S. Jayram, and J.~Nelson.
\newblock Simple analyses of the sparse {J}ohnson-{L}indenstrauss transform.
\newblock In {\em 1st {S}ymposium on {S}implicity in {A}lgorithms}, volume~61
  of {\em OASIcs OpenAccess Ser. Inform.}, pages Art. No. 15, 9. Schloss
  Dagstuhl. Leibniz-Zent. Inform., Wadern, 2018.

\bibitem{10.1145/3219819.3220098}
Y.~Dahiya, D.~Konomis, and D.~P. Woodruff.
\newblock An empirical evaluation of sketching for numerical linear algebra.
\newblock In {\em Proceedings of the 24th ACM SIGKDD International Conference
  on Knowledge Discovery \& Data Mining}, KDD '18, pages 1292--1300, New York,
  NY, USA, 2018. Association for Computing Machinery.

\bibitem{MR1943859}
S.~Dasgupta and A.~Gupta.
\newblock An elementary proof of a theorem of {J}ohnson and {L}indenstrauss.
\newblock {\em Random Structures Algorithms}, 22(1):60--65, 2003.

\bibitem{MR1863696}
K.~R. Davidson and S.~J. Szarek.
\newblock Local operator theory, random matrices and {B}anach spaces.
\newblock In {\em Handbook of the geometry of {B}anach spaces, {V}ol. {I}},
  pages 317--366. North-Holland, Amsterdam, 2001.

\bibitem{MR2270673}
T.~A. Davis.
\newblock {\em Direct methods for sparse linear systems}, volume~2 of {\em
  Fundamentals of Algorithms}.
\newblock Society for Industrial and Applied Mathematics (SIAM), Philadelphia,
  PA, 2006.

\bibitem{10.1145/2049662.2049670}
T.~A. Davis.
\newblock Algorithm 915, {S}uite{S}parse{QR}: multifrontal multithreaded
  rank-revealing sparse {QR} factorization.
\newblock {\em ACM Trans. Math. Software}, 38(1):Art. 1, 1--22, 2011.

\bibitem{10.1145/2049662.2049663}
T.~A. Davis and Y.~Hu.
\newblock The {U}niversity of {F}lorida sparse matrix collection.
\newblock {\em ACM Trans. Math. Software}, 38(1):Art. 1, 1--25, 2011.

\bibitem{dolan2002benchmarking}
E.~D. Dolan and J.~J. Mor{\'e}.
\newblock Benchmarking optimization software with performance profiles.
\newblock {\em Mathematical programming}, 91(2):201--213, 2002.

\bibitem{DonohoMutualCoherence}
D.~L. Donoho and M.~Elad.
\newblock Optimally sparse representation in general (nonorthogonal)
  dictionaries via {$l^1$} minimization.
\newblock {\em Proc. Natl. Acad. Sci. USA}, 100(5):2197--2202, 2003.

\bibitem{10.5555/1109557.1109682}
P.~Drineas, M.~W. Mahoney, and S.~Muthukrishnan.
\newblock Sampling algorithms for l2 regression and applications.
\newblock In {\em Proceedings of the Seventeenth Annual ACM-SIAM Symposium on
  Discrete Algorithm}, SODA '06, page 1127–1136, USA, 2006. Society for
  Industrial and Applied Mathematics.

\bibitem{Facchinei2015}
F.~Facchinei, G.~Scutari, and S.~Sagratella.
\newblock Parallel selective algorithms for nonconvex big data optimization.
\newblock {\em IEEE Transactions on Signal Processing}, 63(7):1874--1889, 2015.

\bibitem{LSMR_2011}
D.~C.-L. Fong and M.~Saunders.
\newblock L{SMR}: An iterative algorithm for sparse least-squares problems.
\newblock {\em SIAM J. Sci. Comput.}, 33(5):2950--2971, 2011.

\bibitem{10.5555/3327345.3327444}
C.~Freksen, L.~Kamma, and K.~G. Larsen.
\newblock Fully understanding the hashing trick.
\newblock In {\em Proceedings of the 32nd International Conference on Neural
  Information Processing Systems}, NIPS'18, pages 5394--5404, Red Hook, NY,
  USA, 2018. Curran Associates Inc.

\bibitem{gnanasekaran2021hierarchical}
A.~{Gnanasekaran} and E.~{Darve}.
\newblock {Hierarchical Orthogonal Factorization: Sparse Least Squares
  Problems}.
\newblock {\em arXiv e-prints}, page arXiv:2102.09878, Feb. 2021.

\bibitem{10.5555/248979}
G.~H. Golub and C.~F. Van~Loan.
\newblock {\em Matrix computations}.
\newblock Johns Hopkins Studies in the Mathematical Sciences. Johns Hopkins
  University Press, Baltimore, MD, third edition, 1996.

\bibitem{Gould:2016vg}
N.~Gould and J.~Scott.
\newblock The state-of-the-art of preconditioners for sparse linear
  least-squares problems: the complete results.
\newblock Technical report, STFC Rutherford Appleton Laboratory, 2015.
\newblock Available at ftp://cuter.rl.ac.uk/pub/nimg/pubs/GoulScot16b_toms.pdf.

\bibitem{10.1145/3014057}
N.~Gould and J.~Scott.
\newblock The state-of-the-art of preconditioners for sparse linear
  least-square problems.
\newblock {\em ACM Trans. Math. Software}, 43(4):Art. 36, 1--35, 2017.

\bibitem{gould2015cutest}
N.~I. Gould, D.~Orban, and P.~L. Toint.
\newblock {CUTEst}: a constrained and unconstrained testing environment with
  safe threads for mathematical optimization.
\newblock {\em Computational Optimization and Applications}, 60(3):545--557,
  2015.

\bibitem{Gower2016a}
R.~Gower, D.~Goldfarb, and P.~Richt{\'{a}}rik.
\newblock Stochastic block {BFGS}: Squeezing more curvature out of data.
\newblock In M.~F. Balcan and K.~Q. Weinberger, editors, {\em Proceedings of
  The 33rd International Conference on Machine Learning}, volume~48 of {\em
  Proceedings of Machine Learning Research}, pages 1869--1878, New York, 2016.
  PMLR.

\bibitem{gower2019rsn}
R.~Gower, D.~Koralev, F.~Lieder, and P.~Richt{\'a}rik.
\newblock {RSN}: Randomized subspace {N}ewton.
\newblock In {\em Advances in Neural Information Processing Systems}, pages
  614--623, 2019.

\bibitem{RichtarikAndGowerLinearSystem}
R.~M. Gower and P.~Richt\'{a}rik.
\newblock Randomized iterative methods for linear systems.
\newblock {\em SIAM J. Matrix Anal. Appl.}, 36(4):1660--1690, 2015.

\bibitem{Gower2020}
R.~M. Gower, P.~Richt{\'{a}}rik, and F.~Bach.
\newblock Stochastic quasi-gradient methods: variance reduction via {J}acobian
  sketching.
\newblock {\em Mathematical Programming}, 2020.

\bibitem{Gratton:2017kz}
S.~Gratton, C.~W. Royer, L.~N. Vicente, and Z.~Zhang.
\newblock Complexity and global rates of trust-region methods based on
  probabilistic models.
\newblock {\em IMA Journal of Numerical Analysis}, 38(3):1579--1597, 2018.

\bibitem{gratton2008recursive}
S.~Gratton, A.~Sartenaer, and P.~L. Toint.
\newblock Recursive trust-region methods for multiscale nonlinear optimization.
\newblock {\em SIAM Journal on Optimization}, 19(1):414--444, 2008.

\bibitem{Griewank}
A.~Griewank.
\newblock The modification of newton’s method for unconstrained optimization
  by bounding cubic terms.
\newblock Technical report, Technical report NA/12, 1981.

\bibitem{Grishchenko2021}
D.~Grishchenko, F.~Iutzeler, and J.~Malick.
\newblock Proximal gradient methods with adaptive subspace sampling.
\newblock {\em Mathematics of Operations Research}, 2021.

\bibitem{inproceedings}
G.~W. Howell and M.~Baboulin.
\newblock Iterative solution of sparse linear least squares using lu
  factorization.
\newblock In {\em Proceedings of the International Conference on High
  Performance Computing in Asia-Pacific Region}, HPC Asia 2018, pages 47--53,
  New York, NY, USA, 2018. Association for Computing Machinery.

\bibitem{MR1715608}
P.~Indyk and R.~Motwani.
\newblock Approximate nearest neighbors: towards removing the curse of
  dimensionality.
\newblock In {\em S{TOC} '98 ({D}allas, {TX})}, pages 604--613. ACM, New York,
  1999.

\bibitem{iwen2020lower}
M.~A. Iwen, D.~Needell, E.~Rebrova, and A.~Zare.
\newblock Lower {M}emory {O}blivious ({T}ensor) {S}ubspace {E}mbeddings with
  {F}ewer {R}andom {B}its: {M}odewise {M}ethods for {L}east {S}quares.
\newblock {\em SIAM J. Matrix Anal. Appl.}, 42(1):376--416, 2021.

\bibitem{IYER2019100547}
C.~Iyer, H.~Avron, G.~Kollias, Y.~Ineichen, C.~Carothers, and P.~Drineas.
\newblock A randomized least squares solver for terabyte-sized dense
  overdetermined systems.
\newblock {\em J. Comput. Sci.}, 36:100547, 2019.

\bibitem{10.5555/3019094.3019103}
C.~{Iyer}, C.~{Carothers}, and P.~{Drineas}.
\newblock Randomized sketching for large-scale sparse ridge regression
  problems.
\newblock In {\em 2016 7th Workshop on Latest Advances in Scalable Algorithms
  for Large-Scale Systems (ScalA)}, pages 65--72, 2016.

\bibitem{NIPS2019_9656}
M.~Jagadeesan.
\newblock Understanding sparse {JL} for feature hashing.
\newblock In {\em Advances in Neural Information Processing Systems},
  volume~32. Curran Associates, Inc., 2019.

\bibitem{ChiJin}
C.~Jin, R.~Ge, P.~Netrapalli, S.~M. Kakade, and M.~I. Jordan.
\newblock How to escape saddle points efficiently.
\newblock In {\em International Conference on Machine Learning}, pages
  1724--1732. PMLR, 2017.

\bibitem{Johnson:1984aa}
W.~B. Johnson and J.~Lindenstrauss.
\newblock Extensions of {L}ipschitz mappings into a {H}ilbert space.
\newblock In {\em Conference in modern analysis and probability ({N}ew {H}aven,
  {C}onn., 1982)}, volume~26 of {\em Contemp. Math.}, pages 189--206. Amer.
  Math. Soc., Providence, RI, 1984.

\bibitem{kahale2020leastsquares}
N.~{Kahale}.
\newblock {Least-squares regressions via randomized Hessians}.
\newblock {\em arXiv e-prints}, page arXiv:2006.01017, June 2020.

\bibitem{MR3167920}
D.~M. Kane and J.~Nelson.
\newblock Sparser {J}ohnson-{L}indenstrauss transforms.
\newblock {\em J. ACM}, 61(1):Art. 4, 23, 2014.

\bibitem{KohlLucc17}
J.~M. Kohler and A.~Lucchi.
\newblock Sub-sampled cubic regularization for non-convex optimization.
\newblock {\em arXiv e-prints}, May 2017.

\bibitem{kozak2019stochastic}
D.~Kozak, S.~Becker, A.~Doostan, and L.~Tenorio.
\newblock Stochastic subspace descent.
\newblock {\em arXiv preprint arXiv:1904.01145}, 2019.

\bibitem{RePEc:spr:coopap:v:79:y:2021:i:2:d:10.1007_s10589-021-00271-w}
D.~Kozak, S.~Becker, A.~Doostan, and L.~Tenorio.
\newblock {A stochastic subspace approach to gradient-free optimization in high
  dimensions}.
\newblock {\em Computational Optimization and Applications}, 79(2):339--368,
  June 2021.

\bibitem{lacotte2019faster}
J.~{Lacotte} and M.~{Pilanci}.
\newblock {Faster Least Squares Optimization}.
\newblock {\em arXiv e-prints}, page arXiv:1911.02675, Nov. 2019.

\bibitem{Lacotte2020}
J.~Lacotte and M.~Pilanci.
\newblock Effective dimension adaptive sketching methods for faster regularized
  least-squares optimization.
\newblock In H.~Larochelle, M.~Ranzato, R.~Hadsell, M.~F. Balcan, and H.~Lin,
  editors, {\em Advances in Neural Information Processing Systems}, volume~33,
  pages 19377--19387. Curran Associates, Inc., 2020.

\bibitem{lacotte2020optimal}
J.~{Lacotte} and M.~{Pilanci}.
\newblock {Optimal Randomized First-Order Methods for Least-Squares Problems}.
\newblock {\em arXiv e-prints}, page arXiv:2002.09488, Feb. 2020.

\bibitem{Lacotte2019}
J.~Lacotte, M.~Pilanci, and M.~Pavone.
\newblock High-dimensional optimization in adaptive random subspaces.
\newblock In H.~Wallach, H.~Larochelle, A.~Beygelzimer, F.~d\textquotesingle
  Alch\'{e}-Buc, E.~Fox, and R.~Garnett, editors, {\em Advances in Neural
  Information Processing Systems}, volume~32. Curran Associates, Inc., 2019.

\bibitem{BenRecht}
J.~D. Lee, M.~Simchowitz, M.~I. Jordan, and B.~Recht.
\newblock Gradient descent only converges to minimizers.
\newblock In {\em Conference on learning theory}, pages 1246--1257. PMLR, 2016.

\bibitem{liu2021extending}
S.~{Liu}, T.~{Liu}, A.~{Vakilian}, Y.~{Wan}, and D.~P. {Woodruff}.
\newblock {Extending and Improving Learned CountSketch}.
\newblock {\em arXiv e-prints}, page arXiv:2007.09890, July 2020.

\bibitem{GlobalOptimisationReview}
M.~Locatelli and F.~Schoen.
\newblock {\em Global Optimization}.
\newblock Society for Industrial and Applied Mathematics, Philadelphia, PA,
  2013.

\bibitem{LoizouEtAl}
N.~Loizou and P.~Richt\'{a}rik.
\newblock Momentum and stochastic momentum for stochastic gradient, {N}ewton,
  proximal point and subspace descent methods.
\newblock {\em Comput. Optim. Appl.}, 77(3):653--710, 2020.

\bibitem{lopes2018error}
M.~Lopes, S.~Wang, and M.~Mahoney.
\newblock Error estimation for randomized least-squares algorithms via the
  bootstrap.
\newblock In J.~Dy and A.~Krause, editors, {\em Proceedings of the 35th
  International Conference on Machine Learning}, volume~80 of {\em Proceedings
  of Machine Learning Research}, pages 3217--3226. PMLR, 10--15 Jul 2018.

\bibitem{Lu2018}
Z.~Lu and L.~Xiao.
\newblock A randomized nonmonotone block proximal gradient method for a class
  of structured nonlinear programming.
\newblock {\em SIAM Journal on Numerical Analysis}, 55(6):2930--2955, 2017.

\bibitem{luo2016efficient}
H.~Luo, A.~Agarwal, N.~Cesa-Bianchi, and J.~Langford.
\newblock Efficient second order online learning by sketching.
\newblock In {\em Advances in Neural Information Processing Systems}, pages
  902--910, 2016.

\bibitem{10.1561/2200000035}
M.~W. Mahoney.
\newblock Randomized algorithms for matrices and data.
\newblock {\em Found. Trends Mach. Learn.}, 3(2):123--224, Feb. 2011.

\bibitem{MR3839333}
M.~W. Mahoney, J.~C. Duchi, and A.~C. Gilbert, editors.
\newblock {\em The mathematics of data}, volume~25 of {\em IAS/Park City
  Mathematics Series}.
\newblock American Mathematical Society, Providence, RI; Institute for Advanced
  Study (IAS), Princeton, NJ, 2018.
\newblock Papers based on the lectures presented at the 26th Annual Park City
  Mathematics Institute Summer Session, July 2016.

\bibitem{martinsson2015blocked}
P.-G. Martinsson.
\newblock {Blocked rank-revealing QR factorizations: How randomized sampling
  can be used to avoid single-vector pivoting}.
\newblock {\em arXiv e-prints}, page arXiv:1505.08115, May 2015.

\bibitem{Martinsson:2017eh}
P.-G. Martinsson, G.~Quintana~Ort\'{\i}, N.~Heavner, and R.~van~de Geijn.
\newblock Householder {QR} factorization with randomization for column pivoting
  ({HQRRP}).
\newblock {\em SIAM J. Sci. Comput.}, 39(2):C96--C115, 2017.

\bibitem{10.1145/2488608.2488621}
X.~Meng and M.~W. Mahoney.
\newblock Low-distortion subspace embeddings in input-sparsity time and
  applications to robust linear regression.
\newblock In {\em S{TOC}'13---{P}roceedings of the 2013 {ACM} {S}ymposium on
  {T}heory of {C}omputing}, pages 91--100. ACM, New York, 2013.

\bibitem{Meng:2014ib}
X.~Meng, M.~A. Saunders, and M.~W. Mahoney.
\newblock L{SRN}: a parallel iterative solver for strongly over- or
  underdetermined systems.
\newblock {\em SIAM J. Sci. Comput.}, 36(2):C95--C118, 2014.

\bibitem{Nelson:te}
J.~Nelson and H.~L. Nguyen.
\newblock O{SNAP}: faster numerical linear algebra algorithms via sparser
  subspace embeddings.
\newblock In {\em 2013 {IEEE} 54th {A}nnual {S}ymposium on {F}oundations of
  {C}omputer {S}cience---{FOCS} 2013}, pages 117--126. IEEE Computer Soc., Los
  Alamitos, CA, 2013.

\bibitem{10.1145/2488608.2488622}
J.~Nelson and H.~L. Nguyen.
\newblock Sparsity lower bounds for dimensionality reducing maps.
\newblock In {\em S{TOC}'13---{P}roceedings of the 2013 {ACM} {S}ymposium on
  {T}heory of {C}omputing}, pages 101--110. ACM, New York, 2013.

\bibitem{Nelson:2014uu}
J.~Nelson and H.~L. Nguyen.
\newblock Lower bounds for oblivious subspace embeddings.
\newblock In {\em Automata, languages, and programming. {P}art {I}}, volume
  8572 of {\em Lecture Notes in Comput. Sci.}, pages 883--894. Springer,
  Heidelberg, 2014.

\bibitem{MR2968857_Nesterov}
Y.~Nesterov.
\newblock Efficiency of coordinate descent methods on huge-scale optimization
  problems.
\newblock {\em SIAM J. Optim.}, 22(2):341--362, 2012.

\bibitem{NesterovTextBook}
Y.~Nesterov.
\newblock {\em Lectures on convex optimization}, volume 137 of {\em Springer
  Optimization and Its Applications}.
\newblock Springer, Cham, 2018.
\newblock Second edition of [ MR2142598].

\bibitem{NesterovAndPolyak}
Y.~Nesterov and B.~T. Polyak.
\newblock Cubic regularization of {N}ewton method and its global performance.
\newblock {\em Math. Program.}, 108(1, Ser. A):177--205, 2006.

\bibitem{Nocedal:2006uv}
J.~Nocedal and S.~J. Wright.
\newblock {\em Numerical optimization}.
\newblock Springer Series in Operations Research and Financial Engineering.
  Springer, New York, second edition, 2006.

\bibitem{10.1145/355984.355989}
C.~C. Paige and M.~A. Saunders.
\newblock L{SQR}: an algorithm for sparse linear equations and sparse least
  squares.
\newblock {\em ACM Trans. Math. Software}, 8(1):43--71, 1982.

\bibitem{Patrascu2015}
A.~Patrascu and I.~Necoara.
\newblock Efficient random coordinate descent algorithms for large-scale
  structured nonconvex optimization.
\newblock {\em Journal of Global Optimization}, 61(1):19--46, 2015.

\bibitem{pilanci2017newton}
M.~Pilanci and M.~J. Wainwright.
\newblock Newton sketch: A near linear-time optimization algorithm with
  linear-quadratic convergence.
\newblock {\em SIAM Journal on Optimization}, 27(1):205--245, 2017.

\bibitem{MR321541}
M.~J.~D. Powell.
\newblock On search directions for minimization algorithms.
\newblock {\em Math. Programming}, 4:193--201, 1973.

\bibitem{Richtarik2015}
P.~Richt{\'{a}}rik and M.~Tak{\'{a}}{\v{c}}.
\newblock Parallel coordinate descent methods for big data optimization.
\newblock {\em Mathematical Programming}, 156:433--484, 2015.

\bibitem{MR3179953_Richtarik}
P.~Richt\'{a}rik and M.~Tak\'{a}\v{c}.
\newblock Iteration complexity of randomized block-coordinate descent methods
  for minimizing a composite function.
\newblock {\em Math. Program.}, 144(1-2, Ser. A):1--38, 2014.

\bibitem{Rokhlin:2008wb}
V.~Rokhlin and M.~Tygert.
\newblock A fast randomized algorithm for overdetermined linear least-squares
  regression.
\newblock {\em Proc. Natl. Acad. Sci. USA}, 105(36):13212--13217, 2008.

\bibitem{10.1109/FOCS.2006.37}
T.~{Sarlos}.
\newblock Improved approximation algorithms for large matrices via random
  projections.
\newblock In {\em 2006 47th Annual IEEE Symposium on Foundations of Computer
  Science (FOCS'06)}, pages 143--152, 2006.

\bibitem{10.1145/2617555}
J.~Scott and M.~T\r{u}ma.
\newblock {HSL_MI}28: An efficient and robust limited-memory incomplete
  cholesky factorization code.
\newblock {\em ACM Trans. Math. Softw.}, 40(4):Art. 36, 1--35, July 2014.

\bibitem{zhen:icml_LLS}
Z.~Shao, C.~Cartis, and F.~Jan.
\newblock A randomised subspace gauss-newton method for nonlinear
  least-squares.
\newblock In {\em Thirty-seventh International Conference on Machine Learning},
  2020.
\newblock In Workshop on Beyond First Order Methods in ML Systems.

\bibitem{tett2017calibrating}
S.~F. Tett, K.~Yamazaki, M.~J. Mineter, C.~Cartis, and N.~Eizenberg.
\newblock Calibrating climate models using inverse methods: case studies with
  {HadAM3, HadAM3P and HadCM3}.
\newblock {\em Geoscientific Model Development}, 10:3567–3589, 2017.

\bibitem{Tropp:wr}
J.~A. Tropp.
\newblock Improved analysis of the subsampled randomized {H}adamard transform.
\newblock {\em Adv. Adapt. Data Anal.}, 3(1-2):115--126, 2011.

\bibitem{MR2946459}
J.~A. Tropp.
\newblock User-friendly tail bounds for sums of random matrices.
\newblock {\em Found. Comput. Math.}, 12(4):389--434, 2012.

\bibitem{OpenML2013}
J.~Vanschoren, J.~N. van Rijn, B.~Bischl, and L.~Torgo.
\newblock {OpenML}: networked science in machine learning.
\newblock {\em SIGKDD Explorations}, 15(2):49--60, 2013.

\bibitem{MR3837109}
R.~Vershynin.
\newblock {\em High-dimensional probability}, volume~47 of {\em Cambridge
  Series in Statistical and Probabilistic Mathematics}.
\newblock Cambridge University Press, Cambridge, 2018.
\newblock An introduction with applications in data science, With a foreword by
  Sara van de Geer.

\bibitem{10.1561/0400000060}
D.~P. Woodruff.
\newblock Sketching as a tool for numerical linear algebra.
\newblock {\em Found. Trends Theor. Comput. Sci.}, 10(1-2):1--157, 2014.

\bibitem{Wright2015}
S.~J. Wright.
\newblock Coordinate descent algorithms.
\newblock {\em Mathematical Programming}, 151:3--–34, 2015.

\bibitem{XuRoosMaho17}
P.~Xu, F.~Roosta-Khorasan, and M.~W. Mahoney.
\newblock {N}ewton-type methods for non-convex optimization under inexact
  {H}essian information.
\newblock {\em arXiv e-prints}, Aug 2017.

\bibitem{XuRoosMaho18}
P.~Xu, F.~Roosta-Khorasan, and M.~W. Mahoney.
\newblock Second-order optimization for non-convex machine learning: {A}n
  empirical study.
\newblock {\em arXiv e-prints}, Aug 2017.

\bibitem{Xu2015}
Y.~Xu and W.~Yin.
\newblock Block stochastic gradient iteration for convex and nonconvex
  optimization.
\newblock {\em SIAM Journal on Optimization}, 25(3):1686--1716, 2015.

\bibitem{Xu2017}
Y.~Xu and W.~Yin.
\newblock A globally convergent algorithm for nonconvex optimization based on
  block coordinate update.
\newblock {\em Journal of Scientific Computing}, 72(2):700--734, 2017.

\bibitem{Yang2020}
Y.~Yang, M.~Pesavento, Z.-Q. Luo, and B.~Ottersten.
\newblock Inexact block coordinate descent algorithms for nonsmooth nonconvex
  optimization.
\newblock {\em IEEE Transactions on Signal Processing}, 68:947--961, 2020.

\bibitem{YaoXuRoosMaho18}
Z.~Yao, P.~Xu, F.~Roosta-Khorasan, and M.~W. Mahoney.
\newblock Inexact non-convex {N}ewton-type methods.
\newblock {\em arXiv e-prints}, Feb 2018.

\bibitem{snobfit_URL}
{Z. Fu}.
\newblock Package {snobfit}.
\newblock \url{http://reflectometry.org/danse/docs/snobfit}, 2009.

\bibitem{zhu2018gradientbased}
R.~Zhu.
\newblock Gradient-based sampling: An adaptive importance sampling for
  least-squares.
\newblock In {\em Proceedings of the 30th International Conference on Neural
  Information Processing Systems}, NIPS'16, pages 406--414, Red Hook, NY, USA,
  2016. Curran Associates Inc.

\end{thebibliography}
\bibliographystyle{abbrv}  

\end{document}